\documentclass[11pt]{article}
\usepackage{amssymb, amsmath, amsthm, mathrsfs, ascmac, color}
\usepackage{graphicx}
\usepackage[subrefformat=parens]{subcaption}
\usepackage[top=3cm, bottom=3cm, left=3.25cm,right=3.25cm]{geometry}
\usepackage{comment}
\usepackage{tikz}
\usepackage[hidelinks]{hyperref}
\theoremstyle{plain}
\newtheorem{thm}{Theorem}[section]
\newtheorem{prop}[thm]{Proposition}
\newtheorem{lem}{Lemma}
\newtheorem{subthm}{Lemma}

\newtheorem{rmk}{Remark}
\makeatletter

\@addtoreset{equation}{section}

\makeatother
\newcommand{\ind}{\bs 1}
\newcommand{\dd}{\mathrm d}
\newcommand{\pd}{\partial}
\newcommand{\ee}{\mathrm e}
\newcommand{\EE}{\mathsf{E}}
\newcommand{\PP}{\mathsf{P}}
\newcommand{\OO}{\mathsf{O}}
\newcommand{\oo}{\mathsf{o}}
\newcommand{\Op}{\mathsf{O}_p}
\newcommand{\op}{\mathsf{o}_p}

\newcommand{\VV}{\mathsf{V}}
\newcommand{\cov}{\mathsf{Cov}}

\newcommand{\GG}{\mathscr G}

\newcommand{\pto}{\stackrel{p}{\to}}
\newcommand{\dto}{\stackrel{d}{\to}}
\newcommand{\TT}{\mathsf T}
\newcommand{\bs}{\boldsymbol}
\newcommand{\mb}{\mathbf}
\newcommand{\mf}{\mathfrak}
\newcommand{\lb}{\langle}
\newcommand{\rb}{\rangle}
\newcommand{\tand}{\widetilde \land}

\allowdisplaybreaks
\title{\textbf{Small noise asymptotics for linear parabolic SPDEs 
in two space dimensions with unknown damping factors}}
\date{}

\author{\textbf{Yozo Tonaki}\thanks{Graduate School of Engineering Science, The University of Osaka, Toyonaka, Japan}
\thanks{Center for Mathematical Modeling and Data Science (MMDS), The University of Osaka, Toyonaka, Japan}
\footnote{email: \texttt{y.tonaki.es@osaka-u.ac.jp}}
\and \textbf{Yusuke Kaino}\thanks{Graduate School of Maritime Sciences, Kobe University, Kobe, Japan}
\and \textbf{Masayuki Uchida}$^{* \dag}$\thanks{CREST, Japan Science and Technology Agency, Kawaguchi, Japan}
\thanks{Graduate School of Mathematical Sciences, The University of Tokyo, Meguro, Japan}}
\begin{document}
\maketitle
\begin{abstract}
We study parametric estimation for 
second order linear parabolic stochastic partial differential equations
in two space dimensions with a small volatility parameter
driven by a $Q$-Wiener process with an unknown damping parameter
using high frequency spatio-temporal data.
We first provide an estimator for the damping parameter of the $Q$-Wiener process
utilizing realized quadratic variations based on spatial and temporal increments.
We next propose minimum contrast estimators of the diffusive and advective parameters 
in the SPDE 
using a contrast function with the proposed estimator 
of the damping parameter.
We then construct a quasi-maximum likelihood estimator of the reaction parameter in the SPDE
using the approximate coordinate process 
derived from the estimators of the diffusive and advective parameters.
We also provide simulation results of the proposed estimators.

\begin{center}
\textbf{Keywords and phrases}
\end{center}
Damping parameter, 
high frequency spatio-temporal data,
linear parabolic stochastic partial differential equations,
parametric estimation, 
$Q$-Wiener process,
reaction parameter,
small volatility parameter 
\end{abstract}

\section{Introduction}
Stochastic partial differential equations (SPDEs) have become indispensable tools 
for modeling complex spatio-temporal phenomena across 
a wide range of scientific disciplines, including geophysical fluid dynamics, 
environmental science, mathematical finance, and biology. 
Among them, parabolic SPDEs play a particularly prominent role 
due to their ability to capture diffusive dynamics and spatial interactions inherent 
in many natural and engineered systems.
Such models have been successfully employed in diverse applications: for example, 
the analysis of sea surface temperature fluctuations \cite{Piterbarg_Ostrovskii1997}, 
the study of climate energy balance dynamics \cite{North_etal2011}, 
and the assessment of plutonium contamination in environmental processes 
\cite{Jones_Zhang1997, Mohapl2000}. 
These examples illustrate the versatility of parabolic SPDEs as statistical 
models capable of representing both the stochastic variability 
and the spatial structure observed in real-world data.

Let $D =(0,1)^2$. 
We consider the following linear parabolic SPDE in two space dimensions
\begin{equation}\label{2d_spde}
\left\{
\begin{split}
\dd X_t(y,z) &= -A_\theta X_t(y,z) \dd t +\epsilon \dd W_t^{Q}(y,z), 
\quad (t,y,z) \in [0,1] \times D,
\\
X_t(y,z) &= 0, \quad (t,y,z) \in [0,1] \times \pd D,
\end{split}
\right.
\end{equation}
with a deterministic initial value $X_0$, 
where the operator $A_{\theta}$ is defined by
\begin{equation}\label{op_A}
-A_\theta = 
\theta_2\biggl(\frac{\pd^2}{\pd y^2} + \frac{\pd^2}{\pd z^2} \biggr)
+ \theta_1\frac{\pd}{\pd y} + \eta_1\frac{\pd}{\pd z} + \theta_0,
\end{equation}
$\theta = (\theta_0, \theta_1, \eta_1, \theta_2) \in \mathbb R^3 \times (0,\infty)$ are unknown parameters, 
$\theta^* = (\theta_0^*, \theta_1^*, \eta_1^*, \theta_2^*)$ denote the true values of $\theta$,
and $\{ W_t^Q \}_{t \ge 0}$ is a $Q$-Wiener process in a Sobolev space on $D$ given by
\begin{equation}\label{QW}
W_t^{Q} = 
\sum_{l_1,l_2 \in \mathbb N} \mu_{l_1,l_2}^{-\alpha/2} 
e_{l_1,l_2} w_{l_1,l_2}(t)
\end{equation}
with an unknown damping parameter $\alpha \in (0,3)$, 
an unknown parameter $\mu_0 \in (-2\pi^2, \infty)$, 
$\mu_{l_1,l_2} = \pi^2(l_1^2 +l_2^2) +\mu_0$, 
the eigenfunctions $\{ e_{l_1,l_2} \}_{l_1,l_2 \in \mathbb N}$ 
defined by \eqref{eigen} below, and independent real-valued standard Wiener processes 
$\{ w_{l_1,l_2} \}_{l_1,l_2 \in \mathbb N}$.
Let $\epsilon \in (0,1)$ be a known small volatility parameter.

Statistical inference for parabolic SPDEs has been studied by many researchers, see 
H\"ubner \textit{et al}.\cite{Hubner_etal1993},
Huebner and Rozovskii \cite{Huebner_Rozovskii1995}, 
Piterbarg and Rozovskii \cite{Piterbarg_Rozovskii1997},
Lototsky \cite{Lototsky2003,Lototsky2009}, 
Markussen \cite{Markussen2003},
Lototsky and Rozovsky \cite{Lototsky_Rozovsky2017},
Cialenco \cite{Cialenco2018},
Mahdi Khalil and Tudor \cite{MahdiKhalil_Tudor2019},
Avetisian and Ralchenko \cite{Avetisian_Ralchenko2020},
Bibinger and Trabs \cite{Bibinger_Trabs2020},
Chong \cite{Chong2020},
Cialenco and Huang \cite{Cialenco_Huang2020},
Altmeyer and Reiss \cite{Altmeyer_Reiss2021},
Hildebrandt and Trabs \cite{Hildebrandt_Trabs2021,Hildebrandt_Trabs2023},
Kaino and Uchida \cite{Kaino_Uchida2021a,Kaino_Uchida2021b},
Cialenco and Kim \cite{Cialenco_Kim2022},
Tudor \cite{Tudor2022},
Bibinger and Bossert \cite{Bibinger_Bossert2023},
Gamain and Tudor \cite{Gamain_Tudor2023},
Gaudlitz and Reiss \cite{Gaudlitz_Reiss2023}
Baltazar-Larios \textit{et al.}\cite{Baltazar-Larios_etal2024},
Jan\'{a}k and Reiss \cite{Janak_Reiss2024},
Strauch and Tiepner \cite{Strauch_Tiepner2024},
Andersson \textit{et al.}\cite{Andersson_etal2025}
and Tonaki \textit{et al}.\cite{TKU2025b}.

Tonaki \textit{et al}.\cite{TKU2023} considered parametric estimation for 
linear parabolic SPDEs in two space dimensions driven by the $Q$-Wiener process \eqref{QW}
with a known damping parameter $\alpha \in (0,1)$ based on high frequency spatio-temporal data, 
and proposed estimators of the coefficient parameters $\theta_2$, $\theta_1$, $\eta_1$ and the volatility parameter 
using realized quadratic variations with respect to temporal increments. 
In order to construct a consistent estimator for $\theta_0$ in SPDE \eqref{2d_spde} 
based on high frequency spatio-temporal data, 
Tonaki \textit{et al}.\cite{TKU2024a} considered small noise asymptotics 
and addressed parametric estimation based on the method presented by \cite{TKU2023} 
and statistical inference for diffusion processes with a small dispersion parameter. 
For statistical inference for diffusion processes with a small dispersion parameter,  
see S{\o}rensen and Uchida \cite{Sorensen_Uchida2003}, Uchida \cite{Uchida2004}, 
Gloter and S{\o}rensen \cite{Gloter_Sorensen2009}, Guy \textit{et al}.\cite{Guy_etal2014}, 
Kaino and Uchida \cite{Kaino_Uchida2018}, and Appendix A in \cite{TKU2024a}.
Furthermore, Tonaki \textit{et al}.\cite{TKU2025a} established parametric estimation 
using realized quadratic variations with respect to temporal and spatial increments 
for linear parabolic SPDEs in two space dimensions driven by the $Q$-Wiener process \eqref{QW} 
with a known damping parameter $\alpha \in (0,2)$. 
As a result, the convergence rate of the estimators for the coefficient parameters in the SPDE was greatly improved 
compared with that of the estimators proposed by \cite{TKU2023}.
Tonaki \textit{et al}.\ \cite{TKU2026} studied parametric estimation for SPDE \eqref{2d_spde} 
driven by the $Q$-Wiener process \eqref{QW} with a known damping parameter $\alpha \in (0,3)$. 
Their approach was based on the method introduced in \cite{TKU2025a} and was motivated 
by the finding that the bias of the estimator of $\theta_2$ significantly affects that of $\theta_0$, 
which was suggested by numerical simulations in \cite{TKU2024a}.

In previous studies, the damping parameter $\alpha$ in the $Q$-Wiener process was assumed to be known. 
However, Bossert \cite{Bossert2024} considered parametric estimation for linear parabolic SPDEs in $d \ (\ge 2)$ space dimensions 
and, in particular, introduced an estimator for the damping parameter $\alpha \in (d/2-1, d/2)$ 
based on realized quadratic variations with respect to temporal increments.
Thereafter, Tonaki \textit{et al}.\cite{TKU2025arXiv1} treated liner parabolic SPDEs in two space dimensions
driven by the $Q$-Wiener process with an unknown damping parameter $\alpha \in (0,2)$, 
proposed an estimator of $\alpha$ based on realized quadratic variations with respect to temporal and spatial increments, 
and constructed estimators of the coefficient parameters using the contrast function obtained by substituting the estimator 
of $\alpha$ into the contrast function proposed by \cite{TKU2025a}.
They showed that, when the damping parameter $\alpha$ is unknown, 
the convergence rate of the estimators of the coefficient parameters is decreased by $1/\log(N)$ compared to when $\alpha$ is known,
where $N$ denotes the number of temporal observations.
They also observed that when the number of spatial observations is small, 
the bias of the estimator of $\alpha$ has a significant effect on the biases of the estimators of the coefficient parameters 
through numerical simulations.

In this paper, we consider SPDE \eqref{2d_spde} driven by the $Q$-Wiener process \eqref{QW}
with an unknown damping parameter $\alpha$ and propose a new estimator of $\alpha$ 
in order to address the issue with the estimator of $\alpha \in (0,3)$ discovered by \cite{TKU2025arXiv1}. 
We then construct estimators for the coefficient parameters $\theta_2$, $\theta_1$, $\eta_1$ and $\theta_0$
based on the estimator of $\alpha$ and the method presented by \cite{TKU2026}.
The main purpose of this paper is to investigate how the asymptotic properties of the estimators of coefficient parameters differ 
when $\alpha$ is unknown compared to when $\alpha$ is known, 
and to examine the effect of the proposed estimator for $\alpha$ on estimators for the coefficient parameters.
The results of this paper show that when $\alpha$ is unknown, 
the convergence rate of the estimators of the coefficient parameters decreases by $1/\log(N)$ compared to 
when $\alpha$ is known as in \cite{TKU2025arXiv1}.
Although the asymptotic property of the proposed estimator of $\alpha$ 
is similar to that of \cite{TKU2025arXiv1}, 
numerical simulations reveal that a significant difference in bias arises when the number of spatial observations is small. 
This finding is important
when dealing with real data, for which a sufficient number of spatial observations may not be available.

This paper is organized as follows. 
In Section \ref{sec2}, we give some notation used in our analysis.
Section \ref{sec3} provides the main results.
We first propose an estimator of the damping parameter $\alpha$ 
of the $Q$-Wiener process given in \eqref{QW}.
We next provide minimum contrast estimators of the diffusive and advective parameters 
$\theta_2$, $\theta_1$ and $\eta_1$ 
in SPDE \eqref{2d_spde}
using a contrast function with the proposed estimator of $\alpha$.
Finally, we construct quasi-maximum likelihood estimators 
of the parameters $\theta_0$ and $\mu_0$ using the approximate coordinate process 
derived from the estimators of $\theta_2$, $\theta_1$ and $\eta_1$.
In Section \ref{sec4}, we present simulation results of the proposed estimators.
Section \ref{sec5} is devoted to the proofs of the results in Section \ref{sec3}.

\section{Preliminaries}\label{sec2}
We here introduce some notation used in our analysis.

\subsection{Eigenfunctions and eigenvalues}
The eigenfunctions $\{ e_{l_1,l_2} \}_{l_1,l_2 \in \mathbb N}$ 
of the operator $A_{\theta}$ given in \eqref{op_A}
and the corresponding eigenvalues $\{ \lambda_{l_1,l_2} \}_{l_1,l_2 \in \mathbb N}$
are given by
\begin{equation}\label{eigen}
e_{l_1,l_2}(y,z) = e_{l_1}^{(1)}(y) e_{l_2}^{(2)}(z) ,
\quad
\lambda_{l_1,l_2} = \theta_2(\pi^2(l_1^2+l_2^2)+\Gamma)
\end{equation}
for $(y,z) \in \overline D$, where
\begin{equation*}
e_{l_1}^{(1)}(y) = e_{l_1}^{(1)}(y; \kappa) 
= \sqrt 2 \sin(\pi l_1 y) \ee^{-\kappa y/2},
\quad
e_{l_2}^{(2)}(z) = e_{l_2}^{(2)}(z; \eta)
= \sqrt 2 \sin(\pi l_2 z) \ee^{-\eta z/2},
\end{equation*}
\begin{equation*}
\kappa = \frac{\theta_1}{\theta_2}, 
\quad
\eta = \frac{\eta_1}{\theta_2},
\quad
\Gamma = -\frac{\theta_0}{\theta_2} +\frac{\kappa^2+\eta^2}{4}.
\end{equation*}
Let $L^2(D) := L_{\kappa,\eta}^2(D)$ be the weighted $L^2$-space equipped with
the inner product
\begin{equation}\label{inner_prod}
\langle u, v \rangle = \langle u, v \rangle_{\kappa,\eta}
= \iint_D u(y,z)v(y,z) \exp(\kappa y +\eta z ) \dd y \dd z, 
\quad
\| u \| = \sqrt{\langle u, u \rangle}
\end{equation}
for $u,v \in L^2(D)$.
The eigenfunctions $\{e_{l_1,l_2}\}_{l_1,l_2 \in \mathbb N}$ 
are orthonormal on $L^2(D)$.
We assume $\lambda_{1,1}^* = 2 \pi^2 \theta_2^* 
+\frac{(\theta_1^*)^2+(\eta_1^*)^2}{4 \theta_2^*} -\theta_0^* > 0$ 
so that $A_\theta$ is a positive definite and self-adjoint operator
with respect to the weighted $L^2$-inner product $\langle \cdot , \cdot \rangle$.

\subsection{Driving noise}
Let $\{ w_{l_1,l_2} \}_{l_1,l_2 \in \mathbb N}$ be independent real-valued standard Wiener processes.
We suppose that the damping parameter $\alpha$ is unknown, 
the parameter space $\Xi_{\alpha}$ of $\alpha$ is a compact convex subset 
of $(0,3)$ and the true value $\alpha^*$ of $\alpha$ belongs to the interior $\Xi_{\alpha}^\circ$ of $\Xi_{\alpha}$. 
Let $\mu_0 \in (-2 \pi^2, \infty)$ be an unknown parameter, and $\mu_0^*$ denotes the true value of $\mu_0$.
The sequence $\{ \mu_{l_1,l_2} \}_{l_1,l_2 \in \mathbb N}$ with $\mu_{l_1,l_2} = \pi^2(l_1^2 +l_2^2) + \mu_0$ is positive.
We consider the $Q$-Wiener process $\{ W_t^{Q} \}_{t \ge 0}$ given in \eqref{QW}, that is, 
\begin{equation*}
W_t^{Q} = \sum_{l_1,l_2 \in \mathbb N} \mu_{l_1,l_2}^{-\alpha/2} e_{l_1,l_2} w_{l_1,l_2}(t).
\end{equation*}

\begin{rmk}
\begin{itemize}
\item[(i)]
The $Q$-Wiener process given in \eqref{QW} is well-defined 
in a Hilbert space larger than $L^2(D)$ which satisfies 
$A_{\theta}^{-1/2} u \in L^2(D)$ for 
$u= \sum_{l_1,l_2 \in \mathbb N} u_{l_1,l_2} e_{l_1,l_2}$ and 
$u_{l_1,l_2} \in \mathbb R$. 
See Chapter 4 in \cite{DaPrato_Zabczyk2014} or Remark 2 in \cite{TKU2023} for details.

\item[(ii)]
We restrict $\alpha > 0$ in order to guarantee that the mild solution $X_t$ 
of SPDE \eqref{2d_spde} satisfies $\sup_{t \in [0,1]} \EE[\| X_t \|^2 ] < \infty$, 
which makes our analysis possible (see Remark 2 in \cite{TKU2023}).
$\alpha < 3$ is a necessary condition for the estimators of the parameters 
constructed through our approach to be consistent.
In fact, we see that for $\alpha \in [3, \infty)$, 
the convergence rate $R_{\alpha,\alpha_0}^{\mathrm{damp}}$ given in \eqref{R_damp} below satisfies 
$R_{\alpha,\alpha_0}^{\mathrm{damp}} \le \sqrt{m N} \cdot N^{2 -\alpha} = \OO(N^{3 -\alpha})$
and $R_{\alpha,\alpha_0}^{\mathrm{damp}} \nrightarrow \infty$,
and that for $\alpha \in (0,3)$, $\phi_{r,\alpha}(\theta_2)$ given in \eqref{phi} below is finite (see \eqref{est_q} below).

\item[(iii)]
The reason Tonaki \textit{et al}.\cite{TKU2025a} (or \cite{TKU2025arXiv1}) restricts $\alpha < 2$ is 
to establish $\sqrt{m N}$-consistent estimators of the coefficient (or damping) parameters.
This arises from the condition $\beta = \alpha < 2$ required in 
Lemma 1-(2) of \cite{TKU2026} with $\gamma = \gamma_1 +\gamma_2$ and $\gamma_1 = \gamma_2 = 1$
in order to ensure the remainder term in Lemma 1-(2) of \cite{TKU2026} is $\OO(h^2)$.
Since one needs only to consider the condition on $\beta = \alpha$ that the remainder term in Lemma 1-(2) of \cite{TKU2026} is $\oo(1)$
in order to obtain consistency of the estimators, one can relax the restriction to $\beta = \alpha < \gamma +1 = 3$.

\end{itemize}
\end{rmk}

\subsection{Mild solution and coordinate processes}

There exists a unique mild solution of SPDE \eqref{2d_spde}, which is given by 
\begin{equation*}
X_t = \ee^{-t A_\theta} X_0 
+\epsilon \int_0^t \ee^{-(t-s)A_\theta} \dd W_s^{Q}
\quad
\mathrm{a.s.}
\end{equation*}
for any $t \ge 0$, where $\ee^{-t A_\theta} u 
= \sum_{l_1,l_2 \in \mathbb N} \ee^{-\lambda_{l_1,l_2}t}
\langle u, e_{l_1,l_2}\rangle e_{l_1,l_2}$ for $u \in L^2(D)$.
The random field $X_t$ is decomposed as follows. 
\begin{equation*}
X_t(y,z) = \sum_{l_1,l_2 \in \mathbb N} 
x_{l_1,l_2}(t) e_{l_1}^{(1)}(y) e_{l_2}^{(2)}(z),
\quad
t \ge 0, \ (y,z) \in \overline{D}
\end{equation*}
with the coordinate process
\begin{equation*}
x_{l_1,l_2}(t) = \langle X_t, e_{l_1,l_2} \rangle 
= \ee^{-\lambda_{l_1,l_2} t} \langle X_0, e_{l_1,l_2} \rangle
+ \epsilon \mu_{l_1,l_2}^{-\alpha/2} \int_0^t \ee^{-\lambda_{l_1,l_2}(t-s)} 
\dd w_{l_1,l_2}(s).
\end{equation*}
The coordinate process $x_{l_1,l_2}(t)$ satisfies the Ornstein-Uhlenbeck process
\begin{equation}\label{OU-process}
\left\{
\begin{split}
\dd x_{l_1,l_2}(t) &= -\lambda_{l_1,l_2} x_{l_1,l_2}(t) \dd t 
+\epsilon \mu_{l_1,l_2}^{-\alpha/2} \dd w_{l_1,l_2}(t),
\\
x_{l_1,l_2}(0) &= \langle X_0, e_{l_1,l_2} \rangle.
\end{split}
\right.
\end{equation}

\subsection{Observations}
Let $N \ge 4$, $M_1 \ge 2$ and $M_2 \ge 2$ be integers and define $\mathrm{M} = (M_1, M_2)$.
We suppose that a mild solution of SPDE \eqref{2d_spde}
is discretely observed on the grid $(t_i, y_j, z_k) \in [0,1] \times [0,1]^2$ with
\begin{equation*}
t_i = t_i^N = i\Delta = \frac{i}{N},
\quad
y_j = y_j^{M_1} = \frac{j}{M_1}, 
\quad
z_k = z_k^{M_2} = \frac{k}{M_2}
\end{equation*}
for $i \in \{0, \ldots, N \}$, 
$j \in \{0, \ldots, M_1 \}$ and $k \in \{0, \ldots, M_2 \}$.
We set the original data
\begin{equation*}
\mathbb X_{\mathrm{M},N} = 
\bigl\{ X_{t_i}(y_j,z_k) \bigr\}_{i \in \{0, \ldots, N \}, j \in \{0, \ldots, M_1 \}, 
k \in \{0, \ldots, M_2 \}}.
\end{equation*}

Our approach relies on the relation between thinned temporal and spatial data.
As in the existing literature
on statistical inference for SPDEs using high frequency observations, 
estimating the parameters $(\kappa, \eta)$ requires spatial subsampling (see, for example, 
\cite{Bibinger_Bossert2023, Hildebrandt_Trabs2021, Bossert2024, TKU2025arXiv1}).
In addition, we control the approximate coordinate process \eqref{approx_cp} below 
by reducing the number of temporal data (see the proof of Proposition 5.2 in \cite{TKU2025arXiv2} for details).
For this reason, we will introduce the following spatial or temporal thinned data.

For $b \in (0,1/2)$ and 
$\mathrm{m} = (m_1, m_2) \in \{1, \ldots, M_1 \} \times \{1, \ldots, M_2 \}$, 
we choose spatial subsampling $(y_j^{b,m_1}, z_k^{b,m_2}) 
\in \{ y_0,\ldots, y_{M_1} \} \times \{ z_0,\ldots, z_{M_2} \}$ such that
\begin{equation*}
b \le y_{0}^{b,m_1} < y_{1}^{b,m_1} < \cdots < y_{m_1}^{b,m_1} \le 1-b,
\quad
b \le z_{0}^{b,m_2} < z_{1}^{b,m_2} < \cdots < z_{m_2}^{b,m_2} \le 1-b.
\end{equation*}
For simplicity, we set $m_1 = m_2 $, $m = m_1 m_2 = \OO(N)$, $N = \OO(m)$
and the thinned spatial resolution 
$\delta = \delta_{m} = \frac{1-2b}{m_1} 
= \frac{1-2b}{m_2} = \frac{1-2b}{\sqrt{m}}$.
We then have $\widetilde y_j = y_j^{b,m_1} = b+j \delta$ and $\widetilde z_k = b+k \delta$,
and set the spatial thinned data $\mathbb X_{\mathrm{m},N}^{b}$ 
of $\mathbb X_{\mathrm{M},N}$ as follows.
\begin{equation}\label{space_thin_data}
\mathbb X_{\mathrm{m},N}^{b} = 
\bigl\{ X_{t_i}(\widetilde y_j, \widetilde z_k) \bigr\}_{i \in \{0, \ldots, N \}, 
j \in \{0, \ldots, m_1 \}, k \in \{0, \ldots, m_2 \}}.
\end{equation}

For $n \in \{ 1, \ldots, N \}$, 
we define the temporal thinned data $\mathbb X_{\mathrm{M},n}^{0}$ 
of $\mathbb X_{\mathrm{M},N}$ with the thinned temporal resolution 
$\Delta_n = \frac{1}{N} \lfloor \frac{N}{n} \rfloor$ as follows.
\begin{equation}\label{time_thin_data}
\mathbb X_{\mathrm{M},n}^{0} = 
\bigl\{ X_{t_i^n}(y_j, z_k) \bigr\}_{i \in \{0, \ldots, n \}, 
j \in \{0, \ldots, m_1 \}, k \in \{0, \ldots, m_2 \}}
\end{equation}
with $t_i^n = i \Delta_n$, where $t_i^n \in \{ t_0, \ldots, t_N \}$.

\begin{rmk}
As it is mentioned in the discussion section of \cite{TKU2026}, for SPDE \eqref{2d_spde}, 
the balance conditions $m = \OO(N)$ and $N = \OO(m)$ are not necessary for parameter identification. 
However, we assume these balance conditions in this paper in order to apply the results of \cite{TKU2025a} and \cite{TKU2026}.
\end{rmk}

\subsection{Some notation}
For a sequence $\{a_n\}$, 
we will write $a_n \equiv a$ if $a_n = a$ for some $a \in \mathbb R$ and all $n$.
For $h \in (0,1)$, $a, b, c \in (0, \infty)$ and $d \in \mathbb R$, we define
\begin{equation*}
h^{a \tand b} =  
\begin{cases}
h^{a}, & a < b,
\\
-h^{b} \log (h), & a = b,
\\
h^{b}, & a > b,
\end{cases}
\end{equation*}
and $h^{d +c(a \tand b)} = h^d \cdot (h^c)^{a \tand b}$.
Furthermore, for $L \in (1,\infty)$, we write
\begin{equation*}
L^{a \tand b} = \frac{1}{(1/L)^{a \tand b}} =
\begin{cases}
L^{a}, & a < b,
\\
L^{b} /\log (L), & a = b,
\\
L^{b}, & a > b.
\end{cases}
\end{equation*}

\section{Main results}\label{sec3}

We set the fractional power $A_{\theta}^{\beta}$ for $\beta \in \mathbb R$ by
\begin{equation*}
A_{\theta}^{\beta} u = \sum_{l_1,l_2 \in \mathbb N} 
\lambda_{l_1,l_2}^{\beta} u_{l_1,l_2} e_{l_1,l_2}
\end{equation*}
for $u = \sum_{l_1,l_2 \in \mathbb N} u_{l_1,l_2} e_{l_1,l_2}$ 
and $u_{l_1,l_2} \in \mathbb R$ such that $A_{\theta}^{\beta} u \in L^2(D)$.

Let $\alpha_0 \in (0,3)$ be a constant that characterizes the regularity of the initial value $X_0$ of SPDE (1.1),
and impose the following condition.
\begin{description}
\item[\textbf{[A1]}$_{\alpha_0}$]
The initial value $X_0 \in L^2(D)$ is deterministic and
$\| A_\theta^{(1+\alpha_0)/2} X_0 \| < \infty$. 

\item[\textbf{[A2]}]
$X_0 \neq 0$, that is, there exists $(\mf l_1, \mf l_2) \in \mathbb N^2$
such that $\langle X_0, e_{\mf l_1, \mf l_2} \rangle \neq 0$.
\end{description}

In order to exclude situations where the initial value dominates 
and to consider situations where our method is beneficial, 
we introduce the following balance condition on $N$ and $\epsilon$.
\begin{description}
\item[{[B]$_{\alpha, \alpha_0}$}]
$\frac{\log(N)}{\epsilon^2 N^{1+\alpha_0 -\alpha}} = \oo(1)$.
\end{description}
This condition guarantees that the convergence rate 
$R_{\alpha,\alpha_0}^{\mathrm{coef}} = R_{\alpha,\alpha_0}^{\mathrm{damp}}/\log(N)$ 
of the estimators for the coefficient parameters $\theta_2$, $\theta_1$ and $\eta_1$ 
satisfies $R_{\alpha,\alpha_0}^{\mathrm{coef}} \to \infty$ and the estimators for the coefficient parameters are consistent. 
In other words, we will make a statistical inference under a balance condition for $\epsilon$ and $N$ 
such that the estimators for the coefficient parameters are consistent.
Additionally, the balance condition in this paper is stricter 
than the balance condition $\frac{1}{\epsilon^2 N^{1+\alpha_0 -\alpha}} = \oo(1)$ given in \cite{TKU2026} 
since $\alpha$ is unknown in this paper, whereas \cite{TKU2026} assumes $\alpha$ is known.

\subsection{Estimation of $\alpha$}
First, we consider estimation for the damping parameter $\alpha$ 
in the $Q$-Wiener process $\{ W_t^Q \}_{t \ge 0}$ given in \eqref{QW}.
The idea behind the estimator of $\alpha$ proposed by Tonaki \textit{et al}.\cite{TKU2025arXiv1} 
was to prepare two types of thinned data that retain 
the ratio $r = \frac{\delta}{\sqrt{\Delta}}$ 
with the spatial resolution $\delta$ and the temporal resolution $\Delta$, 
and to construct a consistent estimator of $\alpha$ 
from the ratio of the realized quadratic variations 
based on triple increments derived from two types of thinned data.
However, this method wastes a significant number of observations
since one of the two types of thinned data contains only
$(\frac{1}{2})^2 \times \frac{1}{4} = \frac{1}{16}$ 
of the observations of the other type in order to maintain the ratio $r$ when constructing triple increments.
As a result, in practical applications where it is not possible to obtain a large number of spatial observations, 
this may cause significant biases in the estimators of $\alpha$ and coefficient parameters.
See Case 3 of Tables 1 and 2 in \cite{TKU2025arXiv1}.

In this study, for the spatial thinned data $\mathbb X_{\mathrm{m},N}^{b}$ 
given in \eqref{space_thin_data}, we define the triple increments
\begin{equation}\label{Old_TX}
\begin{split}
T_{i,j,k} X 
&= X_{t_{i}}(\widetilde y_{j},\widetilde z_{k})
-X_{t_{i}}(\widetilde y_{j-1},\widetilde z_{k})
-X_{t_{i}}(\widetilde y_{j},\widetilde z_{k-1})
+X_{t_{i}}(\widetilde y_{j-1},\widetilde z_{k-1})
\\
&\quad
-X_{t_{i-1}}(\widetilde y_{j},\widetilde z_{k})
+X_{t_{i-1}}(\widetilde y_{j-1},\widetilde z_{k})
+X_{t_{i-1}}(\widetilde y_{j},\widetilde z_{k-1})
-X_{t_{i-1}}(\widetilde y_{j-1},\widetilde z_{k-1})
\\
&= \sum_{l_1,l_2 \in \mathbb N}
(x_{l_1,l_2}(t_i) -x_{l_1,l_2}(t_{i-1}))
(e_{l_1}^{(1)}(\widetilde y_{j}) -e_{l_1}^{(1)}(\widetilde y_{j-1}))
(e_{l_2}^{(2)}(\widetilde z_{k}) -e_{l_2}^{(2)}(\widetilde z_{k-1}))
\end{split}
\end{equation}
with the temporal and spatial resolutions $(\Delta, \delta)$
for $i \in \{1,\ldots,N \}$, $j \in \{1,\ldots, m_1\}$
and $k \in \{1,\ldots, m_2\}$
and
\begin{equation}\label{New_TX}
\widehat T_{i,j,k} X
= \sum_{k'=0}^1 \sum_{j'=0}^1 \sum_{i'=0}^3 T_{i+i',j+j',k+k'} X
\end{equation}
for $i \in \{1,\ldots,N-3\}$, $j \in \{1,\ldots, m_1-1\}$
and $k \in \{1,\ldots, m_2-1\}$. 
We then obtain 
\begin{align*}
\widehat T_{i,j,k} X 
&= X_{t_{i+3}}(\widetilde y_{j+1},\widetilde z_{k+1})
-X_{t_{i+3}}(\widetilde y_{j-1},\widetilde z_{k+1})
-X_{t_{i+3}}(\widetilde y_{j+1},\widetilde z_{k-1})
+X_{t_{i+3}}(\widetilde y_{j-1},\widetilde z_{k-1})
\\
&\quad
-X_{t_{i-1}}(\widetilde y_{j+1},\widetilde z_{k+1})
+X_{t_{i-1}}(\widetilde y_{j-1},\widetilde z_{k+1})
+X_{t_{i-1}}(\widetilde y_{j+1},\widetilde z_{k-1})
-X_{t_{i-1}}(\widetilde y_{j-1},\widetilde z_{k-1})
\\
&= \sum_{l_1,l_2 \in \mathbb N}
(x_{l_1,l_2}(t_{i+3}) -x_{l_1,l_2}(t_{i-1}))
(e_{l_1}^{(1)}(\widetilde y_{j+1}) -e_{l_1}^{(1)}(\widetilde y_{j-1}))
(e_{l_2}^{(2)}(\widetilde z_{k+1}) -e_{l_2}^{(2)}(\widetilde z_{k-1}))
\end{align*}
and note that the temporal and spatial resolutions $(\Delta', \delta')$ 
of the triple increments $\widehat T_{i,j,k} X$ satisfy
\begin{equation*}
\Delta' = t_{i+3} -t_{i-1} = 4\Delta,
\quad
\delta' = \widetilde y_{j+1} -\widetilde y_{j-1} = \widetilde z_{k+1} -\widetilde z_{k-1} = 2\delta.
\end{equation*}
Let
\begin{equation}\label{R_damp}
R_{\alpha,\alpha_0}^{\mathrm{damp}} = 
\frac{\sqrt{m N}}{\Delta^{(\alpha \tand 2) -\alpha} 
\lor \epsilon^{-2} \Delta^{\alpha_0 -\alpha}}
= \sqrt{m N} (N^{(\alpha \tand 2) -\alpha} \land \epsilon^{2} N^{\alpha_0 -\alpha}).
\end{equation}
We have for $\delta'/\sqrt{\Delta'} 
= \delta/\sqrt{\Delta} \equiv r \in (0,\infty)$,
\begin{equation*}
\frac{1}{\epsilon^2 m N (4\Delta)^\alpha} 
\sum_{k = 1}^{m_2 -1} \sum_{j = 1}^{m_1 -1} \sum_{i=1}^{N -3} 
(\widehat T_{i,j,k} X)^2
= g_{r,\alpha}^{(\mathrm{m})}(\vartheta) 
+\Op \bigl( (R_{\alpha,\alpha_0}^{\mathrm{damp}})^{-1} \bigr),
\end{equation*}
where 
\begin{equation}\label{g_r}
g_{r,\alpha}^{(\mathrm{m})}(\vartheta) = 
\frac{\phi_{r,\alpha}(\theta_2)}{4(1-2b)^2}
\sum_{k=1}^{m_2-1} \sum_{j=1}^{m_1-1}
\int_{\widetilde z_{k-1}}^{\widetilde z_{k+1}}
\int_{\widetilde y_{j-1}}^{\widetilde y_{j+1}} 
\ee^{-\kappa y -\eta z} \dd y \dd z
\end{equation}
for $\vartheta = (\kappa, \eta, \theta_2)$ and 
$\phi_{r,\alpha}(\theta_2)$ is given in \eqref{phi} below.
Since we also have 
\begin{align*}
\frac{1}{\epsilon^2 m N \Delta^{\alpha}}
\sum_{k = \underline{k}+1}^{\overline{k}} 
\sum_{j = \underline{j}+1}^{\overline{j}} \sum_{i=1}^N (T_{i,j,k}X)^2
= \frac{\phi_{r,\alpha}(\theta_2)}{(1-2b)^2}
\int_{\widetilde z_{\underline{k}}}^{\widetilde z_{\overline{k}}}
\int_{\widetilde y_{\underline{j}}}^{\widetilde y_{\overline{j}}}
\ee^{-\kappa y -\eta z} \dd y \dd z
+\Op \bigl( (R_{\alpha,\alpha_0}^{\mathrm{damp}})^{-1} \bigr)
\end{align*}
for $0 \le \underline{j} < \overline{j} \le m_1$ 
and $0 \le \underline{k} < \overline{k} \le m_2$,
we obtain
\begin{align*}
&\frac{1}{\epsilon^2 m N \Delta^\alpha} 
\biggl(
\sum_{k=1}^{m_2} \sum_{j=1}^{m_1} \sum_{i=1}^N (T_{i,j,k}X)^2
+\sum_{k=2}^{m_2-1} \sum_{j=1}^{m_1} \sum_{i=1}^N (T_{i,j,k}X)^2
\\
&\quad+\sum_{k=1}^{m_2} \sum_{j=2}^{m_1-1} \sum_{i=1}^N (T_{i,j,k}X)^2
+\sum_{k=2}^{m_2-1} \sum_{j=2}^{m_1-1} \sum_{i=1}^N (T_{i,j,k}X)^2
\biggr)
\\
&= 4g_{r,\alpha}^{(\mathrm{m})}(\vartheta) 
+\Op \bigl( (R_{\alpha,\alpha_0}^{\mathrm{damp}})^{-1} \bigr).
\end{align*}
The derivation of the above results can be found in Subsection \ref{sec5-1} below. 
We then set
\begin{align*}
S_1 &=
\sum_{k=1}^{m_2-1} \sum_{j=1}^{m_1-1} \sum_{i=1}^{N-3} (\widehat T_{i,j,k} X)^2,
\\
S_2 &=
\sum_{k=1}^{m_2} \sum_{j=1}^{m_1} \sum_{i=1}^N (T_{i,j,k}X)^2
+\sum_{k=2}^{m_2-1} \sum_{j=1}^{m_1} \sum_{i=1}^N (T_{i,j,k}X)^2
\\
&\quad+\sum_{k=1}^{m_2} \sum_{j=2}^{m_1-1} \sum_{i=1}^N (T_{i,j,k}X)^2
+\sum_{k=2}^{m_2-1} \sum_{j=2}^{m_1-1} \sum_{i=1}^N (T_{i,j,k}X)^2
\end{align*}
and have
\begin{equation*}
\frac{S_1}{S_2/4}
= 4^\alpha \cdot \frac{S_1/\epsilon^2 m N (4\Delta)^\alpha}{S_2/4\epsilon^2 m N \Delta^\alpha}
= 4^\alpha \bigl(1 +\Op \bigl( (R_{\alpha,\alpha_0}^{\mathrm{damp}})^{-1} \bigr) \bigr).
\end{equation*}
Therefore, we define the estimator of the damping parameter $\alpha$ as follows.
\begin{equation}\label{est_alpha}
\widehat \alpha = 
\frac{\log(S_1) -\log(S_2/4)}{\log(4)}
= \frac{\log(S_1) -\log(S_2)}{\log(4)} +1.
\end{equation}
We obtain the following theorem.
\begin{thm}\label{th1}
Let $\alpha_0 \in (0,3)$ and $\delta/\sqrt{\Delta} \equiv r \in (0,\infty)$.
It holds that under [A1]$_{\alpha_0}$ and [B]$_{\alpha^*,\alpha_0}$, 
\begin{equation*}
R_{\alpha^*,\alpha_0}^{\mathrm{damp}} (\widehat \alpha -\alpha^*) = \Op(1)
\end{equation*}
as $m \to \infty$, $N \to \infty$ and $\epsilon \to 0$.
\end{thm}

\begin{rmk}
\begin{itemize}
\item[(i)]
In a similar manner to the estimator of $\alpha$ given in (3.1) of \cite{TKU2025arXiv1},
we define the estimator of $\alpha$ in the following simple form
\begin{equation*}
\widetilde \alpha = \frac{\log(S_1) -\log(\widetilde S_2)}{\log(4)},
\quad
\widetilde S_2 = \sum_{k=1}^{m_2} \sum_{j=1}^{m_1} \sum_{i=1}^N (T_{i,j,k}X)^2.
\end{equation*}
We then find from Lemma \ref{lem_tilde_alpha} below that
$\bigl( R_{\alpha^*,\alpha_0}^{\mathrm{damp}} \land \sqrt{m} \bigr) (\widetilde \alpha -\alpha^*) = \Op(1)$,
which implies that $\widetilde \alpha$ is the inferior estimator compared to $\widehat \alpha$. 
Therefore, we employ the estimator of the form \eqref{est_alpha}.

\item[(ii)]
Since the true value $\alpha^*$ is known, we have no way to directly check the sufficient condition [B]$_{\alpha^*, \alpha_0}$ 
in Theorem \ref{th1}. In practice, it is natural to assume that [B]$_{\alpha^*, \alpha_0}$ holds for all $\alpha \in \Xi_{\alpha}$,
in which case we can check [B]$_{\sup(\Xi_{\alpha}), \alpha_0}$. 
Similarly, for the conditions involving the true value $\alpha^*$ that appear below, 
we can check the conditions obtained by replacing true value $\alpha^*$ with $\inf(\Xi_{\alpha})$ or $\sup(\Xi_{\alpha})$.

\item[(iii)]
If there exist sequences $\{ r_{j,n} \}_{n \in \mathbb N}$ ($j=1,2$) and $\{ c_n \}_{n \in \mathbb N}$ such that 
$\lim_{n \to \infty} r_{j,n} = \infty$, $\lim_{n \to \infty} c_n \in (0, \infty)$ and $r_{j,n}(\mathcal Y_{j,n} - c_n) = \Op(1)$, 
then $\mathcal Y_n = \log(\mathcal Y_{1,n}) -\log(\mathcal Y_{2,n})$ satisfies $(r_{1,n} \land r_{2,n}) \mathcal Y_n = \Op(1)$
(see \eqref{taylor} and \eqref{eq-z1/z2} below).
For simplicity, we consider the case where the convergence rate given in \eqref{R_damp} 
achieves $R_{\alpha^*,\alpha_0}^{\mathrm{damp}} = \sqrt{m N}$. 
We then see that the convergence rate of the estimator of $\alpha$ proposed by (3.1) in \cite{TKU2025arXiv1} is, 
more precisely, $\sqrt{m N} \land \sqrt{m' N'} = \sqrt{m N}/4$.
On the other hand, we find that the convergence rate of the estimator $\widehat \alpha$ given in \eqref{est_alpha} is, more precisely,
\begin{equation*}
\sqrt{m N} \land \sqrt{(m_1-1)(m_2-1)(N-3)} = \frac{\sqrt{m N}}{q_{\mathrm{m},N}}
\end{equation*}
with
\begin{equation*}
q_{\mathrm{m},N} =\frac{\sqrt{m N}}{\sqrt{(m_1-1)(m_2-1)(N-3)}} = \frac{1}{\sqrt{(1-\frac{1}{m_1})(1-\frac{1}{m_2})(1-\frac{3}{N})}}.
\end{equation*}
Since $(1 \le)\ q_{\mathrm{m},N} \le 4$ holds for any $N \ge 4$ and $m_1, m_2 \ge 2$, 
we have $\sqrt{m N}/q_{\mathrm{m},N} \ge \sqrt{m N}/4$. 
By using this approach, we can reduce the bias of the estimator for the damping parameter $\alpha$ in practical applications, 
where we cannot obtain a large number of spatial observations.
We will provide numerical simulations on the differences between these two estimators in Section \ref{sec4}.
\end{itemize}

\end{rmk}

\subsection{Estimation of $\theta_1$, $\eta_1$ and $\theta_2$}
We next consider estimation for the coefficient parameters 
$\theta_1$, $\eta_1$ and $\theta_2$ based on the thinned data 
$\mathbb X_{\mathrm{m},N}^{b}$ given in \eqref{space_thin_data}.
Let $J_0$ be the Bessel function of the first kind of order $0$:
\begin{equation}\label{Bessel}
J_0(x) = 1 +\sum_{k =1}^{\infty} \frac{(-1)^k}{(k!)^2} \Bigl( \frac{x}{2} \Bigr)^{2k}.
\end{equation}
For $r \in (0,\infty)$, $\alpha \in (0,3)$ and $\theta_2 \in (0,\infty)$, we define
\begin{equation}\label{phi}
\phi_{r,\alpha}(\theta_2)
=\frac{2}{\theta_2^{1-\alpha} \pi}
\int_0^\infty 
\frac{1-\ee^{-x^2}}{x^{1+2\alpha}}
\biggl(
J_0\Bigl(\frac{\sqrt{2}r x}{\sqrt{\theta_2}}\Bigr)
-2J_0\Bigl(\frac{r x}{\sqrt{\theta_2}}\Bigr)+1
\biggr) \dd x.
\end{equation}
Tonaki \textit{et al}.\cite{TKU2026} supposed that the damping parameter $\alpha$ is known
and showed 
\begin{equation}\label{TKU26_2.8}
\frac{1}{\epsilon^2 N \Delta^{\alpha}}\sum_{i=1}^{N} \EE[(T_{i,j,k}X)^2]
= \ee^{-\kappa \overline y_j -\eta \overline z_k} \phi_{r,\alpha}(\theta_2)
+ \OO(\Delta^{1-\alpha +(\alpha \tand 2)} 
\lor \epsilon^{-2} \Delta^{1+\alpha_0 -\alpha})
\end{equation}
where $\overline y_j = (\widetilde y_{j-1} +\widetilde y_j)/2$
and $\overline z_k = (\widetilde z_{k-1} +\widetilde z_k)/2$ (see (2.8) in \cite{TKU2026}). 
Using the above result, they defined the contrast function as 
\begin{equation}\label{contrast_U}
U_{m,N}^{\epsilon}(\kappa, \eta, \theta_2;\alpha) = 
\frac{1}{m}\sum_{k=1}^{m_2}\sum_{j=1}^{m_1} 
\Biggl\{
\frac{1}{\epsilon^2 N \Delta^{\alpha}}\sum_{i=1}^{N} (T_{i,j,k}X)^2
-\ee^{-\kappa \overline y_j -\eta \overline z_k} \phi_{r,\alpha}(\theta_2)
\Biggr\}^2,
\end{equation}
and established the minimum contrast estimator
\begin{equation*}
(\widehat \kappa^{\mathrm{s}}, \widehat \eta^{\mathrm{s}}, \widehat \theta_2^{\mathrm{s}}) 
= \underset{(\kappa, \eta, \theta_2) \in \Xi_{\vartheta}}{\mathrm{argmin}}\,
U_{m,N}^{\epsilon}(\kappa, \eta, \theta_2;\alpha^*)
\end{equation*}
with the known damping parameter $\alpha^*$, 
where  the parameter space $\Xi_{\kappa,\eta}$ of $(\kappa, \eta)$ 
is a compact convex subset of $\mathbb R^2$,  
the parameter space $\Xi_{\theta_2} \subset (0,\infty)$ of $\theta_2$ 
is a compact convex subset of $(0,\infty)$, and 
$\Xi_{\vartheta} = \Xi_{\kappa, \eta} \times \Xi_{\theta_2}$.
They further showed that the estimators 
$\widehat \theta_1^{\mathrm{s}} = \widehat \kappa^{\mathrm{s}} \widehat \theta_2^{\mathrm{s}}$,
$\widehat \eta_1^{\mathrm{s}} = \widehat \eta^{\mathrm{s}} \widehat \theta_2^{\mathrm{s}}$
and $\widehat \theta_2^{\mathrm{s}}$ have
$R_{\alpha^*,\alpha_0}^{\mathrm{damp}}$-consistency
(see Theorem 2.2 in \cite{TKU2026}).

In this study, we consider the contrast function obtained 
by substituting the estimator $\widehat \alpha$ given in \eqref{est_alpha} 
into \eqref{contrast_U}, and define the estimators of 
$(\theta_1, \eta_1, \theta_2)$ as follows.
\begin{equation}\label{est_coef}
\left\{
\begin{split}
(\widehat \kappa, \widehat \eta, \widehat \theta_2) 
&= \underset{(\kappa, \eta, \theta_2) \in \Xi_{\vartheta}}{\mathrm{argmin}}\,
U_{m,N}^{\epsilon}(\kappa, \eta, \theta_2;\widehat \alpha),
\\
\widehat \theta_1 &= \widehat \kappa \widehat \theta_2,
\\
\widehat \eta_1 &= \widehat \eta \widehat \theta_2.
\end{split}
\right.
\end{equation}
We then obtain the following result.
\begin{thm}\label{th2}
Let $\alpha_0 \in (0,3)$ and $\delta/\sqrt{\Delta} \equiv r \in(0,\infty)$.
It holds that under [A1]$_{\alpha_0}$ and [B]$_{\alpha^*, \alpha_0}$, 
\begin{equation*}
\frac{R_{\alpha,\alpha_0}^{\mathrm{damp}}}{\log(N)}
\begin{pmatrix}
\widehat \theta_1 -\theta_1^*
\\
\widehat \eta_1 -\eta_1^* 
\\
\widehat \theta_2 -\theta_2^*
\end{pmatrix}
= \Op(1)
\end{equation*}
as $m \to \infty$, $N \to \infty$ and $\epsilon \to 0$.
\end{thm}

\begin{rmk}
As in Theorem 4.3 in \cite{TKU2025arXiv1}, we find from Theorem \ref{th2} that 
for SPDE \eqref{2d_spde}, when the damping parameter $\alpha$ is unknown, 
the convergence rates of the estimators for the coefficient parameters $\theta_1$, $\eta_1$ and $\theta_2$
are slower by a factor of $\frac{1}{\log(N)}$ compared with the case where $\alpha$ is known (\cite{TKU2026}).
This decay in the convergence rate is due to
\begin{equation*}
\Delta^{\alpha^*} \biggl| \frac{1}{\Delta^{\widehat \alpha}} -\frac{1}{\Delta^{\alpha^*}} \biggr|
= \Op \biggl( \frac{\log(N)}{R_{\alpha,\alpha_0}^{\mathrm{damp}}} \biggr)
\end{equation*}
in the proof of Lemma \ref{lem_XYZ}-(3) below.
\end{rmk}

\subsection{Estimation of $\theta_0$ and $\mu_0$}
Finally, we consider parametric estimation for $\theta_0$ and $\mu_0$ 
using the thinned data $\mathbb X_{\mathrm{M},n}^{0}$ 
given in \eqref{time_thin_data}.
Let $D_{j,k} = (y_{j-1}, y_{j}] \times (z_{k-1}, z_{k}]$ 
for $(j,k) \in \{1,\ldots, M_1 \} \times \{1,\ldots, M_2 \}$. 
Although we cannot directly observe the coordinate process $x_{l_1,l_2}(t)$, 
we can obtain by \eqref{inner_prod}, 
\begin{align*}
x_{l_1,l_2}(t) 
&= \iint_{D} X_t(y, z) e_{l_1}^{(1)}(y;\kappa^*)e_{l_2}^{(2)}(z;\eta^*) 
\exp(\kappa^* y +\eta^* z) \dd y \dd z
\\
&= \sum_{k = 1}^{M_2} \sum_{j = 1}^{M_1} 
\iint_{D_{j,k}} X_t(y, z) 
e_{l_1}^{(1)}(y;\kappa^*)e_{l_2}^{(2)}(z;\eta^*)
\exp(\kappa^* y +\eta^* z) \dd y \dd z.
\end{align*}
Since the estimators $\widehat \kappa$ and $\widehat \eta$ 
obtained from \eqref{est_coef} and the observations $X_{t_i} (y_j,z_k)$ 
on the grid $(t_i, y_j,z_k)$ are available, 
we define the approximate coordinate process as follows.
\begin{equation}\label{approx_cp0}
\widehat x_{l_1,l_2}(t_i) 
= \sum_{k = 1}^{M_2} \sum_{j = 1}^{M_1} 
\iint_{D_{j,k}} X_{t_i}(y_{j}, z_{k}) 
e_{l_1}^{(1)}(y;\widehat \kappa)
e_{l_2}^{(2)}(z;\widehat \eta) 
\exp(\widehat \kappa y +\widehat \eta z) \dd y \dd z.
\end{equation}
Since it follows that 
\begin{equation*}
\int \ee^{a x} \sin(b x) \dd x
= \frac{\ee^{a x}}{a^2 +b^2} (a \sin(b x) -b \cos(b x))
=: \mathrm{ES}(x;a,b),
\end{equation*}
we have
\begin{align}
\widehat x_{l_1,l_2}(t_i) 
&= \sum_{k = 1}^{M_2} \sum_{j = 1}^{M_1} 
X_{t_i}(y_{j}, z_{k}) \iint_{D_{j,k}} 
e_{l_1}^{(1)}(y;\widehat \kappa)
e_{l_2}^{(2)}(z;\widehat \eta) 
\exp(\widehat \kappa y +\widehat \eta z) \dd y \dd z
\nonumber
\\
&= \sum_{k = 1}^{M_2} \sum_{j = 1}^{M_1}  X_{t_i}(y_{j}, z_{k})
\int_{y_{j-1}}^{y_{j}} 
\sqrt{2} \exp \Bigl(\frac{\widehat \kappa y}{2} \Bigr) \sin(\pi l_1 y) \dd y
\nonumber
\\
&\qquad \times
\int_{z_{k-1}}^{z_{k}} 
\sqrt{2} \exp \Bigl(\frac{\widehat \eta z}{2} \Bigr) \sin(\pi l_2 z) \dd z
\nonumber
\\
&= \sum_{k = 1}^{M_2} \sum_{j = 1}^{M_1}  X_{t_i}(y_{j}, z_{k})
\delta_{j}^{(1)} g_{l_1}(\widehat \kappa) 
\delta_{k}^{(2)} g_{l_2}(\widehat \eta),
\label{approx_cp}
\end{align}
where 
\begin{align*}
\delta_j^{(1)} g_l(\widehat \kappa) &= g_l(y_{j};\widehat \kappa) -g_l(y_{j-1};\widehat \kappa),
\quad
\delta_k^{(2)} g_l(\widehat \eta) = g_l(z_{k};\widehat \eta) - g_l(z_{k-1};\widehat \eta),
\\
g_l(x;a) &= \sqrt{2} \mathrm{ES} \Bigl(x; \frac{a}{2}, \pi l \Bigr),
\quad a, x \in \mathbb R, \ l \in \mathbb N.
\end{align*}

Tonaki \textit{et al}.\cite{TKU2024a} considered parameter estimation for the Ornstein-Uhlenbeck process 
\begin{equation}\label{process}
\dd x(t) = -\lambda x(t) \dd t +\epsilon \mu^{-\alpha/2} \dd w(t)
\end{equation}
with a deterministic initial value $x(0) \neq 0$ and a small volatility parameter $\epsilon \in (0,1)$, 
where $(\lambda, \mu) \in \mathbb R \times (0,\infty)$ are unknown parameters, 
$\alpha \in (0, 3)$ and $\{ w(t) \}_{t \ge 0}$ is a real-valued standard Wiener process.
For discrete observations $\mb x = \{ x(t_{i}^n ) \}_{i=1}^n$ obtained from the process \eqref{process},
they proposed the quasi-maximum likelihood estimators of $(\lambda, \mu)$ based on the quasi log-likelihood function
\begin{equation}\label{contrast_V}
V_n^{\epsilon} (\lambda, \mu; \mb x, \alpha)
=\sum_{i=1}^n
\frac{(x(t_i^n)- \ee^{-\lambda \Delta_n} x(t_{i-1}^n))^2}
{\frac{\epsilon^2 (1 -\ee^{-2\lambda \Delta_n})}{2 \lambda \mu^\alpha}}
+n\log \biggl( \frac{1-\ee^{-2\lambda \Delta_n}}{2\lambda \mu^\alpha \Delta_n} \biggr).
\end{equation}
See the appendix section in \cite{TKU2024a} for details.

In this study, for $(\mf l_1, \mf l_2) \in \mathbb N^2$ obtained from [A2], 
we construct the approximate coordinate process 
$\widehat {\mb x}_{\mf l_1,\mf l_2} 
= \{ \widehat x_{\mf l_1, \mf l_2}(t_{i}^n) \}_{i=1}^n$ 
by \eqref{approx_cp}, 
and consider the quasi log-likelihood function obtained 
by substituting the process $\widehat {\mb x}_{\mf l_1,\mf l_2}$ 
and the estimator $\widehat \alpha$ given in \eqref{est_alpha} 
into \eqref{contrast_V}, and define the estimators of 
$(\theta_0, \mu_0)$ as follows.
\begin{equation}\label{estimator_lam_mu}
\left\{
\begin{split}
(\widehat \lambda_{\mf l_1,\mf l_2}, \widehat \mu_{\mf l_1,\mf l_2})
&= \underset{(\lambda,\mu) \in \Xi_{\nu}}{\mathrm{arginf}}\, 
V_n^{\epsilon} (\lambda,\mu; 
\widehat {\mb x}_{\mf l_1,\mf l_2}, \widehat \alpha),
\\
\widehat \theta_0 = \widehat \theta_{0,\mf l_1, \mf l_2}
&= -\widehat \lambda_{\mf l_1,\mf l_2}
+\widehat \theta_2 \biggl( \frac{\widehat \kappa^2 + \widehat \eta^2}{4} 
+\pi^2(\mf l_1^2 +\mf l_2^2) \biggr),
\\
\widehat \mu_0 = \widehat \mu_{0,\mf l_1, \mf l_2} 
&= \widehat \mu_{\mf l_1,\mf l_2}- \pi^2(\mf l_1^2 +\mf l_2^2),
\end{split}
\right.
\end{equation}
where $\Xi_{\lambda}, \Xi_{\mu}$ are compact convex subsets of $(0,\infty)$ and 
$\Xi_{\nu} = \Xi_{\lambda} \times \Xi_{\mu}$. 
For $(\lambda, \mu) \in (0,\infty)^2$, 
$x \in \mathbb R \setminus \{0\}$ and $\alpha \in (0,3)$, we define 
\begin{equation*}
G(\lambda,\mu;x,\alpha) = 
\frac{1 -\ee^{-2\lambda}}{2\lambda} \mu^\alpha x^2,
\quad 
H(\mu;\alpha) = \frac{\alpha^2}{2\mu^2},
\quad
I(\lambda, \mu; x,\alpha) = 
\begin{pmatrix}
G(\lambda, \mu; x,\alpha) & 0
\\
0 & H(\mu;\alpha)
\end{pmatrix}
\end{equation*}
and
\begin{equation*}
\mathcal I = I(\lambda_{\mf l_1, \mf l_2}^*,
\mu_{\mf l_1, \mf l_2}^*; x_{\mf l_1, \mf l_2}(0), \alpha^*)^{-1}.
\end{equation*}
Let $M_{(1)} = M_1 \land M_2$.
For the rate $R_{\alpha,\alpha_0}^{\mathrm{damp}}$ given in \eqref{R_damp}, we define
\begin{equation*}
R_{\alpha,\alpha_0}^{\mathrm{coef}} = \frac{R_{\alpha,\alpha_0}^{\mathrm{damp}}}{\log(N)}.
\end{equation*}
In order to control the approximate process $\widehat {\mb x}_{l_1,l_2}$, we consider the following conditions. 
\begin{description}
\item[{[C1]$_{\alpha,\alpha_0}$}]
The following balance conditions hold for $n$, $\epsilon$ and $R_{\alpha,\alpha_0}^{\mathrm{coef}}$. 
\begin{enumerate}
\item[(i)]
$\frac{n^3 \epsilon^2 \Delta_n^{\alpha_0 \tand 2}}{(R_{\alpha,\alpha_0}^{\mathrm{coef}})^2} = \oo(1)$
and $\frac{\epsilon^{-4} \Delta_n^{\alpha_0 \tand 2}}{(R_{\alpha,\alpha_0}^{\mathrm{coef}})^2} = \oo(1)$.

\item[(ii)]
$\frac{n^3 \epsilon^4 \Delta_n^{\alpha \tand 1}}{(R_{\alpha,\alpha_0}^{\mathrm{coef}})^2} = \oo(1)$, 
$\frac{\epsilon^{-2} \Delta_n^{\alpha \tand 1}}{(R_{\alpha,\alpha_0}^{\mathrm{coef}})^2} = \oo(1)$
and $\frac{n^2 \Delta_n^{\alpha \tand 1}}{(R_{\alpha,\alpha_0}^{\mathrm{coef}})^2} = \oo(1)$.

\item[(iii)]
$\frac{\epsilon^{-1}}{R_{\alpha,\alpha_0}^{\mathrm{coef}}} = \oo(1)$.
\end{enumerate}

\item[{[C2]$_{\alpha,\alpha_0}$}]
The following balance conditions hold for $n$, $\epsilon$ and $M_{(1)}$.
\begin{enumerate}
\item[(i)]
$\frac{n^3 \epsilon^2}{M_{(1)}^{2\alpha_0 \tand 2}} = \oo(1)$ 
and $\frac{\epsilon^{-4}}{M_{(1)}^{2\alpha_0 \tand 2}} = \oo(1)$.

\item[(ii)]
$\frac{n^3 \epsilon^4}{M_{(1)}^{2\alpha \tand 2}} = \oo(1)$ 
and $\frac{\epsilon^{-2}}{M_{(1)}^{2\alpha \tand 2}} = \oo(1)$.

\item[(iii)]
$\frac{n \sqrt{n}}{M_{(1)}^{2\alpha \tand 2}} = \oo(1)$.
\end{enumerate}

\item[{[C3]$_{\alpha,\alpha_0}$}]
The following balance conditions hold for $n$, $\epsilon$ and $R_{\alpha,\alpha_0}^{\mathrm{damp}}$.
\begin{enumerate}
\item[(i)]
$\frac{n \epsilon^2}{R_{\alpha,\alpha_0}^{\mathrm{damp}}} = \oo(1)$ 
and $\frac{(n \epsilon^2)^{-1}}{R_{\alpha,\alpha_0}^{\mathrm{damp}}} = \oo(1)$. 

\item[(ii)]
$\frac{\sqrt{n}}{R_{\alpha,\alpha_0}^{\mathrm{damp}}} = \oo(1)$.
\end{enumerate}
\end{description}

We then obtain the following theorem.
\begin{thm}\label{th3}
Let $\alpha_0 \in (0,3)$ and $\delta/\sqrt{\Delta} \equiv r \in (0,\infty)$.
Assume that [A1]$_{\alpha_0}$, [A2] and [B]$_{\alpha^*,\alpha_0}$ hold. 
As $n \to \infty$ and $\epsilon \to 0$, it holds that
under [C1]$_{\alpha^*, \alpha_0}$-(i), (ii), [C2]$_{\alpha^*, \alpha_0}$-(i), (ii) 
and [C3]$_{\alpha^*, \alpha_0}$-(i), 
\begin{equation}\label{cons}
(\widehat \theta_0, \widehat \mu_0) \pto (\theta_0^*, \mu_0^*),
\end{equation}
and that under [C1]$_{\alpha^*, \alpha_0}$--[C3]$_{\alpha^*, \alpha_0}$,  
\begin{equation}\label{asym_norm}
\begin{pmatrix}
\epsilon^{-1}(\widehat \theta_0 - \theta_0^*)
\\
\sqrt n(\widehat \mu_0-\mu_0^*)
\end{pmatrix}
\dto \mathrm{N} (0, \mathcal I).
\end{equation}
\end{thm}

\begin{rmk}
\begin{itemize}
\item[(i)]
The error between the approximate coordinate process given in \eqref{approx_cp}
and the true coordinate process given in \eqref{OU-process} consists of three factors: 
the estimation error of $\alpha$, the estimation errors of $\kappa$ and $\eta$, and the spatial discretization error.
[C1]$_{\alpha, \alpha_0}$--[C3]$_{\alpha, \alpha_0}$ are the conditions that control 
the estimation error of $\kappa$ and $\eta$, the spatial discretization error, 
and the estimation error of $\alpha$, respectively. See Remarks \ref{rmk_W} and \ref{rmk_AB} below for details.

\item[(ii)]
In this paper, we allow for less restrictive initial conditions.
Consequently, the effect of the initial value may propagate to the conditions 
[C1]$_{\alpha^*, \alpha_0}$, [C2]$_{\alpha^*, \alpha_0}$ and [C3]$_{\alpha^*, \alpha_0}$.
However, under ideal conditions (for example, the case where $\alpha_0 > 2$ and 
$R_{\alpha,\alpha_0}^{\mathrm{damp}} = \sqrt{m N}$ are achieved), 
[C1]$_{\alpha^*, \alpha_0}$, [C2]$_{\alpha^*, \alpha_0}$ and [C3]$_{\alpha^*, \alpha_0}$ are independent of the value of $\alpha_0$.

\item[(iii)]
It is difficult to determine the optimal choice of $(\mathfrak l_1,\mathfrak l_2)$ in [A2] since $G(\lambda, \mu; x, \alpha)$ 
is decreasing with respect to $\lambda \in (0,\infty)$ for any $\mu, x \in (0,\infty)$ and $\alpha \in (0,3)$, 
while is increasing with respect to $\mu,x \in (0,\infty)$ for any $\lambda \in (0,\infty)$ and $\alpha \in (0,\infty)$.  
Nevertheless, selecting excessively large values of $\mf l_1$ and $\mf l_2$ is not desirable 
because there exists a constant $C>0$ 
such that $\mu_{l_1, l_2}^{\alpha^*} x_{l_1, l_2}(0)^2 \le C \lambda_{l_1,l_2}^{1+\alpha_0} x_{l_1,l_2}(0)^2$,
and hence 
\begin{equation*}
G(\lambda_{l_1,l_2}, \mu_{l_1,l_2};x_{l_1,l_2}(0), \alpha^*) 
\le C \frac{1 -\ee^{-2\lambda_{l_1,l_2}}}{2\lambda_{l_1,l_2}} \cdot \lambda_{l_1,l_2}^{1+\alpha_0} x_{l_1,l_2}(0)^2
\to 0
\end{equation*}
as $l_1 ,l_2 \to \infty$, provided that $[A1]_{\alpha_0}$ with $\alpha_0 > \alpha^* - 1$ is satisfied.
It is also a direction for future work to improve the estimators in the case where there are 
multiple pairs $(\mf l_1,\mf l_2)\in\mathbb{N}^2$ satisfying [A2].

\item[(iv)]
We consider the following condition, which is stricter than [A2].
\begin{description}
\item[{[A2]$'$}]
There exist two indices $(\mf p_1, \mf q_1), (\mf p_2, \mf q_2) \in \mathbb N^2$ 
such that $(\mf p_1, \mf q_1) \neq (\mf p_2, \mf q_2)$ and $\langle X_0, e_{\mf p_j, \mf q_j} \rangle \neq 0$ for $j \in \{ 1, 2 \}$.
\end{description}
We now define the estimators $\widehat \theta^{\mathrm{w}} = (\widehat \theta_0^{\mathrm{w}}, \widehat \theta_1^{\mathrm{w}}, 
\widehat \eta_1^{\mathrm{w}}, \widehat \theta_2^{\mathrm{w}})^\TT$ 
and $\widehat \mu_0^{\mathrm{w}}$ of $\theta = (\theta_0, \theta_1, \eta_1, \theta_2)^\TT$ and $\mu_0$ by
\begin{equation*}
\widehat \theta_2^{\mathrm{w}} = \frac{\widehat \lambda_{\mf p_1, \mf q_1} -\widehat \lambda_{\mf p_2, \mf q_2}}
{\pi^2 (\mf p_1^2 +\mf q_1^2 -\mf p_2^2 -\mf q_2^2)},
\quad
\widehat \theta_1^{\mathrm{w}} = \widehat \theta_1 \frac{\widehat \theta_2^{\mathrm{w}}}{\widehat \theta_2},
\quad
\widehat \eta_1^{\mathrm{w}} = \widehat \eta_1 \frac{\widehat \theta_2^{\mathrm{w}}}{\widehat \theta_2},
\end{equation*}
$\widehat \theta_0^{\mathrm{w}} = \widehat \theta_{0,\mf p_1, \mf q_1}$
and $\widehat \mu_0^{\mathrm{w}} = \widehat \mu_{0,\mf p_1, \mf q_1}$
based on the estimators $(\widehat \lambda_{\mf p_1, \mf q_1}, \widehat \lambda_{\mf p_2, \mf q_2})$ 
obtained from \eqref{estimator_lam_mu}, where $a^\TT$ denotes the transpose of $a$.
It then holds that as $n \to \infty$ and $\epsilon \to 0$, under [C1]$_{\alpha^*, \alpha_0}$, [C2]$_{\alpha^*, \alpha_0}$
and [C3]$_{\alpha^*, \alpha_0}$, 
\begin{equation*}
\begin{pmatrix}
\epsilon^{-1} (\widehat \theta^{\mathrm{w}} -\theta^*)
\\
\sqrt{n} (\widehat \mu_0^{\mathrm{w}} -\mu_0^*)
\end{pmatrix}
\dto \mathrm{N} (0, \mathcal J)
\end{equation*}
with the degenerate variance-covariance matrix
\begin{equation*}
\mathcal J = 
\begin{pmatrix}
\mb G_{\mf p_1, \mf q_1:\mf p_2, \mf q_2} & O
\\
O & H(\mu_{\mf p_1, \mf q_1}; \alpha^*)^{-1}
\end{pmatrix}
\in \mathbb R^{5 \otimes 5},
\end{equation*}
where
\begin{equation*}
\mb v = \frac{1}{\pi^2 (\mf p_1^2 +\mf q_1^2 -\mf p_2^2 -\mf q_2^2) \theta_2^*} (\theta_1^*, \eta_1^*, \theta_2^*),
\quad
G_{j} = G(\lambda_{\mf p_j, \mf q_j}, \mu_{\mf p_j, \mf q_j}; x_{\mf p_j, \mf q_j}(0), \alpha^*)
\end{equation*}
for $j \in \{ 1, 2 \}$ and
\begin{equation*}
\quad
\mb G_{\mf p_1, \mf q_1:\mf p_2, \mf q_2} = 
\begin{pmatrix}
G_1^{-1} & -G_1^{-1} \mb v
\\
-G_1^{-1} \mb v^\TT & (G_1^{-1} +G_2^{-1}) \mb v^\TT \mb v
\end{pmatrix}.
\end{equation*}
For this proof, see the proof of Theorem 3.3 or 3.6 in \cite{TKU2024a}.
\end{itemize}
\end{rmk}

We conclude this section by providing the following quick reference, which yields the dominant conditions 
for the sufficient conditions [C1]$_{\alpha^*,\alpha_0}$, [C2]$_{\alpha^*,\alpha_0}$ and [C3]$_{\alpha^*,\alpha_0}$ 
under the balance conditions between each $n$ and $\epsilon$.

For $a > b \ge 0$, we label the situation where the balance conditions 
$n \epsilon^a = \oo(1)$ and $\frac{1}{n \epsilon^b} = \OO(1)$ are satisfied as $\mathrm{Bal}(a,b)$.
For convenience, we denote $\mathrm{Bal}(\infty,b)$ as the situation where $\frac{1}{n \epsilon^b} = \OO(1)$ is satisfied.
From Lemma \ref{lem_bal} below, we can summarize parts of the sufficient conditions for \eqref{cons} and \eqref{asym_norm} 
in Theorem \ref{th3} as shown in Table \ref{tab0}.
\begin{table}[h]
\caption{Quick reference of the sufficient conditions for Theorem \ref{th3}.}
\label{tab0}
\begin{center}
\begin{tikzpicture}
  \def\cellw{2.5}
  \def\cellh{0.8}

  \draw (0,0) rectangle (6*\cellw,-5*\cellh);

  \foreach \i in {1}{
    \draw (0,-\i*\cellh) -- (6*\cellw,-\i*\cellh);
  }
  \foreach \i in {2,3,4}{
    \draw (0,-\i*\cellh) -- (0.98*\cellw,-\i*\cellh);
  }
  \foreach \i in {3}{
    \draw (\cellw,-\i*\cellh) -- (2*\cellw,-\i*\cellh);
  }
  \foreach \i in {2,4}{
    \draw (2*\cellw,-\i*\cellh) -- (3*\cellw,-\i*\cellh);
  }
  \foreach \i in {3}{
    \draw (3*\cellw,-\i*\cellh) -- (6*\cellw,-\i*\cellh);
  }
  \foreach \j in {0.98,1,2,3,4,5}{
    \draw (\j*\cellw,0) -- (\j*\cellw,-5*\cellh);
  }
  \node at ({(2-0.5)*\cellw},{-(1-0.5)*\cellh}) {[C1]$_{\alpha,\alpha_0}$-(i)};
  \node at ({(3-0.5)*\cellw},{-(1-0.5)*\cellh}) {[C1]$_{\alpha,\alpha_0}$-(ii)};
  \node at ({(4-0.5)*\cellw},{-(1-0.5)*\cellh}) {[C2]$_{\alpha,\alpha_0}$-(i)};
  \node at ({(5-0.5)*\cellw},{-(1-0.5)*\cellh}) {[C2]$_{\alpha,\alpha_0}$-(ii)};
  \node at ({(6-0.5)*\cellw},{-(1-0.5)*\cellh}) {[C3]$_{\alpha,\alpha_0}$-(i)};

  \node at ({(1-0.5)*\cellw},{-(2-0.5)*\cellh}) {$\mathrm{Bal}(\infty,4)$};
  \node at ({(1-0.5)*\cellw},{-(3-0.5)*\cellh}) {$\mathrm{Bal}(4,2)$};
  \node at ({(1-0.5)*\cellw},{-(4-0.5)*\cellh}) {$\mathrm{Bal}(2,1)$};
  \node at ({(1-0.5)*\cellw},{-(5-0.5)*\cellh}) {$\mathrm{Bal}(1,0)$};

  \node at ({(2-0.5)*\cellw},{-(2.5-0.5)*\cellh}) 
  {{\scriptsize$\frac{n^3 \epsilon^2 \Delta_n^{\alpha_0 \tand 2}}{(R_{\alpha,\alpha_0}^{\mathrm{coef}})^2} = \oo(1)$}};
  \node at ({(2-0.5)*\cellw},{-(4.5-0.5)*\cellh}) 
  {{\scriptsize$\frac{\epsilon^{-4} \Delta_n^{\alpha_0 \tand 2}}{(R_{\alpha,\alpha_0}^{\mathrm{coef}})^2} = \oo(1)$}};

  \node at ({(3-0.5)*\cellw},{-(2-0.5)*\cellh}) 
  {{\scriptsize$\frac{n^3 \epsilon^4 \Delta_n^{\alpha \tand 1}}{(R_{\alpha,\alpha_0}^{\mathrm{coef}})^2} = \oo(1)$}};
  \node at ({(3-0.5)*\cellw},{-(3.5-0.5)*\cellh}) 
  {{\scriptsize$\frac{n^2 \Delta_n^{\alpha \tand 1}}{(R_{\alpha,\alpha_0}^{\mathrm{coef}})^2} = \oo(1)$}};
  \node at ({(3-0.5)*\cellw},{-(5-0.5)*\cellh}) 
  {{\scriptsize$\frac{\epsilon^{-2} \Delta_n^{\alpha \tand 1}}{(R_{\alpha,\alpha_0}^{\mathrm{coef}})^2} = \oo(1)$}};

  \node at ({(4-0.5)*\cellw},{-(2.5-0.5)*\cellh}) 
  {{\scriptsize$\frac{n^3 \epsilon^2}{M_{(1)}^{2\alpha_0 \tand 2}} = \oo(1)$}};
  \node at ({(4-0.5)*\cellw},{-(4.5-0.5)*\cellh}) 
  {{\scriptsize$\frac{\epsilon^{-4}}{M_{(1)}^{2\alpha_0 \tand 2}} = \oo(1)$}};

  \node at ({(5-0.5)*\cellw},{-(2.5-0.5)*\cellh}) 
  {{\scriptsize$\frac{n^3 \epsilon^4}{M_{(1)}^{2\alpha \tand 2}} = \oo(1)$}};
  \node at ({(5-0.5)*\cellw},{-(4.5-0.5)*\cellh}) 
  {{\scriptsize$\frac{\epsilon^{-2}}{M_{(1)}^{2\alpha \tand 2}} = \oo(1)$}};

  \node at ({(6-0.5)*\cellw},{-(2.5-0.5)*\cellh}) 
  {{\scriptsize$\frac{n \epsilon^2}{R_{\alpha,\alpha_0}^{\mathrm{damp}}} = \oo(1)$}};
  \node at ({(6-0.5)*\cellw},{-(4.5-0.5)*\cellh}) 
  {{\scriptsize$\frac{(n \epsilon^2)^{-1}}{R_{\alpha,\alpha_0}^{\mathrm{damp}}} = \oo(1)$}};
\end{tikzpicture}
\end{center}

\begin{center}
\begin{tikzpicture}
  \def\cellw{2.5}
  \def\cellh{0.8}

  \draw (0,0) rectangle (6*\cellw,-7*\cellh);

  \foreach \i in {1}{
    \draw (0,-\i*\cellh) -- (6*\cellw,-\i*\cellh);
  }
  \foreach \i in {2,3,4,5,6}{
    \draw (0,-\i*\cellh) -- (0.98*\cellw,-\i*\cellh);
  }
  \foreach \i in {4}{
    \draw (\cellw,-\i*\cellh) -- (2*\cellw,-\i*\cellh);
  }
  \foreach \i in {2,6}{
    \draw (2*\cellw,-\i*\cellh) -- (3*\cellw,-\i*\cellh);
  }
  \foreach \i in {4}{
    \draw (3*\cellw,-\i*\cellh) -- (4*\cellw,-\i*\cellh);
  }
  \foreach \i in {3,5}{
    \draw (4*\cellw,-\i*\cellh) -- (5*\cellw,-\i*\cellh);
  }
  \foreach \i in {2,5}{
    \draw (5*\cellw,-\i*\cellh) -- (6*\cellw,-\i*\cellh);
  }
  \foreach \j in {0.98,1,2,3,4,5}{
    \draw (\j*\cellw,0) -- (\j*\cellw,-7*\cellh);
  }
  \node at ({(2-0.5)*\cellw},{-(1-0.5)*\cellh}) {[C1]$_{\alpha,\alpha_0}$-(i)};
  \node at ({(3-0.5)*\cellw},{-(1-0.5)*\cellh}) {[C1]$_{\alpha,\alpha_0}$-(ii)};
  \node at ({(4-0.5)*\cellw},{-(1-0.5)*\cellh}) {[C2]$_{\alpha,\alpha_0}$-(i)};
  \node at ({(5-0.5)*\cellw},{-(1-0.5)*\cellh}) {{\small [C2]$_{\alpha,\alpha_0}$-(ii),(iii)}};
  \node at ({(6-0.5)*\cellw},{-(1-0.5)*\cellh}) {[C3]$_{\alpha,\alpha_0}$};

  \node at ({(1-0.5)*\cellw},{-(2-0.5)*\cellh}) {$\mathrm{Bal}(\infty,4)$};
  \node at ({(1-0.5)*\cellw},{-(3-0.5)*\cellh}) {$\mathrm{Bal}(4,8/3)$};
  \node at ({(1-0.5)*\cellw},{-(4-0.5)*\cellh}) {$\mathrm{Bal}(8/3,2)$};
  \node at ({(1-0.5)*\cellw},{-(5-0.5)*\cellh}) {$\mathrm{Bal}(2,4/3)$};
  \node at ({(1-0.5)*\cellw},{-(6-0.5)*\cellh}) {$\mathrm{Bal}(4/3,1)$};
  \node at ({(1-0.5)*\cellw},{-(7-0.5)*\cellh}) {$\mathrm{Bal}(1,0)$};

  \node at ({(2-0.5)*\cellw},{-(3-0.5)*\cellh}) 
  {{\scriptsize$\frac{n^3 \epsilon^2 \Delta_n^{\alpha_0 \tand 2}}{(R_{\alpha,\alpha_0}^{\mathrm{coef}})^2} = \oo(1)$}};
  \node at ({(2-0.5)*\cellw},{-(6-0.5)*\cellh}) 
  {{\scriptsize$\frac{\epsilon^{-4} \Delta_n^{\alpha_0 \tand 2}}{(R_{\alpha,\alpha_0}^{\mathrm{coef}})^2} = \oo(1)$}};

  \node at ({(3-0.5)*\cellw},{-(2-0.5)*\cellh}) 
  {{\scriptsize$\frac{n^3 \epsilon^4 \Delta_n^{\alpha \tand 1}}{(R_{\alpha,\alpha_0}^{\mathrm{coef}})^2} = \oo(1)$}};
  \node at ({(3-0.5)*\cellw},{-(4.5-0.5)*\cellh}) 
  {{\scriptsize$\frac{n^2 \Delta_n^{\alpha \tand 1}}{(R_{\alpha,\alpha_0}^{\mathrm{coef}})^2} = \oo(1)$}};
  \node at ({(3-0.5)*\cellw},{-(7-0.5)*\cellh}) 
  {{\scriptsize$\frac{\epsilon^{-2} \Delta_n^{\alpha \tand 1}}{(R_{\alpha,\alpha_0}^{\mathrm{coef}})^2} = \oo(1)$}};

  \node at ({(4-0.5)*\cellw},{-(3-0.5)*\cellh}) 
  {{\scriptsize$\frac{n^3 \epsilon^2}{M_{(1)}^{2\alpha_0 \tand 2}} = \oo(1)$}};
  \node at ({(4-0.5)*\cellw},{-(6-0.5)*\cellh}) 
  {{\scriptsize$\frac{\epsilon^{-4}}{M_{(1)}^{2\alpha_0 \tand 2}} = \oo(1)$}};

  \node at ({(5-0.5)*\cellw},{-(2.5-0.5)*\cellh}) 
  {{\scriptsize$\frac{n^3 \epsilon^4}{M_{(1)}^{2\alpha \tand 2}} = \oo(1)$}};
  \node at ({(5-0.5)*\cellw},{-(4.5-0.5)*\cellh}) 
  {{\scriptsize$\frac{n \sqrt{n}}{M_{(1)}^{2\alpha \tand 2}} = \oo(1)$}};
  \node at ({(5-0.5)*\cellw},{-(6.5-0.5)*\cellh}) 
  {{\scriptsize$\frac{\epsilon^{-2}}{M_{(1)}^{2\alpha \tand 2}} = \oo(1)$}};

  \node at ({(6-0.5)*\cellw},{-(2-0.5)*\cellh}) 
  {{\scriptsize$\frac{n \epsilon^2}{R_{\alpha,\alpha_0}^{\mathrm{damp}}} = \oo(1)$}};
  \node at ({(6-0.5)*\cellw},{-(4-0.5)*\cellh}) 
  {{\scriptsize$\frac{\sqrt{n}}{R_{\alpha,\alpha_0}^{\mathrm{damp}}} = \oo(1)$}};
  \node at ({(6-0.5)*\cellw},{-(6.5-0.5)*\cellh}) 
  {{\scriptsize$\frac{(n \epsilon^2)^{-1}}{R_{\alpha,\alpha_0}^{\mathrm{damp}}} = \oo(1)$}};
\end{tikzpicture}
\end{center}
\end{table}

\section{Simulation results}\label{sec4}
A numerical solution of SPDE \eqref{2d_spde} is generated by using the truncation method as follows.
\begin{equation*}
\widetilde X_{t_{i}}(y_j, z_k)
= \sum_{l_1=1}^{L_1} \sum_{l_2=1}^{L_2} 
\widetilde x_{l_1,l_2}(t_{i}) e_{l_1,l_2}(y_j, z_k)
\end{equation*}
for $i \in \{ 1, \ldots, N \}$, $j \in \{ 1, \ldots, M_1 \}$ 
and $k \in \{ 1, \ldots, M_2 \}$ with
\begin{equation*}
\left\{
\begin{split}
\widetilde x_{l_1,l_2}(0) &= \langle X_0, e_{l_1,l_2} \rangle,
\\
\widetilde x_{l_1,l_2}(t_i) &= 
\ee^{-\lambda_{l_1,l_2}/N} \widetilde x_{l_1,l_2}(t_{i-1})
+ \epsilon \sqrt{\frac{1-\ee^{-2\lambda_{l_1,l_2}/N}}
{2 \lambda_{l_1,l_2} \mu_{l_1,l_2}^{\alpha}}} Z_{i,l_1,l_2},
\quad 
i \in \{ 1, \ldots, N \},
\end{split}
\right.
\end{equation*}
where $\{ Z_{i,l_1,l_2} \}$ are independent standard normal random variables.
We set $N = 10^3$, $M_1 = M_2 = 200$, $L_1 = L_2 = 10^4$, 
$X_0 = 3 e_{1,1}$, $\epsilon \in \{0.1, 0.5 \}$ and the true values of parameters 
$(\theta_0^*, \theta_1^*, \eta_1^*, \theta_2^*) = (0, 0.2, 0.2, 0.2)$, 
$\mu_0^* = -19.5$ and $\alpha^* = 0.5$.
We use R language to compute the estimators of Theorems \ref{th1}--\ref{th3}.
The number of Monte Carlo iterations is $200$.

We conduct numerical simulations according to the following steps. 
First, we estimate the damping parameter $\alpha$ using the estimator $\widehat \alpha$ given in \eqref{est_alpha} 
constructed from the spatial thinned data $\mathbb X_{\mathrm{m},N}^{b}$ with $m_1 = m_2 = m_0$, $b = b_0$, 
\begin{equation*}
(m_0, b_0) \in \mathbf{S} = \{ (200, 0.005), (100, 0.01), (50, 0.02), (30, 0.033) \}
\end{equation*}
and $N = 10^3$.
Next, we estimate the parameters $\theta_1$, $\eta_1$ and $\theta_2$ 
by the estimators $\widehat \theta_1$, $\widehat \eta_1$ and $\widehat \theta_2$ given in \eqref{est_coef}
based on the spatial thinned data $\mathbb X_{\mathrm{m},N}^{b}$ with $m_1 = m_2 = m_0$, $b = b_0$, 
$(m_0, b_0) \in \mathbf{S}$ and $N = 10^3$.
Finally, we calculate the estimators $\widehat \theta_0 = \widehat \theta_{0,1,1}$ and $\widehat \mu_0 = \widehat \mu_{0,1,1}$
given in \eqref{estimator_lam_mu} based on the temporal thinned data $\mathbb X_{\mathrm{M},n}^{0}$ 
with $M_1 = M_2 = 200$ and $n = 100$.
Tables \ref{tab1} and \ref{tab3} show the simulation results of the means and the standard deviations of 
$\widehat \alpha$, $\widehat \theta_1$, $\widehat \eta_1$, $\widehat \theta_2$, 
$\widehat \theta_0$ and $\widehat \mu_0$ for $\epsilon \in \{ 0.1, 0.5\}$, respectively.

We also calculate the estimators $\widehat \alpha^{\mathrm{p}}$, 
$\widehat \theta_1^{\mathrm{p}}$, $\widehat \eta_1^{\mathrm{p}}$, $\widehat \theta_2^{\mathrm{p}}$, 
$\widehat \theta_0^{\mathrm{p}}$ and $\widehat \mu_0^{\mathrm{p}}$
of $\alpha$, $\theta_1$, $\eta_1$, $\theta_2$, $\theta_0$ and $\mu_0$ given below according to the same steps as above
in order to compare them with the proposed estimators. 
We first compute the estimator $\widehat \alpha^{\mathrm{p}}$ of $\alpha$ 
proposed by \cite{TKU2025arXiv1} based on two thinned data 
$\mathbb X_{\mathrm{m},N}^{b}$ and $\mathbb X_{\mathrm{m}/2,N/4}^{b}$ with $m_1 = m_2 = m_0$, $b = b_0$,
$(m_0, b_0) \in \mathbf{S}$ and $N = 10^3$.
Next, by using the spatial thinned data $\mathbb X_{\mathrm{m},N}^{b}$ with $m_1 = m_2 = m_0$, $b = b_0$, 
$(m_0, b_0) \in \mathbf{S}$ and $N = 10^3$,  
we calculate the estimators $\widehat \theta_1^{\mathrm{p}}$, $\widehat \eta_1^{\mathrm{p}}$ and 
$\widehat \theta_2^{\mathrm{p}}$ of $\theta_1$, $\eta_1$ and $\theta_2$ 
based on the contrast function obtained by substituting $\widehat \alpha^{\mathrm{p}}$ instead of $\widehat \alpha$ 
into \eqref{contrast_U} as follows.
\begin{equation*}
(\widehat \kappa^{\mathrm{p}}, \widehat \eta^{\mathrm{p}}, \widehat \theta_2^{\mathrm{p}})
= \underset{(\kappa, \eta, \theta_2) \in \Xi_{\vartheta}}{\mathrm{argmin}}\,
U_{m,N}^{\epsilon}(\kappa, \eta, \theta_2; \widehat \alpha^{\mathrm{p}}),
\quad
\widehat \theta_1^{\mathrm{p}} = \widehat \kappa^{\mathrm{p}} \widehat \theta_2^{\mathrm{p}},
\quad
\widehat \eta_1^{\mathrm{p}} = \widehat \eta^{\mathrm{p}} \widehat \theta_2^{\mathrm{p}}.
\end{equation*}
We further construct the approximate coordinate process 
$\widehat {\mb x}_{1,1}^{\mathrm{p}} = \{ \widehat x_{1,1}^{\mathrm{p}}(t_{i}^n) \}_{i=1}^n$ 
by using \eqref{approx_cp}, $\widehat \kappa^{\mathrm{p}}$ and $\widehat \eta^{\mathrm{p}}$, that is, we define
\begin{equation*}
\widehat x_{1,1}^{\mathrm{p}}(t_i^n) 
= \sum_{k = 1}^{M_2} \sum_{j = 1}^{M_2} \widetilde X_{t_i^n}(y_{j}, z_{k})
\delta_{j}^{(1)} g_{1}(\widehat \kappa^{\mathrm{p}}) \delta_{k}^{(2)} g_{2}(\widehat \eta^{\mathrm{p}})
\end{equation*}
based on the temporal thinned data $\mathbb X_{\mathrm{M},n}^{0}$ with $M_1 = M_2 = 200$ and $n = 100$
and set the estimators $\widehat \theta_0^{\mathrm{p}}$ and $\widehat \mu_0^{\mathrm{p}}$ of $\theta_0$ and $\mu_0$ by
\begin{equation*}
(\widehat \lambda_{1,1}^{\mathrm{p}}, \widehat \mu_{1,1}^{\mathrm{p}})
= \underset{(\lambda,\mu) \in \Xi_{\nu}}{\mathrm{arginf}}\, 
V_n^{\epsilon} (\lambda,\mu; 
\widehat {\mb x}_{1,1}^{\mathrm{p}}, \widehat \alpha^{\mathrm{p}}),
\end{equation*}
\begin{equation*}
\widehat \theta_0^{\mathrm{p}} 
= -\widehat \lambda_{1,1}^{\mathrm{p}}
+\widehat \theta_2^{\mathrm{p}} \biggl( \frac{(\widehat \kappa^{\mathrm{p}})^2 + (\widehat \eta^{\mathrm{p}})^2}{4} 
+2\pi^2 \biggr),
\quad
\widehat \mu_0^{\mathrm{p}} 
= \widehat \mu_{1,1}^{\mathrm{p}} -2\pi^2.
\end{equation*}
Tables \ref{tab2} and \ref{tab4} show the simulation results of the means and the standard deviations of 
$\widehat \alpha^{\mathrm{p}}$, $\widehat \theta_1^{\mathrm{p}}$, $\widehat \eta_1^{\mathrm{p}}$, 
$\widehat \theta_2^{\mathrm{p}}$, $\widehat \theta_0^{\mathrm{p}}$ and $\widehat \mu_0^{\mathrm{p}}$ 
for $\epsilon \in \{ 0.1, 0.5 \}$, respectively.

We see from Tables \ref{tab1}-\ref{tab4} that for all $\epsilon \in \{ 0.1, 0.5 \}$, 
the biases of both the estimators $\widehat \alpha$ and $\widehat \alpha^{\mathrm{p}}$ decrease as $m_0$ increases, 
and that these biases affect the estimators of $\theta_1$, $\eta_1$, and $\theta_2$, as well as the estimator of $\theta_0$.
We also observe that for all scenarios $(m_0, b_0) \in \mathbf{S}$ and $\epsilon \in \{0.1, 0.5\}$,
the biases of $\widehat \alpha$, $\widehat \theta_1$, $\widehat \eta_1$, $\widehat \theta_2$
and $\widehat \theta_0$ are smaller than those of 
$\widehat \alpha^{\mathrm{p}}$, $\widehat \theta_1^{\mathrm{p}}$, $\widehat \eta_1^{\mathrm{p}}$, 
$\widehat \theta_2^{\mathrm{p}}$ and $\widehat \theta_0^{\mathrm{p}}$, respectively.
In particular, it appears that the proposed method substantially suppresses the increase in biases of the estimators 
for the coefficient parameters when $m_0$ is small.
Furthermore, we find from Tables \ref{tab1} and \ref{tab3} that
for all $(m_0, b_0) \in \mathbf{S}$, the mean squared errors of the estimator $\widehat \theta_0$ get smaller as $\epsilon$ decreases.

\begin{table}[h]
\caption{Means and standard deviations of the proposed estimators with $\epsilon = 0.1$.}
\label{tab1}
\begin{center}
\begin{tabular}{ccccccc} \hline
&$\widehat \alpha$ & $\widehat \theta_0$ & $\widehat \theta_1$ & $\widehat \eta_1$ & $\widehat \theta_2$ & $\widehat \mu_0$
\rule[0mm]{0cm}{4.5mm} \\ \cline{2-7}
$(m_0,b_0)$ 
& $\alpha^* = 0.5$ & $\theta_0^* = 0$ & $\theta_1^* = 0.2$ & $\eta_1^* = 0.2$ & $\theta_2^* = 0.2$ & $\mu_0^* = -19.5$ 
\rule[0mm]{0cm}{4.5mm} \\ \hline
$(200, 0.005)$ & $0.499$ & $0.077$ & $0.204$ & $0.204$ & $0.204$ & $-19.487$
 \\
& ($0.0004$) & ($0.115$) & ($0.001$) & ($0.001$) & ($0.001$) & ($0.064$)
 \\   \hline
$(100, 0.01)$ & $0.493$ & $0.401$ & $0.220$ & $0.220$ & $0.220$ & $-19.486$
 \\
& ($0.0007$) & ($0.128$) & ($0.002$) & ($0.002$) & ($0.002$) & ($0.064$)
 \\   \hline
$(50, 0.02)$ & $0.487$ & $0.667$ & $0.234$ & $0.234$ & $0.233$ & $-19.492$
 \\
& ($0.0017$) & ($0.203$) & ($0.008$) & ($0.008$) & ($0.008$) & ($0.064$)
 \\   \hline
$(30, 0.033)$ & $0.484$ & $1.046$ & $0.253$ & $0.253$ & $0.252$ & $-19.495$
 \\
& ($0.0031$) & ($0.200$) & ($0.009$) & ($0.009$) & ($0.009$) & ($0.064$)
 \\   \hline
\end{tabular}
\end{center}
\end{table}

\begin{table}[h]
\caption{Means and standard deviations of the comparative estimators with $\epsilon = 0.1$.}
\label{tab2}
\begin{center}
\begin{tabular}{ccccccc} \hline
&$\widehat \alpha^{\mathrm{p}}$ & $\widehat \theta_0^{\mathrm{p}}$ & $\widehat \theta_1^{\mathrm{p}}$ & $\widehat \eta_1^{\mathrm{p}}$ & $\widehat \theta_2^{\mathrm{p}}$ & $\widehat \mu_0^{\mathrm{p}}$
\rule[0mm]{0cm}{4.5mm} \\ \cline{2-7}
$(m_0,b_0)$ 
& $\alpha^* = 0.5$ & $\theta_0^* = 0$ & $\theta_1^* = 0.2$ & $\eta_1^* = 0.2$ & $\theta_2^* = 0.2$ & $\mu_0^* = -19.5$ 
\rule[0mm]{0cm}{4.5mm} \\ \hline
$(200, 0.005)$ & $0.497$ & $0.235$ & $0.211$ & $0.211$ & $0.211$ & $-19.484$
 \\
& ($0.001$) & ($0.144$) & ($0.004$) & ($0.004$) & ($0.004$) & ($0.077$)
 \\   \hline
$(100, 0.01)$ & $0.483$ & $1.023$ & $0.251$ & $0.251$ & $0.250$ & $-19.495$
 \\
& ($0.002$) & ($0.192$) & ($0.007$) & ($0.007$) & ($0.007$) & ($0.076$)
 \\   \hline
$(50, 0.02)$ & $0.467$ & $2.120$ & $0.303$ & $0.303$ & $0.304$ & $-19.503$
 \\
& ($0.005$) & ($0.382$) & ($0.015$) & ($0.015$) & ($0.018$) & ($0.076$)
 \\   \hline
$(30, 0.033)$ & $0.447$ & $4.327$ & $0.387$ & $0.387$ & $0.415$ & $-19.485$
 \\
& ($0.008$) & ($0.847$) & ($0.033$) & ($0.033$) & ($0.042$) & ($0.085$)
 \\   \hline
\end{tabular}
\end{center}
\end{table}

\begin{table}[h]
\caption{Means and standard deviations of the proposed estimators with $\epsilon = 0.5$.}
\label{tab3}
\begin{center}
\begin{tabular}{ccccccc} \hline
&$\widehat \alpha$ & $\widehat \theta_0$ & $\widehat \theta_1$ & $\widehat \eta_1$ & $\widehat \theta_2$ & $\widehat \mu_0$
\rule[0mm]{0cm}{4.5mm} \\ \cline{2-7}
$(m_0,b_0)$ 
& $\alpha^* = 0.5$ & $\theta_0^* = 0$ & $\theta_1^* = 0.2$ & $\eta_1^* = 0.2$ & $\theta_2^* = 0.2$ & $\mu_0^* = -19.5$ 
\rule[0mm]{0cm}{4.5mm} \\ \hline
$(200, 0.005)$ & $0.498$ & $0.105$ & $0.208$ & $0.208$ & $0.208$ & $-19.488$
 \\
& ($0.0017$) & ($0.715$) & ($0.017$) & ($0.017$) & ($0.017$) & ($0.069$)
 \\   \hline
$(100, 0.01)$ & $0.492$ & $0.410$ & $0.223$ & $0.223$ & $0.223$ & $-19.492$
 \\
& ($0.0069$) & ($0.711$) & ($0.016$) & ($0.016$) & ($0.017$) & ($0.069$)
 \\   \hline
$(50, 0.02)$ & $0.486$ & $0.656$ & $0.236$ & $0.236$ & $0.235$ & $-19.496$
 \\
& ($0.0046$) & ($0.669$) & ($0.012$) & ($0.012$) & ($0.012$) & ($0.069$)
 \\   \hline
$(30, 0.033)$ & $0.483$ & $1.076$ & $0.257$ & $0.257$ & $0.256$ & $-19.499$
 \\
& ($0.0041$) & ($0.660$) & ($0.011$) & ($0.011$) & ($0.011$) & ($0.069$)
 \\   \hline
\end{tabular}
\end{center}
\end{table}

\begin{table}[h]
\caption{Means and standard deviations of the comparative estimators with $\epsilon = 0.5$.}
\label{tab4}
\begin{center}
\begin{tabular}{ccccccc} \hline
&$\widehat \alpha^{\mathrm{p}}$ & $\widehat \theta_0^{\mathrm{p}}$ & $\widehat \theta_1^{\mathrm{p}}$ & $\widehat \eta_1^{\mathrm{p}}$ & $\widehat \theta_2^{\mathrm{p}}$ & $\widehat \mu_0^{\mathrm{p}}$
\rule[0mm]{0cm}{4.5mm} \\ \cline{2-7}
$(m_0,b_0)$ 
& $\alpha^* = 0.5$ & $\theta_0^* = 0$ & $\theta_1^* = 0.2$ & $\eta_1^* = 0.2$ & $\theta_2^* = 0.2$ & $\mu_0^* = -19.5$ 
\rule[0mm]{0cm}{4.5mm} \\ \hline
$(200, 0.005)$& $0.496$ & $0.239$ & $0.214$ & $0.214$ & $0.213$ & $-19.485$
 \\
& ($0.006$) & ($0.703$) & ($0.015$) & ($0.015$) & ($0.015$) & ($0.077$)
 \\   \hline
$(100, 0.01)$ & $0.482$ & $1.027$ & $0.253$ & $0.253$ & $0.252$ & $-19.495$
 \\
& ($0.006$) & ($0.721$) & ($0.016$) & ($0.016$) & ($0.018$) & ($0.076$)
 \\   \hline
$(50, 0.02)$ & $0.466$ & $2.101$ & $0.304$ & $0.304$ & $0.305$ & $-19.503$
 \\
& ($0.006$) & ($0.772$) & ($0.018$) & ($0.018$) & ($0.021$) & ($0.077$)
 \\   \hline
$(30, 0.033)$ & $0.446$ & $4.437$ & $0.394$ & $0.394$ & $0.422$ & $-19.485$
 \\
& ($0.008$) & ($1.038$) & ($0.036$) & ($0.036$) & ($0.042$) & ($0.091$)
 \\   \hline
\end{tabular}
\end{center}
\end{table}

\section{Proofs}\label{sec5}

We set the following notation.
\begin{enumerate}
\item
For families $\{ a_\lambda \}, \{ b_\lambda \} \subset \mathbb R$, 
we write $a_\lambda \lesssim b_\lambda$ 
if $|a_\lambda| \le C |b_\lambda|$ for some universal constant $C > 0$ 
and any $\lambda$, and we write $a_\lambda \sim b_\lambda$ 
if $a_\lambda \lesssim b_\lambda$ and $b_\lambda \lesssim a_\lambda$.

\item
For two functions $f, g: \mathbb R^d \to \mathbb R$, we write
$f(x) \lesssim g(x)$ ($x \to a$) if $f(x) \lesssim g(x)$ in a neighborhood of $x = a$.

\item
For $x = (x_1, \ldots, x_d) \in \mathbb R^d$ and 
$f:\mathbb R^d \to \mathbb R$, 
we write $\pd_{x_i} f(x) = \frac{\pd}{\pd x_i}f(x)$,
$\pd_x f(x) = (\pd_{x_1}f(x), \ldots, \pd_{x_d}(x))$
and $\pd_x^2 f(x) = (\pd_{x_j} \pd_{x_i} f(x))_{i,j = 1}^d$.

\item
For a subset $S$ of $\mathbb R$, 
we define $C^{1,0}(\Xi_{\alpha} \times S)$ as the set of all functions 
on $\Xi_{\alpha} \times S$ that are differentiable with respect to $\alpha \in \Xi_{\alpha}$
for all $s \in S$ and whose partial derivatives are continuous on $\Xi_{\alpha} \times S$.

\item
Let $\ind_A$ be the indicator function of $A$.

\item
Let $\GG_{i,\ell} = \GG_{i,\ell}^n = \sigma[\{ w_{\ell}(s) \}_{s \le t_i^n}]$
be the $\sigma$-algebra generated by $\{w_{\ell}(s)\}_{s \le t_i^n}$ with $\ell = (\mf l_1, \mf l_2) \in \mathbb N^2$ obtained from [A2].
\end{enumerate}

For simplicity, we set $R_1 = R_{\alpha^*,\alpha_0}^{\mathrm{damp}}$
and $R_2 = R_{\alpha^*,\alpha_0}^{\mathrm{coef}}$.

\subsection{Proof of Theorem \ref{th1}}\label{sec5-1}

Proposition 4.1 in \cite{TKU2026} established
the following results on the covariance of $T_{i,j,k} X$ given in \eqref{Old_TX}.
Under [A1]$_{\alpha_0}$ with $\alpha_0 \in (0,3)$, 
it follows that for $\alpha \in (0,3)$ and $\delta = r \sqrt{\Delta}$,
\begin{align}
\cov \bigl[T_{i,j,k}X, T_{i',j',k'}X \bigr]
&= \OO \Biggl(
\frac{\epsilon^2 \Delta^{\alpha \tand 2}}{|i-i'| +1}
\biggl( \Delta + \frac{1}{(|j-j'|+1)(|k-k'|+1)} \biggr)
\Biggr)
\nonumber
\\
&\quad+ \OO \Biggl(
\frac{\epsilon^2 \Delta^{1/2+(\alpha \tand 2)}}{|i-i'| +1}
\biggl( \frac{\ind_{\{ j \neq j' \}}}{|j-j'|+1} 
+ \frac{\ind_{\{ k \neq k' \}}}{|k-k'|+1} \biggr) \Biggr)
\label{cov-1}
\end{align}
and 
\begin{equation}\label{cov-2}
\small
\left\{
\begin{split}
\sum_{k,k'=1}^{m_2} \sum_{j,j'=1}^{m_1} \sum_{i,i'=1}^{N} 
\cov \bigl[T_{i,j,k}X, T_{i',j',k'}X \bigr]^2
&= \OO \bigl( \epsilon^4 m N (\Delta^{\alpha \tand 2})^2 \bigr),
\\
\sum_{k,k'=1}^{m_2} \sum_{j,j'=1}^{m_1} \sum_{i,i'=1}^{N} 
\cov \bigl[(T_{i,j,k}X)^2, (T_{i',j',k'}X)^2 \bigr]
&= \OO \bigl( \epsilon^4 m N (\Delta^{\alpha \tand 2})^2 
(\epsilon^{-2} \Delta^{\alpha_0 -(\alpha_0 \tand 2)} \lor 1)
\bigr).
\end{split}
\right.
\end{equation}

We obtain the following results on the covariance of $\widehat T_{i,j,k}X$
given in \eqref{New_TX}.
\begin{lem}\label{lem1}
Let $\alpha \in (0,3)$.
Assume that [A1]$_{\alpha_0}$ with $\alpha_0 \in (0,3)$ holds.
Then, it holds that for $\delta = r \sqrt{\Delta}$,  
\begin{align*}
\cov \bigl[\widehat T_{i,j,k}X, \widehat T_{i',j',k'}X \bigr]
&= \OO \Biggl(
\frac{\epsilon^2 \Delta^{\alpha \tand 2}}{|i-i'| +1}
\biggl( \Delta + \frac{1}{(|j-j'|+1)(|k-k'|+1)} \biggr)
\Biggr)
\\
&\quad+ \OO \Biggl(
\frac{\epsilon^2 \Delta^{1/2+(\alpha \tand 2)}}{|i-i'| +1}
\biggl( \frac{1}{|j-j'|+1} + \frac{1}{|k-k'|+1} \biggr) \Biggr).
\end{align*}
In particular, we have
\begin{equation*}
\sum_{k,k'=1}^{m_2-1} \sum_{j,j'=1}^{m_1-1} \sum_{i,i'=1}^{N-3} 
\cov \bigl[\widehat T_{i,j,k}X, \widehat T_{i',j',k'}X \bigr]^2
= \OO \bigl( \epsilon^4 m N (\Delta^{\alpha \tand 2})^2 \bigr),
\end{equation*}
\begin{equation*}
\sum_{k,k'=1}^{m_2-1} \sum_{j,j'=1}^{m_1-1} \sum_{i,i'=1}^{N-3} 
\cov \bigl[(\widehat T_{i,j,k}X)^2, (\widehat T_{i',j',k'}X)^2 \bigr]
= \OO \bigl( \epsilon^4 m N (\Delta^{\alpha \tand 2})^2 
(\epsilon^{-2} \Delta^{\alpha_0 -(\alpha_0 \tand 2)} \lor 1)
\bigr).
\end{equation*}
\end{lem}

\begin{proof}
For $a, b \in \mathbb R$, we have
\begin{equation}\label{ineq-tri}
\frac{1}{|a+b| +1} \le \frac{|b|+1}{|a|+1}.
\end{equation}
Indeed, we have by the triangle inequality, 
\begin{align*}
|a| +1 &= |a+b-b| +1 
\\
&\le |a+b| +|b| +1 
\\
&\le (|b|+1)|a+b| +|b|+1
\\
&=
(|b|+1)(|a+b| +1).
\end{align*}
Therefore, we obtain by \eqref{New_TX}, \eqref{cov-1} and \eqref{ineq-tri},
\begin{align*}
&\cov \bigl[\widehat T_{i,j,k}X, \widehat T_{i',j',k'}X \bigr]
\\
&= \sum_{k_1, k_2 =0}^1 \sum_{j_1, j_2 = 0}^1 \sum_{i_1, i_2 = 0}^3 
\cov \bigl[T_{i+i_1,j+j_1,k+k_1}X, T_{i'+i_2,j'+j_2,k'+k_2}X \bigr]
\\
&= \OO \Biggl(
\sum_{k_1, k_2 =0}^1 \sum_{j_1, j_2 = 0}^1 \sum_{i_1, i_2 = 0}^3
\frac{\epsilon^2 \Delta^{\alpha \tand 2}}{|i+i_1-(i'+i_2)| +1}
\\
&\qquad \times
\biggl( \Delta + \frac{1}{(|j+j_1-(j'+j_2)|+1)(|k+k_1-(k'+k_2)|+1)} \biggr)
\Biggr)
\\
&\quad+ \OO \Biggl(
\sum_{k_1, k_2 =0}^1 \sum_{j_1, j_2 = 0}^1 \sum_{i_1, i_2 = 0}^3
\frac{\epsilon^2 \Delta^{1/2+(\alpha \tand 2)}}{|i+i_1-(i'+i_2)| +1}
\\
&\qquad \times
\biggl( 
\frac{\ind_{\{j+j_1 \neq j'+j_2\}}}{|j+j_1-(j'+j_2)|+1}
+ \frac{\ind_{\{k+k_1 \neq k'+k_2\}}}{|k+k_1-(k'+k_2)|+1} 
\biggr) \Biggr)
\\
&= \OO \Biggl(
\sum_{k_1, k_2 =0}^1 \sum_{j_1, j_2 = 0}^1 \sum_{i_1, i_2 = 0}^3
\frac{\epsilon^2 \Delta^{\alpha \tand 2}(|i_1-i_2|+1)}{|i-i'| +1}
\\
&\qquad \times
\biggl( \Delta + \frac{(|j_1-j_2|+1)(|k_1-k_2|+1)}{(|j-j'|+1)(|k-k'|+1)} \biggr)
\Biggr)
\\
&\quad+ \OO \Biggl(
\sum_{k_1, k_2 =0}^1 \sum_{j_1, j_2 = 0}^1 \sum_{i_1, i_2 = 0}^3
\frac{\epsilon^2 \Delta^{1/2+(\alpha \tand 2)}(|i_1-i_2|+1)}{|i-i'| +1}
\\
&\qquad \times
\biggl( 
\frac{\ind_{\{j+j_1 \neq j'+j_2\}}(|j_1-j_2|+1)}{|j-j'|+1}
+ \frac{\ind_{\{k+k_1 \neq k'+k_2\}}(|k_1-k_2|+1)}{|k-k'|+1} 
\biggr) \Biggr)
\\
&= \OO \Biggl(
\frac{\epsilon^2 \Delta^{\alpha \tand 2}}{|i-i'| +1}
\biggl( \Delta + \frac{1}{(|j-j'|+1)(|k-k'|+1)} \biggr)
\Biggr)
\\
&\quad+ \OO \Biggl(
\frac{\epsilon^2 \Delta^{1/2+(\alpha \tand 2)}}{|i-i'| +1}
\biggl( \frac{1}{|j-j'|+1} + \frac{1}{|k-k'|+1} \biggr) \Biggr).
\end{align*}
Moreover, we have
\begin{align*}
&\frac{1}{\epsilon^4 m N (\Delta^{\alpha \tand 2})^2}
\sum_{k,k'=1}^{m_2-1}\sum_{j,j'=1}^{m_1-1}\sum_{i,i' =1}^{N-3} 
\cov \bigl[ \widehat T_{i,j,k}X, \widehat T_{i',j',k'}X \bigr]^2
\\
&=
\OO \Biggl( 
\frac{1}{N}\sum_{i,i'=1}^{N-3} \frac{1}{(|i-i'|+1)^2}
\\
&\qquad \times
\biggl(
m\Delta^2
+\frac{1}{m_1}\sum_{j,j'=1}^{m_1-1} \frac{1}{(|j-j'|+1)^2}
\times \frac{1}{m_2}\sum_{k,k'=1}^{m_2-1} \frac{1}{(|k-k'|+1)^2}
\biggr)
\Biggr)
\\
&\quad
+\OO \Biggl( \frac{1}{N}\sum_{i,i'=1}^{N-3} \frac{1}{(|i-i'|+1)^2}
\biggl(
m_2\Delta
\times \frac{1}{m_1}\sum_{j,j'=1}^{m_1-1} \frac{1}{(|j-j'|+1)^2}
\\
&\qquad
+m_1\Delta
\times \frac{1}{m_2}\sum_{k,k'=1}^{m_2-1} \frac{1}{(|k-k'|+1)^2}
\biggr)
\Biggr)
\\
&=\OO(1)+ \OO(\Delta^{1/2})
\\
&=\OO(1).
\end{align*}
In the same manner as in the proof of (6) of Proposition 4.1 in \cite{TKU2026}, 
we obtain by Isserlis' theorem, 
\begin{align*}
&\cov \bigl[ (\widehat T_{i,j,k}X)^2, (\widehat T_{i',j',k'}X)^2 \bigr] 
\\
&=4\cov \bigl[\widehat T_{i,j,k} X, \widehat T_{i',j',k'} X \bigr] 
\\
&\qquad \times 
\sum_{ l_1, l_1' \in \mathbb N} \widehat A_{i, l_1, l_1'} 
(e_{l_1}^{(1)}(\widetilde y_{j+1})-e_{ l_1}^{(1)}(\widetilde y_{j-1}))
(e_{l_1'}^{(2)}(\widetilde z_{k+1})-e_{ l_1'}^{(2)}(\widetilde z_{k-1}))
\\
&\qquad\times \sum_{ l_2, l_2' \in \mathbb N} \widehat A_{i', l_2, l_2'} 
(e_{l_2}^{(1)}(\widetilde y_{j'+1})-e_{l_2}^{(1)}(\widetilde y_{j'-1}))
(e_{l_2'}^{(2)}(\widetilde z_{k'+1})-e_{l_2'}^{(2)}(\widetilde z_{k'-1}))
\\
&\quad+2 \cov \bigl[\widehat T_{i,j,k} X, \widehat T_{i',j',k'} X \bigr]^2,
\end{align*}
where
\begin{equation*}
\widehat A_{i, l_1, l_2}
= - \lb X_0, e_{l_1,l_2} \rb (1-\ee^{-2\lambda_{l_1,l_2} \Delta})
\ee^{-\lambda_{l_1,l_2} (i-1)\Delta}.
\end{equation*}
It follows from the proof of Proposition 2.1 in \cite{TKU2026} that 
under [A1]$_{\alpha_0}$ with $\alpha_0 \in (0,3)$, 
\begin{equation*}
\widehat {\mathbf A}_{N,j,k} = 
\sum_{i=1}^{N-3} 
\biggl(
\sum_{ l_1, l_2 \in \mathbb N} \widehat A_{i, l_1, l_2} 
(e_{l_1}^{(1)}(\widetilde y_{j+1})-e_{l_1}^{(1)}(\widetilde y_{j-1}))
(e_{l_2}^{(2)}(\widetilde z_{k+1})-e_{l_2}^{(2)}(\widetilde z_{k-1}))
\biggr)^2
=\OO(\Delta^{\alpha_0})
\end{equation*}
uniformly in $j,k$.  
Hence, it holds from the Schwarz inequality and $m = \OO(N)$ that
\begin{align*}
&\biggl| \sum_{k,k'=1}^{m_2-1} \sum_{j,j'=1}^{m_1-1} \sum_{i,i'=1}^{N-3} 
\cov \bigl[ (\widehat T_{i,j,k}X)^2, (\widehat T_{i',j',k'}X)^2 \bigr] \biggr|
\\
&\le 4 \biggl( 
\sum_{k,k'=1}^{m_2-1} \sum_{j,j'=1}^{m_1-1} \sum_{i,i'=1}^{N-3}
\cov \bigl[\widehat T_{i,j,k} X, \widehat T_{i',j',k'} X \bigr]^2 \biggr)^{1/2}
\sum_{k=1}^{m_2-1} \sum_{j=1}^{m_1-1} \widehat {\mathbf A}_{N,j,k}
\\
&\quad+2
\sum_{k,k'=1}^{m_2-1} \sum_{j,j'=1}^{m_1-1} \sum_{i,i'=1}^{N-3}
\cov \bigl[\widehat T_{i,j,k} X, \widehat T_{i',j',k'} X \bigr]^2
\\
&= \OO \bigl( 
\epsilon^2 \sqrt{m N} \Delta^{\alpha \tand 2} \cdot m \Delta^{\alpha_0} \bigr) 
+\OO \bigl( \epsilon^4 m N (\Delta^{\alpha \tand 2})^2 \bigr)
\\
&=\OO \bigl( \epsilon^4 m N (\Delta^{\alpha \tand 2})^2 
(\epsilon^{-2} \Delta^{\alpha_0 -(\alpha \tand 2)} \lor 1 ) \bigr).
\end{align*}
\end{proof}

We define
\begin{equation*}
\mathcal Z_1 = \frac{S_1}{\epsilon^2 m N (4\Delta)^{\alpha^*}},
\quad
\mathcal Z_2 = \frac{S_2}{4 \epsilon^2 m N \Delta^{\alpha^*}}.
\end{equation*}
We then obtain
\begin{align*}
R_1 (\widehat \alpha -\alpha^*)
&= R_1 \biggl( \frac{\log(S_1) -\log(S_2)}{\log(4)} +1 -\alpha^* \biggr)
\\
&=
R_1 \biggl( \frac{\log(4^{\alpha^*} \mathcal Z_1) 
-\log(4 \mathcal Z_2)}{\log(4)} +1 -\alpha^* \biggr)
\\
&=
R_1 \Biggl(
\biggl(
\frac{\log(\mathcal Z_1) -\log(\mathcal Z_2)}{\log(4)}
+\alpha^* -1 \biggr)
+1-\alpha^*
\Biggr)
\\
&= \frac{R_1}{\log(4)}
\log \biggl( \frac{\mathcal Z_1}{\mathcal Z_2} \biggr).
\end{align*}
By the Taylor expansion
\begin{equation}\label{taylor}
\log(x) = (x-1) - \int_0^1 \frac{1-u}{(1+u(x-1))^2} \dd u (x-1)^2,
\end{equation}
we have
\begin{equation*}
\begin{split}
R_1 \log \biggl( \frac{\mathcal Z_1}{\mathcal Z_2} \biggr)
&= R_1 \biggl( \frac{\mathcal Z_1}{\mathcal Z_2} -1 \biggr)
- \frac{1}{R_1}
\int_0^1 \frac{1-u}{(1+u(\mathcal Z_1/\mathcal Z_2-1))^2} \dd u
\biggl( 
R_1 \biggl( \frac{\mathcal Z_1}{\mathcal Z_2} -1 \biggr)
\biggr)^2.
\end{split}
\end{equation*}
If $|x-1| \le 1/2$, then it follows that
$|1+u(x-1)| \ge 1 -u|x-1| \ge 1/2$ for any $u \in [0,1]$ and
\begin{equation*}
\int_0^1 \frac{1-u}{(1+u(x-1))^2} \dd u \le 4.
\end{equation*}
Therefore, for any sequence of random variables $\{ Y_n \}$ such that $Y_n \pto 1$ 
and any positive number $\upsilon$, there thus exists 
$N \in \mathbb N$ such that for any integer $n \ge N$, 
$\PP(|Y_n -1| > 1/2) \le \upsilon$. 
By choosing $K > 4$, we have
\begin{align*}
&\PP \biggl( \biggl| \int_0^1 \frac{1-u}{(1+u(Y_n-1))^2} \dd u \biggr| \ge K \biggr)
\\
&\le \PP \biggl( 
\biggl\{ \biggl| \int_0^1 \frac{1-u}{(1+u(Y_n-1))^2} \dd u \biggr| \ge K \biggr\}
\cap \Bigl\{ |Y_n -1| < \frac{1}{2} \Bigr\}
\biggr)
+ \PP \Bigl( |Y_n -1| \ge \frac{1}{2} \Bigr)
\\
&\le \PP ( 4 \ge K ) + \PP \Bigl( |Y_n -1| \ge \frac{1}{2} \Bigr)
\\
&\le \upsilon
\end{align*}
for $n \ge N$, and thus 
$\displaystyle \int_0^1 \frac{1-u}{(1+u(Y_n-1))^2} \dd u = \Op(1)$. 
We hence show
\begin{equation*}
R_1 \biggl( \frac{\mathcal Z_1}{\mathcal Z_2} -1 \biggr) = \Op(1)
\end{equation*}
in order to prove $R_1 \log(\mathcal Z_1/\mathcal Z_2) = \Op(1)$.

Remind the definition of $g_{r,\alpha}^{(\mathrm{m})}(\vartheta)$ given in \eqref{g_r}.
Since it holds that for $f \in C([b,1-b])$, 
\begin{align}
\sum_{j=1}^{m_1-1}
\int_{\widetilde y_{j-1}}^{\widetilde y_{j+1}} f(y) \dd y
&= \sum_{j=1}^{m_1-1} \biggl(
\int_{\widetilde y_{j-1}}^{\widetilde y_j} f(y) \dd y
+\int_{\widetilde y_j}^{\widetilde y_{j+1}} f(y) \dd y
\biggr)
\nonumber
\\
&= \int_{\widetilde y_{0}}^{\widetilde y_{m_1-1}} f(y) \dd y
+\int_{\widetilde y_1}^{\widetilde y_{m_1}} f(y) \dd y
\nonumber
\\
&= \int_{\widetilde y_{0}}^{\widetilde y_{m_1}} f(y) \dd y
+\int_{\widetilde y_1}^{\widetilde y_{m_1-1}} f(y) \dd y
\nonumber
\\
&= \int_{b}^{1-b} f(y) \dd y
+\int_{\widetilde y_1}^{\widetilde y_{m_1-1}} f(y) \dd y,
\label{eq-ff}
\end{align}
we obtain
\begin{align*}
g_{r,\alpha}^{(\mathrm{m})}(\vartheta) &= 
\frac{\phi_{r,\alpha}(\theta_2)}{4(1-2b)^2}
\biggl(
\int_{b}^{1-b} \ee^{-\kappa y} \dd y 
+\int_{\widetilde y_1}^{\widetilde y_{m_1-1}} \ee^{-\kappa y} \dd y
\biggr)
\nonumber
\\
&\quad\times
\biggl(
\int_{b}^{1-b} \ee^{-\eta z} \dd z 
+\int_{\widetilde z_1}^{\widetilde z_{m_2-1}} \ee^{-\eta z} \dd z
\biggr).
\end{align*}
Since we have
\begin{equation}\label{eq-z1/z2}
\biggl| \frac{\mathcal Z_1}{\mathcal Z_2}-1 \biggr|
= \biggl| \frac{\mathcal Z_1 -g_{r,\alpha^*}^{(\mathrm{m})}(\vartheta) 
-(\mathcal Z_2 -g_{r,\alpha^*}^{(\mathrm{m})}(\vartheta))}
{\mathcal Z_2 -g_{r,\alpha^*}^{(\mathrm{m})}(\vartheta) +g_{r,\alpha^*}^{(\mathrm{m})}(\vartheta)} \biggr|
\le 
\frac{|\mathcal Z_1- g_{r,\alpha^*}^{(\mathrm{m})}(\vartheta)| 
+|\mathcal Z_2 -g_{r,\alpha^*}^{(\mathrm{m})}(\vartheta)|}
{|g_{r,\alpha^*}^{(\mathrm{m})}(\vartheta) -|\mathcal Z_2 -g_{r,\alpha^*}^{(\mathrm{m})}(\vartheta)||},
\end{equation}
we obtain 
\begin{equation*}
R_1 \biggl| \frac{\mathcal Z_1}{\mathcal Z_2}-1 \biggr|
\le 
\frac{R_1 |\mathcal Z_1- g_{r,\alpha^*}^{(\mathrm{m})}(\vartheta)| 
+R_1|\mathcal Z_2 -g_{r,\alpha^*}^{(\mathrm{m})}(\vartheta)|}
{|g_{r,\alpha^*}^{(\mathrm{m})}(\vartheta) -|\mathcal Z_2 -g_{r,\alpha^*}^{(\mathrm{m})}(\vartheta)||}.
\end{equation*}
It thus suffices to show the following (E1)--(E2).
\begin{description}
\item[(E1)]
$R_1 \bigl( \mathcal Z_1 -g_{r,\alpha^*}^{(\mathrm{m})}(\vartheta) \bigr) = \Op(1)$,

\item[(E2)]
$R_1 \bigl( \mathcal Z_2 -g_{r,\alpha^*}^{(\mathrm{m})}(\vartheta) \bigr) = \Op(1)$.
\end{description}

\begin{proof}[\bf{Proof of (E2)}]
We first show
\begin{equation}\label{eq-Z1}
\EE [ \mathcal Z_2 ] = g_{r,\alpha^*}^{(\mathrm{m})}(\vartheta) +\OO (R_1^{-1}).
\end{equation}
By \eqref{TKU26_2.8} and $\Delta = \OO(\frac{1}{\sqrt{m N}})$, we have
\begin{align}
\frac{1}{\epsilon^2 N \Delta^{\alpha^*}}
\sum_{i=1}^N \EE \bigl[ (T_{i,j,k}X)^2 \bigr]
&= \ee^{-\kappa \overline y_j -\eta \overline z_k} \phi_{r,\alpha^*}(\theta_2)
+ \OO \bigl(\Delta^{1 -\alpha^* +(\alpha^* \tand 2)} 
\lor \epsilon^{-2} \Delta^{1+\alpha_0 -\alpha^*} \bigr)
\nonumber
\\
&= \ee^{-\kappa \overline y_j -\eta \overline z_k} \phi_{r,\alpha^*}(\theta_2)
+ \OO \biggl( \frac{1}
{\sqrt{m N} (N^{(\alpha^* \tand 2) -\alpha^*} \land \epsilon^2 N^{\alpha_0-\alpha^*})}
\biggr)
\nonumber
\\
&= \ee^{-\kappa \overline y_j -\eta \overline z_k} \phi_{r,\alpha^*}(\theta_2)
+ \OO ( R_1^{-1})
\label{eq-TX1}
\end{align}
with the middle points $\overline y_j = (\widetilde y_{j-1} +\widetilde y_j)/2$
and $\overline z_k = (\widetilde z_{k-1} +\widetilde z_k)/2$
for $j= 1, \ldots, m_1$, $k = 1,\ldots, m_2$.
It follows from the Taylor expansion 
and $\widetilde y_j -\widetilde y_{j-1} = (1-2b)/m_1 = \delta$
that for $f \in C^2([b,1-b])$ and 
$\underline{j}, \overline{j} \in \{ 0, 1,\ldots, m_1\}$ such that 
$\underline{j} < \overline{j}$,
\begin{align*}
&\Biggl| 
\frac{1}{m_1} \sum_{j = \underline{j}+1}^{\overline{j}} f(\overline y_j)
- \frac{1}{1-2b} \int_{\widetilde y_{\underline{j}}}^{\widetilde y_{\overline{j}}} f(y) \dd y
\Biggr|
\\
&=
\Biggl| \frac{1}{1-2b} 
\sum_{j = \underline{j}+1}^{\overline{j}}
\int_{\widetilde y_{j-1}}^{\widetilde y_j} 
(f(\overline y_j) -f(y)) \dd y
\Biggr|
\\
&\le 
\Biggl| \frac{1}{1-2b} 
\sum_{j = \underline{j}+1}^{\overline{j}}
\int_{\widetilde y_{j-1}}^{\widetilde y_j} 
f'(\overline y_j)(y -\overline y_j) \dd y
\Biggr|
\\
&\quad +
\Biggl| \frac{1}{1-2b} 
\sum_{j = \underline{j}+1}^{\overline{j}}
\int_{\widetilde y_{j-1}}^{\widetilde y_j} 
\biggl(
\int_0^1 (1-u) f''(\overline y_j +u(y -\overline y_j)) \dd u 
\biggr)
(y -\overline y_j)^2 \dd y
\Biggr|
\\
&=
0 + \OO \Biggl( \sum_{j = \underline{j}+1}^{\overline{j}}
\int_{\widetilde y_{j-1}}^{\widetilde y_j} (y -\overline y_j)^2 \dd y \Biggr)
\\
&= \OO(\delta^2) = \OO(\Delta).
\end{align*}
Since we obtain
\begin{equation*}
\Delta R_1 = \Delta 
\sqrt{m N} (N^{(\alpha^* \tand 2) -\alpha^*} \land \epsilon^2 N^{\alpha_0-\alpha^*})
\le \Delta \sqrt{m N} N^{\alpha^* -\alpha^*} = \OO(1),
\end{equation*}
and $\Delta = \OO(R_1^{-1})$, 
we find that for 
$\underline{j}, \overline{j} \in \{ 0, 1,\ldots, m_1\}$ such that 
$\underline{j} < \overline{j}$
and $\underline{k}, \overline{k} \in \{ 0, 1,\ldots, m_2\}$ such that
$\underline{k} < \overline{k}$, 
\begin{align}
&\frac{1}{\epsilon^2 m N \Delta^{\alpha^*}} 
\sum_{k = \underline{k}+1}^{\overline{k}} 
\sum_{j = \underline{j}+1}^{\overline{j}} \sum_{i=1}^N 
\EE \bigl[ (T_{i,j,k}X)^2 \bigr]
\nonumber
\\
&= \phi_{r,\alpha^*}(\theta_2) \times
\frac{1}{m} 
\sum_{k = \underline{k}+1}^{\overline{k}} 
\sum_{j = \underline{j}+1}^{\overline{j}} 
\ee^{-\kappa \overline y_j -\eta \overline z_k}
+ \OO (R_1^{-1})
\nonumber
\\
&= \frac{\phi_{r,\alpha^*}(\theta_2)}{(1-2b)^2}
\int_{\widetilde z_{\underline{k}}}^{\widetilde z_{\overline{k}}}
\int_{\widetilde y_{\underline{j}}}^{\widetilde y_{\overline{j}}}
\ee^{-\kappa y -\eta z} \dd y \dd z
+ \OO(\Delta) +\OO (R_1^{-1})
\nonumber
\\
&= \frac{\phi_{r,\alpha^*}(\theta_2)}{(1-2b)^2}
\int_{\widetilde z_{\underline{k}}}^{\widetilde z_{\overline{k}}}
\int_{\widetilde y_{\underline{j}}}^{\widetilde y_{\overline{j}}}
\ee^{-\kappa y -\eta z} \dd y \dd z
+ \OO (R_1^{-1}).
\label{eq-pf-2}
\end{align}
Therefore, we obtain
\begin{align*}
\EE [ \mathcal Z_2 ] &= 
\frac{1}{4\epsilon^2 m N \Delta^{\alpha^*}}
\Biggl(
\sum_{k=1}^{m_2} \sum_{j=1}^{m_1} \sum_{i=1}^N \EE \bigl[ (T_{i,j,k}X)^2 \bigr]
+\sum_{k=2}^{m_2-1} \sum_{j=1}^{m_1} \sum_{i=1}^N \EE \bigl[ (T_{i,j,k}X)^2 \bigr]
\\
&\qquad
+\sum_{k=1}^{m_2} \sum_{j=2}^{m_1-1} \sum_{i=1}^N \EE \bigl[ (T_{i,j,k}X)^2 \bigr]
+\sum_{k=2}^{m_2-1} \sum_{j=2}^{m_1-1} \sum_{i=1}^N \EE \bigl[ (T_{i,j,k}X)^2 \bigr]
\Biggr)
\\
&= 
\frac{\phi_{r,\alpha^*}(\theta_2)}{4(1-2b)^2}
\Biggl(
\int_{\widetilde z_{0}}^{\widetilde z_{m_2}}
\int_{\widetilde y_{0}}^{\widetilde y_{m_1}}
\ee^{-\kappa y -\eta z} \dd y \dd z
+\int_{\widetilde z_{1}}^{\widetilde z_{m_2-1}}
\int_{\widetilde y_{0}}^{\widetilde y_{m_1}}
\ee^{-\kappa y -\eta z} \dd y \dd z
\\ 
&\qquad
+\int_{\widetilde z_{0}}^{\widetilde z_{m_2}}
\int_{\widetilde y_{1}}^{\widetilde y_{m_1-1}}
\ee^{-\kappa y -\eta z} \dd y \dd z
+\int_{\widetilde z_{1}}^{\widetilde z_{m_2-1}}
\int_{\widetilde y_{1}}^{\widetilde y_{m_1-1}}
\ee^{-\kappa y -\eta z} \dd y \dd z
\Biggr)
+\OO (R_1^{-1})
\\
&= 
\frac{\phi_{r,\alpha^*}(\theta_2)}{4(1-2b)^2}
\Biggl(
\int_{b}^{1-b} \ee^{-\kappa y} \dd y
\int_{b}^{1-b} \ee^{-\eta z} \dd z
+\int_{b}^{1-b} \ee^{-\kappa y} \dd y
\int_{\widetilde z_{1}}^{\widetilde z_{m_2-1}} \ee^{-\eta z} \dd z
\\ 
&\qquad
+\int_{\widetilde y_{1}}^{\widetilde y_{m_1-1}} \ee^{-\kappa y} \dd y
\int_{b}^{1-b} \ee^{-\eta z} \dd z
+\int_{\widetilde y_{1}}^{\widetilde y_{m_1-1}} \ee^{-\kappa y} \dd y
\int_{\widetilde z_{1}}^{\widetilde z_{m_2-1}} \ee^{-\eta z} \dd z
\Biggr)
+\OO (R_1^{-1})
\\
&= 
\phi_{r,\alpha^*}(\theta_2) \times
\frac{1}{4(1-2b)^2} \biggl(
\int_{b}^{1-b} \ee^{-\kappa y} \dd y
+ \int_{\widetilde y_1}^{\widetilde y_{m_1-1}} \ee^{-\kappa y} \dd y 
\biggr) 
\\
&\quad \times
\biggl(
\int_{b}^{1-b} \ee^{-\eta z} \dd z
+ \int_{\widetilde z_1}^{\widetilde z_{m_2-1}} \ee^{-\eta z} \dd z 
\biggr)
+\OO (R_1^{-1})
\\
&= g_{r,\alpha^*}^{(\mathrm{m})}(\vartheta) +\OO (R_1^{-1}).
\end{align*}
This concludes the proof of \eqref{eq-Z1}.

We next show 
\begin{equation}\label{eq-Z1-1}
R_1 \bigl( \mathcal Z_2 -\EE [\mathcal Z_2] \bigr) = \Op(1).
\end{equation}
We find from \eqref{cov-2} that 
\begin{align*}
\EE \Bigl[ \bigl( S_2 -\EE [S_2] \bigr)^2 \Bigr]
&\lesssim 
\EE \Biggl[ \biggl(
\sum_{k=1}^{m_2} \sum_{j=1}^{m_1}
\sum_{i=1}^N \Bigl( (T_{i,j,k}X)^2 -\EE[(T_{i,j,k}X)^2] \Bigr)
\biggr)^2 \Biggr]
\\
&\quad+\EE \Biggl[ \biggl(
\sum_{k=2}^{m_2-1} \sum_{j=1}^{m_1}
\sum_{i=1}^N \Bigl( (T_{i,j,k}X)^2 -\EE[(T_{i,j,k}X)^2] \Bigr)
\biggr)^2 \Biggr]
\\
&\quad+
\EE \Biggl[ \biggl(
\sum_{k=1}^{m_2} \sum_{j=2}^{m_1-1}
\sum_{i=1}^N \Bigl( (T_{i,j,k}X)^2 -\EE[(T_{i,j,k}X)^2] \Bigr)
\biggr)^2 \Biggr]
\\
&\quad+\EE \Biggl[ \biggl(
\sum_{k=2}^{m_2-1} \sum_{j=2}^{m_1-1}
\sum_{i=1}^N \Bigl( (T_{i,j,k}X)^2 -\EE[(T_{i,j,k}X)^2] \Bigr)
\biggr)^2 \Biggr]
\\
&= 
\sum_{k,k'=1}^{m_2} \sum_{j,j'=1}^{m_1} \sum_{i,i'=1}^{N}
\cov \Bigl[(T_{i,j,k}X)^2, (T_{i',j',k'}X)^2 \Bigr]
\\
&\quad+
\sum_{k,k'=2}^{m_2-1} \sum_{j,j'=1}^{m_1} \sum_{i,i'=1}^{N}
\cov \Bigl[(T_{i,j,k}X)^2, (T_{i',j',k'}X)^2 \Bigr]
\\
&\quad+
\sum_{k,k'=1}^{m_2} \sum_{j,j'=2}^{m_1-1} \sum_{i,i'=1}^{N}
\cov \Bigl[(T_{i,j,k}X)^2, (T_{i',j',k'}X)^2 \Bigr]
\\
&\quad+
\sum_{k,k'=2}^{m_2-1} \sum_{j,j'=2}^{m_1-1} \sum_{i,i'=1}^{N}
\cov \Bigl[(T_{i,j,k}X)^2, (T_{i',j',k'}X)^2 \Bigr]
\\
&= \OO \Bigl(\epsilon^4 m N (\Delta^{\alpha^* \tand 2})^2 
(1 \lor \epsilon^{-2} \Delta^{\alpha_0 -(\alpha^* \tand 2)}) \Bigr)
\end{align*}
and
\begin{align*}
&R_1^2
\EE \Bigl[ \bigl(\mathcal Z_2 -\EE [\mathcal Z_2] \bigr)^2 \Bigr]
\\
&\lesssim \frac{R_1^2}{\epsilon^4 (m N)^2 \Delta^{2\alpha^*}}
\EE \Bigl[ \bigl(S_2 -\EE [S_2] \bigr)^2 \Bigr]
\\
&= \biggl( \frac{\sqrt{m N}}{\Delta^{(\alpha^* \tand 2) -\alpha^*} 
\lor \epsilon^{-2} \Delta^{\alpha_0 -\alpha^*}} \biggr)^2
\times \frac{\epsilon^4 m N (\Delta^{\alpha^* \tand 2})^2 
(1 \lor \epsilon^{-2} \Delta^{\alpha_0 -(\alpha^* \tand 2)})
}{\epsilon^4 (m N)^2 \Delta^{2\alpha^*}}
\\
&= \biggl( \frac{\Delta^{\alpha^*}}{
\Delta^{\alpha^* \tand 2} 
(1 \lor \epsilon^{-2} \Delta^{\alpha_0 -(\alpha^* \tand 2)})} \biggr)^2
\times \frac{(\Delta^{\alpha^* \tand 2})^2 
(1 \lor \epsilon^{-2} \Delta^{\alpha_0 -(\alpha^* \tand 2)})}{\Delta^{2\alpha^*}}
\\
&= \frac{1}{1 \lor \epsilon^{-2} \Delta^{\alpha_0 -(\alpha^* \tand 2)}}
\\
&\le 1.
\end{align*}
This proves the desired result. 

It thus follows from \eqref{eq-Z1} and \eqref{eq-Z1-1} that 
\begin{align*}
R_1 \bigl( \mathcal Z_2 -g_{r,\alpha^*}^{(\mathrm{m})}(\vartheta) \bigr)
&= R_1 \bigl( \mathcal Z_2 -\EE [\mathcal Z_2] \bigr)
+R_1 \bigl( \EE [\mathcal Z_2] -g_{r,\alpha^*}^{(\mathrm{m})}(\vartheta) \bigr)
\\
&= \Op(1) +\OO(1)
\\
&= \Op(1).
\end{align*}
\end{proof}

\begin{proof}[\bf{Proof of (E1)}]
First, we prove 
\begin{equation}\label{eq-Z2}
\EE [ \mathcal Z_1 ] = g_{r,\alpha^*}^{(\mathrm{m})}(\vartheta) + \OO (R_1^{-1}).
\end{equation}
Since $\widehat T_{i,j,k}X$ given in \eqref{New_TX} is the triple increment 
with the temporal resolution $\Delta' = 4 \Delta$ and 
the spatial resolution $\delta' = 2 \delta$,
$\widetilde y_j$ is the middle point of $\widetilde y_{j-1}$ and $\widetilde y_{j+1}$, 
and $r' = \delta'/\sqrt{\Delta'} = \delta/\sqrt{\Delta} = r \in (0,\infty)$,
we find from \eqref{TKU26_2.8} that 
\begin{align*}
\EE \bigl[ (\widehat T_{i,j,k}X)^2 \bigr]
&= \epsilon^2 
\Bigl( (\Delta')^{\alpha^*}
\ee^{-\kappa (\widetilde y_{j-1} +\widetilde y_{j+1})/2 
-\eta (\widetilde z_{k-1} +\widetilde z_{k+1})/2} \phi_{r',\alpha^*}(\theta_2)
\\
&\qquad +R_{i,j,k}(\alpha^*)
+\OO((\Delta')^{1+(\alpha^* \tand 2)})
\Bigr)
+ R_{i,j,k}(\alpha_0)
\\
&= \epsilon^2 
\Bigl( (4 \Delta)^{\alpha^*}
\ee^{-\kappa \widetilde y_j -\eta \widetilde z_k} \phi_{r,\alpha^*}(\theta_2)
+R_{i,j,k}(\alpha^*)
+\OO(\Delta^{1+(\alpha^* \tand 2)})
\Bigr)
+ R_{i,j,k}(\alpha_0)
\end{align*}
where $R_{i,j,k}(\alpha)$ satisfies 
$\sum_{i=1}^{N-3} R_{i,j,k}(\alpha) = \OO((\Delta')^\alpha) = \OO(\Delta^\alpha)$
uniformly in $j,k$ for $\alpha \in (0,3)$. 
We hence have
\begin{align}
&\frac{1}{\epsilon^2 (N-3) (4\Delta)^{\alpha^*}}
\sum_{i=1}^{N-3} \EE \bigl[ (\widehat T_{i,j,k}X)^2 \bigr]
\nonumber
\\
&= \ee^{-\kappa \widetilde y_j -\eta \widetilde z_k} \phi_{r,\alpha^*}(\theta_2)
+ \OO \bigl(\Delta^{1 -\alpha^* +(\alpha^* \tand 2)} 
\lor \epsilon^{-2} \Delta^{1+\alpha_0 -\alpha^*} \bigr)
\nonumber
\\
&= \ee^{-\kappa \widetilde y_j -\eta \widetilde z_k} \phi_{r,\alpha^*}(\theta_2)
+ \OO (R_1^{-1})
\label{m2_hat_TX}
\end{align}
in the same way as \eqref{eq-TX1}.
It follows from the Taylor expansion and \eqref{eq-ff} that 
for $f \in C^2([b,1-b])$ and 
$\underline{j}, \overline{j} \in \{ 0, 1,\ldots, m_1\}$ such that 
$\underline{j} < \overline{j}$,
\begin{align*}
&\Biggl| 
\frac{1}{m_1} \sum_{j = 1}^{m_1-1} f(\widetilde y_j)
- \frac{1}{2(1-2b)} \biggl(
\int_{b}^{1-b} f(y) \dd y
+ \int_{\widetilde y_1}^{\widetilde y_{m_1-1}} f(y) \dd y 
\biggr)
\Biggr|
\\
&=
\Biggl| \frac{1}{2(1-2b)} 
\sum_{j = 1}^{m_1-1}
\int_{\widetilde y_{j-1}}^{\widetilde y_{j+1}} 
(f(\widetilde y_j) -f(y)) \dd y
\Biggr|
\\
&\le 
\Biggl| \frac{1}{2(1-2b)} 
\sum_{j = 1}^{m_1 -1}
\int_{\widetilde y_{j-1}}^{\widetilde y_{j+1}} 
f'(\widetilde y_j)(y -\widetilde y_j) \dd y
\Biggr|
\\
&\qquad +
\Biggl| \frac{1}{2(1-2b)} 
\sum_{j = 1}^{m_1-1} \int_{\widetilde y_{j-1}}^{\widetilde y_{j+1}} 
\biggl(
\int_0^1 (1-u) f''(\widetilde y_j +u(y -\widetilde y_j)) \dd u 
\biggr)
(y -\widetilde y_j)^2 \dd y
\Biggr|
\\
&=
0 + \OO \Biggl( \sum_{j = 1}^{m_1-1}
\int_{\widetilde y_{j-1}}^{\widetilde y_{j+1}} (y -\widetilde y_j)^2 \dd y \Biggr)
\\
&= \OO(\delta^2) = \OO(\Delta),
\end{align*}
which is involved in the reminder of \eqref{m2_hat_TX}. Therefore, we obtain 
\begin{align*}
\EE [ \mathcal Z_1 ] 
&= \frac{1}{\epsilon^2 m N (4\Delta)^{\alpha^*}}
\sum_{k=1}^{m_2-1} \sum_{j=1}^{m_1-1}
\sum_{i=1}^{N-3} \EE \bigl[ (\widehat T_{i,j,k}X)^2 \bigr]
\\
&= \frac{N-3}{N} \times \phi_{r,\alpha^*}(\theta_2) \times
\frac{1}{m} \sum_{k=1}^{m_2-1} \sum_{j=1}^{m_1-1} 
\ee^{-\kappa \widetilde y_j -\eta \widetilde z_k} +\OO(R_1^{-1})
\\
&= \phi_{r,\alpha^*}(\theta_2) \times
\frac{1}{4(1-2b)^2} \biggl(
\int_{b}^{1-b} \ee^{-\kappa y} \dd y
+ \int_{\widetilde y_1}^{\widetilde y_{m_1-1}} \ee^{-\kappa y} \dd y 
\biggr) 
\\
&\quad \times
\biggl(
\int_{b}^{1-b} \ee^{-\eta z} \dd z
+ \int_{\widetilde z_1}^{\widetilde z_{m_2-1}} \ee^{-\eta z} \dd z 
\biggr) 
+ \OO(\Delta) +\OO(R_1^{-1})
\\
&= g_{r,\alpha^*}^{(\mathrm{m})}(\vartheta)+ \OO(R_1^{-1}).
\end{align*}
This completes the proof of \eqref{eq-Z2}.

Next, we show 
\begin{equation}\label{eq-Z2-2}
R_1 \bigl( \mathcal Z_1 -\EE [\mathcal Z_1] \bigr) = \Op(1).
\end{equation}
Since it follows from Lemma \ref{lem1} that 
\begin{align*}
\EE \Bigl[ \bigl( S_1 -\EE [S_1] \bigr)^2 \Bigr]
&= \EE \Biggl[ \biggl(
\sum_{k=1}^{m_2-1} \sum_{j=1}^{m_1-1}
\sum_{i=1}^{N-3} \Bigl( (\widehat T_{i,j,k}X)^2 -\EE[(\widehat T_{i,j,k}X)^2] \Bigr)
\biggr)^2 \Biggr]
\\
&= 
\sum_{k,k'=1}^{m_2-1} 
\sum_{j,j'=1}^{m_1-1}
\sum_{i,i'=1}^{N-3}
\cov \Bigl[(\widehat T_{i,j,k}X)^2, (\widehat T_{i',j',k'}X)^2 \Bigr]
\\
&= \OO \Bigl(\epsilon^4 m N (\Delta^{\alpha^* \tand 2})^2 
(\epsilon^{-2} \Delta^{\alpha_0 -(\alpha^* \tand 2)} \lor 1) \Bigr),
\end{align*}
we obtain
\begin{align*}
&R_1^2
\EE \Bigl[ \bigl(\mathcal Z_1 -\EE [\mathcal Z_1] \bigr)^2 \Bigr]
\nonumber
\\
&= \frac{R_1^2}{\epsilon^4 (m N)^2 (4\Delta)^{2\alpha^*}}
\EE \Bigl[ \bigl(S_1 -\EE [S_1] \bigr)^2 \Bigr]
\nonumber
\\
&\lesssim \biggl( \frac{\sqrt{m N}}{\Delta^{(\alpha^* \tand 2) -\alpha^*} 
\lor \epsilon^{-2} \Delta^{\alpha_0 -\alpha^*}} \biggr)^2
\times \frac{\epsilon^4 m N (\Delta^{\alpha^* \tand 2})^2 
(1 \lor \epsilon^{-2} \Delta^{\alpha_0 -(\alpha^* \tand 2)})
}{\epsilon^4 (m N)^2 \Delta^{2\alpha^*}}
\\
&= \biggl( \frac{\Delta^{\alpha^*}}{
\Delta^{\alpha^* \tand 2} 
(1 \lor \epsilon^{-2} \Delta^{\alpha_0 -(\alpha^* \tand 2)})} \biggr)^2
\times \frac{(\Delta^{\alpha^* \tand 2})^2 
(1 \lor \epsilon^{-2} \Delta^{\alpha_0 -(\alpha^* \tand 2)})}{\Delta^{2\alpha^*}}
\\
&= \frac{1}{1 \lor \epsilon^{-2} \Delta^{\alpha_0 -(\alpha^* \tand 2)}}
\\
&\le 1,
\end{align*}
which yields the desired result.

We therefore find from \eqref{eq-Z2} and \eqref{eq-Z2-2} that
\begin{align*}
R_1 \bigl(\mathcal Z_1 -g_{r,\alpha^*}^{(\mathrm{m})}(\vartheta) \bigr)
&= R_1 \bigl( \mathcal Z_1 -\EE [\mathcal Z_1] \bigr)
+R_1 \bigl( \EE [\mathcal Z_1] 
-g_{r,\alpha^*}^{(\mathrm{m})}(\vartheta) \bigr)
\\
&= \Op(1) +\OO(1)
\\
&= \Op(1).
\end{align*}
\end{proof}

\subsection{Proof of Theorem \ref{th2}}
We define
$f_{r,\alpha}(\mb{y};\vartheta) = \exp(-{\bs \kappa}^\TT {\mb y}) \phi_{r,\alpha}(\theta_2)$
for $\mb y = (y, z)^\TT \in D$ and $\bs \kappa = (\kappa, \eta)^\TT \in \mathbb R^2$,
$h_{j,k}^{(l)}(\vartheta;\alpha) = \pd_\vartheta^l f_{r,\alpha} (\overline {\bf y}_{j,k}; \vartheta)$
for $\overline {\bf y}_{j,k} = (\overline y_j, \overline z_k)^\TT$ and $l \in \{ 0,1,2 \}$ and
\begin{equation}\label{def_mbU}
\mb{U}_{j,k}^{\epsilon}(\vartheta;\alpha)
= \frac{1}{\epsilon^2 N \Delta^{\alpha}}\sum_{i=1}^{N} (T_{i,j,k}X)^2
-f_{r,\alpha}(\overline {\bf y}_{j,k}; \vartheta).
\end{equation}
We also define
\begin{equation}\label{def_XYZ}
\left\{
\begin{split}
\mb{X}_{m,N}^{\epsilon, l}
&= \frac{1}{m} \sup_{\vartheta \in \Xi_{\vartheta}} \sum_{k=1}^{m_2} \sum_{j=1}^{m_1}
\bigl| h_{j,k}^{(l)}(\vartheta;\widehat \alpha) 
-h_{j,k}^{(l)}(\vartheta;\alpha^*) \bigr|^2,
\quad l \in \{ 0, 1, 2 \},
\\
\mb{Y}_{m,N}^{\epsilon}
&= \frac{1}{m} \sup_{\vartheta \in \Xi_{\vartheta}}
\sum_{k=1}^{m_2}\sum_{j=1}^{m_1} 
\mathbf{U}_{j,k}^{\epsilon}(\vartheta;\alpha^*)^2,
\\
\mb{Z}_{m,N}^{\epsilon}
&= \frac{1}{m} \sup_{\vartheta \in \Xi_{\vartheta}}
\sum_{k=1}^{m_2}\sum_{j=1}^{m_1} 
\bigl(
\mathbf{U}_{j,k}^{\epsilon}(\vartheta;\widehat \alpha) 
-\mathbf{U}_{j,k}^{\epsilon}(\vartheta;\alpha^*)
\bigr)^2.
\end{split}
\right.
\end{equation}
We then obtain the following result.
\begin{lem}\label{lem_XYZ}
Assume that [A1]$_{\alpha_0}$ with $\alpha_0 \in (0,3)$ holds. 
It then holds that
\begin{itemize}
\item[(1)]
$\mb{X}_{m,N}^{\epsilon,l} = \Op( (R_{\alpha^*,\alpha_0}^{\mathrm{damp}})^{-2} )$
for any $l \in \{0, 1, 2 \}$,

\item[(2)]
$\mb{Y}_{m,N}^{\epsilon} = \Op(1)$,

\item[(3)]
$\mb{Z}_{m,N}^{\epsilon} 
= \Op \biggl( 
\Bigl(\frac{R_{\alpha^*,\alpha_0}^{\mathrm{damp}}}{\log(N)} \Bigr)^{-2} 
\biggr) 
= \Op( (R_{\alpha^*,\alpha_0}^{\mathrm{coef}})^{-2} )$.
\end{itemize}
\end{lem}
\begin{proof}
Note that the function $\Xi_{\alpha} \times \Xi_{\theta_2} \ni (\alpha, \theta_2) 
\mapsto \pd_{\theta_2}^l \phi_{r,\alpha}(\theta_2)$
belongs to $C^{1,0}(\Xi_{\alpha} \times \Xi_{\theta_2})$ for $r \in (0, \infty)$ and $l \in \{ 0, 1, 2 \}$ (see Lemma \ref{lem_phi} below).
\begin{enumerate}
\item[(1)]
Let $S$ be a compact subset of $\mathbb R$.
Since $\alpha^* \in \Xi_{\alpha}^{\circ}$ and
there exists $\gamma \in (0,\infty)$ such that
$\widehat \alpha_u = \alpha^* +u(\widehat \alpha -\alpha^*) \in \Xi_{\alpha}$ 
for any $u \in [0,1]$ 
on $\Omega_{\gamma} := \{ |\widehat \alpha -\alpha^*| < \gamma \}$,
it follows from the Taylor expansion that for $f \in C^{1,0}(\Xi_{\alpha} \times S)$, 
\begin{align}
\sup_{s \in S} |f(\widehat \alpha, s) -f(\alpha^*, s)| 
&= \sup_{s \in S} \biggl| \int_0^1 \pd_\alpha f(\widehat \alpha_u, s) \dd u \biggr| 
|\widehat \alpha -\alpha^*|
\nonumber
\\
&\le
\sup_{(\alpha,s) \in \Xi_{\alpha} \times S} |\pd_\alpha f(\alpha,s)| 
\times |\widehat \alpha -\alpha^*|
\nonumber
\\
&\lesssim
|\widehat \alpha -\alpha^*|
\label{est_f_alpha}
\end{align}
on $\Omega_{\gamma}$. 
It also follows from Theorem \ref{th1} that $\PP(\Omega_{\gamma}) \to 1$ 
as $m \to \infty$, $N \to \infty$ and 
$\epsilon \searrow 0$.

Since 
\begin{equation*}
f_{r,\widehat \alpha}(\mb{y};\vartheta) -f_{r,\alpha^*}(\mb{y};\vartheta)
= \exp(-{\bs \kappa}^\TT \mb{y}) 
(\phi_{r,\widehat \alpha}(\theta_2) -\phi_{r,\alpha^*}(\theta_2)),
\end{equation*}
the function $(\alpha, \theta_2) \mapsto \pd_{\theta_2}^{p} \phi_{r,\alpha}(\theta_2)$ 
belongs to $C^{1,0}(\Xi_{\alpha} \times \Xi_{\theta_2})$ for $p \in \{ 0,1,2 \}$
and the function
$\overline D \times \Xi_{\bs \kappa} \ni (\mb y, \bs \kappa) 
\mapsto \pd_{\bs \kappa}^{p} \exp(-\bs \kappa^\TT \mb{y})$ is continuous, 
we see from \eqref{est_f_alpha} that 
there exists $\gamma \in (0,\infty)$ such that
\begin{equation*}
\sup_{\theta_2 \in \Xi_{\theta_2}} 
\bigl| \pd_{\theta_2}^{p} \phi_{r,\widehat \alpha}(\theta_2) 
-\pd_{\theta_2}^{p} \phi_{r,\alpha^*}(\theta_2) \bigr|
\lesssim |\widehat \alpha -\alpha^*|
\end{equation*}
on $\Omega_{\gamma}$, and 
\begin{align*}
\mathbf{X}_{m,N}^{\epsilon, l}
&= \frac{1}{m} \sup_{\vartheta \in \Xi_{\vartheta}} \sum_{k=1}^{m_2} \sum_{j=1}^{m_1}
\bigl| h_{j,k}^{(l)}(\vartheta;\widehat \alpha) 
-h_{j,k}^{(l)}(\vartheta;\alpha^*) \bigr|^2
\\
&= \frac{1}{m} \sup_{\vartheta \in \Xi_{\vartheta}} \sum_{k=1}^{m_2} \sum_{j=1}^{m_1}
\bigl| \pd_\vartheta^l f_{r,\widehat \alpha} (\overline {\bf y}_{j,k}; \vartheta)
-\pd_\vartheta^l f_{r,\alpha^*} (\overline {\bf y}_{j,k}; \vartheta) \bigr|^2
\\
&\le
\sup_{(\mb{y}, \vartheta) \in \overline D \times \Xi_{\vartheta}} \bigl| 
\pd_\vartheta^l f_{r,\widehat \alpha}(\mb{y};\vartheta)
-\pd_\vartheta^l f_{r,\alpha^*}(\mb{y};\vartheta) \bigr|^2
\\
&\le
\sup_{(\mb{y}, \vartheta) \in \overline D \times \Xi_{\vartheta}}
\Biggl(
\sum_{p = 0}^{l} {{l}\choose{p}}
|\pd_{\bs \kappa}^{p} \exp(-\bs \kappa^\TT \mb{y})|
\bigl| \pd_{\theta_2}^{l -p} \phi_{r,\widehat \alpha}(\theta_2) 
-\pd_{\theta_2}^{l-p} \phi_{r,\alpha^*}(\theta_2) \bigr|
\Biggr)^2
\\
&\le
(l!)^2
\Biggl( \sum_{p = 0}^{l} \sup_{(\mb{y}, \bs \kappa) 
\in \overline D \times \Xi_{\bs \kappa}} 
|\pd_{\bs \kappa}^{p} \exp(-\bs \kappa^\TT \mb{y})|^2
\Biggr)
\\
&\quad \times \Biggl(
\sum_{p = 0}^{l}
\sup_{\theta_2 \in \Xi_{\theta_2}} 
\bigl| \pd_{\theta_2}^{p} \phi_{r,\widehat \alpha}(\theta_2) 
-\pd_{\theta_2}^{p} \phi_{r,\alpha^*}(\theta_2) \bigr|^2
\Biggr)
\\
&\lesssim |\widehat \alpha -\alpha^*|^2
\end{align*}
on $\Omega_{\gamma}$ for $l \in \{ 0,1,2 \}$.
Since it follows from Theorem \ref{th1} that
for any $\upsilon \in (0, \infty)$, there exists $M \in (0, \infty)$ such that
$\PP(R_1 |\widehat \alpha -\alpha^*| > M) \le \upsilon/2$
and $\PP(\Omega_\gamma^{\mathrm{c}}) \le \upsilon/2$ 
for large $m$, $N$ and small $\epsilon$, we obtain by setting $M' = C M^2$, 
\begin{align*}
\PP \bigl( R_1^2 |\mathbf{X}_{m,N}^{\epsilon, l}| > M' \bigr)
&\le \PP \bigl( \{ R_1^2 |\mathbf{X}_{m,N}^{\epsilon, l}| > C M^2 \} \cap \Omega_\gamma \bigr)
+\PP( \Omega_\gamma^{\mathrm{c}})
\\
&\le \PP \bigl( (R_1 |\widehat \alpha -\alpha^*|)^2 > M^2 \bigr)
+\PP( \Omega_\gamma^{\mathrm{c}})
\\
&\le \upsilon
\end{align*}
for large $m$, $N$ and small $\epsilon$, and get the desired result.

\item[(2)]
We have by \eqref{def_mbU},
\begin{align*}
\mb{Y}_{m,N}^{\epsilon}
&= \frac{1}{m} \sup_{\vartheta \in \Xi_{\vartheta}}
\sum_{k=1}^{m_2}\sum_{j=1}^{m_1} 
\mathbf{U}_{j,k}^{\epsilon}(\vartheta;\alpha^*)^2
\\
&\lesssim
\frac{1}{m}\sum_{k=1}^{m_2}\sum_{j=1}^{m_1}
\biggl(
\frac{1}{\epsilon^2 N \Delta^{\alpha^*}}\sum_{i=1}^{N} (T_{i,j,k}X)^2
\biggr)^2
+ \frac{1}{m}
\sup_{\vartheta \in \Xi_{\vartheta}}
\sum_{k=1}^{m_2}\sum_{j=1}^{m_1} f_{r,\alpha^*}(\overline {\bf y}_{j,k}; \vartheta)^2
\\
&=
\frac{1}{m}\sum_{k=1}^{m_2}\sum_{j=1}^{m_1}
\biggl(
\frac{1}{\epsilon^2 N \Delta^{\alpha^*}}\sum_{i=1}^{N} (T_{i,j,k}X)^2
\biggr)^2
+ \frac{1}{m}
\sup_{\vartheta \in \Xi_{\vartheta}}
\sum_{k=1}^{m_2}\sum_{j=1}^{m_1} h_{j,k}^{(0)}(\vartheta;\alpha^*)^2.
\end{align*}
It follows from \eqref{TKU26_2.8} and \eqref{cov-2} that 
under [B]$_{\alpha^*, \alpha_0}$, 
\begin{align}
&\EE \Biggl[
\biggl( \frac{1}{\epsilon^2 N \Delta^{\alpha^*}}\sum_{i=1}^{N} (T_{i,j,k}X)^2 \biggr)^2
\Biggr]
\nonumber
\\
&\lesssim \EE \Biggl[
\biggl( \frac{1}{\epsilon^2 N \Delta^{\alpha^*}}\sum_{i=1}^{N} \bigl( (T_{i,j,k}X)^2 
-\EE[(T_{i,j,k}X)^2] \bigr)
\biggr)^2
\Biggr]
\nonumber
\\
&\quad+ \biggl( \frac{1}{\epsilon^2 N \Delta^{\alpha^*}}\sum_{i=1}^{N} 
\EE[(T_{i,j,k}X)^2]
\biggr)^2
\nonumber
\\
&= \VV \Biggl[\frac{1}{\epsilon^2 N \Delta^{\alpha^*}}\sum_{i=1}^{N} (T_{i,j,k}X)^2 \Biggr]
+\OO(1)
\nonumber
\\
&= \frac{1}{(\epsilon^2 N \Delta^{\alpha^*})^2} \sum_{i,i'=1}^{N} 
\cov \bigl[(T_{i,j,k}X)^2, (T_{i',j,k}X)^2 \bigr]
+\OO(1)
\nonumber
\\
&= \OO \biggl( \frac{\epsilon^{-2} \Delta^{\alpha_0 -\alpha^*+1/2} \lor 1 }{N}
\biggr) +\OO(1)
\nonumber
\\
&= \OO \biggl( \biggl( 
\frac{\log(N)}{\epsilon^2 N^{1+\alpha_0 -\alpha^*}} 
\cdot \frac{1}{N^{1/2} \log(N)} \biggr) \lor \frac{1}{N} \biggr) 
+ \OO(1)
\nonumber
\\
&=\oo(1) +\OO(1) = \OO(1)
\label{est_TX}
\end{align}
uniformly in $j,k$, which together with 
\begin{equation}\label{unif_conv}
\sup_{\vartheta \in \Xi_{\vartheta}} 
\Biggl| \frac{1}{m}\sum_{k=1}^{m_2}\sum_{j=1}^{m_1} 
h_{j,k}^{(0)}(\vartheta;\alpha^*)^2
- \frac{1}{(1-2b)^2} \int_{[b,1-b]^2} 
f_{r,\alpha^*}(\mb{y};\vartheta)^2 \dd \mb{y} \Biggr|
= \oo(1)
\end{equation}
yields the desired result.

\item[(3)]
We have
\begin{align*}
&\bigl| \mathbf{U}_{j,k}^{\epsilon}(\vartheta;\widehat \alpha) 
-\mathbf{U}_{j,k}^{\epsilon}(\vartheta;\alpha^*) \bigr|
\\
&\le
\Delta^{\alpha^*} 
\biggl| \frac{1}{\Delta^{\widehat \alpha}} -\frac{1}{\Delta^{\alpha^*}} \biggr|
\frac{1}{\epsilon^2 N \Delta^{\alpha^*}}\sum_{i=1}^{N} (T_{i,j,k}X)^2
+\bigl| f_{r,\widehat \alpha}(\overline {\mb y}_{j,k}; \vartheta)
-f_{r,\alpha^*}(\overline {\mb y}_{j,k}; \vartheta) \bigr|
\\
&=
\Delta^{\alpha^*} 
\biggl| \frac{1}{\Delta^{\widehat \alpha}} -\frac{1}{\Delta^{\alpha^*}} \biggr|
\frac{1}{\epsilon^2 N \Delta^{\alpha^*}}\sum_{i=1}^{N} (T_{i,j,k}X)^2
+\bigl| h_{j,k}^{(0)}(\widehat \alpha; \vartheta)
-h_{j,k}^{(0)}(\alpha^*; \vartheta) \bigr|.
\end{align*}
By the Taylor expansion for the exponential function $a^x$ ($a \in (0, \infty)$) at $x_0 \in \mathbb{R}$: 
\begin{equation*}
a^x = a^{x_0} + a^{x_0} \log(a) \int_0^1 a^{u(x -x_0)} \dd u (x-x_0),
\end{equation*}
Theorem \ref{th1} and [B]$_{\alpha^*,\alpha_0}$, we estimate 
\begin{align*}
\Delta^{\alpha^*} 
\biggl| \frac{1}{\Delta^{\widehat \alpha}} -\frac{1}{\Delta^{\alpha^*}} \biggr|
&= \log \biggl( \frac{1}{\Delta} \biggr)
\int_0^1 \biggl( \frac{1}{\Delta} \biggr)^{u(\widehat \alpha -\alpha^*)} \dd u
|\widehat \alpha -\alpha^*|
\\
&= \log(N)
\int_0^1 N^{u(\widehat \alpha -\alpha^*)} \dd u |\widehat \alpha -\alpha^*|
\\
&\le 
\log(N) N^{|\widehat \alpha -\alpha^*|} |\widehat \alpha -\alpha^*|
\\
&= 
\log(N) \exp \bigl(|\widehat \alpha -\alpha^*| \log(N) \bigr)
|\widehat \alpha -\alpha^*|
\\
&= 
\frac{\log(N)}{R_1} 
\exp \biggl(R_1 |\widehat \alpha -\alpha^*| \times \frac{\log (N)}{R_1} \biggr)
R_1 |\widehat \alpha -\alpha^*|
\\
&= 
\frac{1}{R_2} 
\exp \biggl(R_1 |\widehat \alpha -\alpha^*| \times \frac{1}{R_2} \biggr)
R_1 |\widehat \alpha -\alpha^*|
\\
&= \Op(R_2^{-1}).
\end{align*}
Therefore, we find from \eqref{est_TX}, (1) 
and $R_1^{-1} = (R_2 \log(N))^{-1} = \oo(R_2^{-1})$ that 
\begin{align*}
\mb{Z}_{m,N}^{\epsilon}
&=
\frac{1}{m} \sup_{\vartheta \in \Xi_{\vartheta}} \sum_{k=1}^{m_2}\sum_{j=1}^{m_1} 
\bigl(
\mathbf{U}_{j,k}^{\epsilon}(\vartheta;\widehat \alpha) 
-\mathbf{U}_{j,k}^{\epsilon}(\vartheta;\alpha^*)
\bigr)^2
\\
&\lesssim
\Delta^{2\alpha^*}
\biggl( \frac{1}{\Delta^{\widehat \alpha}} -\frac{1}{\Delta^{\alpha^*}} \biggr)^2
\times
\frac{1}{m}\sum_{k=1}^{m_2}\sum_{j=1}^{m_1} 
\biggl( \frac{1}{\epsilon^2 N \Delta^{\alpha^*}}\sum_{i=1}^{N} (T_{i,j,k}X)^2 \biggr)^2
\\
&\quad+
\frac{1}{m} \sup_{\vartheta \in \Xi_{\vartheta}} \sum_{k=1}^{m_2}\sum_{j=1}^{m_1} 
\bigl( h_{j,k}^{(0)}(\widehat \alpha; \vartheta)
-h_{j,k}^{(0)}(\alpha^*; \vartheta) \bigr)^2
\\
&= \Op(R_2^{-2} \cdot 1) +\Op(R_1^{-2})
\\
&= \Op(R_2^{-2}).
\end{align*}
\end{enumerate}
\end{proof}

For $b \in (0,1/2)$ and $u,v \in L^2((b,1-b)^2)$, we define
\begin{equation*}
\langle u,v \rangle_b = 
\frac{1}{(1-2b)^2} \int_{(b,1-b)^2} u(\mb{y}) v(\mb{y})\dd \mb{y},
\quad
\| u \|_b = \sqrt{\langle u,u \rangle_b}.
\end{equation*}
Let $\vartheta = (\kappa, \eta ,\theta_2)$ and 
$\vartheta^* = (\bs{\kappa}^*, \theta_2^*) = (\kappa^*, \eta^* ,\theta_2^*)$. We define
\begin{align*}
U(\vartheta,\vartheta^*;\alpha) 
&= \|f_{r,\alpha}(\cdot:\vartheta) -f_{r,\alpha}(\cdot:\vartheta^*) \|_b^2,
\\
\mathcal U(\vartheta; \alpha) &= 
\Bigl(
\bigl\langle 
\pd_{\vartheta_p} f_{r,\alpha}(\cdot:\vartheta), 
\pd_{\vartheta_q} f_{r,\alpha}(\cdot:\vartheta)
\bigr\rangle_b
\Bigr)_{1\le p,q \le 3}.
\end{align*}

Through the proof of Theorem 2.2 in \cite{TKU2026}, 
we have already obtained the following results.
\begin{prop}\label{propR1}
Let $\alpha^*, \alpha_0 \in (0,3)$ and $\delta/\sqrt{\Delta} \equiv r \in (0,\infty)$.
Under [A1]$_{\alpha_0}$ and [B]$_{\alpha^*, \alpha_0}$, it holds that
\begin{description}
\item[(R1)]
$U(\vartheta,\vartheta^*;\alpha^*) = U(\vartheta^*,\vartheta^*;\alpha^*)$
if and only if $\vartheta = \vartheta^*$,

\item[(R2)]
$\mathcal U(\vartheta^*;\alpha^*)$ is positive definite,

\item[(R3)]
$\displaystyle  \sup_{\vartheta \in \Xi_{\vartheta}} 
\bigl| U_{m,N}^{\epsilon}(\vartheta;\alpha^*) -U(\vartheta,\vartheta^*;\alpha^*) \bigr|
= \op(1)$,

\item[(R4)]
$R_1 \pd_{\vartheta} U_{m,N}(\vartheta^*;\alpha^*) = \Op(1)$,

\item[(R5)]
$\displaystyle 
\sup_{\vartheta; |\vartheta -\vartheta^*| \le \delta_{m,N}^\epsilon}
\bigl| \pd_{\vartheta}^2 U_{m,N}^{\epsilon}(\vartheta; \alpha^*) 
- 2\mathcal U(\vartheta^*;\alpha^*) \bigr| = \op(1)$
for $\delta_{m,N}^\epsilon \searrow 0$ as $m, N \to \infty$ and $\epsilon \searrow 0$.
\end{description}
\end{prop}
Since we have (R1) in Proposition \ref{propR1} and 
$U_{m,N}^{\epsilon} (\widehat \vartheta; \widehat \alpha)
\le U_{m,N}^{\epsilon} (\vartheta; \widehat \alpha)$ 
for any $\vartheta \in \Xi_{\vartheta}$,
we see that for any $\upsilon \in (0,\infty)$, 
there exists $\rho \in (0,\infty)$ such that
\begin{align*}
\PP \bigl( |\widehat \vartheta -\vartheta^*| \ge \upsilon \bigr)
&\le \PP \bigl( U(\widehat \vartheta,\vartheta^*;\alpha^*) 
-U(\vartheta^*,\vartheta^*;\alpha^*) \ge \rho \bigr)
\\
&= \PP \bigl( U(\widehat \vartheta,\vartheta^*;\alpha^*) 
-U_{m,N}^{\epsilon} (\widehat \vartheta; \widehat \alpha)
\\
&\qquad
+U_{m,N}^{\epsilon} (\widehat \vartheta; \widehat \alpha)
-U_{m,N}^{\epsilon} (\vartheta^*; \widehat \alpha)
\\
&\qquad
+U_{m,N}^{\epsilon} (\vartheta^*; \widehat \alpha)
-U(\vartheta^*,\vartheta^*;\alpha^*) \ge \rho \bigr)
\\
&\le \PP \biggl( U(\widehat \vartheta,\vartheta^*;\alpha^*) 
-U_{m,N}^{\epsilon} (\widehat \vartheta; \widehat \alpha)
\ge \frac{\rho}{3} \biggr)
\\
&\quad
+\PP \biggl( 
U_{m,N}^{\epsilon} (\widehat \vartheta; \widehat \alpha)
-U_{m,N}^{\epsilon} (\vartheta^*; \widehat \alpha)
\ge \frac{\rho}{3} \biggr)
\\
&\quad
+\PP \biggl( U_{m,N}^{\epsilon} (\vartheta^*; \widehat \alpha)
-U(\vartheta^*,\vartheta^*;\alpha^*) \ge \frac{\rho}{3} \biggr)
\\
&\le 2 \PP \biggl( \sup_{\vartheta \in \Xi_{\vartheta}} 
\bigl| U_{m,N}^{\epsilon}(\vartheta; \widehat \alpha) -U(\vartheta,\vartheta^*;\alpha^*)
\bigr| \ge \frac{\rho}{3} \biggr)
\\
&\quad
+\PP \biggl( 
U_{m,N}^{\epsilon} (\widehat \vartheta, \widehat \alpha)
-U_{m,N}^{\epsilon} (\vartheta^*, \widehat \alpha)
\ge \frac{\rho}{3} \biggr)
\\
&= 2 \PP \biggl( \sup_{\vartheta \in \Xi_{\vartheta}} 
\bigl| U_{m,N}^{\epsilon}(\vartheta; \widehat \alpha) -U(\vartheta,\vartheta^*;\alpha^*)
\bigr| \ge \frac{\rho}{3} \biggr).
\end{align*}
It thus suffices to show the following (F1) in order to prove the consistency 
of $\widehat \vartheta$.

\begin{description}
\item[(F1)]
$\displaystyle \widehat {\mathbb S}_{1,m,N}^\epsilon
= \sup_{\vartheta \in \Xi_{\vartheta}} 
\bigl| 
U_{m,N}^{\epsilon}(\vartheta; \widehat \alpha) -U(\vartheta,\vartheta^*;\alpha^*)
\bigr|
= \op(1)$.
\end{description}

Furthermore, we have by the Taylor expansion,
\begin{equation*}
-R_2 \pd_{\vartheta} U_{m,N}^{\epsilon}(\vartheta^*;\widehat \alpha)
=\int_0^1 \pd_{\vartheta}^2 
U_{m,N}^{\epsilon}(\widehat \vartheta_u;\widehat \alpha) \dd u 
R_2 (\widehat\vartheta-\vartheta^*),
\end{equation*}
where $\widehat \vartheta_u = \vartheta^* +u(\widehat \vartheta -\vartheta^*)$.
Since we have (R2) in Proposition \ref{propR1}, we show
\begin{description}
\item[(F2)]
$R_2 \pd_{\vartheta} U_{m,N}^{\epsilon}(\vartheta^*;\widehat \alpha) = \Op(1)$,

\item[(F3)]
$\displaystyle \widehat {\mathbb S}_{3,m,N}^\epsilon
= \sup_{u \in [0,1]}
\bigl| \pd^2 U_{m,N}^{\epsilon}(\widehat \vartheta_u;\widehat \alpha) 
- 2\mathcal U(\vartheta^*;\alpha^*) \bigr| = \op(1)$
\end{description}
in order to show Theorem \ref{th2}.

We introduce the following notation to show (F1)--(F3).
\begin{equation*}
\left\{
\begin{split}
\widehat {\mathbb T}_{1,m,N}^{\epsilon}
&= \sup_{\vartheta \in \Xi_{\vartheta}} \bigl| 
U_{m,N}^{\epsilon}(\vartheta;\widehat \alpha)
-U_{m,N}^{\epsilon}(\vartheta;\alpha^*)
\bigr|,
\\
\widehat {\mathbb T}_{2,m,N}^{\epsilon}
&= R_2
\bigl(
\pd_{\vartheta} U_{m,N}^{\epsilon}(\vartheta^*;\widehat \alpha)
-\pd_{\vartheta} U_{m,N}^{\epsilon}(\vartheta^*;\alpha^*)
\bigr),
\\
\widehat {\mathbb T}_{3,m,N}^{\epsilon}
&= \sup_{\vartheta \in \Xi_{\vartheta}} \bigl| 
\pd_{\vartheta}^2 U_{m,N}^{\epsilon}(\vartheta;\widehat \alpha)
-\pd_{\vartheta}^2 U_{m,N}^{\epsilon}(\vartheta;\alpha^*)
\bigr|.
\end{split}
\right.
\end{equation*}
Note that
\begin{equation}\label{dU}
\left\{
\begin{split}
U_{m,N}^{\epsilon}(\vartheta;\alpha)
&= \frac{1}{m} \sum_{k=1}^{m_2} \sum_{j=1}^{m_1} 
\mathbf{U}_{j,k}^{\epsilon}(\vartheta;\alpha)^2,
\\
\pd_{\vartheta} U_{m,N}^{\epsilon}(\vartheta;\alpha)
&= -\frac{2}{m} \sum_{k=1}^{m_2}\sum_{j=1}^{m_1} 
\mathbf{U}_{j,k}^{\epsilon}(\vartheta;\alpha) h_{j,k}^{(1)}(\vartheta;\alpha),
\\
\pd_{\vartheta}^2 U_{m,N}^{\epsilon}(\vartheta;\alpha)
&= \frac{2}{m} \sum_{k=1}^{m_2}\sum_{j=1}^{m_1} 
\bigl(
h_{j,k}^{(1)}(\vartheta;\alpha)
h_{j,k}^{(1)}(\vartheta;\alpha)^\TT
- \mathbf{U}_{j,k}^{\epsilon}(\vartheta;\alpha) h_{j,k}^{(2)}(\vartheta;\alpha)
\bigr).
\end{split}
\right.
\end{equation}

\begin{proof}[\bf{Proof of (F1)}]
Let
\begin{equation*}
\mathbb S_{1,m,N}^{\epsilon} = \sup_{\vartheta \in \Xi_{\vartheta}} 
\bigl| U_{m,N}^{\epsilon}(\vartheta;\alpha^*) -U(\vartheta,\vartheta^*;\alpha^*) \bigr|.
\end{equation*}
We then find from (R3) in Proposition \ref{propR1} 
that $\mathbb S_{1,m,N}^{\epsilon} = \op(1)$.
Since
\begin{align*}
\widehat {\mathbb S}_{1,m,N}^{\epsilon} 
&= \sup_{\vartheta \in \Xi_{\vartheta}} 
\bigl| 
U_{m,N}^{\epsilon}(\vartheta; \widehat \alpha) 
-U(\vartheta,\vartheta^*;\alpha^*)
\bigr|
\\
&\le 
\sup_{\vartheta \in \Xi_{\vartheta}} 
\bigl| 
U_{m,N}^{\epsilon}(\vartheta; \widehat \alpha) 
-U_{m,N}^{\epsilon}(\vartheta; \alpha^*)
\bigr|
+\sup_{\vartheta \in \Xi_{\vartheta}} 
\bigl| 
U_{m,N}^{\epsilon}(\vartheta; \alpha^*)
-U(\vartheta,\vartheta^*;\alpha^*)
\bigr|
\\
&= \widehat {\mathbb T}_{1,m,N}^{\epsilon} +{\mathbb S}_{1,m,N}^{\epsilon}
\\
&= \widehat {\mathbb T}_{1,m,N}^{\epsilon} +\op(1),
\end{align*}
we will show $\widehat {\mathbb T}_{1,m,N}^{\epsilon} = \op(1)$.

Note that for matrices $a, \widehat a \in \mathbb R^{d_1} \otimes \mathbb R^{d_2}$ 
and $b, \widehat b \in \mathbb R^{d_2} \otimes \mathbb R^{d_3}$ ($d_1, d_2, d_3 \in \mathbb N$),
\begin{equation}\label{decomp-1}
\widehat a \widehat b -ab = (\widehat a -a)(\widehat b -b)
+a(\widehat b -b) +(\widehat a -a)b.
\end{equation}
By \eqref{dU}, \eqref{decomp-1} and the Schwarz inequality, we obtain
\begin{align*}
&\bigl| U_{m,N}^{\epsilon}(\vartheta;\widehat \alpha) 
-U_{m,N}^{\epsilon}(\vartheta;\alpha^*) \bigr|
\\
&= \Biggl|
\frac{1}{m} \sum_{k=1}^{m_2} \sum_{j=1}^{m_1} 
\bigl(
\mathbf{U}_{j,k}^{\epsilon}(\vartheta;\widehat \alpha)^2 
-\mathbf{U}_{j,k}^{\epsilon}(\vartheta;\alpha^*)^2
\bigr)
\Biggr|
\\
&\le
\frac{1}{m}\sum_{k=1}^{m_2}\sum_{j=1}^{m_1} 
\bigl(
\mathbf{U}_{j,k}^{\epsilon}(\vartheta;\widehat \alpha) -\mathbf{U}_{j,k}^{\epsilon}(\vartheta;\alpha^*)
\bigr)^2
\\
&\quad+
2 \Biggl|
\frac{1}{m}\sum_{k=1}^{m_2}\sum_{j=1}^{m_1} 
\mathbf{U}_{j,k}^{\epsilon}(\vartheta;\alpha^*)
\bigl(
\mathbf{U}_{j,k}^{\epsilon}(\vartheta;\widehat \alpha) -\mathbf{U}_{j,k}^{\epsilon}(\vartheta;\alpha^*)
\bigr)
\Biggr|
\\
&\le
\frac{1}{m}\sum_{k=1}^{m_2}\sum_{j=1}^{m_1} 
\bigl(
\mathbf{U}_{j,k}^{\epsilon}(\vartheta;\widehat \alpha) -\mathbf{U}_{j,k}^{\epsilon}(\vartheta;\alpha^*)
\bigr)^2
\\
&\quad+
2\Biggl\{\frac{1}{m}\sum_{k=1}^{m_2}\sum_{j=1}^{m_1} 
\mathbf{U}_{j,k}^{\epsilon}(\vartheta;\alpha^*)^2 \Biggr\}^{1/2}
\Biggl\{
\frac{1}{m}\sum_{k=1}^{m_2}\sum_{j=1}^{m_1} 
\bigl(
\mathbf{U}_{j,k}^{\epsilon}(\vartheta;\widehat \alpha) -\mathbf{U}_{j,k}^{\epsilon}(\vartheta;\alpha^*)
\bigr)^2 \Biggr\}^{1/2}
\end{align*}
and see from the definitions of $\mathbf{Y}_{m,N}^{\epsilon}$ and 
$\mathbf{Z}_{m,N}^{\epsilon}$ given in \eqref{def_XYZ}, 
Lemma \ref{lem_XYZ} and [B]$_{\alpha^*,\alpha_0}$ that 
\begin{align*}
\widehat {\mathbb T}_{1,m,N}^{\epsilon}
&\le \mathbf{Z}_{m,N}^{\epsilon} 
+2 \sqrt{\mathbf{Y}_{m,N}^{\epsilon} \mathbf{Z}_{m,N}^{\epsilon}}
\\
&= \Op( R_2^{-2} ) +\Op( 1 \cdot R_2^{-1} )
\\
&= \op(1).
\end{align*}
\end{proof}

\begin{proof}[\bf{Proof of (F2)}]
Since we have $R_2/R_1 = 1/\log(N) = \oo(1)$ and (R4) in Proposition \ref{propR1}, we find 
\begin{align*}
R_2 \pd_{\vartheta} U_{m,N}^{\epsilon}(\vartheta^*;\widehat \alpha)
&= R_2 \bigl(
\pd_{\vartheta} U_{m,N}^{\epsilon}(\vartheta^*;\widehat \alpha)
-\pd_{\vartheta} U_{m,N}^{\epsilon}(\vartheta^*;\alpha^*) \bigr)
+ R_2 \pd_{\vartheta} U_{m,N}(\vartheta^*;\alpha^*)
\\
&= \widehat {\mathbb T}_{2,m,N}^{\epsilon}
+ \frac{R_2}{R_1} \times R_1 \pd_{\vartheta} U_{m,N}(\vartheta^*;\alpha^*)
\\
&= \widehat {\mathbb T}_{2,m,N}^{\epsilon} +\op(1).
\end{align*}
We thus prove $\widehat {\mathbb T}_{2,m,N}^{\epsilon} = \Op(1)$.
We see from \eqref{dU}, \eqref{decomp-1} and the Schwarz inequality that
\begin{align*}
&\frac{1}{2}\bigl| \pd_{\vartheta} U_{m,N}(\vartheta;\widehat \alpha)
-\pd_{\vartheta} U_{m,N}(\vartheta;\alpha^*) \bigr|
\\
&=
\Biggl|
\frac{1}{m} \sum_{k=1}^{m_2} \sum_{j=1}^{m_1}
\bigl(
\mathbf{U}_{j,k}^{\epsilon}(\vartheta;\widehat \alpha)
h_{j,k}^{(1)}(\vartheta;\widehat \alpha)
-\mathbf{U}_{j,k}^{\epsilon}(\vartheta;\alpha^*) h_{j,k}^{(1)}(\vartheta;\alpha^*)
\bigr)
\Biggr|
\\
&\le 
\Biggl|
\frac{1}{m} \sum_{k=1}^{m_2} \sum_{j=1}^{m_1}
\bigl(\mathbf{U}_{j,k}^{\epsilon}(\vartheta;\widehat \alpha)
-\mathbf{U}_{j,k}^{\epsilon}(\vartheta;\alpha^*) \bigr)
\bigl( 
h_{j,k}^{(1)}(\vartheta;\widehat \alpha)- h_{j,k}^{(1)}(\vartheta;\alpha^*) 
\bigr) 
\Biggr|
\\
&\quad +\Biggl|
\frac{1}{m} \sum_{k=1}^{m_2} \sum_{j=1}^{m_1}
\mathbf{U}_{j,k}^{\epsilon}(\vartheta;\alpha^*)
\bigl( 
h_{j,k}^{(1)}(\vartheta;\widehat \alpha)- h_{j,k}^{(1)}(\vartheta;\alpha^*) \bigr) 
\Biggr|
\\
&\quad+ \Biggl|
\frac{1}{m} \sum_{k=1}^{m_2} \sum_{j=1}^{m_1}
\bigl(\mathbf{U}_{j,k}^{\epsilon}(\vartheta;\widehat \alpha)
-\mathbf{U}_{j,k}^{\epsilon}(\vartheta;\alpha^*) \bigr)
h_{j,k}^{(1)}(\vartheta;\alpha^*) 
\Biggr|
\\
&\le 
\Biggl\{
\frac{1}{m} \sum_{k=1}^{m_2} \sum_{j=1}^{m_1}
\bigl( \mathbf{U}_{j,k}^{\epsilon}(\vartheta;\widehat \alpha)
-\mathbf{U}_{j,k}^{\epsilon}(\vartheta;\alpha^*) \bigr)^2
\Biggr\}^{1/2}
\\
&\qquad \times
\Biggl\{
\frac{1}{m} \sum_{k=1}^{m_2} \sum_{j=1}^{m_1}
\bigl| h_{j,k}^{(1)}(\vartheta;\widehat \alpha)
-h_{j,k}^{(1)}(\vartheta;\alpha^*) \bigr|^2
\Biggr\}^{1/2}
\\
&\quad +
\Biggl\{
\frac{1}{m} \sum_{k=1}^{m_2} \sum_{j=1}^{m_1}
\mathbf{U}_{j,k}^{\epsilon}(\vartheta;\alpha^*)^2
\Biggr\}^{1/2}
\Biggl\{
\frac{1}{m} \sum_{k=1}^{m_2} \sum_{j=1}^{m_1}
\bigl| h_{j,k}^{(1)}(\vartheta;\widehat \alpha)
-h_{j,k}^{(1)}(\vartheta;\alpha^*) \bigr|^2
\Biggr\}^{1/2}
\\
&\quad + \Biggl\{
\frac{1}{m} \sum_{k=1}^{m_2} \sum_{j=1}^{m_1}
\bigl( \mathbf{U}_{j,k}^{\epsilon}(\vartheta;\widehat \alpha)
-\mathbf{U}_{j,k}^{\epsilon}(\vartheta;\alpha^*) \bigr)^2
\Biggr\}^{1/2}
\Biggl\{
\frac{1}{m} \sum_{k=1}^{m_2} \sum_{j=1}^{m_1}
\bigl| h_{j,k}^{(1)}(\vartheta;\alpha^*) \bigr|^2
\Biggr\}^{1/2}.
\end{align*}
Since
\begin{equation}\label{def_W}
\mathbf{W}_{m,N}^l = \sup_{\vartheta \in \Xi_{\vartheta}} 
\frac{1}{m} \sum_{k=1}^{m_2} \sum_{j=1}^{m_1}
\bigl| h_{j,k}^{(l)}(\vartheta;\alpha^*) \bigr|^2
\end{equation}
satisfies
\begin{equation}\label{est_W}
\mathbf{W}_{m,N}^l = \OO(1)
\end{equation}
by the uniform convergence as in \eqref{unif_conv}, 
it holds from the definitions of $\mathbf{X}_{m,N}^{\epsilon,l}$, 
$\mathbf{Y}_{m,N}^{\epsilon}$ and $\mathbf{Z}_{m,N}^{\epsilon}$ 
given in \eqref{def_XYZ}, Lemma \ref{lem_XYZ}, 
[B]$_{\alpha^*,\alpha_0}$ and $R_2 = \oo(R_1)$ that
\begin{align*}
\frac{1}{2}\widehat {\mathbb T}_{2,m,N}^{\epsilon}
&\le R_2 \Bigl(
\sqrt{\mathbf{Z}_{m,N}^{\epsilon} \mathbf{X}_{m,N}^{\epsilon,1} }
+\sqrt{\mathbf{Y}_{m,N}^{\epsilon} \mathbf{X}_{m,N}^{\epsilon,1}}
+\sqrt{\mathbf{Z}_{m,N}^{\epsilon} \mathbf{W}_{m,N}^{1}}
\Bigr)
\\
&= \Op( R_2 \cdot R_2^{-1} \cdot R_1^{-1} ) 
+\Op( R_2 \cdot 1 \cdot R_1^{-1} ) +\Op( R_2 \cdot R_2^{-1} \cdot 1 )
\\
&= \Op( R_1^{-1} ) 
+\op( R_1 \cdot R_1^{-1} ) +\Op(1)
\\
&= \Op(1).
\end{align*}
\end{proof}

\begin{proof}[\bf{Proof of (F3)}]
Let
\begin{equation*}
{\mathbb S}_{3,m,N}^\epsilon
= \sup_{u \in [0,1]}
\bigl| \pd_\vartheta^2 U_{m,N}^{\epsilon}(\widehat \vartheta_u; \alpha^*) 
- 2\mathcal U(\vartheta^*;\alpha^*) \bigr|.
\end{equation*}
Since it follows from $\widehat \vartheta \pto \vartheta^*$ 
and (R5) in Proposition \ref{propR1} that $\mathbb S_{3,m,N}^{\epsilon} = \op(1)$,
we have 
\begin{align*}
\widehat {\mathbb S}_{3,m,N}^{\epsilon}
&= \sup_{u \in [0,1]}
\bigl| \pd^2 U_{m,N}^{\epsilon}(\widehat \vartheta_u;\widehat \alpha) 
- 2\mathcal U(\vartheta^*;\alpha^*) \bigr|
\\
&\le \sup_{u \in [0,1]}
\bigl| \pd^2 U_{m,N}^{\epsilon}(\widehat \vartheta_u;\widehat \alpha) 
- \pd^2 U_{m,N}^{\epsilon}(\widehat \vartheta_u;\alpha^*) \bigr|
+\sup_{u \in [0,1]}
\bigl| \pd^2 U_{m,N}^{\epsilon}(\widehat \vartheta_u;\alpha^*) 
- 2\mathcal U(\vartheta^*;\alpha^*) \bigr|
\\
&\le \sup_{\vartheta \in \Xi_{\vartheta}}
\bigl| \pd^2 U_{m,N}^{\epsilon}(\vartheta;\widehat \alpha) 
- \pd^2 U_{m,N}^{\epsilon}(\vartheta;\alpha^*) \bigr|
+\sup_{u \in [0,1]}
\bigl| \pd^2 U_{m,N}^{\epsilon}(\widehat \vartheta_u;\alpha^*) 
- 2\mathcal U(\vartheta^*;\alpha^*) \bigr|
\\
&= \widehat {\mathbb T}_{3,m,N}^{\epsilon} +{\mathbb S}_{3,m,N}^{\epsilon}
\\
&= \widehat {\mathbb T}_{3,m,N}^{\epsilon} +\op(1).
\end{align*}
We will show $\widehat {\mathbb T}_{3,m,N}^{\epsilon} = \op(1)$.
Since we find from \eqref{dU}, \eqref{decomp-1} and the Schwarz inequality that
\begin{align*}
&\frac{1}{2} \bigl| \pd_{\vartheta}^2 U_{m,N}(\vartheta;\widehat \alpha)
-\pd_{\vartheta}^2 U_{m,N}(\vartheta;\alpha^*) \bigr|
\\
&=
\Biggl| 
\frac{1}{m} \sum_{k=1}^{m_2} \sum_{j=1}^{m_1}
\Bigl( 
\bigl( h_{j,k}^{(1)}(\vartheta;\widehat \alpha) 
h_{j,k}^{(1)}(\vartheta;\widehat \alpha)^\TT
-h_{j,k}^{(1)}(\vartheta;\alpha^*) h_{j,k}^{(1)}(\vartheta;\alpha^*)^\TT 
\bigr) 
\\
&\qquad -\bigl( 
\mathbf{U}_{j,k}^{\epsilon}(\vartheta;\widehat \alpha)
h_{j,k}^{(2)}(\vartheta;\widehat \alpha)
-\mathbf{U}_{j,k}^{\epsilon}(\vartheta;\alpha^*) h_{j,k}^{(2)}(\vartheta;\alpha^*)
\bigr)
\Bigr)
\Biggr|
\\
&\le 
\Biggl| 
\frac{1}{m} \sum_{k=1}^{m_2} \sum_{j=1}^{m_1}
\bigl( h_{j,k}^{(1)}(\vartheta;\widehat \alpha) 
-h_{j,k}^{(1)}(\vartheta;\alpha^*) \bigr)
\bigl( h_{j,k}^{(1)}(\vartheta;\widehat \alpha) 
-h_{j,k}^{(1)}(\vartheta;\alpha^*) \bigr)^\TT
\Biggr|
\\
&\quad+
2 \Biggl| 
\frac{1}{m} \sum_{k=1}^{m_2} \sum_{j=1}^{m_1}
\bigl( h_{j,k}^{(1)}(\vartheta;\widehat \alpha) 
-h_{j,k}^{(1)}(\vartheta;\alpha^*) \bigr)
h_{j,k}^{(1)}(\vartheta;\alpha^*)^\TT
\Biggr|
\\
&\quad+
\Biggl|
\frac{1}{m} \sum_{k=1}^{m_2} \sum_{j=1}^{m_1}
\bigl(\mathbf{U}_{j,k}^{\epsilon}(\vartheta;\widehat \alpha)
-\mathbf{U}_{j,k}^{\epsilon}(\vartheta;\alpha^*) \bigr)
\bigl( 
h_{j,k}^{(2)}(\vartheta;\widehat \alpha)- h_{j,k}^{(2)}(\vartheta;\alpha^*) 
\bigr) 
\Biggr|
\\
&\quad + \Biggl|
\frac{1}{m} \sum_{k=1}^{m_2} \sum_{j=1}^{m_1}
\mathbf{U}_{j,k}^{\epsilon}(\vartheta;\alpha^*)
\bigl( 
h_{j,k}^{(2)}(\vartheta;\widehat \alpha)- h_{j,k}^{(2)}(\vartheta;\alpha^*) \bigr) 
\Biggr|
\\
&\quad+ \Biggl|
\frac{1}{m} \sum_{k=1}^{m_2} \sum_{j=1}^{m_1}
\bigl(\mathbf{U}_{j,k}^{\epsilon}(\vartheta;\widehat \alpha)
-\mathbf{U}_{j,k}^{\epsilon}(\vartheta;\alpha^*) \bigr)
h_{j,k}^{(2)}(\vartheta;\alpha^*) 
\Biggr|
\\
&\le
\frac{1}{m} \sum_{k=1}^{m_2} \sum_{j=1}^{m_1}
\bigl| h_{j,k}^{(1)}(\vartheta;\widehat \alpha)
-h_{j,k}^{(1)}(\vartheta;\alpha^*) \bigr|^2
\\
&\quad+
2\Biggl\{
\frac{1}{m} \sum_{k=1}^{m_2} \sum_{j=1}^{m_1}
\bigl| h_{j,k}^{(1)}(\vartheta;\widehat \alpha)
-h_{j,k}^{(1)}(\vartheta;\alpha^*) \bigr|^2
\Biggr\}^{1/2}
\Biggl\{
\frac{1}{m} \sum_{k=1}^{m_2} \sum_{j=1}^{m_1}
\bigl| h_{j,k}^{(1)}(\vartheta;\alpha^*) \bigr|^2
\Biggr\}^{1/2}
\\
&\quad + \Biggl\{
\frac{1}{m} \sum_{k=1}^{m_2} \sum_{j=1}^{m_1}
\bigl( \mathbf{U}_{j,k}^{\epsilon}(\vartheta;\widehat \alpha)
-\mathbf{U}_{j,k}^{\epsilon}(\vartheta;\alpha^*) \bigr)^2
\Biggr\}^{1/2}
\\
&\qquad \times
\Biggl\{
\frac{1}{m} \sum_{k=1}^{m_2} \sum_{j=1}^{m_1}
\bigl| h_{j,k}^{(2)}(\vartheta;\widehat \alpha)
-h_{j,k}^{(2)}(\vartheta;\alpha^*) \bigr|^2
\Biggr\}^{1/2}
\\
&\quad +
\Biggl\{
\frac{1}{m} \sum_{k=1}^{m_2} \sum_{j=1}^{m_1}
\mathbf{U}_{j,k}^{\epsilon}(\vartheta;\alpha^*)^2
\Biggr\}^{1/2}
\Biggl\{
\frac{1}{m} \sum_{k=1}^{m_2} \sum_{j=1}^{m_1}
\bigl| h_{j,k}^{(2)}(\vartheta;\widehat \alpha)
-h_{j,k}^{(2)}(\vartheta;\alpha^*) \bigr|^2
\Biggr\}^{1/2}
\\
&\quad + \Biggl\{
\frac{1}{m} \sum_{k=1}^{m_2} \sum_{j=1}^{m_1}
\bigl( \mathbf{U}_{j,k}^{\epsilon}(\vartheta;\widehat \alpha)
-\mathbf{U}_{j,k}^{\epsilon}(\vartheta;\alpha^*) \bigr)^2
\Biggr\}^{1/2}
\Biggl\{
\frac{1}{m} \sum_{k=1}^{m_2} \sum_{j=1}^{m_1}
\bigl| h_{j,k}^{(2)}(\vartheta;\alpha^*) \bigr|^2
\Biggr\}^{1/2},
\end{align*}
it follows from the definitions of $\mathbf{X}_{m,N}^{\epsilon,l}$, 
$\mathbf{Y}_{m,N}^{\epsilon}$, $\mathbf{Z}_{m,N}^{\epsilon}$ 
given in \eqref{def_XYZ} and $\mathbf{W}_{m,N}^l$ given in \eqref{def_W},
Lemma \ref{lem_XYZ}, \eqref{est_W} and [B]$_{\alpha^*,\alpha_0}$ that
\begin{align*}
\frac{1}{2}\widehat {\mathbb T}_{3,m,N}^{\epsilon}
&\le 
\mathbf{X}_{m,N}^{\epsilon,1}
+ 2\sqrt{\mathbf{X}_{m,N}^{\epsilon,1} \mathbf{W}_{m,N}^{1} }
+ \sqrt{\mathbf{Z}_{m,N}^{\epsilon} \mathbf{X}_{m,N}^{\epsilon,2} }
+ \sqrt{\mathbf{Y}_{m,N}^{\epsilon} \mathbf{X}_{m,N}^{\epsilon,2}}
+ \sqrt{\mathbf{Z}_{m,N}^{\epsilon} \mathbf{W}_{m,N}^{2}}
\\
&= \Op( R_1^{-2} )
+ \Op( R_1^{-1} \cdot 1 )
+ \Op( R_2^{-1} \cdot R_1^{-1} )
+ \Op( 1 \cdot R_1^{-1} )
+ \Op( R_2^{-1} \cdot 1 )
\\
&= \op(1).
\end{align*}
\end{proof}

\subsection{Proof of Theorem \ref{th3}}
For $\lambda, \mu \in (0,\infty)$, 
a discretely observed stochastic process 
$\widetilde {\mb x} = \{ \widetilde x(t_i^n) \}_{i=0}^n$ 
and $\alpha \in (0,3)$, we define
\begin{align*}
V_n^{\epsilon} (\lambda,\mu;\widetilde {\mb x}, \alpha)
&= \sum_{i=1}^n
\frac{(\widetilde x(t_i^n)
- \ee^{-\lambda \Delta_n} \widetilde x(t_{i-1}^n))^2}
{\frac{\epsilon^2(1-\ee^{-2\lambda \Delta_n})}{2\lambda \mu^\alpha}}
+n\log \biggl( \frac{1-\ee^{-2\lambda \Delta_n}}{2\lambda \mu^\alpha \Delta_n} \biggr),
\\
L_n^{\epsilon} (\lambda,\mu;\widetilde {\mb x},\alpha)
&=
\begin{pmatrix}
-\epsilon\partial_\lambda V_n^{\epsilon} (\lambda,\mu;\widetilde {\mb x},\alpha)
\\
-\frac{1}{\sqrt n}\partial_\mu V_n^{\epsilon} (\lambda,\mu;\widetilde {\mb x},\alpha)
\end{pmatrix},
\\
K_n^{\epsilon} (\lambda,\mu;\widetilde {\mb x},\alpha)
&=
\begin{pmatrix}
\epsilon^2 \partial_\lambda^2 V_n^{\epsilon} (\lambda,\mu;\widetilde {\mb x},\alpha) 
& \frac{\epsilon}{\sqrt n}
\partial_\lambda \partial_\mu V_n^{\epsilon} (\lambda,\mu;\widetilde {\mb x},\alpha)
\\
\frac{\epsilon}{\sqrt n}
\partial_\mu \partial_\lambda V_n^{\epsilon} (\lambda,\mu;\widetilde {\mb x},\alpha)
& \frac{1}{n} \partial_\mu^2 V_n^{\epsilon} (\lambda,\mu;\widetilde {\mb x},\alpha)
\end{pmatrix}.
\end{align*}

For $\ell = (\mf l_1, \mf l_2) \in \mathbb N^2$ obtained from [A2], 
we set the discretely observed true coordinate process 
$\textbf x = {\mb x}_{\ell} = \{ x_{\ell}(t_i^n) \}_{i=0}^n$
given in \eqref{OU-process} with the initial value $x_0 = x_{\ell}(0)$,
the approximate coordinate process
$\widehat {\mb x} = \widehat {\mb x}_{\ell}
= \{ \widehat x_{\ell}(t_i^n) \}_{i=0}^n$
defined by \eqref{approx_cp0},
and the estimators 
$(\widehat \lambda, \widehat \mu) = (\widehat \lambda_{\ell}, \widehat \mu_{\ell})$
given in \eqref{estimator_lam_mu}, 
and the true values $(\lambda^*,\mu^*) = (\lambda_{\ell}^*, \mu_{\ell}^*)$,
where
\begin{equation}\label{true_values}
\lambda_{\ell}^* = \theta_2^* \biggl( \pi^2 |\ell|^2 + \frac{(\kappa^*)^2 +(\eta^*)^2}{4} \biggr) -\theta_0^*, 
\quad
\mu_{\ell}^* = \mu_0^* +\pi^2 |\ell|^2.
\end{equation}
For $\lambda_1,\lambda_2, \mu_1,\mu_2 \in (0,\infty)$ and $\alpha \in (0,3)$, 
we define
\begin{equation*}
\begin{split}
V_1(\lambda_1, \mu_1, \lambda_2;\alpha) 
&= \mu_1^\alpha (\lambda_1 -\lambda_2)^2 \frac{1- \exp(-2\lambda_2)}{2 \lambda_2} x_0^2,
\\
V_2(\mu_1, \mu_2;\alpha) 
&= \biggl( \frac{\mu_1}{\mu_2} \biggr)^\alpha -1
-\alpha \log \biggl( \frac{\mu_1}{\mu_2} \biggr).
\end{split}
\end{equation*}

Through the proof of Theorem A.2 in \cite{TKU2024a}, 
we have already obtained the following results.
\begin{prop}\label{propR2}
Under $n \to \infty$ and $\epsilon \searrow 0$, it holds that
for estimators $(\widetilde \lambda, \widetilde \mu) \in \Xi_{\nu}$,
\begin{description}
\item[(R6)]
$V_1(\lambda, \mu, \lambda^*;\alpha^*) = V_1(\lambda^*, \mu, \lambda^*;\alpha^*) (= 0)$
if and only if $\lambda = \lambda^*$,

\item[(R7)]
$V_2(\mu, \mu^*;\alpha^*) = V_2(\mu^*, \mu^*;\alpha^*) (= 0)$ 
if and only if $\mu = \mu^*$,

\item[(R8)]
$\displaystyle  
\sup_{(\lambda,\mu) \in \Xi_{\nu}} \Bigl|
\epsilon^2 \bigl( V_n^{\epsilon} (\lambda,\mu; {\mb x},\alpha^*)
-V_n^{\epsilon} (\lambda^*,\mu; {\mb x},\alpha^*) \bigr) 
- V_1(\lambda,\mu,\lambda^*; \alpha^*)
\Bigr| = \op(1)$,

\item[(R9)]
$\displaystyle \sup_{\mu \in \Xi_{\mu}}
\biggl|
\frac{1}{n}
\bigl( 
V_n^{\epsilon} (\widetilde \lambda,\mu; {\mb x},\alpha^*)
-V_n^{\epsilon} (\widetilde \lambda,\mu^*; {\mb x},\alpha^*)
\bigr) 
- V_2(\mu,\mu^*; \alpha^*)
\biggr| = \op(1)$
under $\epsilon^{-1}(\widetilde \lambda -\lambda^*) = \Op(1)$,

\item[(R10)]
$\displaystyle \epsilon \sup_{\mu \in \Xi_{\mu}} 
\bigl| \pd_\lambda V_n^{\epsilon} (\lambda^*, \mu; {\mb x},\alpha^*) \bigr|
= \op(1)$,

\item[(R11)]
$\displaystyle \sup_{(\lambda,\mu) \in \Xi_{\nu}} 
\bigl| \epsilon^2 \pd_\lambda^2 V_n^{\epsilon} (\lambda, \mu; {\mb x},\alpha^*) 
-2 G(\lambda^*,\mu; x_0, \alpha^*) \bigr| = \op(1)$,

\item[(R12)]
$L_n^{\epsilon} (\lambda^*,\mu^*; {\mb x},\alpha^*)
\dto \mathrm{N}(0, 4 \mathcal I^{-1})$,

\item[(R13)]
$\displaystyle \sup_{u \in [0,1]}
\bigl| 
K_n^{\epsilon} (\widetilde \lambda_u, \widetilde \mu_u; {\mb x},\alpha^*) 
-2 \mathcal I^{-1}
\bigr| = \op(1)$
under $\epsilon^{-1}(\widetilde \lambda -\lambda^*) = \Op(1)$ and $\widetilde \mu -\mu^* = \op(1)$,
where $\widetilde \lambda_u = \lambda^* +u(\widetilde \lambda -\lambda^*)$
and $\widetilde \mu_u = \mu^* +u(\widetilde \mu -\mu^*)$ for $u \in [0,1]$.
\end{description}
\end{prop}

Since we have (R6) and (R7) in Proposition \ref{propR2}, 
$V_1(\lambda^*, \mu, \lambda^*;\alpha^*) = 0$, 
$V_2(\mu^*, \lambda^*;\alpha^*) = 0$ and 
$V_n^{\epsilon} (\widehat \lambda,\widehat \mu; \widehat {\mb x},\widehat \alpha)
\le V_n^{\epsilon} (\lambda, \mu;\widehat {\mb x}, \widehat \alpha)$ 
for the estimators $(\widehat \lambda, \widehat \mu)$
and any $(\lambda,\mu) \in \Xi_{\nu}$, 
we find that for any $\upsilon \in (0,\infty)$, 
there exists $\rho \in (0,\infty)$ such that
\begin{align*}
\PP \bigl( |\widehat \lambda -\lambda^*| \ge \upsilon \bigr)
&\le \PP \bigl( V_1(\widehat \lambda, \widehat \mu, \lambda^*;\alpha^*) 
-V_1(\lambda^*, \widehat \mu, \lambda^*;\alpha^*)
\ge \rho \bigr)
\\
&= \PP \Bigl( 
V_1(\widehat \lambda, \widehat \mu, \lambda^*;\alpha^*) 
- \epsilon^2 
\bigl( 
V_n^{\epsilon} (\widehat \lambda,\widehat \mu; \widehat {\mb x},\widehat \alpha)
-V_n^{\epsilon} (\lambda^*,\widehat \mu; \widehat {\mb x},\widehat \alpha)
\bigr) 
\\
&\qquad
+\epsilon^2 
\bigl( 
V_n^{\epsilon} (\widehat \lambda,\widehat \mu; \widehat {\mb x},\widehat \alpha)
-V_n^{\epsilon} (\lambda^*,\widehat \mu; \widehat {\mb x},\widehat \alpha)
\bigr)
-V_1(\lambda^*, \widehat \mu, \lambda^*;\alpha^*)
\ge \rho \Bigr)
\\
&\le \PP \biggl( 
V_1(\widehat \lambda, \widehat \mu, \lambda^*;\alpha^*) 
- \epsilon^2 
\bigl( 
V_n^{\epsilon} (\widehat \lambda,\widehat \mu; \widehat {\mb x},\widehat \alpha)
-V_n^{\epsilon} (\lambda^*,\widehat \mu; \widehat {\mb x},\widehat \alpha)
\bigr)
\ge \frac{\rho}{2} \biggr)
\\
&\quad
+\PP \biggl( \epsilon^2 
\bigl( 
V_n^{\epsilon} (\widehat \lambda,\widehat \mu; \widehat {\mb x},\widehat \alpha)
-V_n^{\epsilon} (\lambda^*,\widehat \mu; \widehat {\mb x},\widehat \alpha)
\bigr)
\ge \frac{\rho}{2} \biggr)
\\
&\le \PP \biggl( 
\sup_{(\lambda,\mu) \in \Xi_{\nu}}
\Bigl|
V_1(\lambda,\mu,\lambda^*; \alpha^*)
- \epsilon^2 
\bigl( 
V_n^{\epsilon} (\lambda,\mu; \widehat {\mb x},\widehat \alpha)
-V_n^{\epsilon} (\lambda^*,\mu; \widehat {\mb x},\widehat \alpha)
\bigr) 
\Bigr|
\ge \frac{\rho}{2} \biggr),
\end{align*}
\begin{align*}
\PP \bigl( |\widehat \mu -\mu^*| \ge \upsilon \bigr)
&\le \PP \bigl( V_2(\widehat \mu, \mu^*;\alpha^*) -V_2(\mu^*, \mu^*;\alpha^*)
\ge \rho \bigr)
\\
&= \PP \biggl( 
V_2(\widehat \mu, \mu^*;\alpha^*) 
- \frac{1}{n}
\bigl( 
V_n^{\epsilon} (\widehat \lambda,\widehat \mu; \widehat {\mb x},\widehat \alpha)
-V_n^{\epsilon} (\widehat \lambda,\mu^*; \widehat {\mb x},\widehat \alpha)
\bigr) 
\\
&\qquad
+\frac{1}{n}
\bigl( 
V_n^{\epsilon} (\widehat \lambda,\widehat \mu; \widehat {\mb x},\widehat \alpha)
-V_n^{\epsilon} (\widehat \lambda,\mu^*; \widehat {\mb x},\widehat \alpha)
\bigr)
-V_2(\mu^*, \mu^*;\alpha^*)
\ge \rho \Bigr)
\\
&\le \PP \biggl( 
V_2(\widehat \mu, \mu^*;\alpha^*) 
- \frac{1}{n}
\bigl( 
V_n^{\epsilon} (\widehat \lambda,\widehat \mu; \widehat {\mb x},\widehat \alpha)
-V_n^{\epsilon} (\widehat \lambda,\mu^*; \widehat {\mb x},\widehat \alpha)
\bigr) 
\ge \frac{\rho}{2} \biggr)
\\
&\quad
+\PP \biggl( 
\frac{1}{n}
\bigl( 
V_n^{\epsilon} (\widehat \lambda,\widehat \mu; \widehat {\mb x},\widehat \alpha)
-V_n^{\epsilon} (\widehat \lambda,\mu^*; \widehat {\mb x},\widehat \alpha)
\bigr)
\ge \frac{\rho}{2} \biggr)
\\
&\le \PP \biggl( 
\sup_{\mu \in \Xi_{\mu}}
\biggl|
V_2(\mu,\mu^*; \alpha^*)
-\frac{1}{n}
\bigl( 
V_n^{\epsilon} (\widehat \lambda,\mu; \widehat {\mb x},\widehat \alpha)
-V_n^{\epsilon} (\widehat \lambda,\mu^*; 
\widehat {\mb x},\widehat \alpha)
\bigr) 
\biggr| 
\ge \frac{\rho}{2} \biggr).
\end{align*}
Therefore, the estimators $\widehat \lambda$ and $\widehat \mu$ satisfy
$\widehat \lambda \pto \lambda^*$ and $\widehat \mu \pto \mu^*$
if the following (G1)--(G2) hold. 
\begin{description}
\item[(G1)]
$\displaystyle
\widehat {\mathbb U}_{1,n}^\epsilon
= \sup_{(\lambda,\mu) \in \Xi_{\nu}}
\Bigl|
\epsilon^2 
\bigl( 
V_n^{\epsilon} (\lambda,\mu; \widehat {\mb x},\widehat \alpha)
-V_n^{\epsilon} (\lambda^*,\mu; \widehat {\mb x},\widehat \alpha)
\bigr) 
- V_1(\lambda,\mu,\lambda^*; \alpha^*)
\Bigr| 
= \op(1)$,

\item[(G2)]
$\displaystyle
\widehat {\mathbb U}_{2,n}^\epsilon
= \sup_{\mu \in \Xi_{\mu}}
\biggl|
\frac{1}{n}
\bigl( 
V_n^{\epsilon} (\widehat \lambda,\mu; \widehat {\mb x},\widehat \alpha)
-V_n^{\epsilon} (\widehat \lambda,\mu^*; 
\widehat {\mb x},\widehat \alpha)
\bigr) 
- V_2(\mu,\mu^*; \alpha^*)
\biggr| 
= \op(1)$.
\end{description}
Moreover, it follows from the Taylor expansion that
\begin{equation*}
L_n^{\epsilon} (\lambda^*, \mu^*; \widehat {\mb x}, \widehat \alpha)
= \int_0^1 
K_n^{\epsilon} (\widehat \lambda_u,\widehat \mu_u; \widehat {\mb x},\widehat \alpha)
\dd u
\begin{pmatrix}
\epsilon^{-1} (\widehat \lambda -\lambda^*)
\\
\sqrt{n} (\widehat \mu -\mu^*)
\end{pmatrix}
\end{equation*}
with $\widehat \lambda_u = \lambda^* +u(\widehat \lambda -\lambda^*)$ and $\widehat \mu_u = \mu^* +u(\widehat \mu -\mu^*)$.
We thus obtain the desired results under (G1)--(G2) and the following (G3)--(G4).
\begin{description}
\item[(G3)]
$L_n^{\epsilon} (\lambda^*,\mu^*; \widehat {\mb x},\widehat \alpha)
\dto \mathrm{N}(0, 4 \mathcal I^{-1})$,

\item[(G4)]
$\displaystyle
\widehat {\mathbb U}_{4,n}^\epsilon
= \sup_{u \in [0,1]}
\bigl| 
K_n^{\epsilon} (\widehat \lambda_u, \widehat \mu_u; 
\widehat {\mb x},\widehat \alpha) -2 \mathcal I^{-1}
\bigr| = \op(1)$.
\end{description}

We introduce the following notation to prove (G1)--(G4).
\begin{equation}\label{hat_V}
\left\{ 
\begin{split}
\widehat {\mathbb V}_{1,n}^\epsilon
&=\epsilon^{2}\sup_{(\lambda,\mu) \in \Xi_{\nu}}
\bigl|
V_n^{\epsilon} (\lambda,\mu;\widehat {\mb x}, \alpha^*)
-V_n^{\epsilon} (\lambda,\mu;{\mb x}, \alpha^*)
\bigr|,
\\
\widehat {\mathbb V}_{2,n}^\epsilon
&=\frac{1}{n}\sup_{(\lambda,\mu) \in \Xi_{\nu}}
\bigl| V_n^{\epsilon} (\lambda,\mu;\widehat {\mb x}, \alpha^*)
-V_n^{\epsilon} (\lambda,\mu;{\mb x}, \alpha^*) \bigr|,
\\
\widehat {\mathbb V}_{3,n}^\epsilon
&=\sup_{\mu \in \Xi_{\mu}}
\bigl| L_n^{\epsilon} (\lambda^*,\mu;\widehat {\mb x}, \alpha^*)
- L_n^{\epsilon} (\lambda^*,\mu;{\mb x}, \alpha^*) \bigr|,
\\
\widehat {\mathbb V}_{4,n}^\epsilon
&=\sup_{(\lambda,\mu) \in \Xi_{\nu}}
\bigl| K_n^{\epsilon} (\lambda,\mu;\widehat {\mb x}, \alpha^*)
-K_n^{\epsilon} (\lambda,\mu;{\mb x}, \alpha^*) \bigr|
\end{split}
\right.
\end{equation}
and
\begin{equation}\label{hat_W}
\left\{ 
\begin{split}
\widehat {\mathbb W}_{1,n}^\epsilon
&=\epsilon^{2}\sup_{(\lambda,\mu) \in \Xi_{\nu}}
\bigl|
V_n^{\epsilon} (\lambda,\mu;\widehat {\mb x},\widehat \alpha)
-V_n^{\epsilon} (\lambda,\mu;\widehat {\mb x}, \alpha^*)
\bigr|,
\\
\widehat {\mathbb W}_{2,n}^\epsilon
&=\frac{1}{n}\sup_{(\lambda,\mu) \in \Xi_{\nu}}
\bigl| V_n^{\epsilon} (\lambda,\mu;\widehat {\mb x},\widehat \alpha)
-V_n^{\epsilon} (\lambda,\mu;\widehat {\mb x}, \alpha^*) \bigr|,
\\
\widehat {\mathbb W}_{3,n}^\epsilon
&=\sup_{\mu \in \Xi_{\mu}}
\bigl| L_n^{\epsilon} (\lambda^*,\mu;\widehat {\mb x},\widehat \alpha)
-L_n^{\epsilon} (\lambda^*,\mu;\widehat {\mb x}, \alpha^*) \bigr|,
\\
\widehat {\mathbb W}_{4,n}^\epsilon
&=\sup_{(\lambda,\mu) \in \Xi_{\nu}}
\bigl| K_n^{\epsilon} (\lambda,\mu;\widehat {\mb x},\widehat \alpha)
-K_n^{\epsilon} (\lambda,\mu;\widehat {\mb x}, \alpha^*) \bigr|.
\end{split}
\right.
\end{equation}
We define $\pd_1 = \epsilon \pd_\lambda$ and $\pd_2 = \frac{1}{\sqrt{n}} \pd_\mu$, 
and set 
\begin{align*}
(\widehat {\mathbb V}_{3,n}^\epsilon)_i 
&= \sup_{\mu \in \Xi_{\mu}} 
\bigl| \pd_i V_n^{\epsilon} (\lambda^*,\mu;\widehat {\mb x}, \alpha^*)
-\pd_i V_n^{\epsilon} (\lambda^*,\mu;{\mb x}, \alpha^*) \bigr|,
\\
(\widehat {\mathbb V}_{4,n}^\epsilon)_{i,j} 
&= \sup_{(\lambda,\mu) \in \Xi_{\nu}} 
\bigl| \pd_j \pd_i V_n^{\epsilon} (\lambda,\mu; \widehat {\mb x}, \alpha^*)
-\pd_j \pd_i V_n^{\epsilon} (\lambda,\mu; {\mb x}, \alpha^*) \bigr|,
\\
(\widehat {\mathbb W}_{3,n}^\epsilon)_i 
&= \sup_{\mu \in \Xi_{\mu}} 
\bigl| \pd_i V_n^{\epsilon} (\lambda^*,\mu; \widehat {\mb x}, \widehat \alpha)
-\pd_i V_n^{\epsilon} (\lambda^*,\mu; \widehat {\mb x}, \alpha^*) \bigr|,
\\
(\widehat {\mathbb W}_{4,n}^\epsilon)_{i,j} 
&= \sup_{(\lambda,\mu) \in \Xi_{\nu}} 
\bigl| \pd_j \pd_i V_n^{\epsilon} (\lambda,\mu;\widehat {\mb x}, \widehat \alpha)
-\pd_j \pd_i V_n^{\epsilon} (\lambda,\mu;\widehat {\mb x}, \alpha^*) \bigr|
\end{align*}
for $i,j \in \{ 1, 2 \}$. 
We then obtain
\begin{equation*}
\widehat {\mathbb V}_{3,n}^\epsilon 
\le \sum_{i \in \{1,2\}} (\widehat {\mathbb V}_{3,n}^\epsilon)_i,
\quad
\widehat {\mathbb W}_{3,n}^\epsilon 
\le \sum_{i \in \{1,2\}} (\widehat {\mathbb W}_{3,n}^\epsilon)_i,
\end{equation*}
\begin{equation*}
\widehat {\mathbb V}_{4,n}^\epsilon 
\le \sum_{i,j \in \{1,2\}} (\widehat {\mathbb V}_{4,n}^\epsilon)_{i,j},
\quad
\widehat {\mathbb W}_{4,n}^\epsilon 
\le \sum_{i,j \in \{1,2\}} (\widehat {\mathbb W}_{4,n}^\epsilon)_{i,j}.
\end{equation*}

We introduce some notation to estimate the differences 
$\widehat {\mathbb V}_{j,n}^\epsilon$, $\widehat {\mathbb W}_{j,n}^\epsilon$
($j \in \{ 1,\ldots, 4 \}$) given in \eqref{hat_V} and \eqref{hat_W}.
For a discretely observed stochastic process 
$\widetilde {\mb x} = \{ \widetilde x(t_i^n) \}_{i=0}^n$, we define
\begin{equation*}
\mathcal M_i(\lambda;\widetilde {\mb x}) = 
\widetilde x(t_i^n) -\ee^{-\lambda \Delta_n} \widetilde x(t_{i-1}^n).
\end{equation*}
For $\textbf x = \{ x(t_i^n) \}_{i=0}^n = {\mb x}_{\ell}$ and
$\widehat {\mb x} = \{ \widehat x(t_i^n) \}_{i=0}^n
= \widehat {\mb x}_{\ell}$, we set
\begin{equation}\label{def_XY}
\left\{
\begin{split}
\mathcal X_n &= 
\sup_{\lambda \in \Xi_{\lambda}}
\Biggl|
\sum_{i=1}^n 
\bigl( 
\mathcal M_i(\lambda;\widehat {\mb x})^2 -\mathcal M_i(\lambda; {\mb x})^2
\bigr)
\Biggr|,
\\
\mathcal Y_n &=
\sup_{\lambda \in \Xi_{\lambda}}
\Biggl|
\sum_{i=1}^n \bigl( 
\mathcal M_i(\lambda;\widehat {\mathbf x}) \widehat x(t_{i-1}^n)
- \mathcal M_i(\lambda; {\mathbf x}) x(t_{i-1}^n)
\bigr)
\Biggr|,
\end{split}
\right.
\end{equation}
\begin{equation}\label{def_XYZ*}
\left\{
\begin{split}
\mathcal X_n^* &= 
\Biggl|
\sum_{i=1}^n 
\bigl( 
\mathcal M_i(\lambda^*;\widehat {\mb x})^2 -\mathcal M_i(\lambda^*; {\mb x})^2
\bigr)
\Biggr|,
\\
\mathcal Y_n^* &=
\Biggl|
\sum_{i=1}^n \bigl( 
\mathcal M_i(\lambda^*;\widehat {\mathbf x}) \widehat x(t_{i-1}^n)
- \mathcal M_i(\lambda^*; {\mathbf x}) x(t_{i-1}^n)
\bigr)
\Biggr|,
\\
\mathcal Z_n^* &= 
\Biggl|
\sum_{i=1}^n 
\bigl( 
\widehat x(t_{i-1}^n)^2 -x(t_{i-1}^n)^2
\bigr)
\Biggr|.
\end{split}
\right.
\end{equation}
We also introduce 
\begin{equation}\label{def_A--E}
\small
\left\{
\begin{split}
\mathcal A_n &= 
\sup_{\lambda \in \Xi_{\lambda}} \sum_{i=1}^n 
\bigl( 
\mathcal M_i(\lambda;\widehat {\mathbf x}) -\mathcal M_i(\lambda; {\mathbf x})
\bigr)^2,
\\
\mathcal C_n &= \sup_{\lambda \in \Xi_{\lambda}} \sum_{i=1}^n \mathcal M_i(\lambda;{\mb x})^2,
\\
\mathcal D_n &= \sup_{\lambda \in \Xi_{\lambda}} \Biggl| \sum_{i=1}^n 
\mathcal M_i(\lambda;{\mb x}) x(t_{i-1}^n) \Biggr|,
\end{split}
\right.
\quad
\left\{
\begin{split}
\mathcal A_n^* &= 
\sum_{i=1}^n 
\bigl( 
\mathcal M_i(\lambda^*;\widehat {\mathbf x}) -\mathcal M_i(\lambda^*; {\mathbf x})
\bigr)^2,
\\
\mathcal B_n^* &= \sum_{i=1}^n 
\bigl( \widehat x(t_{i-1}^n) -x(t_{i-1}^n) \bigr)^2,
\\
\mathcal C_n^* &= \sum_{i=1}^n \mathcal M_i(\lambda^*;{\mb x})^2,
\\
\mathcal D_n^* &= \Biggl| 
\sum_{i=1}^n \mathcal M_i(\lambda^*;{\mb x}) x(t_{i-1}^n) 
\Biggr|,
\\
\mathcal E_n^* &= \sum_{i=1}^n x(t_{i-1}^n)^2.
\end{split}
\right.
\end{equation}
We then obtain the following lemmas to control 
$\widehat {\mathbb V}_{j,n}^\epsilon$, $\widehat {\mathbb W}_{j,n}^\epsilon$
($j \in \{ 1,\ldots, 4 \}$).

\begin{lem}\label{lem_est_V}
For $\widehat {\mathbb V}_{j,n}^\epsilon (j \in \{1,\ldots, 4\})$ 
given in \eqref{hat_V}, we have the following four assertions.
\begin{enumerate}
\item[(1)]
If $n \mathcal X_n = \op(1)$, then $\widehat {\mathbb V}_{1,n}^\epsilon = \op(1)$.

\item[(2)]
If $\epsilon^{-2} \mathcal X_n = \op(1)$, then 
$\widehat {\mathbb V}_{2,n}^\epsilon = \op(1)$.

\item[(3)]
If $\frac{\sqrt{n}}{\epsilon^2} \mathcal X_n^* = \op(1)$ and
$\epsilon^{-1} \mathcal Y_n^* = \op(1)$, then 
$\widehat {\mathbb V}_{3,n}^\epsilon = \op(1)$.
In particular, 
$(\widehat {\mathbb V}_{3,n}^\epsilon)_{1} = \op(1)$ holds
under $\epsilon^{-1} \mathcal X_n^* = \op(1)$ and
$\epsilon^{-1} \mathcal Y_n^* = \op(1)$.

\item[(4)]
If $\epsilon^{-2} \mathcal X_n = \op(1)$, 
$\frac{1}{\epsilon \sqrt{n}} \mathcal Y_n = \op(1)$ 
and $\frac{1}{n} \mathcal Z_n^* = \op(1)$, then 
$\widehat {\mathbb V}_{4,n}^\epsilon = \op(1)$.
In particular, 
$(\widehat {\mathbb V}_{4,n}^\epsilon)_{1,1} = \op(1)$ holds 
under $\frac{1}{n} \mathcal X_n = \op(1)$, 
$\frac{1}{n} \mathcal Y_n = \op(1)$ and $\frac{1}{n} \mathcal Z_n^* = \op(1)$.

\end{enumerate}
\end{lem}

\begin{lem}\label{lem_est_W}
For $\widehat {\mathbb W}_{j,n}^\epsilon (j \in \{1,\ldots, 4\})$ 
given in \eqref{hat_W}, we have the following four assertions.
\begin{enumerate}
\item[(1)]
If $n \mathcal X_n = \op(1)$ and
$\frac{n \epsilon^2}{R_{\alpha^*,\alpha_0}^{\mathrm{damp}}} = \oo(1)$, 
then $\widehat {\mathbb W}_{1,n}^\epsilon = \op(1)$.

\item[(2)]
If $\epsilon^{-2} \mathcal X_n = \op(1)$ 
and $\frac{(n \epsilon^2)^{-1}}{R_{\alpha^*,\alpha_0}^{\mathrm{damp}}} = \oo(1)$, then
$\widehat {\mathbb W}_{2,n}^\epsilon = \op(1)$.

\item[(3)]
If $\frac{\sqrt{n}}{\epsilon^2} \mathcal X_n^* =\op(1)$, 
$\epsilon^{-1} \mathcal Y_n^* = \op(1)$ 
and $\frac{\sqrt{n}}{R_{\alpha^*,\alpha_0}^{\mathrm{damp}}} = \oo(1)$,
then $\widehat {\mathbb W}_{3,n}^\epsilon = \op(1)$.
In particular, $(\widehat {\mathbb W}_{3,n}^\epsilon)_{1} = \op(1)$ holds
under $\epsilon^{-1} \mathcal X_n^* = \op(1)$ 
and $\epsilon^{-1} \mathcal Y_n^* = \op(1)$.

\item[(4)]
If $\epsilon^{-2} \mathcal X_n = \op(1)$, 
$\frac{1}{\epsilon \sqrt{n}} \mathcal Y_n = \op(1)$, 
$\frac{1}{n} \mathcal Z_n^* = \op(1)$, 
and $\frac{(n \epsilon^2)^{-1}}{R_{\alpha^*,\alpha_0}^{\mathrm{damp}}} = \oo(1)$,
then $\widehat {\mathbb W}_{4,n}^\epsilon = \op(1)$. 
In particular, $(\widehat {\mathbb W}_{4,n}^\epsilon)_{1,1} = \op(1)$ holds
under $\frac{1}{n} \mathcal X_n = \op(1)$, 
$\frac{1}{n} \mathcal Y_n = \op(1)$ and $\frac{1}{n} \mathcal Z_n^* = \op(1)$.
\end{enumerate}
\end{lem}

\begin{rmk}\label{rmk_W}
$\widehat {\mathbb W}_{j,n}^\epsilon (j \in \{1,\ldots, 4\})$ 
are the differences between 
the quasi log-likelihood function $V_n^{\epsilon}$,
the score function $L_n^{\epsilon}$
and the observed information $K_n^{\epsilon}$
with the estimator $\widehat \alpha$ and those with the true value $\alpha^*$, 
and are quantities that do not appear in \cite{TKU2026} since $\alpha$ is assumed to be known.
We find from Lemmas \ref{lem_est_V} and \ref{lem_est_W} that the sufficient conditions for $\widehat {\mathbb W}_{j,n}^\epsilon = \op(1)$ 
are the conditions that combine those for $\widehat {\mathbb V}_{j,n}^\epsilon = \op(1)$ 
with the condition containing $R_{\alpha^*,\alpha_0}^{\mathrm{damp}}$.
We thus see that the conditions involving $R_{\alpha^*,\alpha_0}^{\mathrm{damp}}$ control the estimation error of $\alpha$.
\end{rmk}

\begin{proof}[\bf{Proof of Lemma \ref{lem_est_V}}]
We define $F(s) = \frac{s}{1-\ee^{-s}}$ and
$F_n(\lambda, \mu;\alpha) = \mu^{\alpha} F(2\lambda \Delta_n)$ 
for $s \in (0,\infty)$, $\lambda, \mu \in (0,\infty)$ and $\alpha \in (0,3)$.
We then have
\begin{equation}\label{eq_V}
V_n^{\epsilon} (\lambda,\mu; \widetilde {\mb x}, \alpha)
= \frac{F_n(\lambda,\mu;\alpha)}{\epsilon^2 \Delta_n} 
\sum_{i=1}^n \mathcal M_i(\lambda;\widetilde {\mb x})^2
-n \log(F_n(\lambda,\mu; \alpha)).
\end{equation}
Since we find
\begin{equation*}
\lim_{s \searrow 0} F(s) = 1,
\quad 
\lim_{s \searrow 0} F'(s) = \frac{1}{2}, 
\quad
\lim_{s \searrow 0} F''(s) = \frac{1}{6}
\end{equation*}
and
\begin{equation*}
\left\{
\begin{split}
\pd_\lambda F_n(\lambda,\mu;\alpha) &= 2 \Delta_n \mu^\alpha F'(2 \lambda \Delta_n),
\\
\pd_\mu F_n(\lambda,\mu;\alpha) &= \alpha \mu^{\alpha-1} F(2 \lambda \Delta_n),
\end{split}
\right.
\quad
\left\{
\begin{split}
\pd_\lambda^2 F_n(\lambda,\mu;\alpha) 
&= (2 \Delta_n)^2 \mu^\alpha F''(2 \lambda \Delta_n),
\\
\pd_\mu^2 F_n(\lambda,\mu;\alpha) 
&= \alpha(\alpha -1) \mu^{\alpha-2} F(2 \lambda \Delta_n),
\\
\pd_\mu \pd_\lambda F_n(\lambda,\mu;\alpha) 
&= 2 \alpha \Delta_n \mu^{\alpha-1} F'(2 \lambda \Delta_n),
\end{split}
\right.
\end{equation*}
we obtain
\begin{equation}\label{asym_Fn}
\left\{
\begin{split}
\lim_{n \to \infty} 
\sup_{(\lambda, \mu, \alpha) \in \Xi_{\nu} \times \Xi_{\alpha}} \bigl| F_n(\lambda, \mu;\alpha) -\mu^{\alpha} \bigr| &= 0,
\\
\lim_{n \to \infty} 
\sup_{(\lambda, \mu, \alpha) \in \Xi_{\nu} \times \Xi_{\alpha}} 
\bigl| \Delta_n^{-1} \pd_\lambda F_n(\lambda, \mu;\alpha) -\mu^{\alpha} \bigr| &= 0,
\\
\lim_{n \to \infty} 
\sup_{(\lambda, \mu, \alpha) \in \Xi_{\nu} \times \Xi_{\alpha}} 
\bigl| \pd_\mu F_n(\lambda, \mu;\alpha) -\alpha \mu^{\alpha-1} \bigr| &= 0,
\\
\lim_{n \to \infty} 
\sup_{(\lambda, \mu, \alpha) \in \Xi_{\nu} \times \Xi_{\alpha}} 
\Bigl| \Delta_n^{-2} \pd_\lambda^2 F_n(\lambda, \mu;\alpha) 
-\frac{2 \mu^\alpha}{3} \Bigr| &= 0,
\\
\lim_{n \to \infty} 
\sup_{(\lambda, \mu, \alpha) \in \Xi_{\nu} \times \Xi_{\alpha}} 
\bigl| \pd_\mu^2 F_n(\lambda, \mu;\alpha) -\alpha(\alpha -1) \mu^{\alpha-2} \bigr| &= 0,
\\
\lim_{n \to \infty} 
\sup_{(\lambda, \mu, \alpha) \in \Xi_{\nu} \times \Xi_{\alpha}} 
\bigl| \Delta_n^{-1} \pd_\mu \pd_\lambda F_n(\lambda, \mu;\alpha) 
-\alpha \mu^{\alpha-1} \bigr| &= 0.
\end{split}
\right.
\end{equation}

\begin{enumerate}
\item[(1)]
By \eqref{eq_V}, we obtain
\begin{align}
&\bigl|
V_n^{\epsilon} (\lambda,\mu;\widehat {\mb x}, \alpha^*)
-V_n^{\epsilon} (\lambda,\mu;{\mb x}, \alpha^*)
\bigr|
\nonumber
\\
&= 
F_n(\lambda,\mu;\alpha^*) \cdot \frac{1}{\epsilon^2 \Delta_n}
\Biggl|
\sum_{i=1}^n 
\bigl( 
\mathcal M_i(\lambda;\widehat {\mb x})^2 -\mathcal M_i(\lambda; {\mb x})^2
\bigr)
\Biggr|.
\label{VV1}
\end{align}
We therefore see from \eqref{asym_Fn} that
\begin{align*}
\widehat {\mathbb V}_{1,n}^\epsilon
&= \epsilon^{2} \sup_{(\lambda,\mu) \in \Xi_{\nu}}
\bigl|
V_n^{\epsilon} (\lambda,\mu;\widehat {\mb x}, \alpha^*)
-V_n^{\epsilon} (\lambda,\mu;{\mb x}, \alpha^*)
\bigr|
\\
&\le 
\sup_{(\lambda,\mu) \in \Xi_{\nu}}
F_n(\lambda,\mu;\alpha^*) 
\cdot \frac{1}{\Delta_n}
\sup_{\lambda \in \Xi_{\lambda}} 
\Biggl|
\sum_{i=1}^n 
\bigl( 
\mathcal M_i(\lambda;\widehat {\mb x})^2 -\mathcal M_i(\lambda; {\mb x})^2
\bigr)
\Biggr|
\\
&\lesssim n \mathcal X_n
\quad (n \to \infty, \epsilon \searrow 0).
\end{align*}

\item[(2)]
We find from \eqref{asym_Fn} and \eqref{VV1} that
\begin{align*}
\widehat {\mathbb V}_{2,n}^\epsilon
&= \frac{1}{n}\sup_{(\lambda,\mu) \in \Xi_{\nu}}
\bigl| V_n^{\epsilon} (\lambda,\mu;\widehat {\mb x}, \alpha^*)
-V_n^{\epsilon} (\lambda,\mu;{\mb x}, \alpha^*) \bigr|
\\
&\le
\sup_{(\lambda,\mu) \in \Xi_{\nu}}
F_n(\lambda,\mu;\alpha^*) 
\cdot \epsilon^{-2} 
\sup_{\lambda \in \Xi_{\lambda}} 
\Biggl|
\sum_{i=1}^n 
\bigl( 
\mathcal M_i(\lambda;\widehat {\mb x})^2 -\mathcal M_i(\lambda; {\mb x})^2
\bigr)
\Biggr|
\\
&\lesssim
\epsilon^{-2} \mathcal X_n
\quad
(n \to \infty, \epsilon \searrow 0).
\end{align*}

\item[(3)]
We have
\begin{equation}\label{eq_dV}
\left\{
\begin{split}
\pd_\lambda V_n^{\epsilon} (\lambda,\mu; \widetilde {\mb x}, \alpha)
&= 
\frac{\pd_\lambda F_n(\lambda,\mu;\alpha)}{\epsilon^2 \Delta_n} 
\sum_{i=1}^n \mathcal M_i(\lambda;\widetilde {\mb x})^2
\\
&\quad 
+\frac{2 \ee^{-\lambda \Delta_n} F_n(\lambda,\mu;\alpha)}{\epsilon^2} 
\sum_{i=1}^n \mathcal M_i(\lambda;\widetilde {\mathbf x}) \widetilde x(t_{i-1}^n)
\\
&\quad 
-n \pd_\lambda \log(F_n(\lambda,\mu; \alpha)),
\\
\pd_\mu V_n^{\epsilon} (\lambda,\mu; \widetilde {\mb x}, \alpha)
&= 
\frac{\pd_\mu F_n(\lambda,\mu;\alpha)}{\epsilon^2 \Delta_n} 
\sum_{i=1}^n \mathcal M_i(\lambda;\widetilde {\mathbf x})^2
\\
&\quad 
-n \pd_\mu \log(F_n(\lambda,\mu; \alpha))
\end{split}
\right.
\end{equation}
and
\begin{equation*}
\left\{
\begin{split}
\pd_\lambda \log(F_n(\lambda,\mu; \alpha))
&= \frac{\dd}{\dd \lambda} \log(F(2\lambda \Delta_n)),
\\
\pd_\mu \log(F_n(\lambda,\mu; \alpha))
&= \alpha \frac{\dd}{\dd \mu}\log(\mu) = \frac{\alpha}{\mu}.
\end{split}
\right.
\end{equation*}
We then see from 
\begin{align*}
&\bigl|
\pd_\lambda V_n^{\epsilon} (\lambda,\mu;\widehat {\mb x}, \alpha^*)
-\pd_\lambda V_n^{\epsilon} (\lambda,\mu;{\mb x}, \alpha^*)
\bigr|
\\
&\le
\frac{|\pd_\lambda F_n(\lambda,\mu;\alpha^*)|}{\Delta_n} \cdot \epsilon^{-2} 
\Biggl|
\sum_{i=1}^n 
\bigl( 
\mathcal M_i(\lambda;\widehat {\mb x})^2 -\mathcal M_i(\lambda; {\mb x})^2
\bigr)
\Biggr|
\\
&\quad+
2 \ee^{-\lambda \Delta_n} F_n(\lambda,\mu;\alpha^*) \cdot \epsilon^{-2}
\Biggl|
\sum_{i=1}^n \bigl( 
\mathcal M_i(\lambda;\widehat {\mathbf x}) \widehat x(t_{i-1}^n)
- \mathcal M_i(\lambda; {\mathbf x}) x(t_{i-1}^n)
\bigr)
\Biggr|
\end{align*}
and \eqref{asym_Fn} that 
\begin{align*}
(\widehat {\mathbb V}_{3,n}^\epsilon)_1
&= \epsilon \sup_{\mu \in \Xi_\mu} \bigl|
\pd_\lambda V_n^{\epsilon} (\lambda^*,\mu; \widehat {\mb x}, \alpha^*)
-\pd_\lambda V_n^{\epsilon} (\lambda^*,\mu; {\mb x}, \alpha^*)
\bigr|
\\
&\le
\sup_{\mu \in \Xi_\mu} 
\frac{|\pd_\lambda F_n(\lambda^*,\mu;\alpha^*)|}{\Delta_n} \cdot \epsilon^{-1} 
\Biggl|
\sum_{i=1}^n 
\bigl( 
\mathcal M_i(\lambda^*;\widehat {\mb x})^2 -\mathcal M_i(\lambda^*; {\mb x})^2
\bigr)
\Biggr|
\\
&\quad+
2 \ee^{-\lambda^* \Delta_n} 
\sup_{\mu \in \Xi_\mu} F_n(\lambda^*,\mu;\alpha^*) \cdot \epsilon^{-1}
\Biggl|
\sum_{i=1}^n \bigl( 
\mathcal M_i(\lambda^*;\widehat {\mathbf x}) \widehat x(t_{i-1}^n)
- \mathcal M_i(\lambda^*; {\mathbf x}) x(t_{i-1}^n)
\bigr)
\Biggr|
\\
&\lesssim \epsilon^{-1} (\mathcal X_n^* +\mathcal Y_n^*)
\quad (n \to \infty, \epsilon \searrow 0),
\end{align*}
and from 
\begin{align*}
&\bigl|
\pd_\mu V_n^{\epsilon} (\lambda,\mu;\widehat {\mb x}, \alpha^*)
-\pd_\mu V_n^{\epsilon} (\lambda,\mu;{\mb x}, \alpha^*)
\bigr|
\\
&\le
|\pd_\mu F_n(\lambda,\mu;\alpha^*)| \cdot \frac{1}{\epsilon^2 \Delta_n} 
\Biggl|
\sum_{i=1}^n 
\bigl( 
\mathcal M_i(\lambda;\widehat {\mb x})^2 -\mathcal M_i(\lambda; {\mb x})^2
\bigr)
\Biggr|
\end{align*}
and \eqref{asym_Fn} that
\begin{align*}
(\widehat {\mathbb V}_{3,n}^\epsilon)_2
&= \frac{1}{\sqrt{n}} \sup_{\mu \in \Xi_\mu} \bigl|
\pd_\mu V_n^{\epsilon} (\lambda^*,\mu; \widehat {\mb x}, \alpha^*)
-\pd_\mu V_n^{\epsilon} (\lambda^*,\mu; {\mb x}, \alpha^*)
\bigr|
\\
&\le
\sup_{\mu \in \Xi_\mu} |\pd_\mu F_n(\lambda^*,\mu;\alpha^*)| \cdot 
\frac{\sqrt{n}}{\epsilon^2} 
\Biggl|
\sum_{i=1}^n 
\bigl( 
\mathcal M_i(\lambda^*;\widehat {\mb x})^2 -\mathcal M_i(\lambda^*; {\mb x})^2
\bigr)
\Biggr|
\\
&\lesssim \frac{\sqrt{n}}{\epsilon^2} \mathcal X_n^*
\quad
(n \to \infty, \epsilon \searrow 0).
\end{align*}
Therefore, we obtain
by $\epsilon^{-1} \lor \frac{\sqrt{n}}{\epsilon^2} = \frac{\sqrt{n}}{\epsilon^2}$,
\begin{align*}
\widehat {\mathbb V}_{3,n}^\epsilon 
\le \sum_{i \in \{1,2\}} (\widehat {\mathbb V}_{3,n}^\epsilon)_i
&\lesssim \epsilon^{-1} (\mathcal X_n^* +\mathcal Y_n^*) 
+\frac{\sqrt{n}}{\epsilon^2} \mathcal X_n^*
\\
&\lesssim \frac{\sqrt{n}}{\epsilon^2} \mathcal X_n^* +\epsilon^{-1} \mathcal Y_n^*
\quad
(n \to \infty, \epsilon \searrow 0).
\end{align*}

\item[(4)]
We have
\begin{equation}\label{eq_d2V}
\left\{
\begin{split}
\pd_\lambda^2 V_n^{\epsilon} (\lambda,\mu; \widetilde {\mb x}, \alpha)
&= 
\frac{\pd_\lambda^2 F_n(\lambda,\mu;\alpha)}{\epsilon^2 \Delta_n} 
\sum_{i=1}^n \mathcal M_i(\lambda;\widetilde {\mathbf x})^2
\\
&\quad 
+\frac{4 \ee^{-\lambda \Delta_n} 
\pd_\lambda F_n(\lambda,\mu;\alpha)}{\epsilon^2} 
\sum_{i=1}^n \mathcal M_i(\lambda;\widetilde {\mathbf x}) \widetilde x(t_{i-1}^n)
\\
&\quad 
+\frac{2 \ee^{-\lambda \Delta_n} 
F_n(\lambda,\mu;\alpha) \Delta_n}{\epsilon^2} 
\sum_{i=1}^n \widetilde x(t_{i-1}^n)^2
\\
&\quad 
-n \pd_\lambda^2 \log(F_n(\lambda,\mu; \alpha)),
\\
\pd_\mu^2 V_n^{\epsilon} (\lambda,\mu; \widetilde {\mb x}, \alpha)
&= 
\frac{\pd_\mu^2 F_n(\lambda,\mu;\alpha)}{\epsilon^2 \Delta_n} 
\sum_{i=1}^n \mathcal M_i(\lambda;\widetilde {\mathbf x})^2
\\
&\quad-n \pd_\mu^2 \log(F_n(\lambda,\mu; \alpha)),
\\
\pd_\mu \pd_\lambda V_n^{\epsilon} (\lambda,\mu; \widetilde {\mb x}, \alpha)
&= 
\frac{\pd_\mu \pd_\lambda F_n(\lambda,\mu;\alpha)}{\epsilon^2 \Delta_n} 
\sum_{i=1}^n \mathcal M_i(\lambda;\widetilde {\mathbf x})^2
\\
&\quad 
+\frac{2 \ee^{-\lambda \Delta_n} \pd_\mu 
F_n(\lambda,\mu;\alpha)}{\epsilon^2} 
\sum_{i=1}^n \mathcal M_i(\lambda;\widetilde {\mathbf x}) \widetilde x(t_{i-1}^n)
\\
&\quad 
-n \pd_\mu \pd_\lambda \log(F_n(\lambda,\mu; \alpha))
\end{split}
\right.
\end{equation}
and
\begin{equation*}
\left\{
\begin{split}
\pd_\lambda^2 \log(F_n(\lambda,\mu; \alpha))
&= \frac{\dd^2}{\dd \lambda^2} \log(F(2\lambda \Delta_n)),
\\
\pd_\mu^2 \log(F_n(\lambda,\mu; \alpha))
&= \alpha \frac{\dd}{\dd \mu} \frac{1}{\mu} = \frac{-\alpha}{\mu^2},
\\
\pd_\mu \pd_\lambda \log(F_n(\lambda,\mu; \alpha))
&= 0.
\end{split}
\right.
\end{equation*}
It then follows from
\begin{align*}
&\bigl|
\pd_\lambda^2 V_n^{\epsilon} (\lambda,\mu;\widehat {\mb x}, \alpha^*)
-\pd_\lambda^2 V_n^{\epsilon} (\lambda,\mu;{\mb x}, \alpha^*)
\bigr|
\\
&\le
\frac{|\pd_\lambda^2 F_n(\lambda,\mu;\alpha^*)|}{\Delta_n^2} 
\cdot \frac{\Delta_n}{\epsilon^2} 
\Biggl|
\sum_{i=1}^n 
\bigl( 
\mathcal M_i(\lambda;\widehat {\mb x})^2 -\mathcal M_i(\lambda; {\mb x})^2
\bigr)
\Biggr|
\\
&\quad+
\frac{4 \ee^{-\lambda \Delta_n} |\pd_\lambda F_n(\lambda,\mu;\alpha^*)|}{\Delta_n} 
\cdot \frac{\Delta_n}{\epsilon^2} 
\Biggl|
\sum_{i=1}^n \bigl( 
\mathcal M_i(\lambda;\widehat {\mathbf x}) \widehat x(t_{i-1}^n)
- \mathcal M_i(\lambda; {\mathbf x}) x(t_{i-1}^n)
\bigr)
\Biggr|
\\
&\quad+
2 \ee^{-\lambda \Delta_n} F_n(\lambda,\mu;\alpha^*) \cdot \frac{\Delta_n}{\epsilon^2} 
\Biggl|
\sum_{i=1}^n \bigl( 
\widehat x(t_{i-1}^n)^2 -x(t_{i-1}^n)^2
\bigr)
\Biggr|
\end{align*}
and \eqref{asym_Fn} that 
\begin{align*}
(\widehat {\mathbb V}_{4,n}^\epsilon)_{1,1} 
&= \epsilon^2 \sup_{(\lambda,\mu) \in \Xi_\nu}
\bigl|
\pd_\lambda^2 V_n^{\epsilon} (\lambda,\mu;\widehat {\mb x}, \alpha^*)
-\pd_\lambda^2 V_n^{\epsilon} (\lambda,\mu;{\mb x}, \alpha^*)
\bigr|
\\
&\le
\sup_{(\lambda,\mu) \in \Xi_\nu} 
\frac{|\pd_\lambda^2 F_n(\lambda,\mu;\alpha^*)|}{\Delta_n^2} 
\cdot \frac{1}{n} 
\sup_{\lambda \in \Xi_\lambda} \Biggl|
\sum_{i=1}^n 
\bigl( 
\mathcal M_i(\lambda;\widehat {\mb x})^2 -\mathcal M_i(\lambda; {\mb x})^2
\bigr)
\Biggr|
\\
&\quad+
4\sup_{(\lambda,\mu) \in \Xi_\nu} 
\frac{\ee^{-\lambda \Delta_n} |\pd_\lambda F_n(\lambda,\mu;\alpha^*)|}{\Delta_n} 
\\
&\qquad \times \frac{1}{n}
\sup_{\lambda \in \Xi_\lambda}
\Biggl|
\sum_{i=1}^n \bigl( 
\mathcal M_i(\lambda;\widehat {\mathbf x}) \widehat x(t_{i-1}^n)
- \mathcal M_i(\lambda; {\mathbf x}) x(t_{i-1}^n)
\bigr)
\Biggr|
\\
&\quad+
2 \sup_{(\lambda,\mu) \in \Xi_\nu} 
\ee^{-\lambda \Delta_n} F_n(\lambda,\mu;\alpha^*) 
\cdot \frac{1}{n}
\Biggl|
\sum_{i=1}^n \bigl( 
\widehat x(t_{i-1}^n)^2 -x(t_{i-1}^n)^2
\bigr)
\Biggr|
\\
&\lesssim \frac{1}{n} (\mathcal X_n +\mathcal Y_n +\mathcal Z_n^*)
\quad
(n \to \infty, \epsilon \searrow 0),
\end{align*}
from
\begin{align*}
&\bigl|
\pd_\mu^2 V_n^{\epsilon} (\lambda,\mu;\widehat {\mb x}, \alpha^*)
-\pd_\mu^2 V_n^{\epsilon} (\lambda,\mu;{\mb x}, \alpha^*)
\bigr|
\\
&\le
\frac{|\pd_\mu^2 F_n(\lambda,\mu;\alpha^*)|}{\epsilon^2 \Delta_n} 
\Biggl|
\sum_{i=1}^n 
\bigl( 
\mathcal M_i(\lambda;\widehat {\mb x})^2 -\mathcal M_i(\lambda; {\mb x})^2
\bigr)
\Biggr|
\end{align*}
and \eqref{asym_Fn} that $(\widehat {\mathbb V}_{4,n}^\epsilon)_{2,2} 
\lesssim \epsilon^{-2} \mathcal X_n$
$(n \to \infty, \epsilon \searrow 0)$, and from 
\begin{align*}
&\bigl|
\pd_\mu \pd_\lambda V_n^{\epsilon} (\lambda,\mu;\widehat {\mb x}, \alpha^*)
-\pd_\mu \pd_\lambda V_n^{\epsilon} (\lambda,\mu;{\mb x}, \alpha^*)
\bigr|
\\
&\le
\frac{|\pd_\mu \pd_\lambda F_n(\lambda,\mu;\alpha^*)|}{\Delta_n} \cdot 
\frac{1}{\epsilon^2} 
\Biggl|
\sum_{i=1}^n 
\bigl( 
\mathcal M_i(\lambda;\widehat {\mb x})^2 -\mathcal M_i(\lambda; {\mb x})^2
\bigr)
\Biggr|
\\
&\quad+
2 \ee^{-\lambda \Delta_n} |\pd_\mu F_n(\lambda,\mu;\alpha^*)|
\cdot \frac{1}{\epsilon^2} 
\Biggl|
\sum_{i=1}^n \bigl( 
\mathcal M_i(\lambda;\widehat {\mathbf x}) \widehat x(t_{i-1}^n)
- \mathcal M_i(\lambda; {\mathbf x}) x(t_{i-1}^n)
\bigr)
\Biggr|
\end{align*}
and \eqref{asym_Fn} that 
\begin{align*}
(\widehat {\mathbb V}_{4,n}^\epsilon)_{1,2} 
&= \frac{\epsilon}{\sqrt{n}} \sup_{(\lambda,\mu) \in \Xi_\nu} \bigl| 
\pd_\mu \pd_\lambda V_n^{\epsilon} (\lambda,\mu;\widehat {\mb x}, \alpha^*)
-\pd_\mu \pd_\lambda V_n^{\epsilon} (\lambda,\mu;{\mb x}, \alpha^*)
\bigr|
\\
&\le
\sup_{(\lambda,\mu) \in \Xi_\nu} 
\frac{|\pd_\mu \pd_\lambda F_n(\lambda,\mu;\alpha^*)|}{\Delta_n} \cdot 
\frac{1}{\epsilon \sqrt{n}}
\sup_{\lambda \in \Xi_\lambda}
\Biggl|
\sum_{i=1}^n 
\bigl( 
\mathcal M_i(\lambda;\widehat {\mb x})^2 -\mathcal M_i(\lambda; {\mb x})^2
\bigr)
\Biggr|
\\
&\quad+
2 \sup_{(\lambda,\mu) \in \Xi_\nu} 
\ee^{-\lambda \Delta_n} |\pd_\mu F_n(\lambda,\mu;\alpha^*)|
\\
&\qquad \times \frac{1}{\epsilon \sqrt{n}}
\sup_{\lambda \in \Xi_\lambda} \Biggl|
\sum_{i=1}^n \bigl( 
\mathcal M_i(\lambda;\widehat {\mathbf x}) \widehat x(t_{i-1}^n)
- \mathcal M_i(\lambda; {\mathbf x}) x(t_{i-1}^n)
\bigr)
\Biggr|
\\
&\lesssim \frac{1}{\epsilon \sqrt{n}} (\mathcal X_n +\mathcal Y_n)
\quad (n \to \infty, \epsilon \searrow 0). 
\end{align*}
We thus see from
$\frac{1}{n} \lor \epsilon^{-2} \lor \frac{1}{\epsilon \sqrt{n}} = \epsilon^{-2}$
and $\frac{1}{n} \lor \frac{1}{\epsilon \sqrt{n}} = \frac{1}{\epsilon \sqrt{n}}$
that
\begin{align*}
\widehat {\mathbb V}_{4,n}^\epsilon 
\le
\sum_{i,j \in \{1,2\}} 
(\widehat {\mathbb V}_{4,n}^\epsilon)_{i,j}
&\lesssim
\frac{1}{n} (\mathcal X_n +\mathcal Y_n +\mathcal Z_n^*)
+\epsilon^{-2} \mathcal X_n
+\frac{1}{\epsilon \sqrt{n}} (\mathcal X_n +\mathcal Y_n)
\\
&\lesssim \epsilon^{-2} \mathcal X_n +\frac{1}{\epsilon \sqrt{n}} \mathcal Y_n
+\frac{1}{n} \mathcal Z_n^*
\quad
(n \to \infty, \epsilon \searrow 0).
\end{align*}
\end{enumerate}
\end{proof}

\begin{proof}[\bf{Proof of Lemma \ref{lem_est_W}}]
It follows from (A.10)--(A.12) in \cite{TKU2024a} that
\begin{equation}\label{est_CDE*}
\mathcal C_n^* = \Op(\epsilon^2),
\quad 
\mathcal D_n^* = \Op(\epsilon),
\quad
\mathcal E_n^* = \Op(n)
\end{equation}
and from $\mathcal M_i(\lambda;{\mb x}) 
= \mathcal M_i(\lambda^*;{\mb x}) 
+(\ee^{-\lambda^* \Delta_n} -\ee^{-\lambda \Delta_n})x(t_{i-1}^n)$ that
\begin{equation}\label{est_CD}
\left\{
\begin{split}
\mathcal C_n &\le 2 \Bigl(\mathcal C_n^* 
+\sup_{\lambda \in \Xi_{\lambda}}(\ee^{-\lambda^* \Delta_n} -\ee^{-\lambda \Delta_n})^2 
\mathcal E_n^* \Bigr)
= \Op(\epsilon^2 \lor n^{-1}),
\\
\mathcal D_n &\le \sqrt{\mathcal C_n \mathcal E_n^*}
= \Op \bigl( (\epsilon \lor n^{-1/2}) n^{1/2} \bigr)
= \Op(\sqrt{n \epsilon^2} \lor 1).
\end{split}
\right.
\end{equation}
\begin{enumerate}
\item[(1)]
By \eqref{eq_V} and the definition 
$F_n(\lambda,\mu;\alpha) = \mu^\alpha F(2 \lambda \Delta_n)$, we obtain
\begin{align*}
& \bigl|
V_n^{\epsilon} (\lambda,\mu; \widehat {\mb x}, \widehat \alpha)
-V_n^{\epsilon} (\lambda,\mu; \widehat {\mb x}, \alpha^*)
\bigr|
\\
&\le
\frac{| F_n(\lambda,\mu;\widehat \alpha) -F_n(\lambda,\mu;\alpha^*) |}
{\epsilon^2 \Delta_n}
\sum_{i=1}^n \mathcal M_i(\lambda;\widehat {\mathbf x})^2
+n \biggl| \log \biggl( \frac{F_n(\lambda,\mu;\widehat \alpha)}
{F_n(\lambda,\mu;\alpha^*)}\biggr) \biggr|
\\
&\le F(2 \lambda \Delta_n) \cdot \frac{| \mu^{\widehat \alpha} -\mu^{\alpha^*} |}
{\epsilon^2 \Delta_n} 
\Biggl( 
\Biggl|
\sum_{i=1}^n 
\bigl( \mathcal M_i(\lambda;\widehat {\mathbf x})^2 
-\mathcal M_i(\lambda; {\mathbf x})^2 \bigr)
\Biggr|
+ \sum_{i=1}^n \mathcal M_i(\lambda; {\mathbf x})^2
\Biggr)
\\
&\quad +n |\widehat \alpha -\alpha^*| |\log (\mu)|.
\end{align*}
It hence holds from \eqref{est_f_alpha} that for some $\gamma_1 \in (0, \infty)$, 
\begin{align*}
\widehat {\mathbb W}_{1,n}^\epsilon
&=\epsilon^{2}\sup_{(\lambda,\mu) \in \Xi_{\nu}}
\bigl|
V_n^{\epsilon} (\lambda,\mu;\widehat {\mb x},\widehat \alpha)
-V_n^{\epsilon} (\lambda,\mu;\widehat {\mb x}, \alpha^*)
\bigr|
\\
&\lesssim 
|\widehat \alpha -\alpha^*|(n\mathcal X_n + n\mathcal C_n)
+n \epsilon^2 |\widehat \alpha -\alpha^*|
\end{align*}

on $\Omega_{\gamma_1}$.
Since it follows from \eqref{est_CD} that
$n \mathcal C_n = \Op(1 \lor n \epsilon^2)$,
we have $\widehat {\mathbb W}_{1,n}^\epsilon = \op(1)$
under $n \mathcal X_n = \op(1)$ and
\begin{equation}\label{alpha-1}
(1 \lor n \epsilon^2) (\widehat \alpha -\alpha^*) = \op(1).
\end{equation}
Since Theorem \ref{th1} yields $\widehat \alpha -\alpha^* = \Op(R_1^{-1})$,
we can translate \eqref{alpha-1} as $\frac{n \epsilon^2}{R_1} = \oo(1)$.

\item[(2)]
We have
\begin{align*}
\widehat {\mathbb W}_{2,n}^\epsilon
&= \frac{1}{n} \sup_{(\lambda,\mu) \in \Xi_{\nu}}
\bigl|
V_n^{\epsilon} (\lambda,\mu;\widehat {\mb x},\widehat \alpha)
-V_n^{\epsilon} (\lambda,\mu;\widehat {\mb x}, \alpha^*)
\bigr|
\\
&\lesssim 
|\widehat \alpha -\alpha^*| (\epsilon^{-2} \mathcal X_n +\epsilon^{-2}  \mathcal C_n)
+|\widehat \alpha -\alpha^*|
\end{align*}
on $\Omega_{\gamma_1}$ as in the proof of (1).
Since it follows from \eqref{est_CD} that
$\epsilon^{-2} \mathcal C_n = \Op(1 \lor (n \epsilon^2)^{-1})$,
it holds that $\widehat {\mathbb W}_{2,n}^\epsilon = \op(1)$
under $\epsilon^{-2} \mathcal X_n = \op(1)$ and
\begin{equation}\label{alpha-2}
(1 \lor (n \epsilon^2)^{-1}) (\widehat \alpha -\alpha^*) = \op(1).
\end{equation}
By Theorem \ref{th1}, we can rewrite \eqref{alpha-2} as 
$\frac{(n \epsilon^2)^{-1}}{R_1} = \oo(1)$.

\item[(3)]
It follows from \eqref{eq_dV} that
\begin{align*}
& \bigl|
\pd_\lambda V_n^{\epsilon} (\lambda,\mu; \widehat {\mb x}, \widehat \alpha)
-\pd_\lambda V_n^{\epsilon} (\lambda,\mu; \widehat {\mb x}, \alpha^*)
\bigr|
\\
&\le 
\frac{|\pd_\lambda (F_n(\lambda,\mu;\widehat \alpha) -F_n(\lambda,\mu;\alpha^*))|}
{\epsilon^2 \Delta_n} 
\sum_{i=1}^n \mathcal M_i(\lambda;\widehat {\mathbf x})^2
\\
&\quad
+\frac{2 \ee^{-\lambda \Delta_n} 
|F_n(\lambda,\mu;\widehat \alpha) -F_n(\lambda,\mu;\alpha^*)|}{\epsilon^2} 
\Biggl|
\sum_{i=1}^n \mathcal M_i(\lambda;\widehat {\mathbf x}) \widehat x(t_{i-1}^n)
\Biggr|
\\
&\quad
+n \biggl| \pd_\lambda \log \biggl( \frac{F_n(\lambda,\mu;\widehat \alpha)}
{F_n(\lambda,\mu;\alpha^*)}\biggr) \biggr|
\\
&\le
\frac{|\frac{\dd}{\dd \lambda} (F(2 \lambda \Delta_n))|}{\Delta_n} 
\cdot
\frac{|\mu^{\widehat \alpha} -\mu^{\alpha^*}|}{\epsilon^2} 
\Biggl( 
\Biggl|
\sum_{i=1}^n 
\bigl( \mathcal M_i(\lambda;\widehat {\mathbf x})^2 
-\mathcal M_i(\lambda; {\mathbf x})^2 \bigr)
\Biggr|
+ \sum_{i=1}^n \mathcal M_i(\lambda; {\mathbf x})^2
\Biggr)
\\
&\quad
+2 \ee^{-\lambda \Delta_n} F(2 \lambda \Delta_n) \cdot
\frac{|\mu^{\widehat \alpha} -\mu^{\alpha^*}|}{\epsilon^2} 
\\
&\qquad \times \Biggl( 
\Biggl|
\sum_{i=1}^n \bigl( 
\mathcal M_i(\lambda;\widehat {\mathbf x}) \widehat x(t_{i-1}^n)
-\mathcal M_i(\lambda; {\mathbf x}) x(t_{i-1}^n) \bigr)
\Biggr|
+\Biggl| \sum_{i=1}^n \mathcal M_i(\lambda; {\mathbf x}) x(t_{i-1}^n) \Biggr|
\Biggr).
\end{align*}
It thus holds from \eqref{est_f_alpha} that on a set $\Omega_{\gamma_1}$,
\begin{align*}
(\widehat {\mathbb W}_{3,n}^\epsilon)_{1}
&= \epsilon \sup_{\mu \in \Xi_\mu}
\bigl|
\pd_\lambda V_n^{\epsilon} (\lambda^*,\mu; \widehat {\mb x}, \widehat \alpha)
-\pd_\lambda V_n^{\epsilon} (\lambda^*,\mu; \widehat {\mb x}, \alpha^*)
\bigr|
\\
&\lesssim |\widehat \alpha -\alpha^*|
(\epsilon^{-1} \mathcal X_n^* +\epsilon^{-1} \mathcal C_n^*
+\epsilon^{-1} \mathcal Y_n^* +\epsilon^{-1} \mathcal D_n^*).
\end{align*}
Since we see from \eqref{est_CDE*} that
$\epsilon^{-1} \mathcal C_n^* = \op(1)$ and $\epsilon^{-1} \mathcal D_n^* = \Op(1)$, 
we obtain $(\widehat {\mathbb W}_{3,n}^\epsilon)_{1} = \op(1)$
under $\epsilon^{-1}\mathcal X_n^* = \Op(1)$, 
$\epsilon^{-1} \mathcal Y_n^* = \Op(1)$ and
$\widehat \alpha -\alpha^* = \op(1)$.
Note that Theorem \ref{th1} implies 
$\widehat \alpha -\alpha^* = \op(1)$ always holds under [B]$_{\alpha^*,\alpha_0}$.

It also follows from \eqref{eq_dV} that
\begin{align*}
&\bigl|
\pd_\mu V_n^{\epsilon} (\lambda,\mu; \widehat {\mb x}, \widehat \alpha)
-\pd_\mu V_n^{\epsilon} (\lambda,\mu; \widehat {\mb x}, \alpha^*)
\bigr|
\\
&\le
\frac{|\pd_\mu (F_n(\lambda,\mu;\widehat \alpha) -F_n(\lambda,\mu;\alpha^*))|}
{\epsilon^2 \Delta_n} 
\sum_{i=1}^n \mathcal M_i(\lambda;\widehat {\mathbf x})^2
+n \biggl| \pd_\mu \log \biggl( \frac{F_n(\lambda,\mu;\widehat \alpha)}
{F_n(\lambda,\mu;\alpha^*)}\biggr) \biggr|
\\
&\le
\frac{F(2 \lambda \Delta_n)|\frac{\dd}{\dd \mu} 
(\mu^{\widehat \alpha} -\mu^{\alpha^*})|}
{\epsilon^2 \Delta_n} 
\Biggl( 
\Biggl|
\sum_{i=1}^n 
\bigl( \mathcal M_i(\lambda;\widehat {\mathbf x})^2 
-\mathcal M_i(\lambda; {\mathbf x})^2 \bigr)
\Biggr|
+ \sum_{i=1}^n \mathcal M_i(\lambda; {\mathbf x})^2
\Biggr)
\\
&\quad
+n \biggl| \frac{\dd}{\dd \mu}\log \bigl(\mu^{\widehat \alpha-\alpha^*} \bigr) \biggr|
\\
&= 
F(2 \lambda \Delta_n) \cdot
\frac{|\widehat \alpha \mu^{\widehat \alpha -1} -\alpha^* \mu^{\alpha^*-1}|}
{\epsilon^2 \Delta_n} 
\Biggl( 
\Biggl|
\sum_{i=1}^n 
\bigl( \mathcal M_i(\lambda;\widehat {\mathbf x})^2 
-\mathcal M_i(\lambda; {\mathbf x})^2 \bigr)
\Biggr|
+ \sum_{i=1}^n \mathcal M_i(\lambda; {\mathbf x})^2
\Biggr)
\\
&\quad
+\frac{n|\widehat \alpha -\alpha^*|}{\mu}.
\end{align*}
By \eqref{est_f_alpha}, there thus exists $\gamma_2 \in (0, \infty)$ 
such that on $\Omega_{\gamma_2}$, 
\begin{align*}
(\widehat {\mathbb W}_{3,n}^\epsilon)_{2}
&= \frac{1}{\sqrt{n}} \sup_{\mu \in \Xi_\mu} 
\bigl|
\pd_\mu V_n^{\epsilon} (\lambda^*,\mu; \widehat {\mb x}, \widehat \alpha)
-\pd_\mu V_n^{\epsilon} (\lambda^*,\mu; \widehat {\mb x}, \alpha^*)
\bigr|
\\
&\lesssim \frac{\sqrt{n}}{\epsilon^2} |\widehat \alpha -\alpha^*|
(\mathcal X_n^* +\mathcal C_n^*)
+\sqrt{n} |\widehat \alpha -\alpha^*|
\\
&= |\widehat \alpha -\alpha^*|
\cdot \frac{\sqrt{n}}{\epsilon^2} \mathcal X_n^* 
+\sqrt{n} |\widehat \alpha -\alpha^*|(\epsilon^{-2} \mathcal C_n^*+1).
\end{align*}
Since we see from \eqref{est_CD} that $\epsilon^{-2} \mathcal C_n^* = \Op(1)$,
we have $(\widehat {\mathbb W}_{3,n}^\epsilon)_{2} = \op(1)$ 
under $\frac{\sqrt{n}}{\epsilon^2} \mathcal X_n^* = \op(1)$ and
\begin{equation}\label{alpha-3}
\sqrt{n} (\widehat \alpha -\alpha^*) = \op(1).
\end{equation}
By Theorem \ref{th1}, 
we can rewrite \eqref{alpha-3} as $\frac{\sqrt{n}}{R_1} = \oo(1)$.

Therefore, we obtain 
\begin{equation*}
\widehat {\mathbb W}_{3,n}^\epsilon 
\le \sum_{i \in \{1,2\}} (\widehat {\mathbb W}_{3,n}^\epsilon)_i
= \op(1)
\end{equation*}
under
$(\epsilon^{-1} \lor \frac{\sqrt{n}}{\epsilon^2}) \mathcal X_n^* = 
\frac{\sqrt{n}}{\epsilon^2} \mathcal X_n^* = \op(1)$,
$\epsilon^{-1} \mathcal Y_n^* = \op(1)$ and $\frac{\sqrt{n}}{R_1} = \oo(1)$.

\item[(4)]
By \eqref{eq_d2V}, we obtain
\begin{align*}
& \bigl|
\pd_\lambda^2 V_n^{\epsilon} (\lambda,\mu; \widehat {\mb x}, \widehat \alpha)
-\pd_\lambda^2 V_n^{\epsilon} (\lambda,\mu; \widehat {\mb x}, \alpha^*)
\bigr|
\\
&\le 
\frac{|\pd_\lambda^2 (F_n(\lambda,\mu;\widehat \alpha) -F_n(\lambda,\mu;\alpha^*))|}
{\epsilon^2 \Delta_n} 
\sum_{i=1}^n \mathcal M_i(\lambda;\widehat {\mathbf x})^2
\\
&\quad
+\frac{4 \ee^{-\lambda \Delta_n} 
|\pd_\lambda (F_n(\lambda,\mu;\widehat \alpha) -F_n(\lambda,\mu;\alpha^*))|}{\epsilon^2} 
\Biggl|
\sum_{i=1}^n \mathcal M_i(\lambda;\widehat {\mathbf x}) \widehat x(t_{i-1}^n)
\Biggr|
\\
&\quad
+\frac{2 \ee^{-\lambda \Delta_n} 
|F_n(\lambda,\mu;\widehat \alpha) -F_n(\lambda,\mu;\alpha^*)| \Delta_n}{\epsilon^2} 
\sum_{i=1}^n \widehat x(t_{i-1}^n)^2 \
\\
&\quad
+n \biggl| \pd_\lambda^2 \log \biggl(
\frac{F_n(\lambda,\mu; \widehat \alpha)}{F_n(\lambda,\mu; \alpha^*)}
\biggr)
\biggr|
\\
&\le
\frac{|\frac{\dd^2}{\dd \lambda^2} (F(2 \lambda \Delta_n))|}{\Delta_n^2} 
\cdot \frac{\Delta_n|\mu^{\widehat \alpha} -\mu^{\alpha^*}|}{\epsilon^2} 
\Biggl( 
\Biggl|
\sum_{i=1}^n 
\bigl( \mathcal M_i(\lambda;\widehat {\mathbf x})^2 
-\mathcal M_i(\lambda; {\mathbf x})^2 \bigr)
\Biggr|
+ \sum_{i=1}^n \mathcal M_i(\lambda; {\mathbf x})^2
\Biggr)
\\
&\quad
+\frac{4 \ee^{-\lambda \Delta_n} 
|\frac{\dd}{\dd \lambda} (F(2 \lambda \Delta_n))|}{\Delta_n} 
\cdot \frac{\Delta_n |\mu^{\widehat \alpha} -\mu^{\alpha^*}|}{\epsilon^2}
\\
&\qquad \times \Biggl( 
\Biggl|
\sum_{i=1}^n \bigl( 
\mathcal M_i(\lambda;\widehat {\mathbf x}) \widehat x(t_{i-1}^n)
-\mathcal M_i(\lambda; {\mathbf x}) x(t_{i-1}^n) \bigr)
\Biggr|
+\Biggl| \sum_{i=1}^n \mathcal M_i(\lambda; {\mathbf x}) x(t_{i-1}^n) \Biggr|
\Biggr)
\\
&\quad
+2 \ee^{-\lambda \Delta_n} F(2 \lambda \Delta_n)
\cdot
\frac{\Delta_n |\mu^{\widehat \alpha} -\mu^{\alpha^*}|}{\epsilon^2}
\Biggl( 
\Biggl|\sum_{i=1}^n \bigl( \widehat x(t_{i-1}^n)^2 -x(t_{i-1}^n)^2 \bigr)
\Biggr|
+\sum_{i=1}^n x(t_{i-1}^n)^2 
\Biggr).
\end{align*}
We thus have
\begin{align*}
(\widehat {\mathbb W}_{4,n}^\epsilon)_{1,1}
&= \epsilon^2 \sup_{(\lambda,\mu) \in \Xi_{\nu}}\bigl|
\pd_\lambda^2 V_n^{\epsilon} (\lambda,\mu; \widehat {\mb x}, \widehat \alpha)
-\pd_\lambda^2 V_n^{\epsilon} (\lambda,\mu; \widehat {\mb x}, \alpha^*)
\bigr|
\\
&\lesssim \frac{1}{n} |\widehat \alpha -\alpha^*|
(\mathcal X_n +\mathcal C_n +\mathcal Y_n +\mathcal D_n +\mathcal Z_n^* +\mathcal E_n^*)
\end{align*}
on $\Omega_{\gamma_1}$.
Since \eqref{est_CDE*} and \eqref{est_CD} yield 
$\frac{1}{n} \mathcal C_n = \op(1)$, $\frac{1}{n} \mathcal D_n = \op(1)$
and $\frac{1}{n}\mathcal E_n^* = \Op(1)$, 
it holds $(\widehat {\mathbb W}_{4,n}^\epsilon)_{1,1} = \op(1)$
under $\frac{1}{n} \mathcal X_n = \op(1)$, $\frac{1}{n}\mathcal Y_n = \op(1)$, 
$\frac{1}{n}\mathcal Z_n^* = \op(1)$ and $\widehat \alpha -\alpha^* = \op(1)$.
We see from Theorem \ref{th1} that 
$\widehat \alpha -\alpha^* = \op(1)$ always holds under [B]$_{\alpha^*,\alpha_0}$.

We also obtain by \eqref{eq_d2V},
\begin{align*}
&\bigl|
\pd_\mu^2 V_n^{\epsilon} (\lambda,\mu; \widehat {\mb x}, \widehat \alpha)
-\pd_\mu^2 V_n^{\epsilon} (\lambda,\mu; \widehat {\mb x}, \alpha^*)
\bigr|
\\
&\le
\frac{|\pd_\mu^2 (F_n(\lambda,\mu;\widehat \alpha) -F_n(\lambda,\mu;\alpha^*))|}
{\epsilon^2 \Delta_n} 
\sum_{i=1}^n \mathcal M_i(\lambda;\widehat {\mathbf x})^2
+n \biggl| \pd_\mu^2 \log \biggl( \frac{F_n(\lambda,\mu;\widehat \alpha)}
{F_n(\lambda,\mu;\alpha^*)}\biggr) \biggr|
\\
&\le
\frac{F(2 \lambda \Delta_n)|\frac{\dd^2}{\dd \mu^2} 
(\mu^{\widehat \alpha} -\mu^{\alpha^*})|}
{\epsilon^2 \Delta_n} 
\Biggl( 
\Biggl|
\sum_{i=1}^n 
\bigl( \mathcal M_i(\lambda;\widehat {\mathbf x})^2 
-\mathcal M_i(\lambda; {\mathbf x})^2 \bigr)
\Biggr|
+ \sum_{i=1}^n \mathcal M_i(\lambda; {\mathbf x})^2
\Biggr)
\\
&\quad
+n \biggl| \frac{\dd^2}{\dd \mu^2} 
\log \bigl(\mu^{\widehat \alpha-\alpha^*} \bigr) \biggr|
\\
&= 
F(2 \lambda \Delta_n) \cdot
\frac{|\widehat \alpha (\widehat \alpha -1) \mu^{\widehat \alpha -2} 
-\alpha^* (\alpha^* -1) \mu^{\alpha^*-2}|}
{\epsilon^2 \Delta_n} 
\\
&\qquad \times \Biggl( 
\Biggl|
\sum_{i=1}^n 
\bigl( \mathcal M_i(\lambda;\widehat {\mathbf x})^2 
-\mathcal M_i(\lambda; {\mathbf x})^2 \bigr)
\Biggr|
+ \sum_{i=1}^n \mathcal M_i(\lambda; {\mathbf x})^2
\Biggr)
\\
&\quad
+\frac{n|\widehat \alpha -\alpha^*|}{\mu^2}.
\end{align*}
We hence find from \eqref{est_f_alpha} that for some $\gamma_3 \in (0, \infty)$, 
\begin{align*}
(\widehat {\mathbb W}_{4,n}^\epsilon)_{2,2}
&= \frac{1}{n} \sup_{(\lambda,\mu) \in \Xi_\nu}
\bigl|
\pd_\mu^2 V_n^{\epsilon} (\lambda,\mu; \widehat {\mb x}, \widehat \alpha)
-\pd_\mu^2 V_n^{\epsilon} (\lambda,\mu; \widehat {\mb x}, \alpha^*)
\bigr|
\\
&\lesssim |\widehat \alpha -\alpha^*| 
(\epsilon^{-2} \mathcal X_n +\epsilon^{-2} \mathcal C_n)
+ |\widehat \alpha -\alpha^*|
\end{align*}
on $\Omega_{\gamma_3}$.
Since $\epsilon^{-2} \mathcal C_n = \Op(1 \lor (n \epsilon^2)^{-1})$, 
we find $(\widehat {\mathbb W}_{4,n}^\epsilon)_{2,2} = \op(1)$ 
under $\epsilon^{-2} \mathcal X_n = \op(1)$ and
\begin{equation}\label{alpha-4}
(1 \lor (n \epsilon^2)^{-1})(\widehat \alpha -\alpha^*) = \op(1).
\end{equation}
By Theorem \ref{th1}, 
we can interpret \eqref{alpha-4} as $\frac{(n \epsilon^{2})^{-1}}{R_1} = \oo(1)$.

Since we see from \eqref{eq_d2V} that
\begin{align*}
&\bigl|
\pd_\mu \pd_\lambda V_n^{\epsilon} (\lambda,\mu; \widehat {\mb x}, \widehat \alpha)
-\pd_\mu \pd_\lambda V_n^{\epsilon} (\lambda,\mu; \widehat {\mb x}, \alpha^*)
\bigr|
\\
&\le
\frac{|\pd_\mu \pd_\lambda
(F_n(\lambda,\mu;\widehat \alpha) -F_n(\lambda,\mu;\alpha^*))|}
{\epsilon^2 \Delta_n} 
\sum_{i=1}^n \mathcal M_i(\lambda;\widehat {\mathbf x})^2
\\
&\quad
+\frac{2 \ee^{-\lambda \Delta_n} 
|\pd_\mu (F_n(\lambda,\mu;\widehat \alpha) -F_n(\lambda,\mu;\alpha^*))|}{\epsilon^2} 
\Biggl|
\sum_{i=1}^n \mathcal M_i(\lambda;\widehat {\mathbf x}) \widehat x(t_{i-1}^n)
\Biggr|
\\
&\quad
+n \biggl| \pd_\mu \pd_\lambda \log \biggl( \frac{F_n(\lambda,\mu;\widehat \alpha)}
{F_n(\lambda,\mu;\alpha^*)}\biggr) \biggr|
\\
&\le
\frac{|\frac{\dd}{\dd \lambda}(F(2 \lambda \Delta_n))| 
|\frac{\dd}{\dd \mu} (\mu^{\widehat \alpha} -\mu^{\alpha^*})|}
{\epsilon^2 \Delta_n} 
\Biggl( 
\Biggl|
\sum_{i=1}^n 
\bigl( \mathcal M_i(\lambda;\widehat {\mathbf x})^2 
-\mathcal M_i(\lambda; {\mathbf x})^2 \bigr)
\Biggr|
+ \sum_{i=1}^n \mathcal M_i(\lambda; {\mathbf x})^2
\Biggr)
\\
&\quad
+\frac{2 \ee^{-\lambda \Delta_n} 
F(2 \lambda \Delta_n) |\frac{\dd}{\dd \mu} (\mu^{\widehat \alpha} -\mu^{\alpha^*})|}
{\epsilon^2} 
\\
&\qquad \times \Biggl( 
\Biggl|
\sum_{i=1}^n \bigl( 
\mathcal M_i(\lambda;\widehat {\mathbf x}) \widehat x(t_{i-1}^n)
-\mathcal M_i(\lambda; {\mathbf x}) x(t_{i-1}^n) \bigr)
\Biggr|
+\Biggl| \sum_{i=1}^n \mathcal M_i(\lambda; {\mathbf x}) x(t_{i-1}^n) \Biggr|
\Biggr)
\\
&= 
\frac{|\frac{\dd}{\dd \lambda}(F(2 \lambda \Delta_n))|}{\Delta_n} 
\cdot
\frac{|\widehat \alpha \mu^{\widehat \alpha -1} -\alpha^* \mu^{\alpha^*-1}|}
{\epsilon^2} 
\\
&\qquad \times \Biggl( 
\Biggl|
\sum_{i=1}^n 
\bigl( \mathcal M_i(\lambda;\widehat {\mathbf x})^2 
-\mathcal M_i(\lambda; {\mathbf x})^2 \bigr)
\Biggr|
+ \sum_{i=1}^n \mathcal M_i(\lambda; {\mathbf x})^2
\Biggr)
\\
&\quad
+2 \ee^{-\lambda \Delta_n} F(2 \lambda \Delta_n) \cdot
\frac{|\widehat \alpha \mu^{\widehat \alpha -1} -\alpha^* \mu^{\alpha^*-1}|}
{\epsilon^2} 
\\
&\qquad \times \Biggl( 
\Biggl|
\sum_{i=1}^n \bigl( 
\mathcal M_i(\lambda;\widehat {\mathbf x}) \widehat x(t_{i-1}^n)
-\mathcal M_i(\lambda; {\mathbf x}) x(t_{i-1}^n) \bigr)
\Biggr|
+\Biggl| \sum_{i=1}^n \mathcal M_i(\lambda; {\mathbf x}) x(t_{i-1}^n) \Biggr|
\Biggr),
\end{align*}
we obtain by \eqref{est_f_alpha},
\begin{align*}
(\widehat {\mathbb W}_{4,n}^\epsilon)_{1,2}
&= \frac{\epsilon}{\sqrt{n}} 
\sup_{(\lambda,\mu) \in \Xi_{\nu}} \bigl|
\pd_\mu \pd_\lambda V_n^{\epsilon} (\lambda,\mu; \widehat {\mb x}, \widehat \alpha)
-\pd_\mu \pd_\lambda V_n^{\epsilon} (\lambda,\mu; \widehat {\mb x}, \alpha^*)
\bigr|
\\
& \lesssim \frac{1}{\sqrt{n} \epsilon} |\widehat \alpha -\alpha^*| 
(\mathcal X_n +\mathcal C_n +\mathcal Y_n +\mathcal D_n)
\end{align*}
on $\Omega_{\gamma_2}$.
Since $\frac{1}{\sqrt{n} \epsilon} \mathcal C_n 
= \Op(\frac{\epsilon}{\sqrt{n}} \lor \frac{1}{n^{3/2} \epsilon})$
and $\frac{1}{\sqrt{n} \epsilon} \mathcal D_n 
= \Op(1 \lor \frac{1}{\sqrt{n} \epsilon })$,
we have $(\widehat {\mathbb W}_{4,n}^\epsilon)_{1,2} = \op(1)$
under $\frac{1}{\sqrt{n} \epsilon} (\mathcal X_n +\mathcal Y_n) = \op(1)$ and
\begin{equation}\label{alpha-5}
\biggl( \frac{\epsilon}{\sqrt{n}} \lor \frac{1}{n^{3/2} \epsilon}
\lor 1 \lor \frac{1}{\sqrt{n} \epsilon} \biggr)
(\widehat \alpha -\alpha^*) 
= \biggl( 1 \lor \frac{1}{\sqrt{n} \epsilon} \biggr)
(\widehat \alpha -\alpha^*) = \op(1).
\end{equation}
By Theorem \ref{th1}, 
we can translate \eqref{alpha-5} as $\frac{(n \epsilon^{2})^{-1/2}}{R_1} = \oo(1)$.

We therefore obtain 
\begin{equation*}
\widehat {\mathbb W}_{4,n}^\epsilon 
\le \sum_{i,j \in \{1,2\}} (\widehat {\mathbb W}_{4,n}^\epsilon)_{i,j}
= \op(1)
\end{equation*}
under $(\frac{1}{n} \lor \epsilon^{-2} \lor \frac{1}{\sqrt{n} \epsilon}) \mathcal X_n 
= \epsilon^{-2} \mathcal X_n = \op(1)$, 
$(\frac{1}{n} \lor \frac{1}{\epsilon \sqrt{n}}) \mathcal Y_n 
= \frac{1}{\epsilon \sqrt{n}} \mathcal Y_n = \op(1)$, 
$\frac{1}{n} \mathcal Z_n^* = \op(1)$
and $\frac{(n \epsilon^{2})^{-1} \lor (n \epsilon^2)^{-1/2}}{R_1}
= \frac{(n \epsilon^2)^{-1}}{R_1} = \oo(1)$.
\end{enumerate}
\end{proof}

We define
\begin{equation}\label{def_A1--B2}
\left\{
\begin{split}
\mathcal A_{1,n}^* &= \Biggl|
\sum_{i=1}^n
\bigl( 
\mathcal M_i(\lambda^*;\widehat {\mathbf x}) -\mathcal M_i(\lambda^*; {\mathbf x})
\bigr)\mathcal M_i(\lambda^*; {\mathbf x})
\Biggr|,
\\
\mathcal A_{2,n}^* &= \Biggl|
\sum_{i=1}^n
\bigl( 
\mathcal M_i(\lambda^*;\widehat {\mathbf x}) -\mathcal M_i(\lambda^*; {\mathbf x})
\bigr) x(t_{i-1}^n)
\Biggr|,
\\
\mathcal B_{1,n}^* &= \Biggl|
\sum_{i=1}^n
\bigl( 
\widehat x(t_{i-1}^n) -x(t_{i-1}^n)
\bigr)\mathcal M_i(\lambda^*; {\mathbf x})
\Biggr|,
\\
\mathcal B_{2,n}^* &= \Biggl|
\sum_{i=1}^n
\bigl( 
\widehat x(t_{i-1}^n) -x(t_{i-1}^n)
\bigr) x(t_{i-1}^n)
\Biggr|.
\end{split}
\right.
\end{equation}
The following lemmas control $\mathcal X_n$, $\mathcal Y_n$, $\mathcal X_n^*$, $\mathcal Y_n^*$, $\mathcal Z_n^*$
given in \eqref{def_XY} and \eqref{def_XYZ*}
with $\mathcal A_n$, $\mathcal A_n^*$, $\mathcal B_n^*$, 
$\mathcal A_{j,n}^*$, $\mathcal B_{j,n}^*$ ($j \in \{1,2\}$)
given in \eqref{def_A--E} and \eqref{def_A1--B2}.
\begin{lem}\label{lem_est_XY}
Let $\{ r_{n,\epsilon} \}$ be a positive sequence depending on $n$ or $\epsilon$.
We then have the following two assertions.
\begin{enumerate}
\item[(1)]
If $r_{n,\epsilon} \mathcal A_n = \op(1)$, 
$\epsilon^2 r_{n,\epsilon}^2 \mathcal A_n = \op(1)$
and $\frac{r_{n,\epsilon}^2}{n} \mathcal A_n = \op(1)$,
then $r_{n,\epsilon} \mathcal X_n = \op(1)$.

\item[(2)]
If $r_{n,\epsilon}^2 \mathcal A_n \mathcal B_n^* = \op(1)$,
$r_{n,\epsilon}^2 (\sqrt{n \epsilon^2} \lor 1) \mathcal A_n = \op(1)$
and $r_{n,\epsilon}^2 (\epsilon^2 \lor n^{-1}) \mathcal B_n^* = \op(1)$,
then $r_{n,\epsilon} \mathcal Y_n = \op(1)$.
\end{enumerate}
\end{lem}

\begin{lem}\label{lem_est_XYZ*}
Let $\{ r_{n,\epsilon} \}$ be a positive sequence depending on $n$ or $\epsilon$.
We then have the following three assertions.
\begin{enumerate}
\item[(1)]
If $r_{n,\epsilon} \mathcal A_n^* = \op(1)$ 
and $r_{n,\epsilon} \mathcal A_{1,n}^* = \op(1)$, 
then $r_{n,\epsilon} \mathcal X_n^* = \op(1)$.

\item[(2)]
If $r_{n,\epsilon}^2 \mathcal A_n^* \mathcal B_n^* = \op(1)$,
$r_{n,\epsilon} \mathcal A_{2,n}^* = \op(1)$ 
and $r_{n,\epsilon} \mathcal B_{1,n}^* = \op(1)$, 
then $r_{n,\epsilon} \mathcal Y_n^* = \op(1)$.

\item[(3)]
If $r_{n,\epsilon} \mathcal B_n^* = \op(1)$ 
and $r_{n,\epsilon} \mathcal B_{2,n}^* = \op(1)$,
then $r_{n,\epsilon} \mathcal Z_n^* = \op(1)$.
\end{enumerate}
\end{lem}

\begin{proof}[\bf{Proof of Lemma \ref{lem_est_XY}}]
Recall the definitions of $\mathcal X_n$ and $\mathcal Y_n$ given in \eqref{def_XY}
and $\mathcal A_n$, $\mathcal B_n^*$, $\mathcal C_n$ and $\mathcal D_n$ given in 
\eqref{def_A--E}.
\begin{enumerate}
\item[(1)]
We have by \eqref{decomp-1} and the Schwarz inequality,
\begin{align*}
r_{n,\epsilon} \mathcal X_n 
&\le
r_{n,\epsilon} \sup_{\lambda \in \Xi_{\lambda}} \sum_{i=1}^n 
\bigl( 
\mathcal M_i(\lambda;\widehat {\mathbf x}) -\mathcal M_i(\lambda; {\mathbf x})
\bigr)^2
\\
&\quad+2 r_{n,\epsilon} \sup_{\lambda \in \Xi_{\lambda}} 
\Biggl|
\sum_{i=1}^n
\bigl( 
\mathcal M_i(\lambda;\widehat {\mathbf x}) -\mathcal M_i(\lambda; {\mathbf x})
\bigr)\mathcal M_i(\lambda; {\mathbf x})
\Biggr|
\\
&\le r_{n,\epsilon} \mathcal A_n 
+2 \sqrt{r_{n,\epsilon}^2 \mathcal A_n \mathcal C_n}
\\
&= r_{n,\epsilon} \mathcal A_n 
+2 \sqrt{r_{n,\epsilon}^2(\epsilon^2 \lor n^{-1}) \mathcal A_n 
(\epsilon^2 \lor n^{-1})^{-1} \mathcal C_n}.
\end{align*}
Since it follows from \eqref{est_CD} that 
$(\epsilon^2 \lor n^{-1})^{-1} \mathcal C_n = \Op(1)$, 
we obtain $r_{n,\epsilon} \mathcal X_n = \op(1)$
nuder $r_{n,\epsilon} \mathcal A_n = \op(1)$, 
$\epsilon^2 r_{n,\epsilon}^2 \mathcal A_n = \op(1)$
and $\frac{r_{n,\epsilon}^2}{n} \mathcal A_n = \op(1)$.

\item[(2)]
We find from \eqref{decomp-1} and the Schwarz inequality that
\begin{align*}
r_{n,\epsilon} \mathcal Y_n 
&\le
r_{n,\epsilon} \sup_{\lambda \in \Xi_{\lambda}} 
\Biggl| \sum_{i=1}^n 
\bigl( 
\mathcal M_i(\lambda;\widehat {\mathbf x}) -\mathcal M_i(\lambda; {\mathbf x})
\bigr)\bigl( \widehat x(t_{i-1}^n) -x(t_{i-1}^n) \bigr)
\Biggr|
\\
&\quad+ r_{n,\epsilon} \sup_{\lambda \in \Xi_{\lambda}} 
\Biggl|
\sum_{i=1}^n
\bigl( 
\mathcal M_i(\lambda;\widehat {\mathbf x}) -\mathcal M_i(\lambda; {\mathbf x})
\bigr) x(t_{i-1}^n)
\Biggr|
\\
&\quad+ r_{n,\epsilon} \sup_{\lambda \in \Xi_{\lambda}} 
\Biggl|
\sum_{i=1}^n
\bigl( \widehat x(t_{i-1}^n) -x(t_{i-1}^n) \bigr) \mathcal M_i(\lambda; {\mathbf x})
\Biggr|
\\
&\le \sqrt{r_{n,\epsilon}^2 \mathcal A_n \mathcal B_n^*}
+ \sqrt{r_{n,\epsilon}^2 (\sqrt{n \epsilon^2} \lor 1) \mathcal A_n 
(\sqrt{n \epsilon^2} \lor 1)^{-1}\mathcal D_n}
\\
&\quad+ 
\sqrt{r_{n,\epsilon}^2 (\epsilon^2 \lor n^{-1}) \mathcal B_n^* 
(\epsilon^2 \lor n^{-1})^{-1}\mathcal C_n}.
\end{align*}
Since $(\epsilon^2 \lor n^{-1})^{-1} \mathcal C_n = \Op(1)$ and 
$(\sqrt{n \epsilon^2} \lor 1)^{-1}\mathcal D_n = \Op(1)$, 
we have $r_{n,\epsilon} \mathcal Y_n = \op(1)$
under $r_{n,\epsilon}^2 \mathcal A_n \mathcal B_n^* = \op(1)$,
$r_{n,\epsilon}^2 (\sqrt{n \epsilon^2} \lor 1) \mathcal A_n = \op(1)$
and $r_{n,\epsilon}^2 (\epsilon^2 \lor n^{-1}) \mathcal B_n^* = \op(1)$.
\end{enumerate}
\end{proof}

\begin{proof}[\bf{Proof of Lemma \ref{lem_est_XYZ*}}]
Recall the definitions of $\mathcal X_n^*$, $\mathcal Y_n^*$ 
and $\mathcal Z_n^*$ given in \eqref{def_XYZ*}
and $\mathcal A_n^*$, $\mathcal B_n^*$, $\mathcal A_{j,n}^*$  and $\mathcal B_{j,n}^*$ 
($j \in \{ 1,2 \}$) given in \eqref{def_A--E} or \eqref{def_A1--B2}.

\begin{enumerate}
\item[(1)]
It follows from \eqref{decomp-1} and the Schwarz inequality that
\begin{align*}
r_{n,\epsilon} \mathcal X_n^*
&\le
r_{n,\epsilon} \sum_{i=1}^n 
\bigl( 
\mathcal M_i(\lambda^*;\widehat {\mathbf x}) -\mathcal M_i(\lambda^*; {\mathbf x})
\bigr)^2
\\
&\quad+2 r_{n,\epsilon} 
\Biggl|
\sum_{i=1}^n
\bigl( 
\mathcal M_i(\lambda^*;\widehat {\mathbf x}) -\mathcal M_i(\lambda^*; {\mathbf x})
\bigr)\mathcal M_i(\lambda^*; {\mathbf x})
\Biggr|
\\
&= r_{n,\epsilon} \mathcal A_n^* +2 r_{n,\epsilon} \mathcal A_{1,n}^*.
\end{align*}
Therefore, we obtain $r_{n,\epsilon} \mathcal X_n^* = \op(1)$
under $r_{n,\epsilon} \mathcal A_n^* = \op(1)$ 
and $r_{n,\epsilon} \mathcal A_{1,n}^* = \op(1)$.

\item[(2)]
We see from \eqref{decomp-1} and the Schwarz inequality that
\begin{align*}
r_{n,\epsilon} \mathcal Y_n^*
&\le 
r_{n,\epsilon} \Biggl| 
\sum_{i=1}^n \bigl( 
\mathcal M_i(\lambda^*; \widehat {\mb x})
-\mathcal M_i(\lambda^*; {\mb x}) \bigr)
\bigl( \widehat x(t_{i-1}^n) -x(t_{i-1}^n) \bigr) 
\Biggr|
\\
&\quad+
r_{n,\epsilon} \Biggl| 
\sum_{i=1}^n \bigl( 
\mathcal M_i(\lambda^*; \widehat {\mb x})
-\mathcal M_i(\lambda^*; {\mb x}) \bigr) x(t_{i-1}^n)
\Biggr|
\\
&\quad+
r_{n,\epsilon} \Biggl| 
\sum_{i=1}^n \mathcal M_i(\lambda^*; {\mb x})
\bigl( \widehat x(t_{i-1}^n) -x(t_{i-1}^n) \bigr) 
\Biggr|
\\
&\le \sqrt{r_{n,\epsilon}^2 \mathcal A_n^* \mathcal B_n^*} 
+ r_{n,\epsilon} \mathcal A_{2,n}^* 
+ r_{n,\epsilon} \mathcal B_{1,n}^*.
\end{align*}
We thus obtain $r_{n,\epsilon} \mathcal Y_n^* = \op(1)$
under $r_{n,\epsilon}^2 \mathcal A_n^* \mathcal B_n^* = \op(1)$,
$r_{n,\epsilon} \mathcal A_{2,n}^* = \op(1)$ 
and $r_{n,\epsilon} \mathcal B_{1,n}^* = \op(1)$.

\item[(3)]
We see from \eqref{decomp-1} and the Schwarz inequality that
\begin{align*}
r_{n,\epsilon} \mathcal Z_n^*
&\le
r_{n,\epsilon} \sum_{i=1}^n 
\bigl( \widehat x(t_{i-1}^n) -x(t_{i-1}^n) \bigr)^2
\\
&\quad+2 r_{n,\epsilon} 
\Biggl|
\sum_{i=1}^n
\bigl( \widehat x(t_{i-1}^n) -x(t_{i-1}^n) \bigr) x(t_{i-1}^n)
\Biggr|
\\
&= r_{n,\epsilon} \mathcal B_n^* +2 r_{n,\epsilon} \mathcal B_{2,n}^*.
\end{align*}
We hence obtain $r_{n,\epsilon} \mathcal Z_n^* = \op(1)$
under $r_{n,\epsilon} \mathcal B_n^* = \op(1)$
and $r_{n,\epsilon} \mathcal B_{2,n}^* = \op(1)$.
\end{enumerate}
\end{proof}

The following lemmas provide sufficient conditions to control 
$\mathcal A_n$, $\mathcal A_n^*$, $\mathcal B_n^*$, 
$\mathcal A_{j,n}^*$ and $\mathcal B_{j,n}^*$ ($j \in \{1, 2\}$)
given in \eqref{def_A--E} and \eqref{def_A1--B2}.
\begin{lem}[Proposition 4.2 in \cite{TKU2026}]\label{lem_est_AB1}
Let $\alpha, \alpha_0 \in (0,3)$, and assume that [A1]$_{\alpha_0}$, 
[A2] and [B]$_{\alpha,\alpha_0}$ hold.
Let $\{ r_{n,\epsilon} \}$ be a positive sequence depending on $n$ or $\epsilon$.
We then have the following two assertions.
\begin{itemize}
\item[(1)]
If $\frac{n r_{n,\epsilon}}{M_{(1)}^{2\alpha_0 \tand 2}} = \oo(1)$,
$\frac{n r_{n,\epsilon} \Delta_n^{\alpha_0 \tand 2}}
{(R_{\alpha, \alpha_0}^{\mathrm{coef}})^2} = \oo(1)$,
$\frac{n \epsilon^2 r_{n,\epsilon}}{M_{(1)}^{2\alpha \tand 2}} = \oo(1)$
and $\frac{n \epsilon^2 r_{n,\epsilon} \Delta_n^{\alpha \tand 1}}
{(R_{\alpha, \alpha_0}^{\mathrm{coef}})^2} = \oo(1)$, 
then $r_{n,\epsilon} \mathcal A_n = \op(1)$.

\item[(2)]
If $\frac{n r_{n,\epsilon}}{M_{(1)}^{2\alpha_0 \tand 2}} = \oo(1)$, 
$\frac{n r_{n,\epsilon}}{(R_{\alpha, \alpha_0}^{\mathrm{coef}})^2} = \oo(1)$ and
$\frac{n \epsilon^2 r_{n,\epsilon}}{M_{(1)}^{2\alpha \tand 2}} = \oo(1)$, 
then $r_{n,\epsilon} \mathcal B_n^* = \op(1)$.
\end{itemize}
\end{lem}

\begin{lem}\label{lem_est_AB2}
Let $\alpha, \alpha_0 \in (0,3)$, and assume that [A1]$_{\alpha_0}$, 
[A2] and [B]$_{\alpha,\alpha_0}$ hold.
Let $\{ r_{n,\epsilon} \}$ be a positive sequence depending on $n$ or $\epsilon$.
We then have the following five assertions.
\begin{enumerate}
\item[(1)]
If $\frac{n \epsilon^2 r_{n,\epsilon} \Delta_n^{\alpha \tand 1}}
{(R_{\alpha, \alpha_0}^{\mathrm{coef}})^2} = \oo(1)$
and $\frac{n \epsilon^2 r_{n,\epsilon}}{M_{(1)}^{2 \alpha \tand 2}} = \oo(1)$,
then $r_{n,\epsilon} \mathcal A_n^* = \op(1)$.

\item[(2)]
If $\frac{n \epsilon^4 r_{n,\epsilon}^2 \Delta_n^{\alpha \tand 1}}
{(R_{\alpha, \alpha_0}^{\mathrm{coef}})^2} = \oo(1)$
and $\frac{\epsilon^4 r_{n,\epsilon}^2}{M_{(1)}^{2\alpha \tand 2}} = \oo(1)$, 
then $r_{n,\epsilon} \mathcal A_{1,n}^* = \op(1)$.

\item[(3)]
If $\frac{n^2 \epsilon^2 r_{n,\epsilon}^2 \Delta_n^{\alpha \tand 1}}
{(R_{\alpha, \alpha_0}^{\mathrm{coef}})^2} = \oo(1)$
and $\frac{n \epsilon^2 r_{n,\epsilon}^2}{M_{(1)}^{2\alpha \tand 2}} = \oo(1)$, 
then $r_{n,\epsilon} \mathcal A_{2,n}^* = \op(1)$.

\item[(4)]
If $\frac{n \epsilon^2 r_{n,\epsilon}^2}{
(R_{\alpha, \alpha_0}^{\mathrm{coef}})^2} = \oo(1)$,
$\frac{\epsilon^4 r_{n,\epsilon}^2}{M_{(1)}^{2\alpha \tand 2}} = \oo(1)$ and
$\frac{\epsilon^2 r_{n,\epsilon}^2}{M_{(1)}^{2\alpha_0 \tand 2}} = \oo(1)$, 
then $r_{n,\epsilon} \mathcal B_{1,n}^* = \op(1)$.

\item[(5)]
If $\frac{n^2 r_{n,\epsilon}^2}{(R_{\alpha, \alpha_0}^{\mathrm{coef}})^2} = \oo(1)$, 
$\frac{n r_{n,\epsilon}}{M_{(1)}} = \oo(1)$, 
$\frac{n \epsilon^2 r_{n,\epsilon}^2}{M_{(1)}^{2\alpha \tand 2}} = \oo(1)$ and
$\frac{n^2 r_{n,\epsilon}^2}{M_{(1)}^{2 \alpha_0 \tand 2}} = \oo(1)$, 
then $r_{n,\epsilon} \mathcal B_{2,n}^* = \op(1)$.
\end{enumerate}
\end{lem}

\begin{proof}[\bf{Proof of Lemma \ref{lem_est_AB2}}]
Let $R = R_{\alpha, \alpha_0}^{\mathrm{coef}}$.
We define $Y_i^n = X_{t_i^n} -\ee^{-\Delta_n A_\theta} X_{t_{i-1}^n}$
and notice that
\begin{equation*}
\mathcal M_i(\lambda_l^*; {\mb x}_l)
= x_l(t_i^n) -\ee^{-\lambda_l^* \Delta_n} x_l(t_{i-1}^n)
= \epsilon \mu_l^{-\alpha/2} 
\int_{t_{i-1}^n}^{t_i^n} \ee^{-\lambda_l^* (t_i^n -s)} \dd w_l(s)
\end{equation*}
and
\begin{equation}\label{eq_Y}
Y_i^n = \sum_{l \in \mathbb N^2} \langle Y_i^n, e_l \rangle e_l 
= \sum_{l \in \mathbb N^2} \mathcal M_i(\lambda_l^*; {\mb x}_l) e_l.
\end{equation}
We also note that
\begin{equation}\label{E_M2}
\EE \bigl[ \mathcal M_i(\lambda_l^*; {\mb x}_l)^2 \bigr] 
= \frac{\epsilon^2(1-\ee^{-2\lambda_l^* \Delta_n})}{2 (\lambda_l^*) \mu_l^{\alpha}} 
\lesssim \frac{\epsilon^2}{(\lambda_l^*)^{1+\alpha}} \land \epsilon^2 \Delta_n,
\quad
\EE \bigl[ \mathcal M_i(\lambda_l^*; {\mb x}_l)^4 \bigr] 
\lesssim \epsilon^4 \Delta_n^2.
\end{equation}
Let $D_{j,k} = (y_{j-1}, y_j] \times (z_{k-1}, z_k]$,
$\mb y_{j,k} = (y_j, z_k)^\TT$,
$\widehat {\bs \kappa} = (\widehat \kappa, \widehat \eta)^\TT$, 
$\bs \kappa = (\kappa^*, \eta^*)^\TT$ 
and $h_l(\mb y; {\bs \kappa}) = e_l(\mb y; {\bs \kappa}) \exp({\bs \kappa}^\TT \mb y)$
for $\mb y = (y, z)^\TT \in D$ and $\bs \kappa = (\kappa, \eta)^\TT \in \mathbb R^2$. 
Since we find from \eqref{approx_cp0} that 
\begin{align*}
\mathcal M_i(\lambda_{\ell}^*; {\mb x}_{\ell})
&= \langle Y_i^n, e_{\ell} \rangle 
= \int_D Y_i^n(\mb y) h_{\ell}(\mb y; {\bs \kappa}^*) \dd \mb y 
= \sum_{j,k} \int_{D_{j,k}} Y_i^n(\mb y) h_{\ell}(\mb y; {\bs \kappa}^*) \dd \mb y,
\\
\mathcal M_i(\lambda_{\ell}^*; \widehat {\mb x}_{\ell}) 
&= \widehat x_{\ell}(t_i^n) 
-\ee^{-\lambda_{\ell}^* \Delta_n} \widehat x_{\ell}(t_{i-1}^n)
= \sum_{j,k} \int_{D_{j,k}} Y_i^n(\mb y_{j,k}) 
h_{\ell}(\mb y;\widehat {\bs \kappa}) \dd \mb y,
\end{align*}
we decompose $\mathcal M_i(\lambda_{\ell}^*; \widehat {\mb x}_{\ell}) 
-\mathcal M_i(\lambda_{\ell}^*; {\mb x}_{\ell})$ as follows.
\begin{align}
\mathcal M_i(\lambda_{\ell}^*; \widehat {\mb x}_{\ell}) 
-\mathcal M_i(\lambda_{\ell}^*; {\mb x}_{\ell})
&= \sum_{j,k} Y_i^n(\mb y_{j,k}) 
\int_{D_{j,k}} (h_{\ell}(\mb y;\widehat {\bs \kappa}) -h_{\ell}(\mb y; \bs \kappa^*)) 
\dd \mb y
\nonumber
\\
&\quad+ \sum_{j,k} \int_{D_{j,k}} 
(Y_i^n(\mb y_{j,k}) -Y_i^n(\mb y)) h_{\ell}(\mb y; {\bs \kappa}^*) \dd \mb y
\nonumber
\\
&=: S_{1,i} +S_{2,i}.
\label{decomp_M}
\end{align}
Since we see from \eqref{approx_cp0} that
\begin{align*}
x(t_{i-1}^n)
&= \langle X_{t_{i-1}^n}, e_{\ell} \rangle 
= \int_D X_{t_{i-1}^n}(\mb y) h_{\ell}(\mb y; {\bs \kappa}^*) \dd \mb y 
= \sum_{j,k} \int_{D_{j,k}} X_{t_{i-1}^n}(\mb y) 
h_{\ell}(\mb y; {\bs \kappa}^*) \dd \mb y,
\\
\widehat x(t_{i-1}^n)
&= \sum_{j,k} \int_{D_{j,k}} X_{t_{i-1}^n}(\mb y_{j,k}) 
h_{\ell}(\mb y;\widehat {\bs \kappa}) \dd \mb y,
\end{align*}
we also decompose $\widehat x(t_{i-1}^n) -x(t_{i-1}^n)$ as follows.
\begin{align}
\widehat x(t_{i-1}^n) -x(t_{i-1}^n)
&= \sum_{j,k} X_{t_{i-1}^n}(\mb y_{j,k}) 
\int_{D_{j,k}} (h_{\ell}(\mb y; \widehat {\bs \kappa}) 
-h_{\ell}(\mb y; \bs \kappa^*)) \dd \mb y
\nonumber
\\
&\quad+ \sum_{j,k} \int_{D_{j,k}} 
(X_{t_{i-1}^n}(\mb y_{j,k}) -X_{t_{i-1}^n}(\mb y)) 
h_{\ell}(\mb y; \bs \kappa^*) \dd \mb y
\nonumber
\\
&=: T_{1,i} +T_{2,i}.
\label{decomp_x}
\end{align}

\begin{itemize}
\item[(1)]
It follows from \eqref{def_A--E} and \eqref{decomp_M} that
\begin{equation*}
r_{n,\epsilon} \mathcal A_n^*
= r_{n,\epsilon} \sum_{i=1}^n (S_{1,i} +S_{2,i})^2
\le 2 r_{n,\epsilon} \sum_{i=1}^n ( S_{1,i}^2 +S_{2,i}^2 ).
\end{equation*}
By the Schwarz inequality and the Taylor expansion, we have
\begin{align*}
S_{1,i}^2 
&= \Biggl( \sum_{j,k} Y_i^n(\mb y_{j,k}) 
\int_{D_{j,k}} (h_{\ell}(\mb y;\widehat {\bs \kappa}) -h_{\ell}(\mb y; \bs \kappa^*)) 
\dd \mb y
\Biggr)^2
\\
&\le 
\sum_{j,k} (Y_i^n(\mb y_{j,k}))^2 
\biggl( 
\int_{D_{j,k}} (h_{\ell}(\mb y;\widehat {\bs \kappa}) -h_{\ell}(\mb y; \bs \kappa^*))
\dd \mb y
\biggr)^2
\sum_{j,k} 1
\\
&\le 
\sum_{j,k} (Y_i^n(\mb y_{j,k}))^2 
\int_{D_{j,k}} (h_{\ell}(\mb y;\widehat {\bs \kappa}) -h_{\ell}(\mb y; \bs \kappa^*))^2
\dd \mb y
\int_{D_{j,k}} \dd \mb y
\sum_{j,k} 1
\\
&=
\sum_{j,k} (Y_i^n(\mb y_{j,k}))^2 
\int_{D_{j,k}} 
(h_{\ell}(\mb y;\widehat {\bs \kappa})-h_{\ell}(\mb y;\bs \kappa^*))^2 \dd \mb y
\\
&\le \sum_{j,k} (Y_i^n(\mb y_{j,k}))^2 
\int_{D_{j,k}} 
\biggl| 
\int_0^1 \pd_{\bs \kappa} h_{\ell}(\mb y;\widehat {\bs \kappa}_u) \dd u
\biggr|^2 |\widehat {\bs \kappa} -{\bs \kappa}^*|^2 \dd \mb y
\end{align*}
with 
$\widehat {\bs \kappa}_u = {\bs \kappa}^* +u(\widehat {\bs \kappa} -{\bs \kappa}^*)$ 
and 
\begin{align*}
r_{n,\epsilon} \sum_{i=1}^n S_{1,i}^2 
&\le 
\frac{r_{n,\epsilon}}{R^2} \sum_{i=1}^n
\sum_{j,k} (Y_i^n(\mb y_{j,k}))^2 
\int_{D_{j,k}} 
\biggl| 
\int_0^1 \pd_{\bs \kappa} h_{\ell}(\mb y;\widehat {\bs \kappa}_u) \dd u
\biggr|^2 \dd \mb y
|R (\widehat {\bs \kappa} -{\bs \kappa}^*)|^2
\\
&=: K_{n} |R (\widehat {\bs \kappa} -{\bs \kappa}^*)|^2.
\end{align*}
Since ${\bs \kappa}^*$ and $\widehat {\bs \kappa}$ belong to 
the compact convex subset $\Xi_{\bs \kappa}$ of $\mathbb R^2$, 
$\widehat {\bs \kappa}_u \in \Xi_{\bs \kappa}$ for all $u \in [0,1]$.
We then obtain
\begin{equation}\label{int_h}
\biggl| 
\int_0^1 \pd_{\bs \kappa} h_{\ell}(\mb y;\widehat {\bs \kappa}_u) \dd u
\biggr|
\le 
\int_0^1 \bigl| \pd_{\bs \kappa} h_{\ell}(\mb y;\widehat {\bs \kappa}_u) 
\bigr| \dd u
\le
\sup_{(\mb y, {\bs \kappa}) \in \overline D \times \Xi_{\bs \kappa}} 
\bigl| \pd_{\bs \kappa} h_{\ell}(\mb y;{\bs \kappa}) \bigr|
\lesssim 1
\end{equation}
and
\begin{equation*}
K_n \lesssim \frac{r_{n,\epsilon}}{R^2} 
\sum_{i=1}^n \sum_{j,k} (Y_i^n(\mb y_{j,k}))^2 \int_{D_{j,k}} \dd \mb y.
\end{equation*}
Therefore, we see from \eqref{eq_Y}, 
the independence of 
$\{ \mathcal M_i(\lambda_{l}^*; {\mb x}_{l}) \}_{l \in \mathbb N^2}$,
\eqref{E_M2} and Lemma 5.4 in \cite{TKU2025arXiv2} that 
\begin{align*}
\EE \bigl[ (Y_i^n(\textbf{y}))^2 \bigr]
&= \sum_{l_1,l_2 \in \mathbb N^2} 
\EE \bigl[ \mathcal M_i(\lambda_{l_1}^*; {\mb x}_{l_1})
\mathcal M_i(\lambda_{l_2}^*; {\mb x}_{l_2}) \bigr] 
e_{l_1}(\mb y) e_{l_2}(\mb y)
\\
&= \sum_{l \in \mathbb N^2} 
\EE \bigl[ \mathcal M_i(\lambda_{l}^*; {\mb x}_{l})^2 \bigr] e_l(\mb y)^2
\\
&=
\epsilon^2 \sum_{l \in \mathbb N^2} 
\frac{1-\ee^{-2\lambda_l^* \Delta_n}}{2 \lambda_l^* \mu_l^{\alpha}} e_l(\mb y)^2
\\
&= \OO( \epsilon^2 \Delta_n^{\alpha \tand 1})
\end{align*}
uniformly in $\mb y \in \overline D$, and see that 
\begin{equation*}
\EE[|K_n|] 
\lesssim 
\frac{r_{n,\epsilon}}{R^2}
\sum_{i=1}^n \sum_{j,k} \EE \bigl[(Y_i^n(\mb y_{j,k}))^2 \bigr]
\int_{D_{j,k}} \dd \mb y
= \OO \biggl( 
\frac{n \epsilon^2 r_{n,\epsilon} \Delta_n^{\alpha \tand 1}}{R^2}
\biggr),
\end{equation*}
which together with Theorem \ref{th2} yields 
\begin{equation}\label{est_S1}
r_{n,\epsilon} \sum_{i=1}^n S_{1,i}^2 
= \Op \biggl( \frac{n \epsilon^2 r_{n,\epsilon} \Delta_n^{\alpha \tand 1}}{R^2} \biggr)
= \op(1).
\end{equation}

We also find from \eqref{eq_Y}, \eqref{E_M2} and Lemma 5.4 in \cite{TKU2025arXiv2} that
\begin{align*}
&\EE \bigl[ (Y_i^n(\mb y) -Y_i^n(\mb z))^2 \bigr] 
\\
&= \sum_{l_1,l_2 \in \mathbb N^2} 
\EE[ \mathcal M_i(\lambda_{l_1}^*; {\mb x}_{l_1}) 
\mathcal M_i(\lambda_{l_2}^*; {\mb x}_{l_2}) ] 
(e_{l_1}(\mb y) -e_{l_1}(\mb z))(e_{l_2}(\mb y) -e_{l_2}(\mb z))
\\
&=
\sum_{l \in \mathbb N^2} 
\EE \bigl[ \mathcal M_i(\lambda_{l}^*; {\mb x}_{l})^2  \bigr] 
(e_l(\mb y) -e_l(\mb z))^2
\\
&=
\epsilon^2 \sum_{l \in \mathbb N^2} 
\frac{1-\ee^{-2\lambda_l^* \Delta_n}}{2 \lambda_l^* \mu_l^{\alpha}} 
(e_l(\mb y) -e_l(\mb z))^2
\\
&= 
\OO \biggl(
\frac{\epsilon^2}{M_{(1)}^{2\alpha \tand 2}}
\biggr) 
\end{align*}
uniformly in $\mb y = (y^{(1)},y^{(2)})$, $\mb z = (z^{(1)},z^{(2)})$ 
with $|y^{(k)} -z^{(k)}| \le 1/M_k$, $k \in \{ 1, 2 \}$.
We hence obtain by the Schwarz inequality, 
\begin{align*}
r_{n,\epsilon} \sum_{i=1}^n S_{2,i}^2
&= r_{n,\epsilon} \sum_{i=1}^n 
\Biggl(
\sum_{j,k} \int_{D_{j,k}} 
(Y_i^n(\mb y_{j,k}) -Y_i^n(\mb y)) h_{\ell}(\mb y; {\bs \kappa}^*) \dd \mb y
\Biggr)^2
\\
&\lesssim 
r_{n,\epsilon} \sum_{i=1}^n \sum_{j,k} \int_{D_{j,k}} 
(Y_i^n(\mb y_{j,k}) -Y_i^n(\mb y))^2
\dd \mb y
\sum_{j,k} \int_{D_{j,k}} \dd \mb y
\\
&=\Op \biggl(
\frac{n \epsilon^2 r_{n,\epsilon}}{M_{(1)}^{2 \alpha \tand 2}}
\biggr) 
\\
&= \op(1).
\end{align*}

\item[(2)]
Let $\mathcal M_{i,l} = \mathcal M_i(\lambda_l^*; {\mb x}_l)$.
By \eqref{def_A1--B2}, \eqref{decomp_M} and the Schwarz inequality, we have
\begin{align*}
r_{n,\epsilon} \mathcal A_{1,n}^*
&=
r_{n,\epsilon} \Biggl| \sum_{i=1}^n (S_{1,i} +S_{2,i}) 
\mathcal M_{i,\ell} \Biggr|
\\
&\le
\sqrt{r_{n,\epsilon}^2 
\sum_{i=1}^n S_{1,i}^2 \sum_{i=1}^n \mathcal M_{i,\ell}^2}
+r_{n,\epsilon} \Biggl| \sum_{i=1}^n S_{2,i} \mathcal M_{i,\ell} \Biggr|
\\
&=
\sqrt{\epsilon^2 r_{n,\epsilon}^2 \sum_{i=1}^n S_{1,i}^2 
\times \epsilon^{-2} \mathcal C_n^*}
+r_{n,\epsilon} 
\Biggl| \sum_{i=1}^n S_{2,i} \mathcal M_{i,\ell} \Biggr|.
\end{align*}
We see from \eqref{est_CDE*} that $\epsilon^{-2} \mathcal C_n^* = \Op(1)$
and from \eqref{est_S1} and 
$\frac{n \epsilon^4 r_{n,\epsilon}^2 \Delta_n^{\alpha \tand 1}}{R_2^2} = \oo(1)$ that
\begin{equation*}
\epsilon^2 r_{n,\epsilon}^2 \sum_{i=1}^n S_{1,i}^2
= \Op \biggl( 
\frac{n \epsilon^4 r_{n,\epsilon}^2 \Delta_n^{\alpha \tand 1}}{R_2^2}
\biggr)
= \op(1).
\end{equation*}
Therefore, all that remains is to show that 
\begin{equation}\label{2-1}
r_{n,\epsilon} \sum_{i=1}^n 
\bigl| \EE[S_{2,i} \mathcal M_{i,\ell} | \GG_{i-1,\ell} ] \bigr| = \op(1),
\end{equation}
\begin{equation}\label{2-2}
r_{n,\epsilon}^2 \sum_{i=1}^n \EE[(S_{2,i} \mathcal M_{i,\ell})^2 | \GG_{i-1,\ell} ] = \op(1)
\end{equation}
hold under $\frac{\epsilon^4 r_{n,\epsilon}^2}{M_{(1)}^{2\alpha \tand 2}} = \oo(1)$
(see Lemma 5.1 in \cite{TKU2025arXiv2}). 

We first show \eqref{2-1}. 
Since it follows from \eqref{eq_Y} that
\begin{align*}
S_{2,i} \mathcal M_{i,\ell}
&= \sum_{j,k} 
\int_{D_{j,k}} (Y_i^n (\mb y_{j,k}) -Y_i^n (\mb y)) h_{\ell}(\mb y) \dd \mb y 
\mathcal M_{i,\ell}
\\
&= \sum_{j,k} 
\int_{D_{j,k}} \sum_{l \in \mathbb N^2}
\mathcal M_{i,\ell} \mathcal M_{i,l}
(e_l(\mb y_{j,k}) -e_l(\mb y)) h_{\ell}(\mb y) \dd \mb y
\end{align*}
and
\begin{equation*}
\EE[ \mathcal M_{i,\ell} \mathcal M_{i,l} | \GG_{i-1,\ell} ] = 
\begin{cases}
\EE[ \mathcal M_{i,\ell}^2] 
& l = \ell,
\\
\EE[ \mathcal M_{i,\ell} ] \EE[ \mathcal M_{i,l} ] 
= 0,
& l \neq \ell,
\end{cases}
\end{equation*}
it holds from \eqref{E_M2} that
\begin{align*}
\bigl| \EE[S_{2,i} \mathcal M_{i,\ell} | \GG_{i-1,\ell} ] \bigr|
&= \Biggl| \sum_{j,k} 
\int_{D_{j,k}} \sum_{l \in \mathbb N^2}
\EE[ \mathcal M_{i,\ell} \mathcal M_{i,l} | \GG_{i-1,\ell} ]
(e_l(\mb y_{j,k}) -e_l(\mb y)) h_{\ell}(\mb y) \dd \mb y 
\Biggr|
\\
&= \Biggl|
\EE[\mathcal M_{i,\ell}^2]
\sum_{j,k} 
\int_{D_{j,k}} 
(e_{\ell}(\mb y_{j,k}) -e_{\ell}(\mb y)) h_{\ell}(\mb y) \dd \mb y 
\Biggr|
\\
&\lesssim
\EE[\mathcal M_{i,\ell}^2]
\sum_{j,k} \int_{D_{j,k}} 
|e_{\ell}(\mb y_{j,k}) -e_{\ell}(\mb y)| \dd \mb y 
\\
&= \OO \biggl( \frac{\epsilon^2 \Delta_n}{M_{(1)}} \biggr) 
\end{align*}
and
\begin{equation*}
r_{n,\epsilon} \sum_{i=1}^n 
\bigl| \EE[S_{2,i} \mathcal M_{i,\ell} | \GG_{i-1,\ell} ] \bigr|
= \OO \biggl( \frac{r_{n,\epsilon} n \epsilon^2 \Delta_n}{M_{(1)}} \biggr)
= \OO \biggl( \frac{\epsilon^2 r_{n,\epsilon}}{M_{(1)}} \biggr)
= \oo(1).
\end{equation*}

We next show \eqref{2-2}. We have
\begin{align*}
\EE \bigl[ (S_{2,i} \mathcal M_{i,\ell})^2 \bigr] 
&\le 
\sum_{j,k} 
\int_{D_{j,k}} 
\EE \Biggl[
\biggl(
\sum_{l \in \mathbb N^2}
\mathcal M_{i,\ell} \mathcal M_{i,l} (e_l(\mb y_{j,k}) -e_l(\mb y)) h_{\ell}(\mb y) 
\biggr)^2
\Biggr] \dd \mb y
\\
&= 
\sum_{j,k} \int_{D_{j,k}}
\sum_{l_1, l_2 \in \mathbb N^2}
\EE \bigl[ \mathcal M_{i,\ell}^2 \mathcal M_{i,l_1} \mathcal M_{i,l_2} \bigr]
\\
&\quad \times 
(e_{l_1}(\mb y_{j,k}) -e_{l_1}(\mb y)) 
(e_{l_2}(\mb y_{j,k}) -e_{l_2}(\mb y)) h_{\ell}(\mb y)^2 \dd \mb y
\end{align*}
and
\begin{equation*}
\EE \bigl[ \mathcal M_{i,\ell}^2
\mathcal M_{i,l_1} \mathcal M_{i,l_2} \bigr]
=
\begin{cases}
\EE \bigl[ \mathcal M_{i,\ell}^4 \bigr], 
& l_1 = l_2 = \ell,
\\
\EE \bigl[ \mathcal M_{i,\ell}^2 \bigr] 
\EE \bigl[ \mathcal M_{i,l_1}^2 \bigr], 
& l_1 = l_2 \neq \ell,
\\
0, & \text{otherwise}.
\end{cases}
\end{equation*}
By \eqref{E_M2} and Lemma 5.4 in \cite{TKU2025arXiv2}, we obtain
\begin{align*}
&\EE \bigl[ (S_{2,i} \mathcal M_{i,\ell})^2 \bigr] 
\\
&\le 
\sum_{j,k} \int_{D_{j,k}}
\sum_{l_1, l_2 \in \mathbb N^2}
\EE \bigl[ \mathcal M_{i,\ell}^2 \mathcal M_{i,l_1} \mathcal M_{i,l_2} \bigr]
\\
&\qquad \times (e_{l_1}(\mb y_{j,k}) -e_{l_1}(\mb y)) 
(e_{l_2}(\mb y_{j,k}) -e_{l_2}(\mb y)) h_{\ell}(\mb y)^2 \dd \mb y
\\
&\lesssim
\EE \bigl[ \mathcal M_{i,\ell}^4 \bigr]
\sum_{j,k} \int_{D_{j,k}} 
(e_{\ell}(\mb y_{j,k}) -e_{\ell}(\mb y))^2 h_{\ell}(\mb y)^2 \dd \mb y
\\
&\quad+
\EE \bigl[ \mathcal M_{i,\ell}^2 \bigr]
\sum_{j,k} \int_{D_{j,k}}
\sum_{l \in \mathbb N^2}
\EE \bigl[ \mathcal M_{i,l}^2 \bigr]
(e_l(\mb y_{j,k}) -e_l(\mb y))^2 h_{\ell}(\mb y)^2 \dd \mb y
\\
&\lesssim
\epsilon^4 \Delta_n^2 
\sum_{j,k} \int_{D_{j,k}} (e_{\ell}(\mb y_{j,k}) -e_{\ell}(\mb y))^2 \dd \mb y
\\
&\qquad+ \epsilon^4 \Delta_n \sum_{j,k} \int_{D_{j,k}}
\sum_{l \in \mathbb N^2}
\frac{(e_l(\mb y_{j,k}) -e_l(\mb y))^2}{(\lambda_l^*)^{1+\alpha}} \dd \mb y
\\
&= \OO \biggl(\frac{\epsilon^4 \Delta_n^2}{M_{(1)}^2} \biggr)
+\OO \biggl(
\frac{\epsilon^4 \Delta_n}{M_{(1)}^{2\alpha \tand 2}}
\biggr) 
\\
&= \OO \biggl(
\frac{\epsilon^4 \Delta_n}{M_{(1)}^{2\alpha \tand 2}}
\biggr) 
\end{align*}
uniformly in $i$ and
\begin{equation*}
r_{n,\epsilon}^2 \sum_{i=1}^n 
\EE[(S_{2,i} \mathcal M_{i,\ell})^2 | \GG_{i-1,\ell} ]
= \OO \biggl(
\frac{\epsilon^4 r_{n,\epsilon}^2}{M_{(1)}^{2\alpha \tand 2}}
\biggr)
= \oo(1).
\end{equation*}

\item[(3)]
It holds from \eqref{def_A1--B2}, \eqref{decomp_M} and the Schwarz inequality that 
\begin{align*}
r_{n,\epsilon} \mathcal A_{2,n}^*
&\le 
\sqrt{r_{n,\epsilon}^2 
\sum_{i=1}^n S_{1,i}^2 \sum_{i=1}^n x(t_{i-1}^n)^2}
+r_{n,\epsilon} 
\Biggl| \sum_{i=1}^n S_{2,i} x(t_{i-1}^n) \Biggr|
\\
&= 
\sqrt{n r_{n,\epsilon}^2 
\sum_{i=1}^n S_{1,i}^2 \times \frac{1}{n} \mathcal E_n^*}
+r_{n,\epsilon} 
\Biggl| \sum_{i=1}^n S_{2,i} x(t_{i-1}^n) \Biggr|.
\end{align*}
It then follows from \eqref{est_CDE*} that $\frac{1}{n} \mathcal E_n^* = \Op(1)$ 
and from \eqref{est_S1} that
\begin{equation*}
n r_{n,\epsilon}^2 \sum_{i=1}^n S_{1,i}^2
= \Op \biggl( 
\frac{n^2 \epsilon^2 r_{n,\epsilon}^2 \Delta_n^{\alpha \tand 1}}{R_2^2} \biggr)
= \op(1).
\end{equation*}
Hence, we will show 
\begin{equation}\label{3-1}
r_{n,\epsilon} \sum_{i=1}^n 
\bigl| \EE[S_{2,i} x(t_{i-1}^n) | \GG_{i-1,\ell} ] \bigr| = \op(1),
\end{equation}
\begin{equation}\label{3-2}
r_{n,\epsilon}^2 \sum_{i=1}^n \EE[(S_{2,i} x(t_{i-1}^n))^2 | \GG_{i-1,\ell} ] = \op(1)
\end{equation}
under $\frac{n \epsilon^2 r_{n,\epsilon}^2}{M_{(1)}^{2\alpha \tand 2}} = \oo(1)$
(see Lemma 5.1 in \cite{TKU2025arXiv2}).

We first show \eqref{3-1}. 
Since we have by \eqref{eq_Y},
\begin{align*}
S_{2,i} x_{\ell}(t_{i-1}^n)
&= \sum_{j,k} 
\int_{D_{j,k}} (Y_i^n (\mb y_{j,k}) -Y_i^n (\mb y)) h_{\ell}(\mb y) \dd \mb y 
x_{\ell}(t_{i-1}^n)
\\
&= \sum_{j,k} 
\int_{D_{j,k}} \sum_{l \in \mathbb N^2}
x_{\ell}(t_{i-1}^n) \mathcal M_{i,l}
(e_l(\mb y_{j,k}) -e_l(\mb y)) h_{\ell}(\mb y) \dd \mb y
\end{align*}
and 
\begin{equation*}
\EE[ x_{\ell}(t_{i-1}^n) \mathcal M_{i,l} | \GG_{i-1,\ell} ] 
= x_{\ell}(t_{i-1}^n) \EE[ \mathcal M_{i,l} ]
= 0,
\end{equation*}
it holds that
\begin{align*}
\bigl| \EE[S_{2,i} x_{\ell}(t_{i-1}^n) | \GG_{i-1,\ell} ] \bigr|
&= \Biggl| \sum_{j,k} 
\int_{D_{j,k}} \sum_{l \in \mathbb N^2}
\EE[ x_{\ell}(t_{i-1}^n) \mathcal M_{i,l} | \GG_{i-1,\ell} ]
(e_l(\mb y_{j,k}) -e_l(\mb y)) h_{\ell}(\mb y) \dd \mb y 
\Biggr|
\\
&= 0
\end{align*}
and
\begin{equation*}
r_{n,\epsilon} \sum_{i=1}^n 
\bigl| \EE[S_{2,i} \mathcal M_i(\lambda^*; {\mb x}) | \GG_{i-1,\ell} ] \bigr|
= 0.
\end{equation*}

We next show \eqref{3-2}. We obtain by the Schwarz inequality,
\begin{align*}
\EE \bigl[ (S_{2,i} x_{\ell}(t_{i-1}^n))^2 \bigr] 
&= \EE \Biggl[ \Biggl( \sum_{j,k} 
\int_{D_{j,k}} \sum_{l \in \mathbb N^2}
x_{\ell}(t_{i-1}^n) \mathcal M_{i,l}
(e_l(\mb y_{j,k}) -e_l(\mb y)) h_{\ell}(\mb y) \dd \mb y
\Biggr)^2 \Biggr]
\\
&\le 
\sum_{j,k} 
\int_{D_{j,k}} 
\EE \Biggl[
\biggl(
\sum_{l \in \mathbb N^2}
x_{\ell}(t_{i-1}^n)
\mathcal M_{i,l} (e_l(\mb y_{j,k}) -e_l(\mb y)) h_{\ell}(\mb y) 
\biggr)^2
\Biggr] \dd \mb y
\\
&= 
\sum_{j,k} \int_{D_{j,k}}
\sum_{l_1, l_2 \in \mathbb N^2}
\EE \bigl[ x_{\ell}(t_{i-1}^n)^2 \mathcal M_{i,l_1} \mathcal M_{i,l_2} \bigr]
\\
&\qquad \times 
(e_{l_1}(\mb y_{j,k}) -e_{l_1}(\mb y)) 
(e_{l_2}(\mb y_{j,k}) -e_{l_2}(\mb y)) h_{\ell}(\mb y)^2 \dd \mb y
\end{align*}
and
\begin{equation*}
\EE \bigl[ x_{\ell}(t_{i-1}^n)^2 \mathcal M_{i,l_1} \mathcal M_{i,l_2} \bigr]
=
\begin{cases}
\EE \bigl[ x_{\ell}(t_{i-1}^n)^2 \bigr] \EE \bigl[ \mathcal M_{i,l_1}^2 \bigr], 
& l_1 = l_2,
\\
0, & \text{otherwise}.
\end{cases}
\end{equation*}
By \eqref{E_M2} and Lemma 5.4 in \cite{TKU2025arXiv2}, we obtain
\begin{align*}
\EE \bigl[ (S_{2,i} x_{\ell}(t_{i-1}^n))^2 \bigr] 
&\le 
\sum_{j,k} \int_{D_{j,k}}
\sum_{l_1, l_2 \in \mathbb N^2}
\EE \bigl[ x_{\ell}(t_{i-1}^n)^2 \mathcal M_{i,l_1} \mathcal M_{i,l_2} \bigr]
\\
&\qquad \times 
(e_{l_1}(\mb y_{j,k}) -e_{l_1}(\mb y)) 
(e_{l_2}(\mb y_{j,k}) -e_{l_2}(\mb y)) h_{\ell}(\mb y)^2 \dd \mb y
\\
&= 
\EE \bigl[ x_{\ell}(t_{i-1}^n)^2 \bigr]
\sum_{j,k} \int_{D_{j,k}}
\sum_{l \in \mathbb N^2}
\EE \bigl[ \mathcal M_{i,l}^2 \bigr]
(e_l(\mb y_{j,k}) -e_l(\mb y))^2 h_{\ell}(\mb y)^2 \dd \mb y
\\
&\lesssim
\epsilon^2 \sum_{j,k} \int_{D_{j,k}}
\sum_{l \in \mathbb N^2}
\frac{(e_l(\mb y_{j,k}) -e_l(\mb y))^2}{(\lambda_l^*)^{1+\alpha}} \dd \mb y
\\
&\le 
\epsilon^2 \max_{j,k} \sup_{\mb y, \mb z \in D_{j,k}}
\sum_{l \in \mathbb N^2}
\frac{(e_l(\mb y) -e_l(\mb z))^2}{(\lambda_l^*)^{1+\alpha}}
\\
&= \OO \biggl(
\frac{\epsilon^2}{M_{(1)}^{2\alpha \tand 2}}
\biggr) 
\end{align*}
uniformly in $i$ and 
\begin{equation*}
r_{n,\epsilon}^2 \sum_{i=1}^n 
\EE[(S_{2,i} x_{\ell}(t_{i-1}^n))^2 | \GG_{i-1,\ell} ]
= \Op \biggl(
\frac{n \epsilon^2 r_{n,\epsilon}^2}{M_{(1)}^{2\alpha \tand 2}}
\biggr)
= \op(1).
\end{equation*}

\item[(4)]
We find from \eqref{def_A1--B2}, \eqref{decomp_x} and the Schwarz inequality that
\begin{align*}
r_{n,\epsilon} \mathcal B_{1,n}^*
&\le 
\sqrt{r_{n,\epsilon}^2 
\sum_{i=1}^n T_{1,i}^2 \sum_{i=1}^n \mathcal M_{i,\ell}^2}
+r_{n,\epsilon} 
\Biggl| \sum_{i=1}^n T_{2,i} \mathcal M_{i,\ell} \Biggr|
\\
&=
\sqrt{\epsilon^2 r_{n,\epsilon}^2 \sum_{i=1}^n T_{1,i}^2 
\times \epsilon^{-2} \mathcal C_n^*}
+r_{n,\epsilon} 
\Biggl| \sum_{i=1}^n T_{2,i} \mathcal M_{i,\ell} \Biggr|.
\end{align*}
We find from \eqref{est_CDE*} that $\epsilon^{-2} \mathcal C_n^* = \Op(1)$.
By the Schwarz inequality and the Taylor expansion, we have
\begin{align*}
T_{1,i}^2 
&= \Biggl( \sum_{j,k} \int_{D_{j,k}} X_{t_{i-1}^n}(\mb y_{j,k}) 
(h_{\ell}(\mb y; \widehat {\bs \kappa}) -h_{\ell}(\mb y; \bs \kappa^*)) \dd \mb y
\Biggr)^2
\\
&\le \sum_{j,k} \int_{D_{j,k}} (X_{t_{i-1}^n}(\mb y_{j,k}))^2
(h_{\ell}(\mb y;\widehat {\bs \kappa})-h_{\ell}(\mb y;\bs \kappa^*))^2 \dd \mb y
\\
&\le \sum_{j,k} (X_{t_{i-1}^n}(\mb y_{j,k}))^2 
\int_{D_{j,k}} 
\biggl| 
\int_0^1 \pd_{\bs \kappa} h_{\ell}(\mb y;\widehat {\bs \kappa}_u) \dd u
\biggr|^2 |\widehat {\bs \kappa} -\bs \kappa^*|^2 \dd \mb y
\end{align*}
and 
\begin{align*}
r_{n,\epsilon} \sum_{i=1}^n T_{1,i}^2 
&\le 
\frac{r_{n,\epsilon}}{R^2} \sum_{i=1}^n
\sum_{j,k} (X_{t_{i-1}^n}(\mb y_{j,k}))^2 
\int_{D_{j,k}} 
\biggl| 
\int_0^1 \pd_{\bs \kappa} h_{\ell}(\mb y;\widehat {\bs \kappa}_u) \dd u
\biggr|^2 \dd \mb y
|R (\widehat {\bs \kappa} -\bs \kappa^*)|^2
\\
&=: L_{n} |R (\widehat {\bs \kappa} -\bs \kappa^*)|^2.
\end{align*}
Since it holds from \eqref{int_h} that
\begin{equation*}
L_n \lesssim \frac{r_{n,\epsilon}}{R^2} 
\sum_{i=1}^n \sum_{j,k} (X_{t_{i-1}^n}(\mb y_{j,k}))^2
\int_{D_{j,k}} \dd \mb y
\end{equation*}
and from Lemma \ref{lem_est_x} below that
\begin{align*}
\EE \bigl[ (X_{t_{i-1}^n}(\mb y))^2 \bigr]
&= \sum_{l_1,l_2 \in \mathbb N^2} 
\EE \bigl[ x_{l_1}(t_{i-1}^n) x_{l_2}(t_{i-1}^n) \bigl] 
e_{l_1}(\mb y) e_{l_2}(\mb y)
\\
&= \sum_{l \in \mathbb N^2} \EE \bigl[ x_l(t_{i-1}^n)^2 \bigr] e_l(\mb y)^2
\\
&\quad+
\sum_{l_1,l_2 \in \mathbb N^2, l_1 \neq l_2} \EE \bigl[ x_{l_1}(t_{i-1}^n) \bigr] \EE \bigl[ x_{l_2}(t_{i-1}^n) \bigr] e_{l_1}(\mb y) e_{l_2}(\mb y)
\\
&= \sum_{l \in \mathbb N^2} \VV \bigl[ x_l(t_{i-1}^n) \bigr] e_l(\mb y)^2
+\biggl( \sum_{l \in \mathbb N^2} \EE \bigl[ x_l(t_{i-1}^n) \bigr] e_l(\mb y) \biggr)^2
\\
&\lesssim \sum_{l \in \mathbb N^2} 
\biggl( 
\frac{\epsilon^2}{(\lambda_l^*)^{1+\alpha}} +\frac{1}{(\lambda_l^*)^{1+\alpha_0}} 
\biggr)
\\
&= \OO(1)
\end{align*}
uniformly in $\mb y \in \overline D$, we have
\begin{equation*}
\EE[|L_n|] 
\lesssim 
\frac{r_{n,\epsilon}}{R^2}
\sum_{i=1}^n \sum_{j,k} \EE \bigl[(X_{t_{i-1}^n}(\mb y_{j,k}))^2 \bigr]
\int_{D_{j,k}} \dd \mb y
= \OO \biggl( \frac{n r_{n,\epsilon}}{R^2} \biggr)
\end{equation*}
and
\begin{equation}\label{est_T1}
r_{n,\epsilon} \sum_{i=1}^n T_{1,i}^2 
= \Op \biggl( \frac{n r_{n,\epsilon}}{R^2} \biggr).
\end{equation}
In particular, we have
\begin{equation*}
\epsilon^2 r_{n,\epsilon}^2 \sum_{i=1}^n T_{1,i}^2 
= \Op \biggl( \frac{n \epsilon^2 r_{n,\epsilon}^2}{R^2} \biggr) = \op(1)
\end{equation*}
under $\frac{n \epsilon^2 r_{n,\epsilon}^2}{R^2} = \oo(1)$.

Therefore, all that remains is to show that 
\begin{equation}\label{4-1}
r_{n,\epsilon} \sum_{i=1}^n 
\bigl| \EE[T_{2,i} \mathcal M_{i,\ell} | \GG_{i-1,\ell} ] \bigr| = \op(1),
\end{equation}
\begin{equation}\label{4-2}
r_{n,\epsilon}^2 \sum_{i=1}^n \EE[(T_{2,i} \mathcal M_{i,\ell})^2 | \GG_{i-1,\ell} ] = \op(1)
\end{equation}
under 
$\frac{\epsilon^4 r_{n,\epsilon}^2}{M_{(1)}^{2\alpha \tand 2}} = \oo(1)$ and
$\frac{\epsilon^2 r_{n,\epsilon}^2}{M_{(1)}^{2\alpha_0 \tand 2}} = \oo(1)$
(see Lemma 5.1 in \cite{TKU2025arXiv2}).

We show \eqref{4-1}. Since
\begin{align*}
T_{2,i} \mathcal M_{i,\ell}
&= \sum_{j,k} 
\int_{D_{j,k}} (X_{t_{i-1}^n} (\mb y_{j,k}) -X_{t_{i-1}^n} (\mb y)) 
h_{\ell}(\mb y) \dd \mb y 
\mathcal M_{i,\ell}
\\
&= \sum_{j,k} 
\int_{D_{j,k}} \sum_{l \in \mathbb N^2}
\mathcal M_{i,\ell} x_l(t_{i-1}^n)
(e_l(\mb y_{j,k}) -e_l(\mb y)) h_{\ell}(\mb y) \dd \mb y
\end{align*}
and
\begin{align*}
\EE[ \mathcal M_{i,\ell} x_l(t_{i-1}^n) | \GG_{i-1,\ell} ] &= 
\begin{cases}
x_l(t_{i-1}^n) \EE[ \mathcal M_{i,\ell}], & l = \ell,
\\
\EE[ x_l(t_{i-1}^n) ] \EE[ \mathcal M_{i,\ell}], & l \neq \ell
\end{cases}
\\
&= 0,
\end{align*}
it holds that
\begin{align*}
\bigl| \EE[T_{2,i} \mathcal M_{i,\ell} | \GG_{i-1,\ell} ] \bigr|
&= \Biggl| \sum_{j,k} 
\int_{D_{j,k}} \sum_{l \in \mathbb N^2}
\EE[ \mathcal M_{i,\ell} x_l(t_{i-1}^n) | \GG_{i-1,\ell} ]
(e_l(\mb y_{j,k}) -e_l(\mb y)) h_{\ell}(y) \dd \mb y 
\Biggr|
\\
&= 0
\end{align*}
and
\begin{equation*}
r_{n,\epsilon} \sum_{i=1}^n 
\bigl| \EE[T_{2,i} \mathcal M_{i,\ell} | \GG_{i-1,\ell} ] \bigr| = 0.
\end{equation*}

We next show \eqref{4-2}. We have
\begin{align*}
\EE \bigl[ (T_{2,i} \mathcal M_{i,\ell})^2 \bigr] 
&\le 
\sum_{j,k} 
\int_{D_{j,k}} 
\EE \Biggl[
\biggl(
\sum_{l \in \mathbb N^2}
\mathcal M_{i,\ell} x_l(t_{i-1}^n) (e_l(\mb y_{j,k}) -e_l(\mb y)) h_{\ell}(\mb y) 
\biggr)^2
\Biggr] \dd \mb y
\\
&= 
\sum_{j,k} \int_{D_{j,k}}
\sum_{l_1, l_2 \in \mathbb N^2}
\EE \bigl[ \mathcal M_{i,\ell}^2
x_{l_1}(t_{i-1}^n) x_{l_2}(t_{i-1}^n) \bigr]
\\
&\qquad \times 
(e_{l_1}(\mb y_{j,k}) -e_{l_1}(\mb y)) 
(e_{l_2}(\mb y_{j,k}) -e_{l_2}(\mb y)) h_{\ell}(\mb y)^2 \dd \mb y
\end{align*}
and
\begin{equation*}
\EE \bigl[ \mathcal M_{i,\ell}^2
x_{l_1}(t_{i-1}^n) x_{l_2}(t_{i-1}^n) \bigr]
=
\begin{cases}
\EE \bigl[ \mathcal M_{i,\ell}^2 \bigr] \EE \bigl[ x_{l_1}(t_{i-1}^n)^2 \bigr], 
& l_1 = l_2,
\\
\EE \bigl[ \mathcal M_{i,\ell}^2 \bigr] 
\EE \bigl[ x_{l_1}(t_{i-1}^n) \bigr] \EE \bigl[ x_{l_2}(t_{i-1}^n) \bigr], 
& l_1 \neq l_2.
\end{cases}
\end{equation*}
By \eqref{E_M2}, Lemma \ref{lem_est_x} below and Lemma 5.4 in \cite{TKU2025arXiv2}, we obtain
\begin{align*}
\EE \bigl[ (T_{2,i} \mathcal M_{i,\ell})^2 \bigr] 
&\le 
\sum_{j,k} \int_{D_{j,k}}
\sum_{l_1, l_2 \in \mathbb N^2}
\EE \bigl[ \mathcal M_{i,\ell}^2 x_{l_1}(t_{i-1}^n) x_{l_2}(t_{i-1}^n) \bigr]
\\
&\qquad \times 
(e_{l_1}(\mb y_{j,k}) -e_{l_1}(\mb y)) 
(e_{l_2}(\mb y_{j,k}) -e_{l_2}(\mb y)) h_{\ell}(\mb y)^2 \dd \mb y
\\
&= 
\EE \bigl[ \mathcal M_{i,\ell}^2 \bigr]
\sum_{j,k} \int_{D_{j,k}} 
\sum_{l \in \mathbb N^2} \EE[x_l(t_{i-1}^n)^2]
(e_{\ell}(\mb y_{j,k}) -e_{\ell}(\mb y))^2 h_{\ell}(\mb y)^2 \dd \mb y
\\
&\quad+
\EE \bigl[ \mathcal M_{i,\ell}^2 \bigr]
\sum_{j,k} \int_{D_{j,k}}
\sum_{l_1,l_2 \in \mathbb N^2, l_1 \neq l_2}
\EE[x_{l_1}(t_{i-1}^n)] \EE[x_{l_2}(t_{i-1}^n)]
\\
&\qquad \times
(e_{l_1}(\mb y_{j,k}) -e_{l_1}(\mb y))
(e_{l_2}(\mb y_{j,k}) -e_{l_2}(\mb y)) h_{\ell}(\mb y)^2 \dd \mb y
\\
&= 
\EE \bigl[ \mathcal M_{i,\ell}^2 \bigr]
\sum_{j,k} \int_{D_{j,k}} 
\sum_{l \in \mathbb N^2} \VV[x_l(t_{i-1}^n)]
(e_{\ell}(\mb y_{j,k}) -e_{\ell}(\mb y))^2 h_{\ell}(\mb y)^2 \dd \mb y
\\
&\quad+
\EE \bigl[ \mathcal M_{i,\ell}^2 \bigr]
\sum_{j,k} \int_{D_{j,k}}
\biggl( \sum_{l \in \mathbb N^2}
\EE[x_l(t_{i-1}^n)] (e_l(\mb y_{j,k}) -e_l(\mb y)) h_{\ell}(\mb y)
\biggr)^2
\dd \mb y
\\
&\lesssim
\epsilon^4 \Delta_n
\sum_{j,k} \int_{D_{j,k}} \sum_{l \in \mathbb N^2}
\frac{(e_l(\mb y_{j,k}) -e_l(\mb y))^2}{(\lambda_l^*)^{1+\alpha}} \dd \mb y
\\
&\quad + \epsilon^2 \Delta_n \sum_{j,k} \int_{D_{j,k}}
\sum_{l \in \mathbb N^2}
\frac{(e_l(\mb y_{j,k}) -e_l(\mb y))^2}{(\lambda_l^*)^{1+\alpha_0}} \dd \mb y
\\
&\le
\epsilon^4 \Delta_n
\max_{j,k} \sup_{\mb y, \mb z \in D_{j,k}} \sum_{l \in \mathbb N^2}
\frac{(e_l(\mb y) -e_l(\mb z))^2}{(\lambda_l^*)^{1+\alpha}}
\\
&\quad + \epsilon^2 \Delta_n 
\max_{j,k} \sup_{\mb y, \mb z \in D_{j,k}} \sum_{l \in \mathbb N^2}
\frac{(e_l(\mb y) -e_l(\mb z))^2}{(\lambda_l^*)^{1+\alpha_0}}
\\
&= 
\OO \biggl(
\frac{\epsilon^4 \Delta_n}{M_{(1)}^{2\alpha \tand 2}}
\biggr) 
+\OO \biggl(
\frac{\epsilon^2 \Delta_n}{M_{(1)}^{2\alpha_0 \tand 2}}
\biggr) 
\end{align*}
uniformly in $i$ and see from 
$\frac{\epsilon^4 r_{n,\epsilon}^2}{M_{(1)}^{2\alpha \tand 2}} = \oo(1)$ and
$\frac{\epsilon^2 r_{n,\epsilon}^2}{M_{(1)}^{2\alpha_0 \tand 2}} = \oo(1)$ that
\begin{equation*}
r_{n,\epsilon}^2 \sum_{i=1}^n 
\EE[(T_{2,i} \mathcal M_i(\lambda^*; {\mb x}))^2 | \GG_{i-1,\ell} ]
= \Op \biggl(
\frac{\epsilon^4 r_{n,\epsilon}^2}{M_{(1)}^{2\alpha \tand 2}}
\biggr) 
+\Op \biggl(
\frac{\epsilon^2 r_{n,\epsilon}^2}{M_{(1)}^{2\alpha_0 \tand 2}}
\biggr) 
= \oo(1).
\end{equation*}

\item[(5)]
By \eqref{def_A1--B2}, \eqref{decomp_x} and the Schwarz inequality, we obtain
\begin{align*}
r_{n,\epsilon} \mathcal B_{2,n}^*
&\le 
\sqrt{r_{n,\epsilon}^2 
\sum_{i=1}^n T_{1,i}^2 \sum_{i=1}^n x(t_{i-1}^n)^2}
+r_{n,\epsilon} 
\Biggl| \sum_{i=1}^n T_{2,i} x(t_{i-1}^n) \Biggr|
\\
&= 
\sqrt{n r_{n,\epsilon}^2 
\sum_{i=1}^n T_{1,i}^2 \times \frac{1}{n} \mathcal E_n^*}
+r_{n,\epsilon} 
\Biggl| \sum_{i=1}^n T_{2,i} x(t_{i-1}^n) \Biggr|.
\end{align*}
We see from \eqref{est_CDE*} that $\frac{1}{n} \mathcal E_n^* = \Op(1)$
and from \eqref{est_T1} and $\frac{n^2 r_{n,\epsilon}^2}{R^2} = \oo(1)$ that
\begin{equation*}
n r_{n,\epsilon}^2 \sum_{i=1}^n T_{1,i}^2
= \Op \biggl( \frac{n^2 r_{n,\epsilon}^2}{R^2} \biggr)
= \op(1).
\end{equation*}

We hence prove 
\begin{equation}\label{5-1}
r_{n,\epsilon} \sum_{i=1}^n 
\bigl| \EE[T_{2,i} x(t_{i-1}^n) | \GG_{i-1,\ell} ] \bigr| = \op(1),
\end{equation}
\begin{equation}\label{5-2}
r_{n,\epsilon}^2 \sum_{i=1}^n \EE[(T_{2,i} x(t_{i-1}^n))^2 | \GG_{i-1,\ell} ] = \op(1)
\end{equation}
under
$\frac{n r_{n,\epsilon}}{M_{(1)}} = \oo(1)$, 
$\frac{n \epsilon^2 r_{n,\epsilon}^2}{M_{(1)}^{2\alpha \tand 2}} = \oo(1)$ and
$\frac{n r_{n,\epsilon}^2}{M_{(1)}^{2 \alpha_0 \tand 2}} = \oo(1)$ 
(see Lemma 5.1 in \cite{TKU2025arXiv2}). 

We show \eqref{5-1}. 
Since
\begin{align*}
T_{2,i} x_{\ell}(t_{i-1}^n)
&= \sum_{j,k} 
\int_{D_{j,k}} (X_{t_{i-1}^n} (\mb y_{j,k}) -X_{t_{i-1}^n} (\mb y)) 
h_{\ell}(\mb y) \dd \mb y 
x_{\ell}(t_{i-1}^n)
\\
&= \sum_{j,k} 
\int_{D_{j,k}} \sum_{l \in \mathbb N^2}
x_{\ell}(t_{i-1}^n) x_l(t_{i-1}^n)
(e_l(\mb y_{j,k}) -e_l(\mb y)) h_{\ell}(\mb y) \dd \mb y
\end{align*}
and
\begin{equation*}
\EE[ x_{\ell}(t_{i-1}^n) x_l(t_{i-1}^n) | \GG_{i-1,\ell} ] = 
\begin{cases}
x_{\ell}(t_{i-1}^n)^2, & l = \ell,
\\
x_{\ell}(t_{i-1}^n) \EE[x_{l}(t_{i-1}^n)], & l \neq \ell,
\end{cases}
\end{equation*} 
it holds that
\begin{align*}
&\bigl| \EE[T_{2,i} x_{\ell}(t_{i-1}^n) | \GG_{i-1,\ell} ] \bigr|
\\
&= \Biggl| \sum_{j,k} 
\int_{D_{j,k}} \sum_{l \in \mathbb N^2}
\EE[ x_{\ell}(t_{i-1}^n) x_l(t_{i-1}^n) | \GG_{i-1,\ell} ]
(e_l(\mb y_{j,k}) -e_l(\mb y)) h_{\ell}(\mb y) \dd \mb y 
\Biggr|
\\
&= \Biggl| 
\sum_{j,k} 
\int_{D_{j,k}} x_{\ell}(t_{i-1}^n)^2 (e_l(\mb y_{j,k}) -e_l(\mb y)) 
h_{\ell}(\mb y) \dd \mb y 
\\
&\qquad+\sum_{j,k} 
\int_{D_{j,k}} \sum_{l \in \mathbb N^2 \setminus \{ \ell \}}
x_{\ell}(t_{i-1}^n) \EE[ x_l(t_{i-1}^n) ]
(e_l(\mb y_{j,k}) -e_l(\mb y)) h_{\ell}(\mb y) \dd \mb y 
\Biggr|
\\
&\lesssim 
x_{\ell}(t_{i-1}^n)^2
\sum_{j,k} 
\int_{D_{j,k}} |e_l(\mb y_{j,k}) -e_l(\mb y)| \dd \mb y 
\\
&\quad+ \Biggl| 
\sum_{j,k} 
\int_{D_{j,k}} \sum_{l \in \mathbb N^2 \setminus \{ \ell \}}
x_{\ell}(t_{i-1}^n) \EE[ x_l(t_{i-1}^n) ]
(e_l(\mb y_{j,k}) -e_l(\mb y)) h_{\ell}(\mb y) \dd \mb y 
\Biggr|
\\
&=: \Op \biggl(\frac{1}{M_{(1)}} \biggr) +\mathcal Q_i.
\end{align*}
It follows from $\frac{n r_{n,\epsilon}}{M_{(1)}} = \oo(1)$ that
\begin{equation*}
r_{n,\epsilon} \sum_{i=1}^n \frac{1}{M_{(1)}} 
= \frac{n r_{n,\epsilon}}{M_{(1)}} 
= \oo(1).
\end{equation*}
Since it also follows from the Schwarz inequality that
\begin{equation*}
\EE \Biggl[ \biggl( r_{n,\epsilon} \sum_{i=1}^n \mathcal Q_i \biggr)^2 \Biggr]
\le n r_{n,\epsilon}^2 \sum_{i=1}^n \EE[\mathcal Q_i^2],
\end{equation*}
it suffices to show
\begin{equation*}
n r_{n,\epsilon}^2 \sum_{i=1}^n \EE[\mathcal Q_i^2] = \oo(1).
\end{equation*}
We obtain by the Schwarz inequality and Lemma \ref{lem_est_x} below and Lemma 5.4 in \cite{TKU2025arXiv2},
\begin{align*}
\EE[\mathcal Q_i^2] &=
\EE \Biggl[ \biggl( 
\sum_{j,k} \int_{D_{j,k}} x_{\ell} (t_{i-1}^n)
\sum_{l \in \mathbb N^2 \setminus \{ \ell \}} 
\EE[ x_l(t_{i-1}^n) ] 
(e_l(\mb y_{j,k}) -e_l(\mb y)) h_{\ell}(\mb y) \dd \mb y 
\biggr)^2 \Biggr] 
\\
&\le
\EE \Biggl[ 
\sum_{j,k} \int_{D_{j,k}} x_{\ell}(t_{i-1}^n)^2
\biggl(
\sum_{l \in \mathbb N^2 \setminus \{ \ell \}} \EE[ x_l(t_{i-1}^n) ]
(e_l(\mb y_{j,k}) -e_l(\mb y)) h_{\ell}(\mb y) \biggr)^2 \dd \mb y 
\\
&\qquad \times
\sum_{j,k} \int_{D_{j,k}} \dd \mb y 
\Biggr] 
\\
&=
\EE[ x_{\ell}(t_{i-1}^n)^2 ]
\sum_{j,k} \int_{D_{j,k}} 
\biggl(
\sum_{l \in \mathbb N^2 \setminus \{ \ell \}} \EE[x_l(t_{i-1}^n)]
(e_l(\mb y_{j,k}) -e_l(\mb y)) h_{\ell}(\mb y) \biggr)^2 \dd \mb y
\\
&\lesssim \sum_{j,k} \int_{D_{j,k}}
\sum_{l \in \mathbb N^2}
\frac{(e_l(\mb y_{j,k}) -e_l(\mb y))^2}{(\lambda_l^*)^{1+\alpha_0}} \dd \mb y
\\
&\le \max_{j,k} \sup_{\mb y, \mb z \in D_{j,k}}
\sum_{l \in \mathbb N^2}
\frac{(e_l(\mb y) -e_l(\mb z))^2}{(\lambda_l^*)^{1+\alpha_0}}
\\
&= 
\OO \biggl(
\frac{1}{M_{(1)}^{2\alpha_0 \tand 2}}
\biggr) 
\end{align*}
uniformly in $i$.
We hence obtain under $\frac{n^2 r_{n,\epsilon}^2}{M_{(1)}^{2\alpha_0 \tand 2}} = \oo(1)$,
\begin{equation*}
n r_{n,\epsilon}^2 \sum_{i=1}^n \EE[\mathcal Q_i^2] 
= \OO \biggl(
\frac{n^2 r_{n,\epsilon}^2}{M_{(1)}^{2\alpha_0 \tand 2}}
\biggr) 
= \oo(1).
\end{equation*}

We next show \eqref{5-2}. We have
\begin{align*}
\EE \bigl[ (T_{2,i} x_{\ell}(t_{i-1}^n))^2 \bigr] 
&\le 
\sum_{j,k} 
\int_{D_{j,k}} 
\EE \Biggl[
\biggl(
\sum_{l \in \mathbb N^2}
x_{\ell}(t_{i-1}^n)
x_l(t_{i-1}^n) (e_l(\mb y_{j,k}) -e_l(\mb y)) h_{\ell}(\mb y) 
\biggr)^2
\Biggr] \dd \mb y
\\
&= 
\sum_{j,k} \int_{D_{j,k}}
\sum_{l_1, l_2 \in \mathbb N^2}
\EE \bigl[ x_{\ell}(t_{i-1}^n)^2
x_{l_1}(t_{i-1}^n) x_{l_2}(t_{i-1}^n) \bigr]
\\
&\qquad \times 
(e_{l_1}(\mb y_{j,k}) -e_{l_1}(\mb y)) 
(e_{l_2}(\mb y_{j,k}) -e_{l_2}(\mb y)) h_{\ell}(\mb y)^2 \dd \mb y
\end{align*}
and
\begin{equation*}
\EE \bigl[ x_{\ell}(t_{i-1}^n)^2
x_{l_1}(t_{i-1}^n) x_{l_2}(t_{i-1}^n) \bigr] = 
\begin{cases}
\EE[x_{\ell}(t_{i-1}^n)^4], & l_1 = l_2 = \ell,
\\
\EE[x_{\ell}(t_{i-1}^n)^2] \EE[x_{l_1}(t_{i-1}^n)^2], 
& l_1 = l_2 \neq \ell,
\\
\EE[x_{\ell}(t_{i-1}^n)^3] \EE[x_{l_2}(t_{i-1}^n)], 
& l_1 = \ell \neq l_2,
\\
\EE[x_{\ell}(t_{i-1}^n)^3] \EE[x_{l_1}(t_{i-1}^n)], 
& l_2 = \ell \neq l_1,
\\
\EE[x_{\ell}(t_{i-1}^n)^2] \EE[x_{l_1}(t_{i-1}^n)] \EE[x_{l_2}(t_{i-1}^n)], 
& \text{otherwise}.
\end{cases}
\end{equation*}
By Lemma \ref{lem_est_x} below and Lemma 5.4 in \cite{TKU2025arXiv2}, we obtain
\begin{align*}
&\EE \bigl[ (T_{2,i} x_{\ell}(t_{i-1}^n))^2 \bigr] 
\\
&\le 
\sum_{j,k} \int_{D_{j,k}}
\sum_{l_1, l_2 \in \mathbb N^2}
\EE \bigl[ x_{\ell}(t_{i-1}^n)^2
x_{l_1}(t_{i-1}^n) x_{l_2}(t_{i-1}^n) \bigr]
\\
&\qquad \times 
(e_{l_1}(\mb y_{j,k}) -e_{l_1}(\mb y)) 
(e_{l_2}(\mb y_{j,k}) -e_{l_2}(\mb y)) h_{\ell}(\mb y)^2 \dd \mb y
\\
&\lesssim
\EE [ x_{\ell}(t_{i-1}^n)^4 ]
\sum_{j,k} \int_{D_{j,k}} 
(e_{\ell}(\mb y_{j,k}) -e_{\ell}(\mb y))^2 h_{\ell}(\mb y)^2 \dd \mb y
\\
&\quad+
\bigl| \EE [ x_{\ell}(t_{i-1}^n)^3] \bigr|
\sum_{j,k} \int_{D_{j,k}}
\bigl| (e_{\ell}(\mb y_{j,k}) -e_{\ell}(\mb y)) h_{\ell}(\mb y) \bigr|
\\
& \qquad \times
\biggl| \sum_{l \in \mathbb N^2}
\EE[x_l(t_{i-1}^n)] (e_l(\mb y_{j,k}) -e_l(\mb y) ) 
h_{\ell}(\mb y) \biggr| \dd \mb y
\\
&\quad+
\EE [ x_{\ell}(t_{i-1}^n)^2]
\sum_{j,k} \int_{D_{j,k}}
\sum_{l \in \mathbb N^2}
\EE[x_l(t_{i-1}^n)^2]
(e_l(\mb y_{j,k}) -e_l(\mb y))^2 h_{\ell}(\mb y)^2 \dd \mb y
\\
&\quad+
\EE [ x_{\ell}(t_{i-1}^n)^2]
\sum_{j,k} \int_{D_{j,k}}
\biggl( \sum_{l \in \mathbb N^2}
\bigl| \EE[x_l(t_{i-1}^n)] 
(e_l(\mb y_{j,k}) -e_l(\mb y)) h_{\ell}(\mb y) \bigr|
\biggr)^2
\dd \mb y
\\
&\lesssim
\sum_{j,k} \int_{D_{j,k}} (e_{\ell}(\mb y_{j,k}) -e_{\ell}(\mb y))^2 \dd \mb y
\\
&\quad+ \sum_{j,k} \int_{D_{j,k}}
|e_{\ell}(\mb y_{j,k}) -e_{\ell}(\mb y)|
\biggl( \sum_{l \in \mathbb N^2} 
\frac{(e_l(\mb y_{j,k}) -e_l(\mb y))^2}{(\lambda_l^*)^{1+\alpha_0}} \biggr)^{1/2} \dd \mb y
\\
&\quad 
+ \sum_{j,k} \int_{D_{j,k}}
\sum_{l \in \mathbb N^2}
\biggl(
\frac{1}{(\lambda_l^*)^{1+\alpha_0}}
+\frac{\epsilon^2}{(\lambda_l^*)^{1+\alpha}} 
\biggr) (e_l(\mb y_{j,k}) -e_l(\mb y))^2 \dd \mb y
\\
&\quad 
+ \sum_{j,k} \int_{D_{j,k}}
\sum_{l \in \mathbb N^2}
\frac{(e_l(\mb y_{j,k}) -e_l(\mb y))^2}{(\lambda_l^*)^{1+\alpha_0}} \dd \mb y
\\
&\le
\max_{j,k} \sup_{\mb y, \mb z \in D_{j,k}} (e_{\ell}(\mb y) -e_{\ell}(\mb z))^2
\\
&\quad+ \max_{j,k} \sup_{\mb y, \mb z \in D_{j,k}}
|e_{\ell}(\mb y) -e_{\ell}(\mb z)|
\biggl( 
\max_{j,k} \sup_{\mb y, \mb z \in D_{j,k}}
\sum_{l \in \mathbb N^2} 
\frac{(e_l(\mb y) -e_l(\mb z))^2}{(\lambda_l^*)^{1+\alpha_0}} \biggr)^{1/2}
\\
&\quad 
+ \max_{j,k} \sup_{\mb y, \mb z \in D_{j,k}}
\sum_{l \in \mathbb N^2}
\biggl(
\frac{1}{(\lambda_l^*)^{1+\alpha_0}}
+\frac{\epsilon^2}{(\lambda_l^*)^{1+\alpha}} 
\biggr) (e_l(\mb y) -e_l(\mb z))^2
\\
&\quad 
+ \max_{j,k} \sup_{\mb y, \mb z \in D_{j,k}}
\sum_{l \in \mathbb N^2}
\frac{(e_l(\mb y) -e_l(\mb z))^2}{(\lambda_l^*)^{1+\alpha_0}}
\\
&= \OO \biggl(\frac{1}{M_{(1)}^2} \biggr)
+\OO \biggl( \frac{1}{M_{(1)}} 
\cdot \frac{1}{(M_{(1)}^{2\alpha_0 \tand 2})^{1/2}} \biggr) 
+\OO \biggl( \frac{1}{M_{(1)}^{2\alpha_0 \tand 2}} 
+ \frac{\epsilon^2}{M_{(1)}^{2\alpha \tand 2}} \biggr) 
+\OO \biggl( \frac{1}{M_{(1)}^{2\alpha_0 \tand 2}} \biggr) 
\\
&= \OO \biggl( \frac{1}{M_{(1)}^{2\alpha_0 \tand 2}} +\frac{\epsilon^2}{M_{(1)}^{2\alpha \tand 2}} \biggr) 
\end{align*}
uniformly in $i$ and
\begin{equation*}
r_{n,\epsilon}^2 \sum_{i=1}^n 
\EE[(T_{2,i} x_{\ell}(t_{i-1}^n))^2 | \GG_{i-1,\ell} ]
= \OO \biggl(
\frac{n r_{n,\epsilon}^2}{M_{(1)}^{2\alpha_0 \tand 2}}
+\frac{n \epsilon^2 r_{n,\epsilon}^2}{M_{(1)}^{2\alpha \tand 2}}
\biggr)
= \oo(1).
\end{equation*}
\end{itemize}
\end{proof}

\begin{rmk}\label{rmk_AB}
As seen in the proof of Lemma \ref{lem_est_AB2},
$\mathcal A_n$, $\mathcal A_n^*$, $\mathcal B_n^*$, 
$\mathcal A_{j,n}^*$ and $\mathcal B_{j,n}^*$ $(j \in \{1, 2\})$
given in \eqref{def_A--E} and \eqref{def_A1--B2}
can be estimated 
by using the estimation error terms $S_{1,i}$ and $T_{1,i}$ of $\bs \kappa$ given in \eqref{decomp_M} and \eqref{decomp_x}  
and the spatial discretization error terms $S_{2,i}$ and $T_{2,i}$ given in \eqref{decomp_M} and \eqref{decomp_x}.
The sufficient conditions for controlling the former and latter errors correspond to the conditions involving 
$R_{\alpha, \alpha_0}^{\mathrm{coef}}$ and $M_{(1)}$, respectively, in Lemmas \ref{lem_est_AB1} and \ref{lem_est_AB2}.
\end{rmk}

Finally, we will prove \eqref{cons} and \eqref{asym_norm}.
\begin{proof}[\bf{Proof of \eqref{cons}}]
Recall the definitions of $\widehat {\mathbb V}_{j,n}^\epsilon$ and $\widehat {\mathbb W}_{j,n}^\epsilon$
($j \in \{1,\ldots,4\}$) given in \eqref{hat_V} and \eqref{hat_W}.
We first show that (G1) and (G2) hold under [C1]$_{\alpha^*,\alpha_0}$-(i), (ii), 
[C2]$_{\alpha^*,\alpha_0}$-(i), (ii) and [C3]$_{\alpha^*,\alpha_0}$-(i) 
in order to obtain $(\widehat \lambda, \widehat \mu) \pto (\lambda^*, \mu^*)$.

We first consider (G1).
We set
\begin{equation*}
{\mathbb U}_{1,n}^\epsilon
= \sup_{(\lambda,\mu) \in \Xi_{\nu}}
\Bigl|
\epsilon^2 
\bigl( 
V_n^{\epsilon} (\lambda,\mu; {\mb x},\alpha^*)
-V_n^{\epsilon} (\lambda^*,\mu; {\mb x},\alpha^*)
\bigr) 
- V_1(\lambda,\mu,\lambda^*; \alpha^*)
\Bigr|. 
\end{equation*}
Since it can be decomposed as
\begin{align*}
&\epsilon^2 \bigl( 
V_n^{\epsilon} (\lambda,\mu; \widehat {\mb x},\widehat \alpha)
-V_n^{\epsilon} (\lambda^*,\mu; \widehat {\mb x},\widehat \alpha)
\bigr) 
- V_1(\lambda,\mu,\lambda^*; \alpha^*)
\\
&= 
\epsilon^2 \Bigl( 
\bigl( 
V_n^{\epsilon} (\lambda,\mu; \widehat {\mb x},\widehat \alpha)
-V_n^{\epsilon} (\lambda^*,\mu; \widehat {\mb x},\widehat \alpha)
\bigr) 
-\bigl( 
V_n^{\epsilon} (\lambda,\mu; \widehat {\mb x},\alpha^*)
-V_n^{\epsilon} (\lambda^*,\mu; \widehat {\mb x},\alpha^*)
\bigr) 
\Bigr)
\\
&\quad+\epsilon^2 \Bigl( 
\bigl( 
V_n^{\epsilon} (\lambda,\mu; \widehat {\mb x},\alpha^*)
-V_n^{\epsilon} (\lambda^*,\mu; \widehat {\mb x},\alpha^*)
\bigr) 
-\bigl( 
V_n^{\epsilon} (\lambda,\mu; {\mb x}, \alpha^*)
-V_n^{\epsilon} (\lambda^*,\mu; {\mb x}, \alpha^*)
\bigr) 
\Bigr)
\\
&\quad+\epsilon^2
\bigl( 
V_n^{\epsilon} (\lambda,\mu; {\mb x}, \alpha^*)
-V_n^{\epsilon} (\lambda^*,\mu; {\mb x}, \alpha^*)
\bigr) 
- V_1(\lambda,\mu,\lambda^*; \alpha^*)
\\
&= 
\epsilon^2 \Bigl( 
\bigl( 
V_n^{\epsilon} (\lambda,\mu; \widehat {\mb x},\widehat \alpha)
-V_n^{\epsilon} (\lambda,\mu; \widehat {\mb x},\alpha^*)
\bigr) 
-\bigl( 
V_n^{\epsilon} (\lambda^*,\mu; \widehat {\mb x},\widehat \alpha)
-V_n^{\epsilon} (\lambda^*,\mu; \widehat {\mb x},\alpha^*)
\bigr) 
\Bigr)
\\
&\quad+\epsilon^2 \Bigl( 
\bigl( 
V_n^{\epsilon} (\lambda,\mu; \widehat {\mb x},\alpha^*)
-V_n^{\epsilon} (\lambda,\mu; {\mb x}, \alpha^*)
\bigr) 
-\bigl( 
V_n^{\epsilon} (\lambda^*,\mu; \widehat {\mb x},\alpha^*)
-V_n^{\epsilon} (\lambda^*,\mu; {\mb x}, \alpha^*)
\bigr) 
\Bigr)
\\
&\quad+\epsilon^2
\bigl( 
V_n^{\epsilon} (\lambda,\mu; {\mb x}, \alpha^*)
-V_n^{\epsilon} (\lambda^*,\mu; {\mb x}, \alpha^*)
\bigr) 
- V_1(\lambda,\mu,\lambda^*; \alpha^*),
\end{align*}
we obtain
\begin{align*}
\widehat {\mathbb U}_{1,n}^\epsilon
&= \sup_{(\lambda,\mu) \in \Xi_{\nu}}
\Bigl|
\epsilon^2 
\bigl( 
V_n^{\epsilon} (\lambda,\mu; \widehat {\mb x},\widehat \alpha)
-V_n^{\epsilon} (\lambda^*,\mu; \widehat {\mb x},\widehat \alpha)
\bigr) 
- V_1(\lambda,\mu,\lambda^*; \alpha^*)
\Bigr| 
\\
&\le 
2 \epsilon^2 
\sup_{(\lambda,\mu) \in \Xi_{\nu}}
\bigl|
V_n^{\epsilon} (\lambda,\mu; \widehat {\mb x}, \widehat \alpha)
-V_n^{\epsilon} (\lambda,\mu; \widehat {\mb x},\alpha^*)
\bigr|
\\
&\quad+ 2 \epsilon^2 
\sup_{(\lambda,\mu) \in \Xi_{\nu}}
\bigl|
V_n^{\epsilon} (\lambda,\mu; \widehat {\mb x}, \alpha^*)
-V_n^{\epsilon} (\lambda,\mu; {\mb x},\alpha^*)
\bigr|
\\
&\quad+\sup_{(\lambda,\mu) \in \Xi_{\nu}}
\Bigl|
\epsilon^2 
\bigl( 
V_n^{\epsilon} (\lambda,\mu; {\mb x},\alpha^*)
-V_n^{\epsilon} (\lambda^*,\mu; {\mb x},\alpha^*)
\bigr) 
- V_1(\lambda,\mu,\lambda^*; \alpha^*)
\Bigr| 
\\
&= 2(\widehat {\mathbb W}_{1,n}^\epsilon +\widehat {\mathbb V}_{1,n}^\epsilon)
+{\mathbb U}_{1,n}^\epsilon.
\end{align*}
Since it follows from (R8) in Proposition \ref{propR2} that 
${\mathbb U}_{1,n}^\epsilon = \op(1)$, 
(G1) holds under 
\begin{equation}\label{VW-1}
\widehat {\mathbb V}_{1,n}^\epsilon = \op(1),
\quad
\widehat {\mathbb W}_{1,n}^\epsilon = \op(1).
\end{equation}
Therefore, we find from Lemmas \ref{lem_est_V}-(1) and \ref{lem_est_W}-(1) that 
\eqref{VW-1} holds under $n \mathcal X_n = \op(1)$ and
\begin{description}
\item[(G11)]
$\frac{n \epsilon^2}{R_{\alpha^*,\alpha_0}^{\mathrm{damp}}} = \oo(1)$ 
\end{description}
and from Lemma \ref{lem_est_XY}-(1) with $r_{n,\epsilon} = n$ that
$n \mathcal X_n = \op(1)$ holds under 
\begin{description}
\item[(G12)]
$n \mathcal A_n = \op(1)$ and $(n \epsilon)^2 \mathcal A_n = \op(1)$.
\end{description}
Eventually, it suffices to show that (G11) and (G12) hold in order to obtain (G1).

Next, we consider (G2).
We define
\begin{equation*}
{\mathbb U}_{2,n}^\epsilon
= \sup_{\mu \in \Xi_{\mu}}
\biggl|
\frac{1}{n}
\bigl( 
V_n^{\epsilon} (\widehat \lambda,\mu; {\mb x},\alpha^*)
-V_n^{\epsilon} (\widehat \lambda,\mu^*; {\mb x},\alpha^*)
\bigr) 
- V_2(\mu,\mu^*; \alpha^*)
\biggr|.
\end{equation*}
We then have
\begin{align*}
&\frac{1}{n} \bigl( 
V_n^{\epsilon} (\widehat \lambda,\mu; \widehat {\mb x},\widehat \alpha)
-V_n^{\epsilon} (\widehat \lambda,\mu^*; \widehat {\mb x},\widehat \alpha)
\bigr) 
- V_2(\mu,\mu^*; \alpha^*)
\\
&= 
\frac{1}{n} \Bigl( 
\bigl( 
V_n^{\epsilon} (\widehat \lambda,\mu; \widehat {\mb x},\widehat \alpha)
-V_n^{\epsilon} (\widehat \lambda,\mu; \widehat {\mb x},\alpha^*)
\bigr) 
-\bigl( 
V_n^{\epsilon} (\widehat \lambda,\mu^*; \widehat {\mb x},\widehat \alpha)
-V_n^{\epsilon} (\widehat \lambda,\mu^*; \widehat {\mb x},\alpha^*)
\bigr) 
\Bigr)
\\
&\quad+\frac{1}{n} \Bigl( 
\bigl( 
V_n^{\epsilon} (\widehat \lambda,\mu; \widehat {\mb x},\alpha^*)
-V_n^{\epsilon} (\widehat \lambda,\mu; {\mb x}, \alpha^*)
\bigr) 
-\bigl( 
V_n^{\epsilon} (\widehat \lambda,\mu^*; \widehat {\mb x},\alpha^*)
-V_n^{\epsilon} (\widehat \lambda,\mu^*; {\mb x}, \alpha^*)
\bigr) 
\Bigr)
\\
&\quad+\frac{1}{n}
\bigl( 
V_n^{\epsilon} (\widehat \lambda,\mu; {\mb x}, \alpha^*)
-V_n^{\epsilon} (\widehat \lambda,\mu^*; {\mb x}, \alpha^*)
\bigr) 
- V_2(\mu,\mu^*; \alpha^*)
\end{align*}
and
\begin{align*}
\widehat {\mathbb U}_{2,n}^\epsilon
&= \sup_{\mu \in \Xi_{\mu}}
\biggl|
\frac{1}{n}
\bigl( 
V_n^{\epsilon} (\widehat \lambda,\mu; \widehat {\mb x},\widehat \alpha)
-V_n^{\epsilon} (\widehat \lambda,\mu^*; 
\widehat {\mb x},\widehat \alpha)
\bigr) 
- V_2(\mu,\mu^*; \alpha^*)
\biggr| 
\\
&\le 
\frac{2}{n}
\sup_{(\lambda,\mu) \in \Xi_{\nu}}
\bigl|
V_n^{\epsilon} (\lambda,\mu; \widehat {\mb x}, \widehat \alpha)
-V_n^{\epsilon} (\lambda,\mu; \widehat {\mb x},\alpha^*)
\bigr|
\\
&\quad+ \frac{2}{n}
\sup_{(\lambda,\mu) \in \Xi_{\nu}}
\bigl|
V_n^{\epsilon} (\lambda,\mu; \widehat {\mb x}, \alpha^*)
-V_n^{\epsilon} (\lambda,\mu; {\mb x},\alpha^*)
\bigr|
\\
&\quad+\sup_{\mu \in \Xi_{\mu}}
\biggl|
\frac{1}{n}
\bigl( 
V_n^{\epsilon} (\widehat \lambda,\mu; {\mb x},\alpha^*)
-V_n^{\epsilon} (\widehat \lambda,\mu^*; {\mb x},\alpha^*)
\bigr) 
- V_2(\mu,\mu^*; \alpha^*)
\biggr| 
\\
&= 2(\widehat {\mathbb W}_{2,n}^\epsilon +\widehat {\mathbb V}_{2,n}^\epsilon)
+{\mathbb U}_{2,n}^\epsilon.
\end{align*}
Since it holds from (R9) in Proposition \ref{propR2} that
${\mathbb U}_{2,n}^\epsilon = \op(1)$ under 
$\epsilon^{-1}(\widehat \lambda -\lambda^*) = \Op(1)$, 
we first show $\epsilon^{-1}(\widehat \lambda -\lambda^*) = \Op(1)$. 
By the Taylor expansion, we have
\begin{equation*}
-\epsilon 
\pd_\lambda V_n^{\epsilon} (\lambda^*,\widehat \mu; 
\widehat {\mb x},\widehat \alpha)
= \epsilon^2 \int_0^1 \pd_\lambda^2 
V_n^{\epsilon} (\widehat \lambda_u,\widehat \mu; 
\widehat {\mb x}, \widehat \alpha) \dd u
\epsilon^{-1} (\widehat \lambda -\lambda^*)
\end{equation*}
with $\widehat \lambda_u = \lambda^* +u(\widehat \lambda -\lambda^*)$.
We define $G(\lambda, \mu) = G(\lambda, \mu; x_0, \alpha^*)$ for $\lambda, \mu \in (0,\infty)$, and
\begin{equation*}
{\mathbb P}_{n}^\epsilon 
= \epsilon \sup_{\mu \in \Xi_{\mu}} 
\bigl| \pd_\lambda V_n^{\epsilon} (\lambda^*, \mu; {\mb x},\alpha^*) \bigr|,
\quad
{\mathbb Q}_{n}^\epsilon 
= \sup_{(\lambda,\mu) \in \Xi_{\nu}} 
\bigl| \epsilon^2 \pd_\lambda^2 V_n^{\epsilon} (\lambda, \mu; {\mb x},\alpha^*) 
-2 G(\lambda^*,\mu) \bigr|.
\end{equation*}
We then obtain
\begin{align*}
&\epsilon 
\bigl| \pd_\lambda V_n^{\epsilon} (\lambda^*,\widehat \mu; 
\widehat {\mb x},\widehat \alpha) \bigr|
\\
&\le 
\epsilon
\Bigl| \pd_\lambda V_n^{\epsilon} (\lambda^*,\widehat \mu; 
\widehat {\mb x},\widehat \alpha)
-\pd_\lambda V_n^{\epsilon} (\lambda^*,\widehat \mu; \widehat {\mb x}, \alpha^*)
\Bigr|
\\
&\quad+\epsilon
\Bigl| 
\pd_\lambda V_n^{\epsilon} (\lambda^*,\widehat \mu; \widehat {\mb x}, \alpha^*)
-\pd_\lambda V_n^{\epsilon} (\lambda^*,\widehat \mu; {\mb x}, \alpha^*)
\Bigr|
+ \epsilon
\bigl| 
\pd_\lambda V_n^{\epsilon} (\lambda^*,\widehat \mu; {\mb x}, \alpha^*)
\bigr|
\\
&\le 
\epsilon
\sup_{\mu \in \Xi_{\mu}} \Bigl| \pd_\lambda V_n^{\epsilon} (\lambda^*, \mu; 
\widehat {\mb x},\widehat \alpha)
-\pd_\lambda V_n^{\epsilon} (\lambda^*, \mu; \widehat {\mb x}, \alpha^*)
\Bigr|
\\
&\quad+\epsilon
\sup_{\mu \in \Xi_{\mu}} \Bigl| 
\pd_\lambda V_n^{\epsilon} (\lambda^*, \mu; \widehat {\mb x}, \alpha^*)
-\pd_\lambda V_n^{\epsilon} (\lambda^*, \mu; {\mb x}, \alpha^*)
\Bigr|
+ \epsilon
\sup_{\mu \in \Xi_{\mu}}
\bigl| 
\pd_\lambda V_n^{\epsilon} (\lambda^*, \mu; {\mb x}, \alpha^*)
\bigr|
\\
&=
(\widehat {\mathbb W}_{3,n}^\epsilon)_1 +(\widehat {\mathbb V}_{3,n}^\epsilon)_1 
+ {\mathbb P}_{n}^\epsilon
\end{align*}
and
\begin{align*}
&\sup_{u \in [0,1]}
\bigl| 
\epsilon^2 \pd_\lambda^2 
V_n^{\epsilon} (\widehat \lambda_u,\widehat \mu; \widehat {\mb x}, \widehat \alpha)
-2 G(\lambda^*,\widehat \mu)
\bigr|
\\
&\le
\epsilon^2 
\sup_{u \in [0,1]}
\Bigl| 
\pd_\lambda^2 
V_n^{\epsilon} (\widehat \lambda_u,\widehat \mu; \widehat {\mb x}, \widehat \alpha)
-\pd_\lambda^2 
V_n^{\epsilon} (\widehat \lambda_u,\widehat \mu; \widehat {\mb x}, \alpha^*)
\Bigr|
\\
&\quad+
\epsilon^2 
\sup_{u \in [0,1]}
\Bigl| 
\pd_\lambda^2 
V_n^{\epsilon} (\widehat \lambda_u,\widehat \mu; \widehat {\mb x}, \alpha^*)
-\pd_\lambda^2 
V_n^{\epsilon} (\widehat \lambda_u,\widehat \mu; {\mb x}, \alpha^*)
\Bigr|
\\
&\quad+
\sup_{u \in [0,1]}
\bigl| 
\epsilon^2 \pd_\lambda^2 
V_n^{\epsilon} (\widehat \lambda_u,\widehat \mu; {\mb x}, \alpha^*)
-2 G(\lambda^*,\widehat \mu)
\bigr|
\\
&\le
\epsilon^2 
\sup_{(\lambda, \mu) \in \Xi_{\nu}}
\Bigl| 
\pd_\lambda^2 
V_n^{\epsilon} (\lambda, \mu; \widehat {\mb x}, \widehat \alpha)
-\pd_\lambda^2 
V_n^{\epsilon} (\lambda, \mu; \widehat {\mb x}, \alpha^*)
\Bigr|
\\
&\quad+
\epsilon^2 
\sup_{(\lambda, \mu) \in \Xi_{\nu}}
\Bigl| 
\pd_\lambda^2 
V_n^{\epsilon} (\lambda, \mu; \widehat {\mb x}, \alpha^*)
-\pd_\lambda^2 
V_n^{\epsilon} (\lambda, \mu; {\mb x}, \alpha^*)
\Bigr|
\\
&\quad+
\sup_{(\lambda, \mu) \in \Xi_{\nu}}
\bigl| 
\epsilon^2 \pd_\lambda^2 
V_n^{\epsilon} (\lambda, \mu; {\mb x}, \alpha^*)
-2 G(\lambda^*,\mu)
\bigr|
\\
&=
(\widehat {\mathbb W}_{4,n}^\epsilon)_{1,1} 
+(\widehat {\mathbb V}_{4,n}^\epsilon)_{1,1} 
+ {\mathbb Q}_{n}^\epsilon.
\end{align*}
Since (R10) and (R11) in Proposition \ref{propR2} yield 
${\mathbb P}_{n}^\epsilon = \Op(1)$ and ${\mathbb Q}_{n}^\epsilon = \op(1)$, 
we obtain $\epsilon^{-1}(\widehat \lambda -\lambda^*) = \Op(1)$ under
\begin{equation}\label{VW-21}
(\widehat {\mathbb V}_{3,n}^\epsilon)_1 = \Op(1),
\quad
(\widehat {\mathbb W}_{3,n}^\epsilon)_1 = \Op(1),
\end{equation}
\begin{equation}\label{VW-22}
(\widehat {\mathbb V}_{4,n}^\epsilon)_{1,1} = \op(1),
\quad
(\widehat {\mathbb W}_{4,n}^\epsilon)_{1,1} = \op(1).
\end{equation}
We therefore obtain (G2) under
\begin{equation}\label{VW-23}
\widehat {\mathbb V}_{2,n}^\epsilon = \op(1),
\quad
\widehat {\mathbb W}_{2,n}^\epsilon = \op(1),
\end{equation}
\eqref{VW-21} and \eqref{VW-22}. 
It follows from Lemmas \ref{lem_est_V}-(2) and \ref{lem_est_W}-(2) that
\eqref{VW-23} under $\epsilon^{-2} \mathcal X_n = \op(1)$ and
\begin{description}
\item[(G21)]
$\frac{(n \epsilon^2)^{-1}}{R_{\alpha^*,\alpha_0}^{\mathrm{damp}}} = \oo(1)$
\end{description}
and from Lemma \ref{lem_est_XY}-(1) with $r_{n,\epsilon} = \epsilon^{-2}$ that $\epsilon^{-2} \mathcal X_n = \op(1)$ under 
\begin{description}
\item[(G22)]
$\epsilon^{-2} \mathcal A_n = \op(1)$ 
and $\frac{1}{n \epsilon^4} \mathcal A_n = \op(1)$.
\end{description}
We also see from Lemmas \ref{lem_est_V}-(3) and \ref{lem_est_W}-(3) that
\eqref{VW-21} holds under $\epsilon^{-1} \mathcal X_n^* = \op(1)$ 
and $\epsilon^{-1} \mathcal Y_n^* = \op(1)$,
and from Lemma \ref{lem_est_XYZ*}-(1) and (2) with $r_{n,\epsilon} = \epsilon^{-1}$ that
$\epsilon^{-1} \mathcal X_n^* = \op(1)$ and 
$\epsilon^{-1} \mathcal Y_n^* = \op(1)$ hold under
\begin{description}
\item[(G23)]
$\epsilon^{-1} \mathcal A_n^* = \op(1)$, $\epsilon^{-1} \mathcal A_{1,n}^* = \op(1)$,

\item[(G24)]
$\epsilon^{-2} \mathcal A_n^* \mathcal B_n^* = \op(1)$,
$\epsilon^{-1} \mathcal A_{2,n}^* = \op(1)$ and
$\epsilon^{-1} \mathcal B_{1,n}^* = \op(1)$.
\end{description}
We further see from Lemmas \ref{lem_est_V}-(4) and \ref{lem_est_W}-(4) that
\eqref{VW-22} holds under 
$\frac{1}{n} \mathcal X_n = \op(1)$, 
$\frac{1}{n} \mathcal Y_n = \op(1)$
and $\frac{1}{n} \mathcal Z_n^* = \op(1)$, 
and from Lemma \ref{lem_est_XY}-(2) and \ref{lem_est_XYZ*}-(3) 
with $r_{n,\epsilon} = \frac{1}{n}$ that 
$\frac{1}{n} \mathcal Y_n = \op(1)$ and
$\frac{1}{n} \mathcal Z_n^* = \op(1)$ hold under
\begin{description}
\item[(G25)]
$\frac{1}{n^2} \mathcal A_n \mathcal B_n^* = \op(1)$,
$\frac{1}{n^2} (\sqrt{n \epsilon^2} \lor 1) \mathcal A_n = \op(1)$,
$\frac{1}{n^2} (\epsilon^2 \lor n^{-1}) \mathcal B_n^* = \op(1)$,

\item[(G26)]
$\frac{1}{n} \mathcal B_n^* = \op(1)$ 
and $\frac{1}{n} \mathcal B_{2,n}^* = \op(1)$.
\end{description}
Note that $\frac{1}{n} \mathcal X_n = \op(1)$ holds under (G22) 
since $\frac{1}{n} \le \epsilon^{-2}$.
Therefore, we need to show that (G21)--(G26) hold in order to prove (G2).

Since (G11) and (G21) are exactly the condition [C3]$_{\alpha^*,\alpha_0}$-(i), 
it suffices to show that (G12) and (G22)--(G26) hold 
under [C1]$_{\alpha^*,\alpha_0}$-(i), (ii), [C2]$_{\alpha^*,\alpha_0}$-(i) and (ii).
\begin{proof}[\bf{Proof of (G12) and (G22)}]
\renewcommand{\qedsymbol}{}
Notice that $a^s b^{1-s} \le a \lor b$ for $a, b \in (0,\infty)$ and $s \in [0,1]$.
Since it can be expressed that
\begin{equation}\label{eps_n-1}
n = \bigl\{(n \epsilon)^2 \bigr\}^{2/3} 
\biggl( \frac{1}{n \epsilon^4} \biggr)^{1/3},
\quad
\epsilon^{-2} = \bigl\{(n \epsilon)^2 \bigr\}^{1/3} 
\biggl( \frac{1}{n \epsilon^4} \biggr)^{2/3},
\end{equation}
we have $n \lor \epsilon^{-2} \le (n \epsilon)^2 \lor \frac{1}{n \epsilon^4}$. 
Therefore, (G12) and (G22) reduce to 
$(n \epsilon)^2 \mathcal A_n = \op(1)$ and
$\frac{1}{n \epsilon^4} \mathcal A_n = \op(1)$.
We see from Lemma \ref{lem_est_AB1}-(1) 
with $r_{n,\epsilon} = (n \epsilon)^2$ and $r_{n,\epsilon} = \frac{1}{n \epsilon^4}$
that the two conditions hold under 
\begin{equation*}
\frac{n^3 \epsilon^2}{M_{(1)}^{2\alpha_0 \tand 2}} = \oo(1),
\quad
\frac{n^3 \epsilon^2 \Delta_n^{\alpha_0 \tand 2}}
{(R_{\alpha^*, \alpha_0}^{\mathrm{coef}})^2} = \oo(1),
\quad
\frac{n^3 \epsilon^4}{M_{(1)}^{2\alpha^* \tand 2}} = \oo(1),
\quad
\frac{n^3 \epsilon^4 \Delta_n^{\alpha^* \tand 1}}
{(R_{\alpha^*, \alpha_0}^{\mathrm{coef}})^2} = \oo(1)
\end{equation*}
and
\begin{equation*}
\frac{\epsilon^{-4}}{M_{(1)}^{2\alpha_0 \tand 2}} = \oo(1),
\quad
\frac{\epsilon^{-4} \Delta_n^{\alpha_0 \tand 2}}
{(R_{\alpha^*, \alpha_0}^{\mathrm{coef}})^2} = \oo(1),
\quad
\frac{\epsilon^{-2}}{M_{(1)}^{2\alpha^* \tand 2}} = \oo(1),
\quad
\frac{\epsilon^{-2} \Delta_n^{\alpha^* \tand 1}}
{(R_{\alpha^*, \alpha_0}^{\mathrm{coef}})^2} = \oo(1),
\end{equation*}
which constitute all the conditions [C1]$_{\alpha^*,\alpha_0}$-(i), (ii), 
[C2]$_{\alpha^*,\alpha_0}$-(i) and (ii) except for the last one of [C1]$_{\alpha^*,\alpha_0}$-(ii).
\end{proof}

\begin{proof}[\bf{Proof of (G23)}]
\renewcommand{\qedsymbol}{}
We see from Lemma \ref{lem_est_AB2}-(1) and (2) 
with $r_{n,\epsilon} = \epsilon^{-1}$ that (G23) holds under
\begin{equation*}
\frac{n \epsilon \Delta_n^{\alpha^* \tand 1}}
{(R_{\alpha^*, \alpha_0}^{\mathrm{coef}})^2} = \oo(1),
\quad
\frac{n \epsilon}{M_{(1)}^{2 \alpha^* \tand 2}} = \oo(1),
\quad
\frac{n \epsilon^2 \Delta_n^{\alpha^* \tand 1}}
{(R_{\alpha^*, \alpha_0}^{\mathrm{coef}})^2} = \oo(1),
\quad
\frac{\epsilon^2}{M_{(1)}^{2\alpha^* \tand 2}} = \oo(1).
\end{equation*}
Since $n \epsilon^2 \le n \epsilon 
= \epsilon (n^3 \epsilon^4)^{1/3} (\epsilon^{-2})^{2/3} 
\le \epsilon (n^3 \epsilon^4 \lor \epsilon^{-2})
\le n^3 \epsilon^4 \lor \epsilon^{-2}$, 
these balance conditions are dominated by [C1]$_{\alpha^*,\alpha_0}$-(ii) and [C2]$_{\alpha^*,\alpha_0}$-(ii). 
\end{proof}

\begin{proof}[\bf{Proof of (G24)}]
\renewcommand{\qedsymbol}{}
Since $\epsilon^{-2} \mathcal A_n^* \mathcal B_n^* 
= (\epsilon^{-2} \mathcal A_n^*) \cdot \mathcal B_n^*$,
(G24) follows under
\begin{equation*}
\epsilon^{-2} \mathcal A_n^* = \op(1),
\quad
\mathcal B_n^* = \op(1),
\quad
\epsilon^{-1} \mathcal A_{2,n}^* = \op(1),
\quad
\epsilon^{-1} \mathcal B_{1,n}^* = \op(1).
\end{equation*}
We find from Lemma \ref{lem_est_AB2}-(1) with $r_{n,\epsilon} = \epsilon^{-2}$ that 
$\epsilon^{-2} \mathcal A_n^* = \op(1)$ holds under
\begin{equation*}
\frac{n \Delta_n^{\alpha^* \tand 1}}
{(R_{\alpha^*, \alpha_0}^{\mathrm{coef}})^2} = \oo(1),
\quad
\frac{n}{M_{(1)}^{2 \alpha^* \tand 2}} = \oo(1),
\end{equation*}
which always hold under [C1]$_{\alpha^*,\alpha_0}$-(ii) and [C2]$_{\alpha^*,\alpha_0}$-(ii)
since $n = (n^3 \epsilon^4)^{1/3} (\epsilon^{-2})^{2/3} 
\le n^3 \epsilon^4 \lor \epsilon^{-2}$.

By Lemma \ref{lem_est_AB1}-(2) with $r_{n,\epsilon} = 1$, we obtain $\mathcal B_{n}^* = \op(1)$ under
\begin{equation*}
\frac{n}{M_{(1)}^{2\alpha_0 \tand 2}} = \oo(1),
\quad
\frac{n}{(R_{\alpha^*, \alpha_0}^{\mathrm{coef}})^2} = \oo(1),
\quad
\frac{n \epsilon^2}{M_{(1)}^{2\alpha^* \tand 2}} = \oo(1).
\end{equation*}
We see from $n \le n^2 = (n^3 \epsilon^2)^{2/3} (\epsilon^{-4})^{1/3} 
\le n^3 \epsilon^2 \lor \epsilon^{-4}$ that
$\frac{n}{M_{(1)}^{2\alpha_0 \tand 2}} = \oo(1)$ holds under [C2]$_{\alpha^*,\alpha_0}$-(i),
and from $n \le n^2 \Delta_n^{\alpha \tand 1}$ that
$\frac{n}{(R_{\alpha^*, \alpha_0}^{\mathrm{coef}})^2} = \oo(1)$ under [C1]$_{\alpha^*,\alpha_0}$-(ii).
We also find from $n \epsilon^2 \le n \epsilon$ and 
the proof of (G23) that
$\frac{n \epsilon^2}{M_{(1)}^{2\alpha^* \tand 2}} = \oo(1)$ holds under [C2]$_{\alpha^*,\alpha_0}$-(i).

It holds from Lemma \ref{lem_est_AB2}-(3) with $r_{n,\epsilon} = \epsilon^{-1}$ that
$\epsilon^{-1} \mathcal A_{2,n}^* = \op(1)$ under
$\frac{n^2 \Delta_n^{\alpha^* \tand 1}}
{(R_{\alpha^*, \alpha_0}^{\mathrm{coef}})^2} = \oo(1)$
and $\frac{n}{M_{(1)}^{2\alpha^* \tand 2}} = \oo(1)$.
The former is exactly the last condition of [C1]$_{\alpha^*,\alpha_0}$-(ii), 
and we has already shown that the latter holds 
in the proof of $\epsilon^{-2} \mathcal A_n^* = \op(1)$.

We see from Lemma \ref{lem_est_AB2}-(4) with $r_{n,\epsilon} = \epsilon^{-1}$ that 
$\epsilon^{-1} \mathcal B_{1,n}^* = \op(1)$ holds under
\begin{equation*}
\frac{n}{(R_{\alpha^*, \alpha_0}^{\mathrm{coef}})^2} = \oo(1),
\quad
\frac{\epsilon^2}{M_{(1)}^{2\alpha^* \tand 2}} = \oo(1),
\quad
\frac{1}{M_{(1)}^{2\alpha_0 \tand 2}} = \oo(1),
\end{equation*}
which follow under [C1]$_{\alpha^*,\alpha_0}$-(ii) from the proof of $\mathcal B_{n}^* = \op(1)$.
\end{proof}

\begin{proof}[\bf{Proof of (G25)}]
\renewcommand{\qedsymbol}{}
Since $\frac{1}{n^2} \mathcal A_n \mathcal B_n^* 
= ((\epsilon^2 \lor n^{-1})^{-1} \mathcal A_n )
(\frac{1}{n^2} (\epsilon^2 \lor n^{-1}) \mathcal B_n^*)$, (G25) holds under 
\begin{equation*}
\biggl( \frac{\epsilon}{n \sqrt{n}} \lor \frac{1}{n^2} 
\lor \frac{1}{\epsilon^2 \lor n^{-1}} \biggr) \mathcal A_n = \op(1),
\quad
\biggl( \frac{\epsilon^2}{n^2} \lor \frac{1}{n^3} \biggr) 
\mathcal B_n^* = \op(1).
\end{equation*}
We find from $\frac{\epsilon}{n \sqrt{n}} \lor \frac{1}{n^2} 
\le 1 \le \frac{1}{\epsilon^2 \lor n^{-1}} 
= \epsilon^{-2} \land n \le (n \epsilon)^2 \lor \frac{1}{n \epsilon^4}$ 
(see \eqref{eps_n-1}) that the former holds under [C1]$_{\alpha^*,\alpha_0}$-(i), (ii), 
[C2]$_{\alpha^*,\alpha_0}$-(i) and (ii), 
and from $\frac{\epsilon^2}{n^2} \lor \frac{1}{n^3} \le 1$ and the proof of (G24) that the latter holds.
\end{proof}

\begin{proof}[\bf{Proof of (G26)}]
\renewcommand{\qedsymbol}{}
$\frac{1}{n} \mathcal B_n^* = \op(1)$ holds by the proof of (G24).
We see from Lemma \ref{lem_est_AB2}-(5) with $r_{n,\epsilon} = \frac{1}{n}$ that 
$\frac{1}{n} \mathcal B_{2,n}^* = \op(1)$ holds under 
\begin{equation*}
\frac{1}{(R_{\alpha^*, \alpha_0}^{\mathrm{coef}})^2} = \oo(1),
\quad 
\frac{1}{M_{(1)}} = \oo(1), 
\quad
\frac{\epsilon^2}{n M_{(1)}^{2\alpha^* \tand 2}} = \oo(1),
\quad
\frac{1}{M_{(1)}^{2 \alpha_0 \tand 2}} = \oo(1),
\end{equation*}
which are clearly always true.
\end{proof}

Therefore, we obtain 
$(\widehat \lambda_{\ell}, \widehat \mu_{\ell}) = (\widehat \lambda, \widehat \mu) 
\pto (\lambda^*, \mu^*) = (\lambda_{\ell}^*, \mu_{\ell}^*)$,
which together with \eqref{estimator_lam_mu}, Theorem \ref{th2} and \eqref{true_values} yields
\begin{align*}
\widehat \theta_0 
&= -\widehat \lambda_{\ell}
+\widehat \theta_2 \biggl( \frac{\widehat \kappa^2 + \widehat \eta^2}{4} +\pi^2 |\ell|^2 \biggr)
\pto -\lambda_{\ell}^* +\theta_2^*\biggl( \frac{(\kappa^*)^2 +(\eta^*)^2}{4} +\pi^2 |\ell|^2 \biggr)
= \theta_0^*,
\\
\widehat \mu_0 
&= \widehat \mu_{\ell} -\pi^2 |\ell|^2
\pto \mu_{\ell}^* -\pi^2 |\ell|^2
= \mu_0^*.
\end{align*}
\end{proof}

\begin{proof}[\bf{Proof of \eqref{asym_norm}}]
We prove that (G1)--(G4) hold under [C1]$_{\alpha^*,\alpha_0}$-(i), (ii), [C2]$_{\alpha^*,\alpha_0}$, [C3]$_{\alpha^*,\alpha_0}$ in order to obtain 
\begin{equation}\label{asym_lam_mu}
\begin{pmatrix}
\epsilon^{-1} (\widehat \lambda -\lambda^*)
\\
\sqrt{n} (\widehat \mu -\mu^*)
\end{pmatrix}
\dto \mathrm{N} (0, \mathcal I).
\end{equation}
Since we have already shown that (G1) and (G2) hold true 
under [C1]$_{\alpha^*,\alpha_0}$-(i), (ii), [C2]$_{\alpha^*,\alpha_0}$-(i), (ii) 
and [C3]$_{\alpha^*,\alpha_0}$-(i) in the proof of \eqref{cons}, 
we will show that the additional conditions [C2]$_{\alpha^*,\alpha_0}$-(iii) and [C3]$_{\alpha^*,\alpha_0}$-(ii)
are necessary in order to obtain (G3) and (G4).

We consider (G3). Since we have
\begin{align*}
\widehat {\mathbb U}_{3,n}^\epsilon
&:= 
\Bigl| 
L_n^{\epsilon} (\lambda^*,\mu^*; \widehat {\mb x},\widehat \alpha) 
-L_n^{\epsilon} (\lambda^*,\mu^*; {\mb x},\alpha^*)
\Bigr|
\\
&\le 
\Bigl| 
L_n^{\epsilon} (\lambda^*,\mu^*; \widehat {\mb x},\widehat \alpha) 
-L_n^{\epsilon} (\lambda^*,\mu^*; \widehat {\mb x},\alpha^*)
\Bigr|
\\
&\quad+\Bigl| 
L_n^{\epsilon} (\lambda^*,\mu^*; \widehat {\mb x}, \alpha^*) 
-L_n^{\epsilon} (\lambda^*,\mu^*; {\mb x},\alpha^*)
\Bigr|
\\
&\le 
\sup_{\mu \in \Xi_{\mu}} \Bigl| 
L_n^{\epsilon} (\lambda^*,\mu; \widehat {\mb x},\widehat \alpha) 
-L_n^{\epsilon} (\lambda^*,\mu; \widehat {\mb x},\alpha^*)
\Bigr|
\\
&\quad
+\sup_{\mu \in \Xi_{\mu}} \Bigl| 
L_n^{\epsilon} (\lambda^*,\mu; \widehat {\mb x}, \alpha^*) 
-L_n^{\epsilon} (\lambda^*,\mu; {\mb x},\alpha^*)
\Bigr|
\\
&= \widehat {\mathbb W}_{3,n}^\epsilon + \widehat {\mathbb V}_{3,n}^\epsilon 
\end{align*}
and (R12) in Proposition \ref{propR2}, we find that (G3) holds under
\begin{equation}\label{VW-3}
\widehat {\mathbb V}_{3,n}^\epsilon = \op(1),
\quad
\widehat {\mathbb W}_{3,n}^\epsilon = \op(1).
\end{equation}
Therefore, we see from Lemmas \ref{lem_est_V}-(3) and \ref{lem_est_W}-(3) that
\eqref{VW-3} holds under $\frac{\sqrt{n}}{\epsilon^2} \mathcal X_n^* =\op(1)$, 
$\epsilon^{-1} \mathcal Y_n^* = \op(1)$ and 
\begin{description}
\item[(G31)]
$\frac{\sqrt{n}}{R_{\alpha^*,\alpha_0}^{\mathrm{damp}}} = \oo(1)$,
\end{description}
and from Lemma \ref{lem_est_XYZ*}-(1) with $r_{n,\epsilon} = \frac{\sqrt{n}}{\epsilon^{2}}$ 
and (2) with $r_{n,\epsilon} = \epsilon^{-1}$ that 
$\frac{\sqrt{n}}{\epsilon^2} \mathcal X_n^* = \op(1)$ 
and $\epsilon^{-1} \mathcal Y_n^* = \op(1)$ hold under
\begin{description}
\item[(G32)]
$\frac{\sqrt{n}}{\epsilon^2} \mathcal A_n^* = \op(1)$, 
$\frac{\sqrt{n}}{\epsilon^2} \mathcal A_{1,n}^* = \op(1)$,

\item[(G33)]
$\epsilon^{-2} \mathcal A_n^* \mathcal B_n^* = \op(1)$,
$\epsilon^{-1} \mathcal A_{2,n}^* = \op(1)$ 
and $\epsilon^{-1} \mathcal B_{1,n}^* = \op(1)$.
\end{description}
Notice that (G33) is the same condition as (G24).
Hence, we need to show (G31)--(G33) in order to obtain (G3).

We next consider (G4). We define
\begin{equation*}
{\mathbb U}_{4,n}^\epsilon
= \sup_{u \in [0,1]}
\bigl| 
K_n^{\epsilon} (\widehat \lambda_u, \widehat \mu_u; {\mb x},\alpha^*) 
-2 \mathcal I^{-1}
\bigr|.
\end{equation*}
Since it holds that
\begin{align*}
\widehat {\mathbb U}_{4,n}^\epsilon
&= \sup_{u \in [0,1]}
\bigl| 
K_n^{\epsilon} (\widehat \lambda_u, \widehat \mu_u; 
\widehat {\mb x},\widehat \alpha) -2 \mathcal I^{-1}
\bigr|
\\
&\le 
\sup_{u \in [0,1]}
\Bigl| 
K_n^{\epsilon} (\widehat \lambda_u, \widehat \mu_u; 
\widehat {\mb x}, \widehat \alpha) 
- K_n^{\epsilon} (\widehat \lambda_u, \widehat \mu_u; 
\widehat {\mb x}, \alpha^*) 
\Bigr|
\\
&\quad+ \sup_{u \in [0,1]}
\Bigl| 
K_n^{\epsilon} (\widehat \lambda_u, \widehat \mu_u; 
\widehat {\mb x}, \alpha^*) 
- K_n^{\epsilon} (\widehat \lambda_u, \widehat \mu_u; {\mb x}, \alpha^*)
\Bigr|
\\
&\quad+ \sup_{u \in [0,1]}
\bigl| 
K_n^{\epsilon} (\widehat \lambda_u, \widehat \mu_u; {\mb x},\alpha^*) 
-2 \mathcal I^{-1}
\bigr|
\\
&\le
\sup_{(\lambda, \mu) \in \Xi_{\nu}}
\Bigl| 
K_n^{\epsilon} (\lambda, \mu; \widehat {\mb x}, \widehat \alpha) 
- K_n^{\epsilon} (\lambda, \mu; \widehat {\mb x}, \alpha^*) 
\Bigr|
\\
&\quad+ \sup_{(\lambda, \mu) \in \Xi_{\nu}} \Bigl| 
K_n^{\epsilon} (\lambda, \mu; \widehat {\mb x}, \alpha^*) 
- K_n^{\epsilon} (\lambda, \mu; {\mb x}, \alpha^*)
\Bigr|
\\
&\quad+\sup_{u \in [0,1]}
\bigl| 
K_n^{\epsilon} (\widehat \lambda_u, \widehat \mu_u; {\mb x},\alpha^*) 
-2 \mathcal I^{-1}
\bigr|
\\
&= \widehat {\mathbb W}_{4,n}^\epsilon + \widehat {\mathbb V}_{4,n}^\epsilon 
+{\mathbb U}_{4,n}^\epsilon
\end{align*}
and from (R13) in Proposition \ref{propR2} 
that ${\mathbb U}_{4,n}^\epsilon = \op(1)$ holds
under $\epsilon^{-1}(\widehat \lambda -\lambda^*) = \Op(1)$ 
and $\widehat \mu \pto \mu^*$, 
we obtain (G4) under
\begin{equation}\label{VW-4}
\widehat {\mathbb V}_{4,n}^\epsilon = \op(1),
\quad
\widehat {\mathbb W}_{4,n}^\epsilon = \op(1).
\end{equation}
We find from Lemmas \ref{lem_est_V}-(4) and \ref{lem_est_W}-(4) that
\eqref{VW-4} holds under
$\epsilon^{-2} \mathcal X_n = \op(1)$, 
$\frac{1}{\epsilon \sqrt{n}} \mathcal Y_n = \op(1)$, 
$\frac{1}{n} \mathcal Z_n^* = \op(1)$ and 
\begin{description}
\item[(G41)]
$\frac{(n \epsilon^2)^{-1/2}}{R_{\alpha^*,\alpha_0}^{\mathrm{damp}}} = \oo(1)$,
\end{description}
and from Lemmas \ref{lem_est_XY}-(1) with $r_{n,\epsilon} = \epsilon^{-2}$, 
(2) with $r_{n,\epsilon} = \frac{1}{\epsilon \sqrt{n}}$ 
and \ref{lem_est_XYZ*}-(3) with $r_{n,\epsilon} = \frac{1}{n}$ that
$\epsilon^{-2} \mathcal X_n = \op(1)$, 
$\frac{1}{\epsilon \sqrt{n}} \mathcal Y_n = \op(1)$
and $\frac{1}{n} \mathcal Z_n^* = \op(1)$ hold under
\begin{description}
\item[(G42)]
$\epsilon^{-2} \mathcal A_n = \op(1)$, 
$\frac{1}{n \epsilon^4} \mathcal A_n = \op(1)$,

\item[(G43)]
$\frac{1}{n \epsilon^2} \mathcal A_n \mathcal B_n^* = \op(1)$,
$\frac{1}{n \epsilon^2} (\sqrt{n \epsilon^2} \lor 1) \mathcal A_n = \op(1)$, 
$\frac{1}{n \epsilon^2} (\epsilon^2 \lor n^{-1}) \mathcal B_n^* = \op(1)$,

\item[(G44)]
$\frac{1}{n} \mathcal B_n^* = \op(1)$ and $\frac{1}{n} \mathcal B_{2,n}^* = \op(1)$.
\end{description}
Note that (G42) and (G44) are the same conditions as (G22) and (G26), respectively.
Consequently, it suffices to prove that (G41)--(G44) hold in order to show (G4).

(G31) is precisely [C3]$_{\alpha^*,\alpha_0}$-(ii). 
Since $(n \epsilon^2)^{-1/2} = (n \epsilon^2)^{1/4} \{(n \epsilon^2)^{-1}\}^{3/4}
\le n \epsilon^2 \lor (n \epsilon^2)^{-1}$,
(G41) always holds when (G11) and (G21) hold. 
(G33), (G42) and (G44) are the same conditions as (G24), (G22) and (G26), respectively,
and we have already shown that all of them hold under [C1]$_{\alpha^*,\alpha_0}$-(i), (ii), 
[C2]$_{\alpha^*,\alpha_0}$-(i), (ii) and [C3]$_{\alpha^*,\alpha_0}$-(i).
Therefore, all that remains is to show that (G32) and (G43) hold 
under [C1]$_{\alpha^*,\alpha_0}$-(i), (ii) and [C2]$_{\alpha^*,\alpha_0}$.
\begin{proof}[\bf{Proof of (G32)}]
\renewcommand{\qedsymbol}{}
We see from Lemma \ref{lem_est_AB2}-(1) and (2) with $r_{n,\epsilon} = \frac{\sqrt{n}}{\epsilon^2}$ 
that (G32) holds under 
$\frac{n \sqrt{n} \Delta_n^{\alpha \tand 1}}
{(R_{\alpha^*, \alpha_0}^{\mathrm{coef}})^2} = \oo(1)$,
$\frac{n \sqrt{n}}{M_{(1)}^{2 \alpha^* \tand 2}} = \oo(1)$, 
$\frac{n^2 \Delta_n^{\alpha^* \tand 1}}
{(R_{\alpha^*, \alpha_0}^{\mathrm{coef}})^2} = \oo(1)$
and $\frac{n}{M_{(1)}^{2\alpha^* \tand 2}} = \oo(1)$, 
which reduce to 
\begin{equation*}
\frac{n^2 \Delta_n^{\alpha \tand 1}}
{(R_{\alpha^*, \alpha_0}^{\mathrm{coef}})^2} = \oo(1),
\quad
\frac{n \sqrt{n}}{M_{(1)}^{2 \alpha^* \tand 2}} = \oo(1).
\end{equation*}
These are precisely the last condition of [C1]$_{\alpha^*,\alpha_0}$-(ii) and [C2]$_{\alpha^*,\alpha_0}$-(iii), respectively.
\end{proof}

\begin{proof}[\bf{Proof of (G43)}]
\renewcommand{\qedsymbol}{}
Since $\frac{1}{n \epsilon^2} \mathcal A_n \mathcal B_n^* 
= (\frac{1}{n \epsilon^2} \mathcal A_n) \cdot \mathcal B_n^*$, 
(G43) holds under
\begin{equation*}
\Bigl( \frac{1}{n \epsilon^2} \lor \frac{1}{\sqrt{n \epsilon^2}} \Bigr) 
\mathcal A_n = \op(1),
\quad 
\Bigl( 1 \lor \frac{1}{n} \lor \frac{1}{(n \epsilon)^2} \Bigr) \mathcal B_n^* = \op(1).
\end{equation*}
Since
$\frac{1}{n \epsilon^2} \lor \frac{1}{\sqrt{n \epsilon^2}} 
\le \epsilon^{-2} \le (n \epsilon)^2 \lor \frac{1}{n \epsilon^4}$ 
(see \eqref{eps_n-1}), 
we find from the proofs of (G12) and (G22) that the former holds. 
We have already shown that
$(1 \lor \frac{1}{n}) \mathcal B_n^* = \mathcal B_n^* =  \op(1)$ holds 
in the proof of (G24).
We see from Lemma \ref{lem_est_AB1}-(2) with $r_{n,\epsilon} = \frac{1}{(n \epsilon)^2}$ that
$\frac{1}{(n \epsilon)^2} \mathcal B_n^* = \op(1)$ holds under
\begin{equation*}
\frac{(n \epsilon^2)^{-1}}{M_{(1)}^{2\alpha_0 \tand 2}} = \oo(1),
\quad 
\frac{(n \epsilon^2)^{-1}}{(R_{\alpha^*, \alpha_0}^{\mathrm{coef}})^2} = \oo(1),
\quad
\frac{1}{n M_{(1)}^{2\alpha^* \tand 2}} = \oo(1).
\end{equation*}
These conditions hold true under [C1]$_{\alpha^*,\alpha_0}$-(ii) and [C2]$_{\alpha^*,\alpha_0}$-(i) since
$(n \epsilon^2)^{-1} \le \epsilon^{-4}$ and $(n \epsilon^2)^{-1} \le \epsilon^{-2} \Delta_n^{\alpha^* \tand 1}$.
\end{proof}

We thus have \eqref{asym_lam_mu}.
Since it follows from Theorem \ref{th2} and [C1]$_{\alpha^*,\alpha_0}$-(iii) that
\begin{align*}
\epsilon^{-1} (\widehat \theta_0 -\theta_0^*)
&= -\epsilon^{-1} (\widehat \lambda -\lambda^*) 
\\
&\quad+\frac{\epsilon^{-1}}{R_{\alpha^*, \alpha_0}^{\mathrm{coef}}} 
\cdot R_{\alpha^*, \alpha_0}^{\mathrm{coef}}
\Biggl\{ \widehat \theta_2 \biggl( \frac{\widehat \kappa^2 + \widehat \eta^2}{4} +\pi^2 |\ell|^2 \biggr)
-\theta_2^*\biggl( \frac{(\kappa^*)^2 +(\eta^*)^2}{4} +\pi^2 |\ell|^2 \biggr)
\Biggr\}
\\
&= -\epsilon^{-1} (\widehat \lambda -\lambda^*) +\op(1)
\\
&= (-1, 0)
\begin{pmatrix}
\epsilon^{-1} (\widehat \lambda -\lambda^*)
\\
\sqrt{n} (\widehat \mu -\mu^*)
\end{pmatrix}
+\op(1)
\end{align*}
and
\begin{equation*}
\sqrt{n} (\widehat \mu_0 -\mu_0^*)
= \sqrt{n} (\widehat \mu -\mu^*)
= (0, 1)
\begin{pmatrix}
\epsilon^{-1} (\widehat \lambda -\lambda^*)
\\
\sqrt{n} (\widehat \mu -\mu^*)
\end{pmatrix},
\end{equation*}
we obtain by \eqref{asym_lam_mu},
\begin{equation*}
\begin{pmatrix}
\epsilon^{-1}(\widehat \theta_0 - \theta_0^*)
\\
\sqrt n(\widehat \mu_0-\mu_0^*)
\end{pmatrix}
= 
\begin{pmatrix}
-1 & 0
\\
0 & 1
\end{pmatrix}
\begin{pmatrix}
\epsilon^{-1} (\widehat \lambda -\lambda^*)
\\
\sqrt{n} (\widehat \mu -\mu^*)
\end{pmatrix}
+\op(1)
\dto \mathrm{N} (0, \mathcal I).
\end{equation*}
\end{proof}

\subsection{Auxiliary results}\label{sec5-4}
We define the estimator of the damping parameter $\alpha$ in the simple form
\begin{equation*}
\widetilde \alpha = \frac{\log(S_1) -\log(\widetilde S_2)}{\log(4)},
\quad
S_1 = \sum_{k=1}^{m_2-1} \sum_{j=1}^{m_1-1} \sum_{i=1}^{N-3} (\widehat T_{i,j,k} X)^2,
\quad
\widetilde S_2 = \sum_{k=1}^{m_2} \sum_{j=1}^{m_1} \sum_{i=1}^N (T_{i,j,k}X)^2
\end{equation*}
instead of $\widehat \alpha$ given in \eqref{est_alpha}. We then obtain the following result.
\begin{subthm}\label{lem_tilde_alpha}
Let $\alpha_0 \in (0,3)$ and $\delta/\sqrt{\Delta} \equiv r \in (0,\infty)$.
It holds that under [A1]$_{\alpha_0}$ and [B]$_{\alpha^*,\alpha_0}$, 
\begin{equation*}
\bigl( R_{\alpha^*,\alpha_0}^{\mathrm{damp}} \land \sqrt{m} \bigr)
(\widetilde \alpha -\alpha^*) = \Op(1)
\end{equation*}
as $m \to \infty$, $N \to \infty$ and $\epsilon \to 0$.
\end{subthm}

\begin{proof}
We define
\begin{equation*}
\widetilde {\mathcal Z}_2 = \frac{\widetilde S_2}{\epsilon^2 m N \Delta^{\alpha^*}},
\quad
\widetilde R_1 = R_1 \land \sqrt{m}.
\end{equation*}
We obtain 
\begin{equation*}
R_1 \bigl( \widetilde {\mathcal Z}_2 -\EE [\widetilde {\mathcal Z}_2] \bigr) = \Op(1)
\end{equation*}
in the same way as in the proof of \eqref{eq-Z1-1}.
Since we have 
\begin{align*}
\int_{b}^{1-b} \ee^{-\kappa y} \dd y
+ \int_{\widetilde y_1}^{\widetilde y_{m_1-1}} \ee^{-\kappa y} \dd y 
&= 2\int_{b}^{1-b} \ee^{-\kappa y} \dd y
- \int_{b}^{\widetilde y_1} \ee^{-\kappa y} \dd y
- \int_{\widetilde y_{m_1-1}}^{1-b} \ee^{-\kappa y} \dd y 
\\
&= 2\int_{b}^{1-b} \ee^{-\kappa y} \dd y + \OO(\delta)
\end{align*}
and
\begin{align*}
g_{r,\alpha^*}^{(\mathrm{m})}(\vartheta)
&=
\frac{\phi_{r,\alpha^*}(\theta_2)}{4(1-2b)^2} \biggl(
\int_{b}^{1-b} \ee^{-\kappa y} \dd y
+ \int_{\widetilde y_1}^{\widetilde y_{m_1-1}} \ee^{-\kappa y} \dd y 
\biggr) 
\\
&\quad \times
\biggl(
\int_{b}^{1-b} \ee^{-\eta z} \dd z
+ \int_{\widetilde z_1}^{\widetilde z_{m_2-1}} \ee^{-\eta z} \dd z 
\biggr) 
\\
&= 
\frac{\phi_{r,\alpha^*}(\theta_2)}{(1-2b)^2}
\int_{b}^{1-b} \ee^{-\kappa y} \dd y
\int_{b}^{1-b} \ee^{-\eta z} \dd z
+ \OO(\delta),
\end{align*}
we see from \eqref{eq-pf-2} that
\begin{align*}
\EE [ \widetilde {\mathcal Z}_2 ] &= 
\frac{1}{\epsilon^2 m N \Delta^{\alpha^*}}
\sum_{k=1}^{m_2} \sum_{j=1}^{m_1} \sum_{i=1}^N \EE \bigl[ (T_{i,j,k}X)^2 \bigr]
\\
&= 
\frac{\phi_{r,\alpha^*}(\theta_2)}{(1-2b)^2}
\int_{b}^{1-b} \ee^{-\kappa y} \dd y
\int_{b}^{1-b} \ee^{-\eta z} \dd z
+\OO (R_1^{-1})
\\
&= g_{r,\alpha^*}^{(\mathrm{m})}(\vartheta) +\OO (R_1^{-1} \lor \delta)
\\
&= g_{r,\alpha^*}^{(\mathrm{m})}(\vartheta) +\OO (\widetilde R_1^{-1}).
\end{align*}
We thus obtain 
\begin{align*}
\widetilde R_1 \bigl( \widetilde {\mathcal Z}_2 -g_{r,\alpha^*}^{(\mathrm{m})}(\vartheta) \bigr)
&= \widetilde R_1 
\bigl( \widetilde {\mathcal Z}_2 -\EE [ \widetilde {\mathcal Z}_2 ] \bigr)
+\widetilde R_1  \bigl( \EE [\mathcal Z_1] 
-g_{r,\alpha^*}^{(\mathrm{m})}(\vartheta) \bigr)
\\
&= \Op(1) +\OO(1)
\\
&= \Op(1),
\end{align*}
which together with 
$\widetilde R_1 \bigl| \mathcal Z_1 -g_{r,\alpha^*}^{(\mathrm{m})}(\vartheta) \bigr|
\le R_1 \bigl| \mathcal Z_1 -g_{r,\alpha^*}^{(\mathrm{m})}(\vartheta) \bigr| = \Op(1)$
yields the desired result.
\end{proof}

We give the following lemma which yields the dominant conditions 
for the sufficient conditions [C1]$_{\alpha^*,\alpha_0}$, [C2]$_{\alpha^*,\alpha_0}$ and [C3]$_{\alpha^*,\alpha_0}$ 
under the balance conditions between each $n$ and $\epsilon$.

\begin{subthm}\label{lem_bal}
Let $\alpha, \alpha_0 \in (0,3)$ and $\delta/\sqrt{\Delta} \equiv r \in (0,\infty)$.
\begin{enumerate}
\item[(1)]
[C1]$_{\alpha,\alpha_0}$-(i) and [C2]$_{\alpha,\alpha_0}$-(i) hold under one of the following conditions.
\begin{enumerate}
\item[a)]
$\frac{1}{n \epsilon^2} = \OO(1)$, 
$\frac{n^3 \epsilon^2 \Delta_n^{\alpha_0 \tand 2}}{(R_{\alpha,\alpha_0}^{\mathrm{coef}})^2} = \oo(1)$
and $\frac{n^3 \epsilon^2}{M_{(1)}^{2\alpha_0 \tand 2}} = \oo(1)$,

\item[b)]
$n \epsilon^2 = \oo(1)$, 
$\frac{\epsilon^{-4} \Delta_n^{\alpha_0 \tand 2}}{(R_{\alpha,\alpha_0}^{\mathrm{coef}})^2} = \oo(1)$
and $\frac{\epsilon^{-4}}{M_{(1)}^{2\alpha_0 \tand 2}} = \oo(1)$.
\end{enumerate}

\item[(2)]
[C1]$_{\alpha,\alpha_0}$-(ii) holds under one of the following conditions.
\begin{enumerate}
\item[a)]
$\frac{1}{n \epsilon^4} = \OO(1)$ and
$\frac{n^3 \epsilon^4 \Delta_n^{\alpha \tand 1}}{(R_{\alpha,\alpha_0}^{\mathrm{coef}})^2} = \oo(1)$,

\item[b)]
$n \epsilon^4 = \oo(1)$, $\frac{1}{n \epsilon} = \OO(1)$ and
$\frac{n^2 \Delta_n^{\alpha \tand 1}}{(R_{\alpha,\alpha_0}^{\mathrm{coef}})^2} = \oo(1)$,

\item[c)]
$n \epsilon = \oo(1)$ and
$\frac{\epsilon^{-2} \Delta_n^{\alpha \tand 1}}{(R_{\alpha,\alpha_0}^{\mathrm{coef}})^2} = \oo(1)$. 
\end{enumerate}

\item[(3)]
[C2]$_{\alpha,\alpha_0}$-(ii) holds under one of the following conditions.
\begin{enumerate}
\item[a)]
$\frac{1}{n \epsilon^2} = \OO(1)$ and
$\frac{n^3 \epsilon^4}{M_{(1)}^{2\alpha \tand 2}} = \oo(1)$,

\item[b)]
$n \epsilon^2 = \oo(1)$ and
$\frac{\epsilon^{-2}}{M_{(1)}^{2\alpha \tand 2}} = \oo(1)$.
\end{enumerate}

\item[(4)]
[C2]$_{\alpha,\alpha_0}$-(ii) and (iii) hold under one of the following conditions.
\begin{enumerate}
\item[a)]
$\frac{1}{n \epsilon^{8/3}} = \OO(1)$ and
$\frac{n^3 \epsilon^4}{M_{(1)}^{2\alpha \tand 2}} = \oo(1)$,

\item[b)]
$n \epsilon^{8/3} = \oo(1)$, $\frac{1}{n \epsilon^{4/3}} = \OO(1)$ and
$\frac{n \sqrt{n}}{M_{(1)}^{2\alpha \tand 2}} = \oo(1)$,

\item[c)]
$n \epsilon^{4/3} = \oo(1)$ and
$\frac{\epsilon^{-2}}{M_{(1)}^{2\alpha \tand 2}} = \oo(1)$.
\end{enumerate}

\item[(5)]
[C3]$_{\alpha,\alpha_0}$-(i) holds under one of the following conditions.
\begin{enumerate}
\item[a)]
$\frac{1}{n \epsilon^2} = \OO(1)$ and
$\frac{n \epsilon^2}{R_{\alpha,\alpha_0}^{\mathrm{damp}}} = \oo(1)$,

\item[b)]
$n \epsilon^2 = \oo(1)$ and
$\frac{(n \epsilon^2)^{-1}}{R_{\alpha,\alpha_0}^{\mathrm{damp}}} = \oo(1)$.
\end{enumerate}

\item[(6)]
[C3]$_{\alpha,\alpha_0}$ holds under one of the following conditions.
\begin{enumerate}
\item[a)]
$\frac{1}{n \epsilon^4} = \OO(1)$ and
$\frac{n \epsilon^2}{R_{\alpha,\alpha_0}^{\mathrm{damp}}} = \oo(1)$,

\item[b)]
$n \epsilon^4 = \oo(1)$, $\frac{1}{n \epsilon^{4/3}} = \oo(1)$
and $\frac{\sqrt{n}}{R_{\alpha,\alpha_0}^{\mathrm{damp}}} = \oo(1)$,

\item[c)]
$n \epsilon^{4/3} = \oo(1)$ and
$\frac{(n \epsilon^2)^{-1}}{R_{\alpha,\alpha_0}^{\mathrm{damp}}} = \oo(1)$.
\end{enumerate}
\end{enumerate}
\end{subthm}

\begin{proof}
Note that the following holds for all $a > b > 0$.
\begin{equation}\label{eps_ab}
\left\{
\begin{split}
\frac{1}{n \epsilon^b} &= \oo(1) \text{ under } \frac{1}{n \epsilon^a} = \OO(1),
\\
n \epsilon^a &= \oo(1) \text{ under } n \epsilon^b = \OO(1).
\end{split}
\right.
\end{equation}
\begin{enumerate}
\item[(1)]
It holds from $n^3 \epsilon^2 \lor \epsilon^{-4} = n^2 (n \epsilon^2 \lor \frac{1}{(n \epsilon^2)^2})$ that 
\begin{equation*}
n^3 \epsilon^2 \lor \epsilon^{-4} = 
\begin{cases}
\OO(n^3 \epsilon^2) & \text{under $\frac{1}{n \epsilon^2} = \OO(1)$},
\\
\OO(\epsilon^{-4}) & \text{under $n \epsilon^2 = \oo(1)$}.
\end{cases}
\end{equation*}

\item[(2)]
Since $n^3 \epsilon^4 \lor \epsilon^{-2} \lor n^2 = n^2 (n \epsilon^4 \lor \frac{1}{(n \epsilon)^2} \lor 1)$, we see from \eqref{eps_ab} with $(a, b)= (4, 1)$ that
\begin{equation*}
n^3 \epsilon^4 \lor \epsilon^{-2} \lor n^2 = 
\begin{cases}
\OO(n^3 \epsilon^4) & \text{under $\frac{1}{n \epsilon^4} = \OO(1)$},
\\
\OO(n^2) & \text{under $n \epsilon^4 = \oo(1)$ and $\frac{1}{n \epsilon} = \OO(1)$},
\\
\OO(\epsilon^{-2}) & \text{under $n \epsilon = \oo(1)$}.
\end{cases}
\end{equation*}

\item[(3)]
Since $n^3 \epsilon^4 \lor \epsilon^{-2} = (n \epsilon)^2 (n \epsilon^2 \lor \frac{1}{(n \epsilon^2)^2})$, we have
\begin{equation*}
n^3 \epsilon^4 \lor \epsilon^{-2} = 
\begin{cases}
\OO(n^3 \epsilon^4) & \text{under $\frac{1}{n \epsilon^2} = \OO(1)$},
\\
\OO(\epsilon^{-2}) & \text{under $n \epsilon^2 = \oo(1)$}.
\end{cases}
\end{equation*}

\item[(4)]
Since $n^3 \epsilon^4 \lor \epsilon^{-2} \lor n^{3/2} = 
n^{3/2} ( (n \epsilon^{8/3})^{3/2} \lor \frac{1}{(n \epsilon^{4/3})^{3/2}} \lor 1 )$, 
we obtain by \eqref{eps_ab} with $(a, b)= (8/3, 4/3)$,
\begin{equation*}
n^3 \epsilon^4 \lor \epsilon^{-2} \lor n^{3/2} = 
\begin{cases}
\OO(n^3 \epsilon^4) & \text{under $\frac{1}{n \epsilon^{8/3}} = \OO(1)$},
\\
\OO(n^{3/2}) & \text{under $n \epsilon^{8/3} = \oo(1)$ 
and $\frac{1}{n \epsilon^{4/3}} = \OO(1)$},
\\
\OO(\epsilon^{-2}) & \text{under $n \epsilon^{4/3} = \oo(1)$}.
\end{cases}
\end{equation*}

\item[(5)]
It is obvious.

\item[(6)]
 It follows from 
$n \epsilon^2 \lor \frac{1}{n \epsilon^2} \lor n^{1/2} 
= n^{1/2} ( (n \epsilon^4)^{1/2} \lor \frac{1}{(n \epsilon^{4/3})^{3/2}} \lor 1 )$
and \eqref{eps_ab} with $(a, b)= (4, 4/3)$ that 
\begin{equation*}
n \epsilon^2 \lor \frac{1}{n \epsilon^2} \lor n^{1/2} = 
\begin{cases}
\OO(n \epsilon^2) & \text{under $\frac{1}{n \epsilon^4} = \OO(1)$},
\\
\OO(n^{1/2}) & \text{under $n \epsilon^4 = \oo(1)$ 
and $\frac{1}{n \epsilon^{4/3}} = \OO(1)$},
\\
\OO((n \epsilon^2)^{-1}) & \text{under $n \epsilon^{4/3} = \oo(1)$}.
\end{cases}
\end{equation*}
\end{enumerate}
\end{proof}

For the function $\phi_{r,\alpha}(\theta_2)$ given in \eqref{phi}, we obtain the following result.
\begin{subthm}\label{lem_phi}
For $r \in (0, \infty)$ and $l \in \{ 0, 1, 2 \}$,
the function $\Xi_{\alpha} \times \Xi_{\theta_2} \ni (\alpha, \theta_2) \mapsto \pd_{\theta_2}^l \phi_{r,\alpha}(\theta_2)$
belongs to $C^{1,0}(\Xi_{\alpha} \times \Xi_{\theta_2})$.
\end{subthm}
\begin{proof}
Since we have 
\begin{align*}
\phi_{r,\alpha} (\theta_2)
&= \frac{2}{\theta_2 \pi} 
\int_0^\infty \frac{1-\ee^{-\theta_2 x^2}}{x^{1+2\alpha}} 
\bigl( J_0(\sqrt{2}r x) -2J_0(r x)+1 \bigr) \dd x
\\
&=: \frac{2}{\theta_2 \pi} \int_0^\infty q_{r}(x;\alpha, \theta_2) \dd x
\\
&=: \frac{2}{\theta_2 \pi} \overline \phi_{r,\alpha}(\theta_2),
\end{align*}
we will show the following three statements in order to arrive at the desired result.
\begin{enumerate}
\item[(a)]
For any $\alpha \in (0,3)$ and $r \in (0, \infty)$,
the function $\Xi_{\theta_2} \ni \theta_2 \mapsto \overline \phi_{r,\alpha}(\theta_2)$ 
is three times differentiable.

\item[(b)]
For any $\theta_2 \in \Xi_{\theta_2}$, $r \in (0, \infty)$ and $l \in \{ 0, 1, 2 \}$,
the function $\Xi_{\alpha} \ni \alpha \mapsto \pd_{\theta_2}^l 
\overline \phi_{r,\alpha}(\theta_2)$ is twice differentiable.

\item[(c)]
For $r \in (0, \infty)$, $k \in \{0, 1 \}$ and $l \in \{ 0, 1, 2 \}$, 
the function $\Xi_{\alpha} \times \Xi_{\theta_2} \ni (\alpha, \theta_2) 
\mapsto \pd_{\alpha}^k \pd_{\theta_2}^l \overline \phi_{r,\alpha}(\theta_2)$ is continuous.

\end{enumerate}

We show (a). Note that 
\begin{equation*}
q_{r}(x;\alpha,\theta_2)
= \frac{1-\ee^{-\theta_2 x^2}}{x^{1+2\alpha}} 
\bigl( J_0(\sqrt{2}r x) -2J_0(r x)+1 \bigr)
\end{equation*}
and
\begin{equation*}
\pd_{\theta_2}^l q_{r}(x;\alpha,\theta_2)
= (-1)^{l+1} \frac{\ee^{-\theta_2 x^2}}{x^{1 -2l +2\alpha}} 
\bigl( J_0(\sqrt{2}r x) -2J_0(r x)+1 \bigr)
\end{equation*}
for $l \in \{1, 2, 3 \}$.
The integral representation of the Bessel function
\begin{equation*}
J_0(x) = \frac{1}{\pi} \int_0^{\pi} \cos( x \sin(t) ) \dd t
\end{equation*}
and \eqref{Bessel} show $\sup_{x \in \mathbb R}|J_0(x)| \le 1$ and
\begin{align*}
J_0(\sqrt{2}x) -2J_0(x)+1 
&= \biggl( 
1 +\sum_{k = 1}^{\infty} \frac{(-1)^k}{(k!)^2} \Bigl( \frac{x^2}{2} \Bigr)^{k}
\biggr)
-2 \biggl(
1 +\sum_{k = 1}^{\infty} \frac{(-1)^k}{(k!)^2} \Bigl( \frac{x^2}{4} \Bigr)^{k}
\biggr)
+1
\\
&= \sum_{k=2}^{\infty} \frac{(-1)^k}{(k!)^2} 
\Bigl( \frac{1}{2^k} -\frac{2}{4^k} \Bigr) x^{2k},
\end{align*}
which yield 
\begin{equation*}
J_0(\sqrt{2} x) -2J_0(x)+1 \lesssim 
\begin{cases}
x^4 & (x \to 0),
\\
1 & (x \to \infty).
\end{cases}
\end{equation*}
Since we have
\begin{equation*}
\sup_{(x, \theta_2) \in \mathbb R \times \Xi_{\theta_2}} 
\frac{1-\ee^{-\theta_2 x^2}}{1-\ee^{-\overline{\theta_2} x^2}} 
\lesssim 1
\end{equation*}
for $\overline{\theta_2} \in (\sup(\Xi_{\theta_2}), \infty)$ and
\begin{equation*}
\sup_{(x, \theta_2) \in [c,\infty) \times \Xi_{\theta_2}} 
\frac{\ee^{-(\theta_2 -\underline{\theta_2})x^2}}{x^{1-2l+2\alpha}}
\lesssim 1
\end{equation*}
for $\underline{\theta_2} \in (0, \inf(\Xi_{\theta_2}))$, $l \in \{ 1, 2, 3 \}$, 
$\alpha \in (0,3)$ and $c \in (0,\infty)$, 
it holds that for $l \in \{1, 2, 3 \}$ and $\alpha \in (0,3)$,
\begin{align}
q_{r}(x;\alpha,\theta_2) 
&= 
\begin{cases}
\dfrac{1-\ee^{-\theta_2 x^2}}{1-\ee^{-\overline{\theta_2} x^2}}
\cdot \dfrac{1-\ee^{-\overline{\theta_2} x^2}}{x^{1+2\alpha}} 
\bigl( J_0(\sqrt{2}r x) -2J_0(r x)+1 \bigr),
\\
\dfrac{1-\ee^{-\theta_2 x^2}}{x^{1+2\alpha}} 
\bigl( J_0(\sqrt{2}r x) -2J_0(r x)+1 \bigr)
\end{cases}
\nonumber
\\
&\lesssim
\begin{cases}
x^{5-2\alpha} & (x \to 0),
\\
x^{-1-2\alpha} & (x \to \infty),
\end{cases}
\label{est_q}
\\
\pd_{\theta_2}^l q_{r}(x;\alpha,\theta_2) 
&= 
\begin{cases}
(-1)^{l+1} \dfrac{\ee^{-\theta_2 x^2}}{x^{1 -2l +2\alpha}} 
\bigl( J_0(\sqrt{2}r x) -2J_0(r x)+1 \bigr),
\\
\ee^{-\underline{\theta_2} x^2} 
\cdot (-1)^{l+1} \dfrac{\ee^{-(\theta_2 -\underline{\theta_2}) x^2}}{x^{1 -2l +2\alpha}} 
\bigl( J_0(\sqrt{2}r x) -2J_0(r x)+1 \bigr)
\end{cases}
\nonumber
\\
&\lesssim
\begin{cases}
x^{3 +2l -2\alpha} & (x \to 0),
\\
\ee^{-\underline{\theta_2} x^2} & (x \to \infty)
\end{cases}
\label{est_pd_the_q}
\end{align}
uniformly in $\theta_2 \in \Xi_{\theta_2}$.
Therefore, we see the function $(0,\infty) \ni x \mapsto \pd_{\theta_2}^l q_{r}(x;\alpha,\theta_2)$ is integrable
for any $l \in \{ 1, 2, 3 \}$, $\alpha \in (0,3)$ and $r \in (0,\infty)$, 
and find from the dominate convergence theorem that 
the function $\Xi_{\theta_2} \ni \theta_2 \mapsto \overline \phi_{r,\alpha}(\theta_2)$ is 
three times differentiable for any $\alpha \in (0,3)$ and $r \in (0,\infty)$.

We next show (b). We have
\begin{align*}
\pd_{\alpha}^k q_{r}(x;\alpha,\theta_2)
&= \frac{1-\ee^{-\theta_2 x^2}}{x^{1+2\alpha}} 
(-2 \log(x))^k \bigl( J_0(\sqrt{2}r x) -2J_0(r x)+1 \bigr)
\\
&= (-2 \log(x))^k q_r(x; \alpha, \theta_2),
\\
\pd_{\alpha}^k \pd_{\theta_2}^l q_{r}(x;\alpha,\theta_2)
&= (-1)^{l+1} \frac{\ee^{-\theta_2 x^2}}{x^{1 -2l +2\alpha}} 
(-2 \log(x))^k \bigl( J_0(\sqrt{2}r x) -2J_0(r x)+1 \bigr)
\\
&= (-2 \log(x))^k \pd_{\theta_2}^l q_r(x; \alpha, \theta_2)
\end{align*}
for $k \in \{1, 2\}$ and $l \in \{1, 2, 3\}$.
Since 
\begin{equation*}
\lim_{x \searrow 0} x^a (\log(x))^k = 0,
\quad
\lim_{x \to \infty} \frac{(\log(x))^k}{x^a} = 0
\end{equation*}
for any $a \in (0, \infty)$ and $k \in \{1, 2 \}$, we obtain
\begin{align}
\pd_{\alpha}^k q_{r}(x;\alpha,\theta_2)
&= 
\begin{cases}
\dfrac{1-\ee^{-\theta_2 x^2}}{x^{1 +2 \overline \alpha}} 
\bigl( J_0(\sqrt{2}r x) -2J_0(r x)+1 \bigr)
\cdot x^{2(\overline \alpha -\alpha)} (-2 \log(x))^k,
\\
\dfrac{1-\ee^{-\theta_2 x^2}}{x^{1 +2 \underline \alpha}} 
\bigl( J_0(\sqrt{2}r x) -2J_0(r x)+1 \bigr)
\cdot \dfrac{(-2 \log(x))^k}{x^{2(\alpha - \underline \alpha)}}  
\end{cases}
\nonumber
\\
&\lesssim
\begin{cases}
q_{r}(x;\overline \alpha,\theta_2) & (x \to 0),
\\
q_{r}(x;\underline \alpha,\theta_2) & (x \to \infty),
\end{cases}
\label{est_pd_alp_q}
\\
\pd_{\alpha}^k \pd_{\theta_2}^l q_{r}(x;\alpha,\theta_2)
&= 
\begin{cases}
(-1)^{l+1} \dfrac{\ee^{-\theta_2 x^2}}{x^{1 -2l +2 \overline \alpha}} 
\bigl( J_0(\sqrt{2}r x) -2J_0(r x)+1 \bigr)
\cdot x^{2(\overline \alpha -\alpha)} (-2 \log(x))^k,
\\
(-1)^{l+1} \dfrac{\ee^{-\theta_2 x^2}}{x^{1 -2l +2 \underline \alpha}} 
\bigl( J_0(\sqrt{2}r x) -2J_0(r x)+1 \bigr)
\cdot \dfrac{(-2 \log(x))^k}{x^{2(\alpha - \underline \alpha)}}  
\end{cases}
\nonumber
\\
&\lesssim
\begin{cases}
\pd_{\theta_2}^l q_{r}(x;\overline \alpha,\theta_2) & (x \to 0),
\\
\pd_{\theta_2}^l q_{r}(x;\underline \alpha,\theta_2) & (x \to \infty),
\end{cases}
\label{est_pd_alp_pd_the_q}
\end{align}
where $\underline \alpha \in (0, \inf(\Xi_{\alpha}))$ 
and $\overline \alpha \in (\sup(\Xi_{\alpha}), 3)$,
which together with \eqref{est_q}, \eqref{est_pd_the_q} and the dominate convergence theorem yield the desired result.

We finally prove (c).
For $(\alpha, \theta_2), (\alpha', \theta_2') \in \Xi_{\alpha} \times \Xi_{\theta_2}$, 
we define $\alpha_u = \alpha' +u(\alpha -\alpha')$ 
and $\theta_{2,u} = \theta_2' +u(\theta_2 -\theta_2')$ for any $u \in [0,1]$.
Note that $\Xi_{\alpha}$ and $\Xi_{\theta_2}$ are convex sets and $(\alpha_u, \theta_{2,u}) \in \Xi_{\alpha} \times \Xi_{\theta_2}$ for any $u \in [0,1]$.
We then obtain by the Taylor expansion and \eqref{est_q}--\eqref{est_pd_alp_pd_the_q},
\begin{align*}
\bigl| 
\pd_{\alpha}^k \pd_{\theta_2}^l \overline \phi_{r,\alpha}(\theta_2)
-\pd_{\alpha}^k \pd_{\theta_2}^l \overline \phi_{r,\alpha'}(\theta_2')
\bigr|
&\le 
\int_0^\infty 
\int_0^1
\bigl|
\pd_{\alpha}^{k+1} \pd_{\theta_2}^l q_{r}(x; \alpha_u, \theta_{2,u}) 
\bigr|
\dd u \dd x
|\alpha - \alpha'|
\\
&\quad+ \int_0^\infty 
\int_0^1
\bigl|
\pd_{\alpha}^k \pd_{\theta_2}^{l+1} q_{r}(x; \alpha_u, \theta_{2,u})
\bigr|
\dd u \dd x
|\theta_2 -\theta_2'|
\\
&\le 
\int_0^\infty 
\sup_{(\alpha, \theta_2) \in \Xi_{\alpha} \times \Xi_{\theta_2}} 
\bigl|
\pd_{\alpha}^{k+1} \pd_{\theta_2}^l q_{r}(x; \alpha, \theta_2) 
\bigr|
\dd x
|\alpha - \alpha'|
\\
&\quad+ \int_0^\infty 
\sup_{(\alpha, \theta_2) \in \Xi_{\alpha} \times \Xi_{\theta_2}} 
\bigl| \pd_{\alpha}^k \pd_{\theta_2}^{l+1} q_{r}(x; \alpha, \theta_2) \bigr|
\dd x
|\theta_2 -\theta_2'|
\\
&\lesssim
|\alpha - \alpha'| +|\theta_2 -\theta_2'|
\end{align*}
for $k \in \{0, 1\}$ and $l \in \{0,1,2\}$, which proves 
the function $\Xi_{\alpha} \times \Xi_{\theta_2} \ni (\alpha, \theta_2) 
\mapsto \pd_{\alpha}^k \pd_{\theta_2}^l \overline \phi_{r,\alpha}(\theta_2)$ is continuous.
\end{proof}

We establish the following lemma on the Ornstein-Uhlenbeck process defined by \eqref{OU-process} in order to prove Lemma \ref{lem_est_AB2}.
\begin{subthm}\label{lem_est_x}
For the Ornstein-Uhlenbeck process $\{ x_l(t) \}_{t \in [0,1]}$, $l = (l_1,l_2)$
defined by \eqref{OU-process}, it follows that under [A1]$_{\alpha_0}$,
\begin{equation*}
\EE[ x_l(t) ] \lesssim \frac{1}{\lambda_l^{(1+\alpha_0)/2}},
\quad
\EE[ x_l(t)^2 ] 
\lesssim \frac{1}{\lambda_l^{1+\alpha_0}} + \frac{\epsilon^2}{\lambda_l^{1+\alpha}},
\quad
\VV[ x_l(t) ] \lesssim \frac{\epsilon^2}{\lambda_l^{1+\alpha}}
\end{equation*}
for any $t \in [0,1]$.
It also holds that for any sequence $\{ a_l \}_{l \in \mathbb N^2}$, 
\begin{equation*}
\biggl( \sum_{l \in \mathbb N^2} a_l \EE [x_l(t)] \biggr)^2
\lesssim \sum_{l \in \mathbb N^2} \frac{a_l^2}{\lambda_l^{1+\alpha_0}}.
\end{equation*}
\end{subthm}
\begin{proof}
Notice that
\begin{equation*}
x_{l}(t) = \ee^{-\lambda_l t} x_l(0) 
+ Z_l^{\epsilon}(t),
\quad
Z_l^{\epsilon}(t)
= \epsilon \mu_l^{-\alpha/2} \int_0^t \ee^{-\lambda_l (t-s)} \dd w_l(s).
\end{equation*}
Since it follows from [A1]$_{\alpha_0}$ that
\begin{equation*}
\lambda_l^{1+\alpha_0} x_l(0)^2 
\le \sum_{l' \in \mathbb N^2} \lambda_{l'}^{1+\alpha_0} x_{l'}(0)^2
= \| A_{\theta}^{(1+\alpha_0)/2} X_0 \|^2 
< \infty
\end{equation*}
and 
\begin{equation*}
\EE [ Z_l^{\epsilon}(t)^2 ] 
= \frac{\epsilon^2 (1 -\ee^{-2\lambda_l t})}{2 \lambda_l \mu_l^{\alpha}}
\lesssim \frac{\epsilon^2}{\lambda_l^{1+\alpha}}
\end{equation*}
for any $l \in \mathbb N^2$ and $t \in [0,\infty)$, we have
\begin{align*}
\EE[ x_l(t) ] &= \ee^{-\lambda_l t} x_l(0) 
\lesssim \frac{1}{\lambda_l^{(1+\alpha_0)/2}},
\\
\EE[ x_l(t)^2 ] &= \ee^{-2 \lambda_l t} x_l(0)^2
+ \EE [ Z_l^{\epsilon}(t)^2 ]
\lesssim 
\frac{1}{\lambda_l^{1+\alpha_0}} + \frac{\epsilon^2}{\lambda_l^{1+\alpha}},
\\
\VV[ x_l(t) ] &= \EE [ Z_l^{\epsilon}(t)^2 ]
\lesssim \frac{\epsilon^2}{\lambda_l^{1+\alpha}}
\end{align*}
for any $t \in [0,1]$.
The latter statement follows from the Schwarz inequality and [A1]$_{\alpha_0}$.
\begin{align*}
\biggl( \sum_{l \in \mathbb N^2} a_l \EE[ x_l(t) ] \biggr)^2
&= \biggl( \sum_{l \in \mathbb N^2} a_l \ee^{-\lambda_l t} x_l(0) \biggr)^2
\\
&= \biggl( \sum_{l \in \mathbb N^2} 
\frac{a_l \ee^{-\lambda_l t}}{\lambda_l^{(1+\alpha_0)/2}} 
\cdot \lambda_l^{(1+\alpha_0)/2} x_l(0) \biggr)^2
\\
&\le
\sum_{l \in \mathbb N^2} 
\frac{a_l^2 \ee^{-2\lambda_l t}}{\lambda_l^{1+\alpha_0}} 
\sum_{l \in \mathbb N^2} \lambda_l^{1+\alpha_0} x_l(0)^2
\\
&\le
\| A_\theta^{(1+\alpha_0)/2} X_0 \|^2 
\sum_{l \in \mathbb N^2} \frac{a_l^2}{\lambda_l^{1+\alpha_0}}.
\end{align*}
\end{proof}


\begin{thebibliography}{99}
\bibitem{Altmeyer_Reiss2021}
Altmeyer, R. and Reiss, M. (2021). 
\newblock Nonparametric estimation for linear SPDEs from local measurements. 
\newblock {\em The Annals of Applied Probability}, 31(1), 1--38.


\bibitem{Andersson_etal2025}
Andersson, M., Avelin, B., Garino, V., Ilmonen, P., and Viitasaari, L. (2025).
\newblock Non-parametric estimation of non-linear diffusion coefficient in parabolic SPDEs.
\newblock{\em arXiv preprint arXiv:2509.12921}.


\bibitem{Avetisian_Ralchenko2020}
Avetisian, D. and Ralchenko, K. (2020).
\newblock Ergodic properties of the solution to a fractional stochastic heat equation, with an application to diffusion parameter estimation. 
\newblock {\em Modern Stochastics: Theory and Applications}, 7(3), 339--356.


\bibitem{Baltazar-Larios_etal2024}
Baltazar-Larios, F.,  Delgado-Vences, F., and Peralta, L. (2024). 
\newblock Statistical inference for a stochastic partial differential equation related to an ecological niche.
\newblock {\em Mathematical Methods in the Applied Sciences}, 47(18), 13672--13689. 


\bibitem{Bibinger_Bossert2023}
Bibinger, M. and Bossert, P. (2023).
\newblock Efficient parameter estimation for parabolic SPDEs based on a log-linear model for realized volatilities. 
\newblock {\em Japanese Journal of Statistics and Data Science}, 6(1), 407--429.


\bibitem{Bibinger_Trabs2020}
Bibinger, M. and Trabs, M. (2020).
\newblock Volatility estimation for stochastic PDEs using high-frequency observations.
\newblock {\em Stochastic Processes and their Applications}, 130(5), 3005--3052.


\bibitem{Bossert2024}
Bossert, P. (2024).
\newblock Parameter estimation for second-order SPDEs in multiple space dimensions.
\newblock{\em Statistical Inference for Stochastic Processes}, 27(3), 485--583.


\bibitem{Chong2020}
Chong, C. (2020).
\newblock High-frequency analysis of parabolic stochastic pdes. 
\newblock{\em The Annals of Statistics}, 48(2), 1143--1167.


\bibitem{Cialenco2018}
Cialenco, I. (2018).
\newblock Statistical inference for SPDEs: an overview.
\newblock {\em Statistical Inference for Stochastic Processes}, 21(2), 309--329.


\bibitem{Cialenco_Huang2020}
Cialenco, I. and Huang, Y. (2020). 
\newblock A note on parameter estimation for discretely sampled spdes. 
\newblock{\em Stochastics and Dynamics}, 20(3), 2050016.


\bibitem{Cialenco_Kim2022}
Cialenco, I. and Kim, H.-J. (2022).
\newblock Parameter estimation for discretely sampled stochastic heat equation driven by space-only noise. 
\newblock {\em Stochastic Processes and their Applications}, 143, 1--30. 


\bibitem{DaPrato_Zabczyk2014}
Da Prato, G. and Zabczyk, J. (2014). 
\newblock {\em Stochastic Equations in Infinite Dimensions}, 2nd edition. 
\newblock Cambridge University Press.


\bibitem{Gamain_Tudor2023}
Gamain, J. and Tudor, C.A. (2023).
\newblock Exact variation and drift parameter estimation for the nonlinear fractional stochastic heat equation. 
\newblock {\em Japanese Journal of Statistics and Data Science}, 6(1), 381--406.


\bibitem{Gaudlitz_Reiss2023}
Gaudlitz, S. and Reiss, M. (2023).
\newblock Estimation for the reaction term in semi-linear SPDEs under small diffusivity.
\newblock {\em Bernoulli} 29(4), 3033--3058.


\bibitem{Gloter_Sorensen2009}
Gloter, A. and S{\o}rensen, M. (2009).
\newblock Estimation for stochastic differential equations with a small diffusion coefficient.
\newblock {\em Stochastic Processes and their Applications}, 119(3), 679--699.


\bibitem{Guy_etal2014}
Guy, R., Laredo, C., and Vergu, E. (2014).
\newblock Parametric inference for discretely observed multidimensional diffusions with small diffusion coefficient.
\newblock {\em Stochastic Processes and their Applications}, 124(1), 51--80.


\bibitem{Hildebrandt_Trabs2021}
Hildebrandt, F. and Trabs, M. (2021).
\newblock Parameter estimation for SPDEs based on discrete observations in time and space.
\newblock {\em Electronic Journal of Statistics}, 15(1), 2716--2776.


\bibitem{Hildebrandt_Trabs2023}
Hildebrandt, F. and Trabs, M. (2023). 
\newblock Nonparametric calibration for stochastic reaction-diffusion equations based on discrete observations. 
\newblock {\em Stochastic Processes and their Applications}, 162, 171--217. 


\bibitem{Hubner_etal1993}
H{\"u}bner, M., Khasminskii, R., and Rozovskii, B.L. (1993).
\newblock {\em Two Examples of Parameter Estimation for Stochastic Partial Differential Equations}, 149--160.
\newblock Springer New York.


\bibitem{Huebner_Rozovskii1995}
Huebner, M. and Rozovskii, B.L. (1995).
\newblock On asymptotic properties of maximum likelihood estimators for parabolic stochastic PDE's.
\newblock {\em Probability Theory and Related Fields}, 103(2), 143--163.


\bibitem{Janak_Reiss2024}
Jan\'{a}k, J. and Reiss, M. (2024). 
\newblock Parameter estimation for the stochastic heat equation with multiplicative noise from local measurements.
\newblock {\em Stochastic Processes and their Applications}, 175, 104385.


\bibitem{Jones_Zhang1997}
Jones, R.H. and Zhang, Y. (1997).
\newblock Models for Continuous Stationary Space-Time Processes. In {\em Modelling Longitudinal and Spatially Correlated Data}. 
\newblock Lecture Notes in Statistics, vol 122. Springer, New York.


\bibitem{Kaino_Uchida2018}
Kaino, Y. and Uchida, M. (2018).
\newblock Hybrid estimators for small diffusion processes based on reduced data.
\newblock {\em Metrika}, 81(7), 745--773.


\bibitem{Kaino_Uchida2021a}
Kaino, Y. and Uchida, M. (2021).
\newblock Parametric estimation for a parabolic linear SPDE model based on discrete observations.
\newblock {\em Journal of Statistical Planning and Inference}, 211, 190--220.


\bibitem{Kaino_Uchida2021b}
Kaino, Y. and Uchida, M. (2021).
\newblock Adaptive estimator for a parabolic linear SPDE with a small noise.
\newblock {\em Japanese Journal of Statistics and Data Science}, 4(1), 513--541.


\bibitem{Lototsky2003}
Lototsky, S.V. (2003).
\newblock Parameter estimation for stochastic parabolic equations: asymptotic properties of a two-dimensional projection-based estimator.
\newblock{\em Statistical Inference for Stochastic Processes}, 6(1), 65--87.


\bibitem{Lototsky2009}
Lototsky, S.V. (2009).
\newblock Statistical inference for stochastic parabolic equations: a spectral approach.
\newblock {\em Publicacions Matem\`{a}tiques}, 53(1), 3--45.


\bibitem{Lototsky_Rozovsky2017}
Lototsky, S.V. and Rozovsky, B.L. (2017).
\newblock {\em Stochastic Partial Differential Equations}.
\newblock Springer.


\bibitem{MahdiKhalil_Tudor2019}
Mahdi Khalil, Z. and Tudor, C. (2019). 
\newblock Estimation of the drift parameter for the fractional stochastic heat equation via power variation. 
\newblock {\em Modern Stochastics: Theory and Applications}, 6(4), 397--417.


\bibitem{Markussen2003}
Markussen, B. (2003).
\newblock Likelihood inference for a discretely observed stochastic partial differential equation. 
\newblock {\em Bernoulli}, 9(5), 745--762.


\bibitem{Mohapl2000}
Mohapl, J. (2000).
\newblock A Stochastic Advection-Diffusion Model for the Rocky Flats Soil Plutonium Data. 
\newblock {\em Annals of the Institute of Statistical Mathematics}, 52(1), 84--107.


\bibitem{North_etal2011}
North, G.R., Wang, J., and Genton, M.G. (2011).
\newblock Correlation models for temperature fields. 
\newblock {\em Journal of climate}, 24(22), 5850--5862.


\bibitem{Piterbarg_Ostrovskii1997}
Piterbarg, L. and Ostrovskii, A. (1997).
\newblock {\em Advection and diffusion in random media: implications for sea surface temperature anomalies}.
\newblock Springer Science \& Business Media.


\bibitem{Piterbarg_Rozovskii1997}
Piterbarg, L. and Rozovskii, B. (1997).
\newblock On asymptotic problems of parameter estimation in stochastic PDE's: the case of discrete time sampling. 
\newblock {\em Mathematical Methods of Statistics}, 6(2), 200--223. 


\bibitem{Sorensen_Uchida2003}
S{\o}rensen, M. and Uchida, M. (2003).
\newblock Small-diffusion asymptotics for discretely sampled stochastic differential equations.
\newblock {\em Bernoulli}, 9(6), 1051--1069.


\bibitem{Strauch_Tiepner2024}
Strauch, C. and Tiepner, A. (2024).
\newblock Nonparametric velocity estimation in stochastic convection-diffusion equations from multiple local measurements.
\newblock {\em arXiv preprint arXiv:2402.08353}.


\bibitem{TKU2023}
Tonaki, Y., Kaino, Y., and Uchida, M. (2023).
\newblock Parameter estimation for linear parabolic SPDEs in two space dimensions based on high frequency data.
\newblock {\em Scandinavian Journal of Statistics}, 50(4), 1568--1589.


\bibitem{TKU2024a}
Tonaki, Y., Kaino, Y., and Uchida, M. (2024).
\newblock Parameter estimation for a linear parabolic SPDE model in two space dimensions with a small noise.
\newblock {\em Statistical Inference for Stochastic Processes}, 27(1), 123--179.


\bibitem{TKU2025a}
Tonaki, Y., Kaino, Y., and Uchida, M. (2025).
\newblock Parametric estimation for linear parabolic SPDEs in two space dimensions based on temporal and spatial increments.
\newblock {\em Metrika}, 88(5), 601--656. 


\bibitem{TKU2025b}
Tonaki, Y., Kaino, Y., and Uchida, M. (2025).
\newblock Small diffusivity asymptotics for a linear parabolic SPDE in two space dimensions.
\newblock{\em Statistical Inference for Stochastic Processes}, 28(2), 11.


\bibitem{TKU2025arXiv2}
Tonaki, Y., Kaino, Y., and Uchida, M. (2025).
\newblock Volatility change point detection for linear parabolic SPDEs.
\newblock{\em arXiv preprint arXiv:2512.01277}.


\bibitem{TKU2026}
Tonaki, Y., Kaino, Y., and Uchida, M. (2026).
\newblock Small dispersion asymptotics for an SPDE in two space dimensions using triple increments.
\newblock{\em Journal of Statistical Planning and Inference}, 241, 106333.


\bibitem{TKU2025arXiv1}
Tonaki, Y., Kaino, Y., and Uchida, M. (2026).
\newblock Estimation for linear parabolic SPDEs in two space dimensions with unknown damping parameters.
\newblock {\em Japanese Journal of Statistics and Data Science}.
\texttt{https://doi.org/10.1007/s42081-026-00348-y}


\bibitem{Tudor2022}
Tudor, C.A. (2022). 
\newblock {\em Stochastic Partial Differential Equations with Additive Gaussian Noise}.
\newblock World Scientific.


\bibitem{Uchida2004}
Uchida, M. (2004).
\newblock Estimation for discretely observed small diffusions based on approximate martingale estimating functions.
\newblock {\em Scandinavian Journal of Statistics}, 31(4), 553--566.

\end{thebibliography}
\end{document}